\documentclass[11pt, a4paper, oneside]{Thesis} 
\usepackage{wrapfig}
\usepackage{lscape}
\usepackage{rotating}
\usepackage{graphicx}
\usepackage{caption}
\usepackage{subcaption}
\usepackage{amscd,amssymb,amsmath,amsthm}
\usepackage{tikz}
\usepackage{tikz-network}
\usepackage[utf8]{inputenc}
\usepackage[english]{babel}
\usepackage{dsfont}
\usepackage{nccmath}
\usepackage{pdfpages}
\allowdisplaybreaks

\theoremstyle{plain}
\newtheorem*{theorem*}{Theorem}
\newtheorem*{cor*}{Corollary}

\usepackage{mathtools}

\usepackage{afterpage}

\newcommand\blankpage{%
	\null
	\thispagestyle{empty}%
	\addtocounter{page}{-1}%
	\newpage}

\def\addsymbol #1: #2#3{$#1$ \> \parbox{5in}{#2 \dotfill \pageref{#3}}\\}

\graphicspath{{Pictures/}} 
\usepackage[square, numbers]{natbib} 
\usepackage{hyperref}
\hypersetup{
    colorlinks=true,
    linkcolor=blue,
    filecolor=darkmagenta,      
    urlcolor=black,
    citecolor=red
} 
\makeatletter
\tikzset{
    dot diameter/.store in=\dot@diameter,
    dot diameter=2pt,
    dot spacing/.store in=\dot@spacing,
    dot spacing=13pt,
    dots/.style={
        line width=\dot@diameter,
        line cap=round,
        dash pattern=on 0pt off \dot@spacing
    }
}
\makeatother

\makeatletter
\def\BState{\State\hskip-\ALG@thistlm}
\makeatother
\numberwithin{equation}{section}

\usepackage{amssymb}
\usepackage[mathscr]{eucal}
\usepackage{xypic}
\usepackage{amsmath}
\usepackage[all]{xy}
\usepackage{lineno}
\usepackage{epsfig}
\usepackage{amscd}
\usepackage{mathrsfs}
\usepackage{enumerate}

\usepackage{hyperref}
\usepackage[nameinlink]{cleveref}
\usepackage[thinc]{esdiff}
\usepackage{tikz-cd}
\usepackage[normalem]{ulem}

\usepackage{xcolor}

\theoremstyle{definition}

\newtheorem{lemma}[theorem]{Lemma}
\newtheorem{proposition}[theorem]{Proposition}
\newtheorem{corollary}[theorem]{Corollary}

\theoremstyle{definition}
\newtheorem{definition}[theorem]{Definition}

\newtheorem{example}[theorem]{Example}

\newtheorem{remark}[theorem]{Remark}

\numberwithin{equation}{section}

\newcommand{\mc}{\mathcal}
\newcommand{\mb}{\mathbb}

\newcommand{\xra}{\xrightarrow}
\newcommand{\ra}{\rightarrow}
\newcommand{\rra}{\rightrightarrows}
\newcommand{\ghta}{(G,H,\tau, \alpha)}

\DeclareMathOperator{\pr}{pr}

\def\og{\leavevmode\raise.3ex\hbox{$\scriptscriptstyle\langle\!\langle$~}}
\def\fg{\leavevmode\raise.3ex\hbox{~$\!\scriptscriptstyle\,\rangle\!\rangle$}}

\title{\ttitle} 

\begin{document}
\makeatletter
\renewcommand*{\NAT@nmfmt}[1]{\textsc{#1}}
\makeatother


\frontmatter 

\setstretch{1.6} 

\fancyhead{} 
\rhead{\thepage} 
\lhead{} 

\pagestyle{fancy} 

\newcommand{\HRule}{\rule{\linewidth}{0.5mm}} 

\hypersetup{pdftitle={\ttitle}}
\hypersetup{pdfsubject=\subjectname}
\hypersetup{pdfauthor=\authornames}
\hypersetup{pdfkeywords=\keywordnames}


\begin{titlepage}
%
%
%
%
%
%
%
%
%
\thispagestyle{empty}
\begin{center}{
		\textbf{\LARGE{On gauge theory and parallel transport in principal 2-bundles over Lie groupoids}}
	}
\end{center}{ \par}

\vspace*{\fill}

\bigskip{}
\begin{center}{
		{\it {\large A thesis submitted for the degree of}}\\
		\vspace {0.5cm}
		{{\Large\bf Doctor of Philosophy}}\\
		\vspace {0.5cm}
		{\it {\large in}}\\
		\vspace {0.5cm} {\bf {\Large Mathematics}}}
\end{center}
\vspace*{\fill}
\begin{center}
\end{center}
\vspace*{\fill}
\begin{center}{\textbf{\Large {\bf Adittya Chaudhuri}}}\\
	IPHD15002
\end{center}{\Large \par}

\vspace*{\fill}

\begin{center}
	\includegraphics[width=0.3\linewidth]{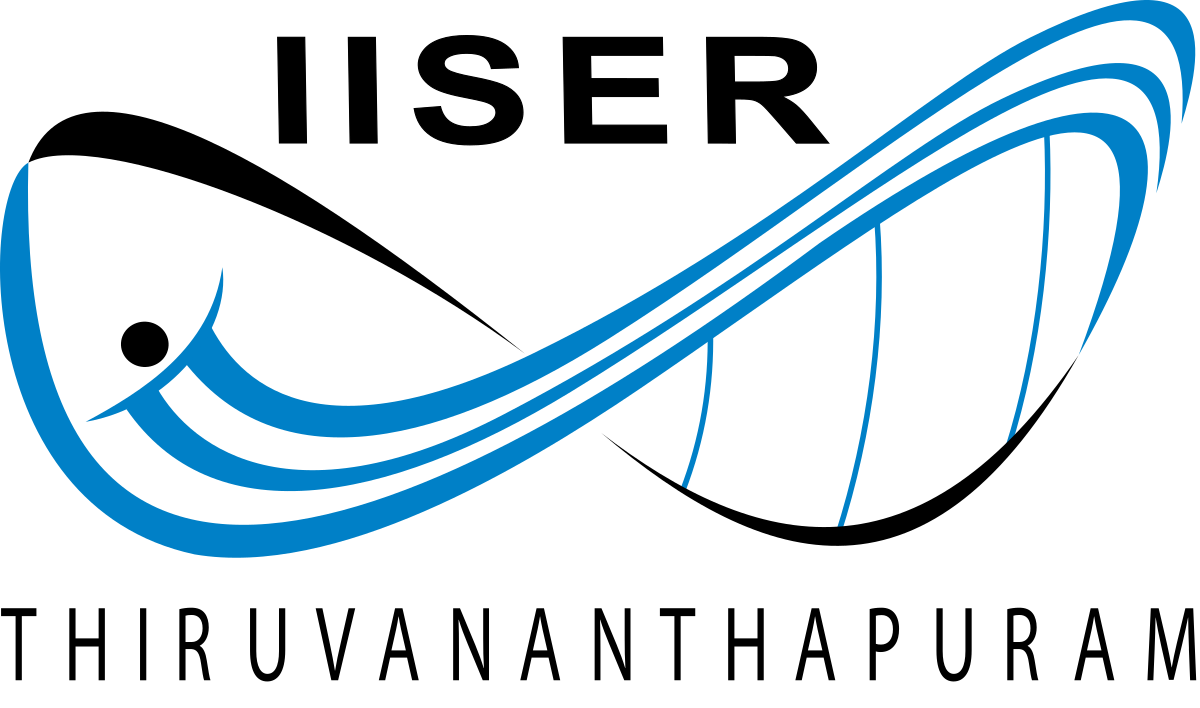}\\
	{\small SCHOOL OF MATHEMATICS}\\
	{\small INDIAN INSTITUTE OF SCIENCE EDUCATION AND RESEARCH}\\
	{\small THIRUVANANTHAPURAM - 695551}\\
	{\small INDIA}\end{center}{\small
	\par}

\begin{center}{\small November 2023}\end{center}{\small \par}
\clearpage

\end{titlepage}

\newpage

\afterpage{\blankpage}

\newpage

\clearpage 

\afterpage{\blankpage}
\includepdf[pages=-]{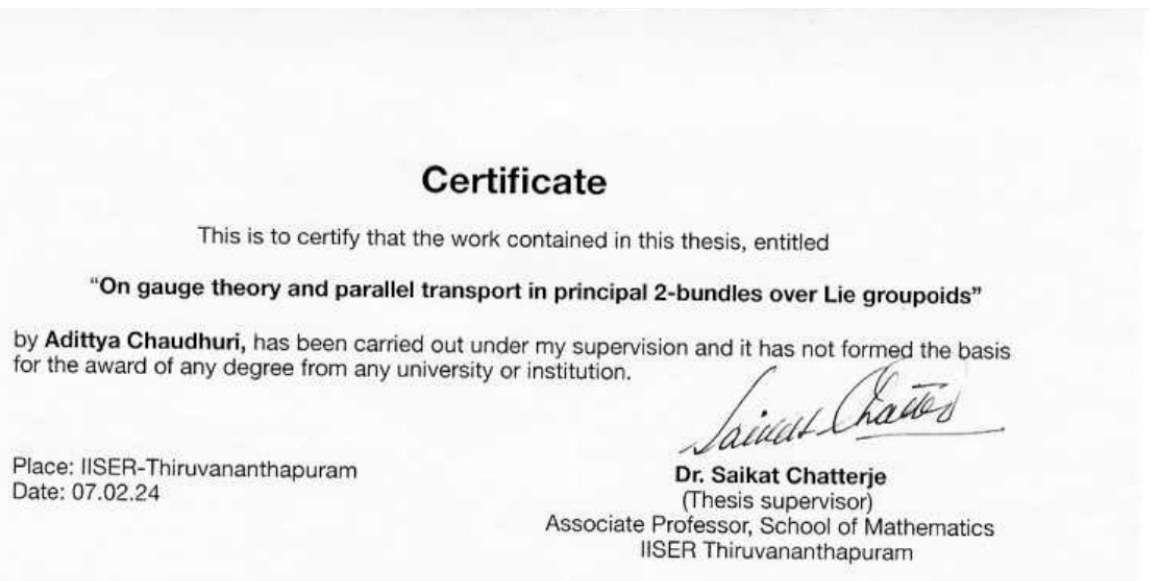}
\clearpage
\includepdf[pages=-]{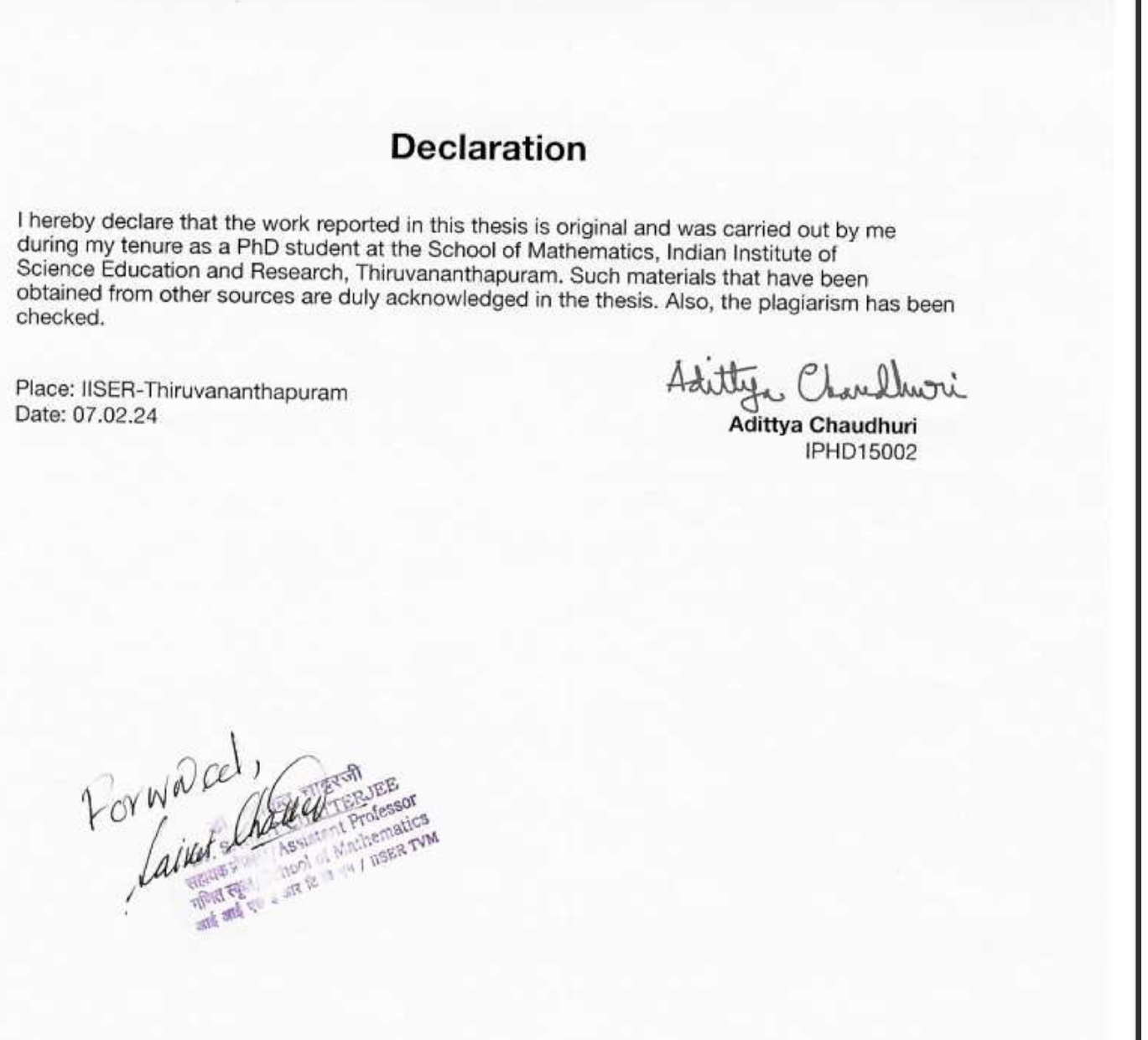}
\thispagestyle{empty}
\begin{figure}
	\begin{center}
		\LARGE{\textit{Dedicated to all who inspired me directly or indirectly to do Mathematics}}
	\end{center}
\end{figure}
\clearpage
\newpage
\thispagestyle{empty}
\bigskip{}
\vspace{1cm} \cleardoublepage


\afterpage{\blankpage}

\clearpage 


%
%
%


\setstretch{1.3} 

\acknowledgements{\addtocontents{toc}{} 

Completing my Ph.D. thesis work would not have been possible without numerous individuals' invaluable and substantial contributions.

 I express my heartfelt gratitude to Dr. Saikat Chatterjee, my Ph.D. supervisor, for his unwavering support and guidance and for providing unconditional freedom in my working style throughout my doctoral journey. 
 
 I thank my Doctoral committee members, Dr. Sarbeswar Pal and Dr. Geetha Thangavelu, for their valuable suggestions and support. 

 Every word falls short of describing the contributions of my family—Baba, Ma, Dadu, and late Didun. My doctoral years witnessed the horror of Covid-19, which shook the entire planet for around two years. Consequently, I completed a significant part of my Ph.D. thesis work at home in Kolkata. Without my family members' unconditional support and love, my thesis work would never have reached the current stage. I sincerely thank them from the bottom of my heart for their roles in every aspect of my life. 

 My ability to appreciate the inherent beauty in abstract Mathematical structures started developing while I was doing my undergraduate studies at the Institute of Mathematics and Applications (IMA), Bhubaneswar. In this regard, I sincerely thank Prof. Swadhin Pattanayak.

 My perspectives on Geometry and Topology underwent significant transformation during the workshop \textit{Characteristic Classes and Cobordism} (4th March to 15th March 2019) organized by the National Centre for Mathematics at the Indian Institute of Technology Bombay (IIT-Bombay), Mumbai, India. In this regard, I want to convey my utmost gratitude to Prof. M. S. Raghunathan, Prof. Nitin Nitsure, and Prof. Anant Shastri for their thought-provoking interactions.

My views on Higher category theory took a giant stride during my visit to the International Centre for Theoretical Sciences, a centre of the Tata Institute of Fundamental Research (I.C.T.S), Bangalore, India (from 6th to 10th January 2020). I am thankful to Dr. Rukmini Dey for arranging the Academic visit. I want to convey my profound gratitude to  Dr. Pranav Pandit for his time to discuss $\infty$-categories during my stay at I.C.T.S. Also, I am indebted to him for allowing me to participate actively in the online \textit{weekly $\infty$-category theory seminar series} he organized during the Covid-19 lockdown period. This seminar series enriched my entire perspective on Higher category theory.

Prof. John Baez plays a pivotal role in shaping my entire research outlook. I am deeply indebted to him for his time, invaluable conversations, and suggestions regarding my research during the whole phase of my Ph.D. life.

I sincerely thank Dr. David Michael Roberts and Dr. Dmitri Pavlov for their invaluable discussions on Higher differential geometry throughout my doctoral period.

I thank Dr. Rajan Mehta for his comments on my research and for inspiring my research work.

I am grateful to Prof. Siddhartha Gadgil for a two-hour-long thought-provoking discussion during his visit to IISER-TVM on 1st November 2019. This discussion shaped my research views significantly.

I want to thank Emilio Minichiello, Chetan Vuppulury, Dr. Praphulla Koushik, Srikanth Pai, Visakh Narayanan, Dr. Arun Debray, and Naga Arjun S J  for enriching my Mathematical knowledge during my Ph.D. tenure.

I convey my sincere gratitude to the MathOverflow community, CATEGORY THEORY-Zulip community, Math Twitter community, nLab, Algebraic Topology (Discord) community, Homotopy Theory (Discord) community and the Topological Quantum Field Theory Club for their invaluable help in expanding my knowledge in Mathematics.

I am thankful to the anonymous referees of my published paper for their valuable suggestions for improving some of the results in this thesis.

I thank Dr. Viji Z. Thomas, Dr. Bindusar Sahoo, Dr. Sarbeswar Pal, and Dr. Sreedhar B Dutta for their insightful discussions.

I am grateful to some people who inspired me to look at greater meanings in life. They include Anup Sen and Chandan Da.

I express my profound gratitude to Mr. Akhil Philip for giving me professional-level personal training in Bodybuilding. Three features of the sport- Consistency, self-control, and getting a sense of achievement from partial progress- helped me incorporate the same in my Ph.D. research.

IISER-TVM has given me a chance to meet some humans whose presence made me feel special. They include Akash bhayia, Tania di, Vatsal, Debopriya di, Ammu, Saurav bhayia, Sharvari, Atmaram, Chandan da, Soham, Atreyi, Atridev, Praphulla, Mahendranath, Rajatava, Adil, Anik, Krishanu, Aparajita, Shakul, Manan, Monami, Tuhin, Soumya, Rajiv, Keshav, Saviour, Ananthakrishna, Joyentanuj, Debojyoti, Pankaj, Abinash, Tathagata, Tahir, Subham, Ambika, Reena, Priyanka, Diksha, Naga Arjun, Abhilash Bhayia, Sabir Bhayia (including Tasty's staff), Thanal's (Uncle and aunty) and Partha da.

I also feel lucky that my campus is near the beach town of Varkala. Many of the crucial ideas in my thesis work originated during my frequent visits to Varkala.

 I want to convey my gratitude to the entire fraternity of IISER-TVM, particularly the School of Mathematics, for the infrastructure and financial support and for allowing me to pursue my doctoral research in one of the most scenic campuses in the country.

Last but not least, Ernest Hemingway's \textit{The Old Man and the Sea} and \textit{A Moveable Feast} instrumentally motivated me during some of the darkest times in my Ph.D. research.

}


%

\afterpage{\blankpage}

\clearpage


\lhead{\emph{Contents}} 
\tableofcontents 

\clearpage

\newpage
\pagestyle{fancy}

\chapter*{Abstract}
\lhead{\emph{Abstract}}
\addcontentsline{toc}{chapter}{Abstract}
We investigate an interplay between some ideas in traditional gauge theory and certain concepts in fibered categories. We accomplish the same by introducing a notion of a principal Lie 2-group bundle over a Lie groupoid and studying its connection structures, gauge transformations, and parallel transport. 

We state and prove a Lie 2-group torsor version of the one-one correspondence between fibered categories and pseudofunctors. This results in a classification of our principal 2-bundles based on their underlying fibration structures. Furthermore, this one-one association leads us to propose a weaker version of the principal Lie group bundle over a Lie groupoid, whose underlying action of the base Lie groupoid on the total space is replaced by a quasi-action. Also, as a consequence of this correspondence, we extend a particular class of our principal 2-bundles to be defined over differentiable stacks presented by the base Lie groupoids.

We construct a short exact sequence of VB-groupoids, namely, the \textit{Atiyah sequence} associated to our principal 2-bundles. As a splitting of the Atiyah sequence and a splitting up to a natural isomorphism, we obtain two notions of connection structures, viz. strict connections and semi-strict connections, respectively, on a principal $2$-bundle over a Lie groupoid. We describe strict and semi-strict connections in terms of Lie 2-algebra valued 1-forms on the total Lie groupoids. The underlying fibration structure of the 2-bundle provides an existence criterion for strict and semi-strict connections. We study the action of the $2$-group of gauge transformations on the groupoid of strict and semi-strict connections, and interestingly, we observe an extended symmetry of semi-strict connections.

We demonstrate an interrelationship between `differential geometric connection-induced horizontal path lifting property in traditional principal bundles' and the `category theoretic cartesian lifting of morphisms in fibered categories'. We illustrate this relationship by developing a theory of connection-induced parallel transport in our framework of principal 2-bundles. In particular, we introduce a notion of parallel transport on a principal 2-bundle along a particular class of Haefliger paths in the base Lie groupoid.  We show that the corresponding parallel transport functor enjoys certain smoothness properties and behaves naturally with fibered products and connection preserving bundle morphisms. Finally, we employ our results to introduce a notion of parallel transport along Haefliger paths in the setup of VB-groupoids.

\clearpage

\afterpage{\blankpage}

%


\mainmatter 

\pagestyle{fancy} 



\chapter{Introduction and overview} 


\lhead{Chapter 1. \emph{Introduction and overview}} %

Principal bundles and their connection structures play a cardinal role at the interface of geometry and physics. In particular, a connection structure on a principal bundle over a manifold describes the dynamics of a particle. Over the last few decades, higher gauge theories have been developed as frameworks to describe the dynamics of string-like extended `higher dimensional objects'. Higher gauge theory generally entails using appropriately categorified versions of `spaces.' For instance, instead of a manifold, one might consider a Lie groupoid or a category internal to smooth spaces. Similarly, a Lie group could be substituted with a Lie 2-group, a smooth map with a smooth functor, and so forth. This process involves adopting a suitable notion of connection that aligns with this categorification.
Typically, they consist of categorified versions of smooth fiber bundles with connection structures consistent with the categorification. Such suitably categorified connection structures are expected to induce a notion of parallel transport. However, the precise description of these categorified objects can vary significantly depending on the specific framework or context in which they are employed.

The goal of this thesis is to investigate some geometric relationships between gauge theory and the theory of fibrations/fibered categories, by studying 
\begin{enumerate}[(i)]
	\item the differential geometric connection structures,
	\item the gauge transformations,
	\item the action of gauge transformations on connections and
	\item the parallel transport
\end{enumerate}
on a  categorified principal bundle that lives in the world of Lie groupoids. More precisely, our categorified principal bundle is a groupoid object in the category of principal bundles. Explicitly, it consists of a morphism of Lie groupoids $\pi:=(\pi_1, \pi_0) \colon [E_1 \rra E_0] \ra [X_1 \rra X_0]$  equipped with a functorial action of a Lie 2-group $\mb{G}:=[G_1 \rra G_0]$ on $\mb{E}:=[E_1 \rra E_0]$ such that both $\pi_1 \colon E_1 \ra X_1$ and $\pi_0 \colon E_0 \ra X_0$ are classical principal $G_1$-bundle and principal $G_0$-bundle over $X_1$ and $X_0$ respectively.

The main results of this thesis mostly revolve around the following five topics:
\begin{enumerate}[(i)]
	\item Studying the underlying fibration structure on our principal Lie 2-group bundles over Lie groupoids that result in
	\begin{itemize}
		\item[(a)] a statement and the proof of a Lie 2-group torsor version of the well-known one-one correspondence (due to Grothendieck) between fibered categories and pseudofunctors;
		\item[(b)] Extending a subclass of our principal 2-bundles to be defined over the differentiable stack represented by the base Lie groupoid.
	\end{itemize}
	\item Studying differential geometric connection structures on our principal 2-bundles as
	\begin{itemize}
		\item[(a)] certain splittings on an associated short exact sequence of VB-groupoids and also as
		\item[(b)] certain Lie 2-algebra valued differential 1-forms on the total Lie groupoid of the Lie 2-group bundle.
	\end{itemize}
	\item Studying gauge transformations on these bundles and investigating their actions on our connection structures.
	\item Studying interrelations between the differential geometric connection structures and the underlying fibration structures on our principal Lie 2-group bundle over a Lie groupoid by developing a notion of parallel transport along certain classes of Haefliger paths in the base Lie groupoid of our 2-bundle.
	\item Finally, checking the sanity of our theory by deriving a notion of parallel transport along such Haefliger paths in certain VB-groupoids associated to our principal 2-bundles.
\end{enumerate}
Before explaining our results, we overview some existing works in higher gauge theory. The list is, of course, far from complete. 

Baez\cite{baez2002higher}, Baez-Schreiber\cite{MR2342821, baez2004higher}, Mackaay-Picken \cite{MR1932333}, Bartels \cite{MR2709030}, Baez-Crans \cite{MR2068522}, Baez-Lauda \cite{MR2068521},  Picken-Martins\cite{MR2661492},  Schreiber-Waldorf's \cite{MR2520993, MR2803871, MR3084724} along with some other papers cited therein comprises some of the earliest works in this area. 

In a more recent time, Wockel in \cite{MR2805195}  introduced a notion of semi-strict principal $2$-bundle over a discrete smooth $2$-space and studies their classification up to Morita equivalence using non-abelian C\v ech cohomology. Schommer-Pries further generalizes Wockel's framework of principal 2-bundles in \cite{MR2800361}. Other approaches to principal $2$-bundles include $G$-gerbes \cite{MR2493616} of Laurent-Gengoux, Sti\'{e}non, and Xu and non-abelian bundle gerbes of Aschieri, Cantini, and Jur\v co \cite{MR2117631}.	\cite{MR3089401} studies the relation between non-abelian bundle gerbes of \cite{MR2117631} and principal $2$-bundles over a smooth manifold. In \cite{MR3480061},  Ginot and Sti\'{e}non introduced the notion of a principal $2$-bundle over a Lie groupoid $\mb{X}$ as 
a  Hilsum \& Skandalis generalized morphism of Lie $2$-groupoids $\mb{X}\ra \mb{G},$ where $\mb{G}$ is the strict structure Lie 2-group. They treated both $\mb{G}$ and $\mb{X}$ as Lie $2$-groupoids. Particularly, they show that a principal automorphism $2$-group bundle is the same as a $G$-gerbe up to a Morita equivalence.

The theory of connection structures on principal $2$-bundles over manifolds/discrete smooth 2-spaces is well studied. Several authors investigated its various aspects which include the works of Breen-Messing \cite{MR2183393},  Baez-Schreiber \cite{MR2342821, schreiber2005loop}, Jurco-Samann-Wolf \cite{MR3351282}, Aschieri-Cantini-Jurco  \cite{MR2117631}, Gengoux-Stienon-Xu  \cite{MR2493616}. More recently, Waldorf developed the notion of a connection on Wockel's principal $2$-bundle in \cite{MR3894086}. Likewise, the theory of parallel transport in a categorified framework is quite an active area of research in the current landscape of higher gauge theory. Below, we list some existing works in the higher gauge theory that discuss holonomy or parallel transport.

Schreiber-Waldorf introduced a model-independent axiomatic approach to the theory of parallel transport in \cite{MR2520993, MR2803871, MR3084724}. An important aspect of their approach is the axiomatic characterization of smoothness conditions for parallel transport functors/2-functors and illustrating its sanity with several existing models like classical principal bundles, associated vector bundles\cite{MR2520993}, connection on non-abelian gerbes\cite{MR2803871, MR3084724} and others to name a few. In \cite{MR3521476}, Collier, Lerman, and Wolbert introduced an alternative but equivalent notion of smoothness for transport functors [Definition 3.5,\cite{MR3521476}]. In \cite{MR3917427}, Waldorf derived a notion of parallel transport on Wockel's principal 2-bundle over a manifold from the global connection data he introduced in \cite{MR3894086}. Their framework permits only local horizontal lifts (that too non-unique) of paths and path homotopies, which extend to the construction of a transport 2-functor from a path 2-groupoid of the base manifold to the bicategory of Lie 2-group equivariant anafunctors. Also, he showed that this 2-functor satisfied the smoothness formalism established in \cite{MR2803871}. Again, in the framework of a principal 2-bundle over a discrete smooth 2-space, Kim and Saemann introduced a notion of generalized higher holonomy functor
via adjusted Weil algebra in \cite{MR4177087} and had been successful in keeping it free from the usual fake-flatness condition. An approach to parallel transport in terms of  Lie crossed module cocycles over a manifold can be found in \cite{MR3357822}, and \cite{MR3529236} investigates its relation to knot theory. The article \cite{MR3645839} provides a gluing algorithm for local 2-holonomies. For approaches in higher gauge theory through double categories, one can check Morton-Picken's \cite{MR4037666} and Zucchini-Soncini's \cite{MR3357822}. 

Nonetheless, as our framework allows for a categorified base space, specifically a Lie groupoid, it is imperative to acknowledge some relevant works in this particular direction. 

In \cite{MR2270285}, Gengoux, Tu, and Xu presented the concept of a holonomy map along a generalized pointed loop for a principal Lie group bundle over a Lie groupoid with a flat connection. In a more recent publication \cite{MR3521476} by Collier, Lerman, and Wolbert, the authors investigated parallel transport within the context of a principal Lie group bundle over a differentiable stack. Specifically, they introduced their definitions for principal bundles, connection structures, and parallel transports as morphisms of stacks, found in Definition 6.1, Definition 6.2, and Definition 6.3 of \cite{MR3521476}, respectively. Furthermore, they demonstrated that when the stack denoted as $\mathcal{X}$ is a quotient stack arising from a Lie groupoid $\mathbb{X}$, their concept of a principal Lie group bundle over $\mathcal{X}$ coincides with the definition of a principal Lie group bundle over $\mathbb{X}$ as given in \cite{MR2270285}. Given a Lie group $G$ and stack $\mc{X}$, their main result gives an equivalence of categories between the category of principal $G$-bundles over $\mc{X}$ and the category of parallel transport functors over $\mc{X}$ [Theorem 6.4, \cite{MR3521476}]. Papers such as \cite{MR3126940, MR3504595, MR3213404, MR2764890} discuss higher gauge theory over path space groupoids. In \cite{MR3566125}, the author considered the base as an affine 2-space. 

For frameworks extending beyond $2$-spaces, interested readers may look at the following references: \cite{MR3548195, bakovic2009simplicial, bakovic2008bigroupoid}, which delve into Kan simplicial manifolds, and \cite{fiorenza2011cech, MR3423073, MR3385700, schreiber2013differential}, which offer a broader perspective involving $\infty$-topos theory.


Let us now return to our work!!

We  introduce a notion of categorified principal bundle, namely a `principal 2-bundle over a Lie groupoid', defined as\\
\textbf{Definition.}[\Cref{Definition:principal $2$-bundle over Liegroupoid}]
	For a Lie 2-group $\mb{G}$, a \textit{principal $\mb{G}$-bundle over a Lie groupoid} $\mathbb{X}$ is defined as a morphism of Lie groupoids $\pi: \mb{E} \rightarrow \mb{X}$ along with a right action $\rho: \mb{E} \times \mb{G} \rightarrow \mb{E}$ of the Lie $2$-group $\mb{G}$ on the Lie groupoid $\mb{E}$ such that, 
	\begin{itemize}
		\item $\pi_0\colon E_0 \rightarrow X_0$ is a principal $G_0$-bundle over the manifold $X_0$,
		\item $\pi_1\colon E_1 \rightarrow X_1$ is a principal $G_1$-bundle over the manifold $X_1$.
	\end{itemize}
We study the underlying fibration structure of these principal 2-bundles and characterize them. With this goal in mind, we introduce two classes of such 2-bundles, which we call \textit{quasi-principal 2-bundles} (\Cref{Definition:Quasicategorical Connection}) and \textit{categorical-principal 2-bundles} (\Cref{Unital connection}). Their set-theoretic analogs correspond respectively to a fibered category equipped with cleavage and a fibered category equipped with a splitting cleavage. The corresponding notion of cleavages and splitting cleavages in our framework have been called by the names \textit{quasi connections} (\Cref{Definition:Quasicategorical Connection}) and \textit{categorical connections} (\Cref{Def:categorical connection}), respectively. We also obtain a weakened version of a principal Lie group bundle over a Lie groupoid (as studied in \cite{MR2270285}), whose underlying action of the base Lie groupoid on the total space is now replaced by a quasi-action (that is not closed under composition and unit map)  and `is an action upto a factor'. We call these geometric objects by the name \textit{pseudo-principal Lie crossed module-bundles over Lie groupoids} (\Cref{Definition:PseudoprincipalLiecrossedmodulebundle}). One can think of these objects as Lie 2-group torsors analogs of pseudofunctors. Our first main result provides a one-one correspondence between quasi-principal 2-bundles and pseudo-principal Lie crossed module-bundles. This result also proves a Lie 2-group torsor version of the well-known one-one correspondence between fibered categories and pseudofunctors, which, according to the best of our knowledge, is a new addition to the existing literature. More precisely, we obtain the following equivalence of categories:
\begin{theorem}[\Cref{Main Theorem 1}]
	For a Lie crossed module $(G,H,\tau, \alpha)$ and a Lie groupoid $\mb{X}$, the groupoid $\rm{Bun}_{\rm{quasi}}(\mb{X}$, $[H \rtimes_{\alpha}G \rra G])$ is equivalent to the groupoid \\${\rm{Pseudo}} \big(\mb{X}, (G,H,\tau, \alpha) \big)$, where $\rm{Bun}_{\rm{quasi}}(\mb{X}$, $[H \rtimes_{\alpha}G \rra G])$ is the groupoid of quasi-principal $[H \rtimes_{\alpha}G \rra G]$-bundles over $\mb{X}$ and ${\rm{Pseudo}} \big(\mb{X}, (G,H,\tau, \alpha) \big)$ is the groupoid of pseudo-principal $(G,H,\tau, \alpha)$-bundles over $\mb{X}$.  		
\end{theorem}

An interesting consequence of the above theorem allows us to extend the categorical principal 2-bundles to be defined over the differentiable stack represented by the base Lie groupoid (\Cref{stack}). 

Moreover, as a side result, we also relate certain aspects of our Lie 2-group bundles theory to Lie groupoid $G$-extensions for a Lie group $G$ (\Cref{eta twisted to Lie groupoid extension}).

Our next step is to develop the theory of connection structures and gauge transformations on principal Lie 2-group bundles over Lie groupoids. 

We start with the construction of the \textit{Atiyah sequence associated to a principal 2-bundle over a Lie groupoid}:
\begin{proposition}[\Cref{Prop:AtiyahLie2gpd}]
	For a Lie 2-group $G$, let $\pi \colon \mb{E} \ra \mb{X}$ be a principal $\mb{G}$-bundle over $\mb{X}$.  Then we have a short exact sequence 
	\begin{equation}\nonumber
		\begin{tikzcd}
			0 \arrow[r, ""] & {\rm Ad}(\mb{E})    \arrow[r, "\delta^{/\mb{G}}"] & {\rm At}(\mb{E})  \arrow[r, "\pi_{*}^{/\mb{G}}"] & T \mb{X} \arrow[r, ""] & 0
		\end{tikzcd}
	\end{equation}
	of VB-groupoids  over $\mb{X}=[X_1\rra X_0],$ where ${\rm Ad}(\mb{E})$, ${\rm At}(\mb{E})$ and $T \mb{X} $ are \textit{Adjoint}, \textit{Atiyah }and tangent VB-groupoids over the Lie groupoid $\mb{X}$. Here, $\delta^{/\mb{G}}$ and $\pi_{*}^{/\mb{G}}$  are induced from the fundamenental vector field functor $\delta$ and the tangent functor $\pi_{*}$.
\end{proposition}

Then, we introduce two notions of connection structures viz. \textit{strict connections }and \textit{semi-strict connections on a principal $2$-bundle over a Lie groupoid } arising from a splitting of the Atiyah sequence and a splitting up to a natural isomorphism (\Cref{strict and semi-strict connection definition}). We describe these connection structures  in terms of $L(\mb{G})$-valued differential 1-forms $\omega \colon T\mb{E} \ra L(\mb{G})$ on the total Lie groupoid $\mb{E}$, where $L(\mb{G})$ is the Lie 2-algebra of the Lie 2-group $\mb{G}$ (\Cref{strict ans semi strict connetion 1-forms}). We construct categorical principal 2-bundles over Lie groupoids equipped with strict and semi-strict connections from the data of principal Lie group bundles (equipped with connection structures) over the base Lie groupoid, (\Cref{Prop:ConOnDeco}). An existential criterion for the strict and semistrict connections on a principal $2$-bundle over an orbifold is also proposed, \Cref{Existence}. Given a principal 2-bundle over a Lie groupoid, we construct respectively the category of strict and semi-strict connections in terms of both splittings of the associated Atiyah sequence and also as $L(\mb{G})$-valued differential 1-forms on $\mb{E}$ (\Cref{category of strict connection}, \Cref{Def:Semisemistriccat}). One of our primary findings results in a categorical one-one correspondence between connections as splittings of the associated Atiyah sequence  and connections as Lie 2-algebra valued differential 1-forms:
\begin{theorem}[\Cref{strict connection=strict forms}]
	For a Lie 2-group $\mb{G}$, let $\pi \colon \mb{E} \ra \mb{X}$ be a principal $\mb{G}$-bundle over a Lie gorupoid $\mb{X}$.
	\begin{enumerate}
		\item The categories $C^{\rm{semi}}_{\mb{E}}$ and $\Omega_{\mb{E}}^{\rm{semi}}$ are isomorphic. 	
		\item The categories $C^{\rm{strict}}_{\mb{E}}$ and $\Omega_{\mb{E}}^{\rm{strict}}$ are isomorphic. 	
	\end{enumerate}	
\end{theorem}
In the above theorem, $C^{\rm{semi}}_{\mb{E}}$ and $\Omega_{\mb{E}}^{\rm{semi}}$ denotes the category of semi-strict connections as splittings and as $L(\mb{G})$-valued differential 1-forms respectively. Notations $C^{\rm{strict}}_{\mb{E}}$ and $\Omega_{\mb{E}}^{\rm{strict}}$ denote likewise for strict connections.

 We study the action of the strict $2$-group of gauge transformations on the groupoid of strict and semi-strict connections (\Cref{Prop:Actiongaugeconnecat}). Our third main result observes an extended symmetry enjoyed by the category of semi-strict connections (\Cref{E:GengaugeactdefwithF}).

Next, we study an interplay between the differential geometric connections and the theory of fibered categories by developing a notion of connection-induced parallel transport on our principal 2-bundles over Lie groupoids.

We introduce a class of Haefliger paths, which we call  \textit{lazy Haefliger paths} in a Lie groupoid (\Cref{Definition:Haefliger path}), and a notion of thin homotopy between them (\Cref{Definition: Thin homotopy of X-paths}). We introduce a notion of the \textit{thin fundamental groupoid $\Pi_{{\rm{thin}}}(\mb{X})$ of a Lie groupoid }$\mb{X}$ (\Cref{Thin fundamental groupoid of a Lie groupoid}) and show it to be a diffeological groupoid (\Cref{Thin fundamental groupoid of a Lie groupoid is a diffeological groupoid}). The multiplicative nature of our connection structures and the underlying fibration structure of a quasi-principal 2-bundle leads us to a construction of the parallel transport functor for a quasi-principal 2-bundle.
\begin{theorem}[\Cref{Theorem: Parallel transport on 2-bundles}]
	Given a quasi-principal $\mb{G}:=[H \rtimes_{\alpha} G \rra G]$-bundle $(\pi: \mb{E} \ra \mb{X}, \mc{C})$  with a strict connection $\omega: T\mb{E} \ra L(\mb{G})$, there is a functor
	\begin{equation}\label{Parallel}
		\mc{T}_{\mc{C}, \omega} \colon  \Pi_{\rm{thin}}(\mb{X}) \ra \overline{\mb{G} \rm{-Tor}},
	\end{equation}
\end{theorem}
In the above theorem, $\overline{\mb{G} \rm{-Tor}}$ is a `suitable'  quotient category of the category of $\mb{G}$-torsors.

We exhibit the naturality of the above parallel transport functor with respect to connection-preserving bundle morphisms (\Cref{Naturality of parallel transport}). 
To be precise, fixing a Lie 2-group $\mb{G}$ and a Lie groupoid $\mb{X}$,  the said observation extends this parallel transport functor \Cref{Parallel} to define a functor 
\Cref{Equivalence of quasi and functors},
\begin{equation}\label{Main functor1}
	\mc{F} \colon \rm{Bun}_{\rm{quasi}}^{\nabla}(\mb{X}, \mb{G}) \ra \rm{Trans}(\mb{X},\mb{G}).
\end{equation}
In \Cref{Main functor1}, $\rm{Bun}_{\rm{quasi}}^{\nabla}(\mb{X}, \mb{G})$ is the category of quasi-principal $\mb{G}$-bundles equipped with strict connections over the Lie groupoid $\mb{X}$. On the other side, $\rm{Trans}(\mb{X},\mb{G})$ is the category of functors from the thin fundamental groupoid of $\mb{X}$ to the priorly mentioned quotiented category of $\mb{G}$-torsors. We also establish the naturality of \Cref{Parallel} with respect to constructions of strong fibered products of Lie groupoids (\Cref{Pullback naturality}). This parallel transport functor is also shown to have an appropriate smoothness property (\Cref{Smoothness of parallel transport functor}), such that in the classical case, this coincides with the one mentioned in \cite{MR3521476}.

 Finally, applying the parallel transport theory developed so far, we investigate parallel transports on VB-groupoids along lazy Haefliger paths.

\subsection*{Significance of our results}

To our knowledge, utilizing the Atiyah sequence approach in higher gauge theory, especially in the framework of Lie 2-group bundles, as presented by us in \cite{MR4403617}, appears to be a novel exploration. 
An interesting outcome of employing this method is the concept of semi-strict connections on a Lie $2$-group bundle over a Lie groupoid. Recently, the work of \cite{MR4529816} delved into the study of a generalized Atiyah sequence within the realm of $\infty$-categories. Notably, in the current year, Herrera-Carmona and Ortiz, as presented in \cite{herreracarmona2023chernweillecomte}, introduced a Chern-Weil map for our principal 2-bundles, utilizing a connection structure similar to ours. Simultaneously, in \cite{MR4598921}, there is a reference to our principal 2-bundles as principal bundle groupoids, and the associated nerves were studied. 

Other than a brief mention to a concept in (Subsection 4.1.3, \cite{MR2218759}), which discusses a distinct setup and context, to the best of our knowledge, we assert that our method of parallel transport (particularly along Haefliger paths in Lie groupoids) is new to the ones already present in the current higher gauge theory literature. We hope our approach will offer a fresh perspective on parallel transport theory (derived from appropriate connection data) on any geometric object with an underlying Lie groupoid fibration structure endowed with a suitable notion of cleavage. 

\subsection*{Limitations of our results}
 It is important to note that there are specific noteworthy topics that have not been addressed in this thesis, such as the construction of our principal 2-bundles with connection structures from the information of a parallel transport functor, and a concept of parallel transport along higher-dimensional entities such as surfaces, and so on. We intend to incorporate these investigations into an upcoming paper to enhance the clarity of the interrelation with other established notions in higher parallel transport theories.

\section{Organization of the thesis}
\subsection{Classical Set-up}
In \Cref{Chapter Classical setup}, we review some standard notions in classical gauge theory, whose categorified versions form the key players in this thesis.  In particular, we recall the following notions:
\begin{itemize}
	\item Fiber bundles viz: traditional principal bundles  and vector bundles over manifolds (Subsection \ref{fibre bundles});
	\item Atiyah sequence associated to a principal bundle over a manifold (Subsection \ref{Atiyah sequence associated to a principal bundle});
	\item Gauge group of a principal bundle (Subsection \ref{subsection: gauge group of a principal bundle});
	\item Connection structure on a principal bundle  as a splitting of the associated Atiyah sequence, as a Lie algebra valued differential 1-form and as a horizontal distribution (Subsection \ref{Subsection: Connection on a principal bundle and its characterizations});
	\item Induced Ehresmann connection on the associated fiber bundle (Subsection \ref{Induced Ehresmann connection on the associated bundle});
	\item Parallel transport of a connection on a principal bundle. This includes the construction of the corresponding parallel transport functor and the induced parallel transport on an associated bundle (Section \ref{Section: Connection structures on a principal bundle}).
\end{itemize}

\subsection{Preliminaries} 
In \Cref{A brief review of bundles over Lie groupoids and their gauge theory}, we recall some standard notions in category theory and higher gauge theory that are necessary or relevant for extending the ideas of \Cref{Chapter Classical setup} in a categorified framework. In particular, we recall the following:
\begin{itemize}
	\item 2-categories (Subsection \ref{2-categories});
	\item Fibered categories, pseudofunctors, and the Grothendieck construction (Subsection \ref{subsection Fibered categories});
	\item \underline{Lie groupoids}:  Basic definitions, properties and examples (Subsection \ref{Basic definitions}), fibered products in Lie groupoids (Subsection \ref{subsection pullbacks in Lie groupoids}), Lie groupoid G-extensions (Subsection \ref{subsection Lie groupoid G extension}), action and quasi-action of a Lie groupoid (Subsection \ref{Action and quasiaction of a Lie groupoid}), anafunctors and Morita equivalence of Lie groupoids (Subsection \ref{subsection Generalized homomorphisms and Morita equivalence of Lie groupoids});
	\item Principal Lie group bundles over Lie groupoids and their connection structures (Section \ref{Principal bundles over Lie groupoids and their connection structures});
	\item \underline{Lie 2-group and its Lie 2-algebra}: Correspondence between Lie 2-groups and Lie crossed modules (Subsection \ref{Lie 2 group from Lie crossed module}), the Lie 2-algebra of a Lie 2-group (Subsection \ref{Lie 2-algebra as Lie crossed module}), adjoint actions of a Lie 2-group (Subsection \ref{SS:adjointLie 2 group}) and an action of a Lie 2-group on a Lie groupoid (Subsection \ref{Action of a Lie 2-group on a Lie grouipoid});
	\item VB-groupoids (Section \ref{subsection VB groupoids});
	\item Baez-Crans 2-vector spaces (Section \ref{Section: 2-vector space});
	\item Haefliger paths and the fundamental groupoid of a Lie groupoid (Section \ref{Haefliger paths and fundamental groupoid of a Lie groupoid});
	\item \underline{Diffeology}: Definitions, basic properties and examples (Subsection \ref{Definitionsbasic properties and examples}) and a discussion on the smoothness property of parallel transport functor of a principal bundle over a manifold (Subsection \ref{Smoothness of traditional parallel transport}).
\end{itemize}

\subsection{Principal 2-bundles over  Lie groupoids and their characterizations}
\Cref{chapter 2-bundles} is mostly based on our papers \cite{MR4403617} and \cite{chatterjee2023parallel}.

Section \ref{A principal 2-bundle over a Lie groupoid} introduces the notion of a \textit{principal 2-bundle over a Lie groupoid} and discusses several examples.

Section \ref{Decorated principal 2-bundles and categorical connections} is splitted into two subsections. In subsection \ref{SS:Decorated}, we construct a principal 2-bundle, namely a \textit{decorted principal 2-bundle over a Lie groupoid} from the data of a Lie crossed module and a principal Lie group bundle over the same base Lie groupoid. We also discuss several examples of decorated principal 2-bundles over Lie groupoids. In subsection \ref{SS:Catconnection}, we introduce a structure called \textit{categorical connection} on our principal 2-bundles. We then characterize decorated principal 2-bundles with respect to these structures. We also relate categorical connections with the triviality of traditional principal bundles.

Section \ref{Section: A Quasi-principal 2-bundle over a Lie groupoid} is splitted into four subsections. Subsection \ref{A quasi-principal 2-bundle over a Lie groupoid} introduces the notion of a \textit{quasi connection }and a \textit{quasi-principal 2-bundle over a Lie groupoid}. Also, we characterize quasi connections in terms of categorical connections. In subsection \ref{Examples of quasiprincipal 2-bundles}, we construct some examples of quasi-connections and quasi-principal 2-bundles. In subsection \ref{Section Quasi-principal 2-bundle over a Lie groupoid as a Grothendieck construction}, we obtain the main result of this chapter \Cref{Main Theorem 1}, `a statement and proof of a Lie 2-group torsor version of the one-one correspondence between fibered categories and pseudofunctors'. As a byproduct of this result, we obtain the notion of a \textit{pseudo-principal Lie crossed module-bundle over a Lie groupoid} and characterize a quasi-principal 2-bundle in terms of it. Finally, in subsection \ref{Quasiconnections as retractions}, we characterize quasi connections in terms of certain retractions.

In section \ref{Towards a principal 2-bundle over a differentiable stack}, we extend a class of our principal 2-bundles to be defined over the differentiable stack that its base Lie groupoid represents.

We end the chapter by introducing a weaker version (in terms of the action of the structure 2-group) of a principal Lie 2-group bundle over a Lie groupoid, namely \textit{$\eta$-twisted principal 2-bundle over a Lie groupoid}, and relate it to a Lie groupoid $G$-extension for a Lie group $G$ (Section \ref{Twisted principal 2-bundles and Lie groupoid $G$-extensions}).

\subsection{Connection structures and gauge transformations on a principal 2-bundle over a Lie groupoid}
Most of the contents of \Cref{chapter 2-bundles copnnection} are borrowed from our paper \cite{MR4403617}.

The whole chapter is split into two sections. Section \ref{Connection structures on a principal 2-bundle over a Lie groupoid} introduces connection structures. While section \ref{Gauge 2-group and its action on the category of connections} studies gauge transformations and investigate their actions on connection structures.

Subsection \ref{Subsection: Atiyah sequence associated to a principal 2-bundle over a Lie groupoid} offers a construction of a short exact sequence of VB-groupoids associated to our principal Lie 2-group bundle over a Lie groupoid, which we call the \textit{Atiyah sequence}. Subsection \ref{subsectionStrict and semi-strict connections as splittings of the Atiyah sequence} introduces \textit{strict} (resp. \textit{semi-strict connection}) on a principal 2-bundle over Lie groupoid as a splitting of the associated Atiyah sequence (resp. splitting up to a natural isomorphism) and construct the corresponding \textit{groupoid of strict} (resp. \textit{semi-strict}) \textit{connections}. In subsection \ref{Strict and semi-strict connections as Lie 2-algebra valued 1-forms on Lie groupoids} we express the strict and semi-strict connections in terms of  Lie $2$-algebra valued differential forms on Lie groupoids and call them \textit{strict} (resp \textit{semi-strict connection 1-forms}) and construct the corresponding \textit{groupoid of strict }(resp. \textit{semi-strict}) \textit{connection 1-forms.} We found a way to measure how much a semi-strict connection 1-form deviates from being a strict connection 1-form. We also devise a procedure to construct semi-strict connections. In subsection \ref{Categorical correspondence strict and semistrict}, we obtain our first main result of this chapter (Subsection \ref{strict connection=strict forms}) by proving an isomorphism between the groupoid of strict (resp. semi-strict) connections and the groupoid of strict (resp. semi-strict) connection 1-forms. In subsection \ref{SS:Conndeco}, we provide a detailed construction of the connection structure on a decorated principal 2-bundle. Also, as a side result, we relate our construction to Cartan connections. Finally, we end the section by proposing an existential criterion for strict and semi-strict connections on a principal 2-bundle over an orbifold in subsection \ref{Existence of strict and semistrict connections}. 

In subsection \ref{Gauge 2-group of a principal 2-bundle over a Lie groupoid}, we introduce the notion of \textit{gauge 2-group of a principal 2-bundle over a Lie groupoid}. We studied the relation between gauge transformations and categorical connections. We also explicitly describe the gauge transformations on a decorated principal 2-bundle. In subsection \ref{Action of Gauge $2$-group on the category of Connections}, we investigate the actions of gauge 2-group on the groupoids of strict and semi-strict connections. In fact, we showed that semi-strict connections enjoy an extended gauge symmetry and finds its relation to the gauge transformations of the connection $1$-forms in higher BF theories (Subsection \ref{An extended symmetry of semi-strict connections}).

\subsection{Parallel transport on quasi-principal 2-bundles and associated VB-groupoids}

\Cref{Chapter: Parallel transport on quasi-principal 2-bundles} is mostly adapted from our paper \cite{chatterjee2023parallel}.

In subsection \ref{Lazy haefliger paths}, we introduce the definitions of a \textit{lazy Haefliger path on a Lie groupoid} and a thin homotopy, namely \textit{lazy $\mb{X}$-path thin homotopy} between them. Subsection \ref{Thin fundamental groupoid of a Lie groupoid} introduces a notion of \textit{thin fundamental groupoid of a Lie groupoid}. Then, the subsection \ref{Subsection Smoothness of thin fundamental groupoid of a Lie groupoid} shows it to be a diffeological groupoid.

Section \ref{Parallel transport on a principal 2-bundle over a Lie groupooid} introduces a notion of parallel transport along a lazy Haefliger path on a quasi-principal 2-bundle equipped with a strict connection. This notion is introduced in three steps. Step-1 uses the underlying quasi-connection structure and forms the content of subsection \ref{Step1}. Step-2 uses the strict connection structure,  and is the content of subsection \ref{Step2}. Finally, in subsection \ref{Step3}, we combine the results of Step-1 and Step-2  to arrive at our intended definition of parallel transport. We also argue that our definition is a reasonable one, as it generalizes the classical one and relates to the existing notion of parallel transport on a principal 2-bundle over a manifold.

Section \ref{Lazy Xpaththinhomotopy} establishes the lazy $\mb{X}$-path thin homotopy invariance of the parallel transport.

In section \ref{Parallel transport functor of a quasi-principal 2-bundle}, we construct the parallel transport functor of a quasi-principal 2-bundle, defined on the thin fundamental groupoid of the base Lie groupoid. Also, we checked its sanity by showing its naturality with respect to connection preserving bundle morphisms in subsection \ref{Naturality with respect to connection preserving morphisms} and strong fibered product constructions in subsection \ref{Naturality with respect to fibered products}. In subsection \ref{Naturality with respect to connection preserving morphisms}, we obtain the main result of this chapter, which extends the parallel transport functor construction to establish \Cref{Main functor1}.

Finally, in section \ref{Associated PAPER VERSION}, we construct a VB-groupoid associated to a quasi-principal 2-bundle and explore the parallel transport on this associated VB-groupoid along a lazy Haefliger path.

\subsection{Future directions of research}
 \Cref{Future} discusses potential future avenues of exploration grounded in the research undertaken throughout the thesis. 

\section{Notations and conventions}\label{Notation conventions}
Here,  we will fix some conventions and notations which will be followed throughout this manuscript.

All manifolds are typically assumed to be smooth, second countable, and Hausdorff. Let $\rm Man$ be the category of such manifolds. Though there are some interesting examples of Lie groupoids, particularly in foliation theory, where morphism space fails to satisfy the Hausdorff property. For the purpose of this theisis we are not seriously required to consider such Lie groupoids other than some cursory mention of some standard examples. We assume our Lie groups to be matrix groups for computational simplicity. For a Lie group $G$, we denote its Lie algebra by 
$ L(G)$. For a smooth map $f\colon M\to N,$ the differential at $m\in M$ will be denoted as $f_{*,m}\colon T_m M\ra T_{f(m)}N.$ A smooth right (resp. left) action of a Lie group $G$ on a smooth manifold $P$ will be denoted as $(p, g)\mapsto p g$ (resp. $(g, p)\mapsto g p$), for $g\in G$ and $p\in P$. The corresponding differentials $T_pP\ra T_{p g}P$ or $T_p P\ra T_{g p }P$ will be denoted respectively 	as $v\mapsto v\cdot g$ or $v\mapsto g\cdot v,$ for $v\in T_p P$. If $M$ and $N$ are manifolds with smooth right (resp. left) action of a Lie group $G$, then a map $f \colon M \ra N$ is said to be \textit{$G$-equivariant} if it satisfies $f(pg)=f(p)g$(resp. $f(gp)=f(g)p$) for all $p \in M$ and $g \in G$. A manifold with a smooth free and transitive right (resp. left) action of a Lie group $G$ will be called a \textit{$G$-torsor}. The collection of $G$-torsors and $G$-equivariant maps form a groupoid, which we denote by $G$-Tor. For a fixed $p\in P$, the differential $\delta_p\colon T_eG\ra T_pP$ of the map $G\ra P$, $g\mapsto pg$, at the identity element, defines a vector field on $P$, the so called fundamental vector field or vertical vector field corresponding to an element $B\in L(G)$. Evaluated at a point $p \in P$, we denote it by $\delta_p(B)$ or $B^{*}(p)$.

For any path $\alpha \colon [0,1] \ra M$ in a manifold $M$, we denote the path $t \mapsto \alpha(1-t)$ by $\alpha^{-1}$.

In any category,  \textit{structure maps}: source, target, unit, and composition will be denoted by the letters $s,t,u$, and $m$, respectively. The composition of a pair of morphisms $f_2, f_1$  is denoted as $f_2\circ f_1,$ where $t(f_1)=s(f_2).$ . For any object $p$ in a category, $1_P$ denotes the element $u(p)$. In a groupoid, the inverse map is denoted by $\mathfrak{i}$ and for any morphism $\gamma$, we use the notation  $\gamma^{-1}$ to denote $\mathfrak{i} (\gamma)$ for brevity. The notation $[C_1 \rra C_0]$ will denote a category $C$ whose object set is $C_0$ and morphism set is $C_1$.  Given a functor $F \colon C \ra D$, the notation $F_0$ and $F_1$ will represent the object level and morphism level map, respectively. The notation $(F_1, F_0)$ will denote the functor $F$. Blackboard bold notation will be used to denote Lie groupoids; that is, $\mb{E}$, $\mb{X}$, etc. unless otherwise stated, these notations will always denote Lie groupoids $[E_1 \rra E_0], [X_1 \rra X_0],..$ etc. respectively throughout the manuscript.


\chapter{Classical Set-up}\label{Chapter Classical setup}



\lhead{Chapter 2. \emph{Classical Set-up}} 

In this chapter, we provide a concise review of elements from \textit{Classical Gauge Theory}. Our goal is to later \textit{categorify} these concepts. To be more specific, we will revisit topics such as gauge transformations, connection structures, and parallel transport on principal bundles over smooth manifolds. This review is intended to serve as a classical foundation for the categorically enriched framework developed in this thesis.


\section{A principal bundle, its Atiyah sequence and its gauge group}\label{Section: Gauge group of a principal bundle over a manifold}
This section will briefly recall fiber bundles, focusing on principal bundles, associated gauge groups, and vector bundles. The content covered in this section is predominantly conventional and can be found in any typical differential geometry textbook. For further details, we recommend consulting references such as \cite{MR3837560}, \cite{MR3585539}, \cite{MR1393940}, and \cite{MR896907}. We commence by recalling the definition of a fibre bundle.
\subsection{Fibre bundles}\label{fibre bundles}

\begin{definition}\label{Definition: fibre bundle}
	Let $M$ and $F$ be manifolds. A \textit{fibre bundle over $M$ with the fibre $F$} is a surjective differentiable map $\pi \colon E \ra M$ such that there exists a cover $\mc{U}:= \cup_{i \in \mc{I}} U_i$ of the manifold $M$ and diffeomorphisms $\phi_i \colon \pi^{-1}(U_i) \ra U_i \times F$ for an index set $\mc{I}$, such that for each $i \in \mc{I}$, the following diagram commutes 
	
	\[
	\begin{tikzcd}
		\pi^{-1}(U_i) \arrow[r, "\phi_i"] \arrow[d, "\pi"'] & U_i \times F \arrow[ld, "{\rm{pr}}_1"] \\
		M                              &                  
	\end{tikzcd}.\]	
\end{definition}
We call the family $(U_i, \phi_i)_{i \in \mc{I}}$ as a\textit{ local trivialization} of the fibre bundle $\pi \colon E \ra M$. The manifolds $M$ and $E$ are called the \textit{base space} and the \textit{total space} of the fiber bundle. Note that the map $\pi \colon E \ra M$ above is always a submersion.

Also, from the local trivialization maps $\phi_i, i \in I$ above, one obtains a family of diffeomorphims $\phi_{j} \circ \phi_i^{-1}|_{(U_i \cap U_j) \times F} \colon (U_i \cap U_j) \times F \ra (U_i \cap U_j) \times F$ called \textit{transition functions}, whenever $U_i \cap U_j \neq \emptyset, i,j \in I$. These transition functions define another family of diffeomorphims $\phi_{jx} \circ \phi^{-1}_{ix} \colon F \ra F$, for each $i,j \in I$ and $x \in U_i \cap U_j$, where $\phi_{jx}:= \phi_j|_{ \lbrace x\rbrace \times F}$ and $\phi_{ix}:= \phi_i|_{ \lbrace x\rbrace \times F}$. They, in turn define smooth maps $\phi_{ji} \colon U_i \cap U_j \ra {\rm{Aut}}(F)$, $x \mapsto \phi_{jx} \circ \phi^{-1}_{ix}$, $i,j \in I$, which satisfy the crucial \textit{cocycle condition} given as
\begin{itemize}
\item $\phi_{ii}(x)= {\rm{id}}_{F}$ for all $x \in U_i$,
\item $\phi_{ij}(x) \circ \phi_{ji}(x)= {\rm{id}}_{F}$ for all $x \in U_i \cap U_j$,
\item $\phi_{ik}(x) \circ \phi_{kj}(x) \circ \phi_{ji}(x)= {\rm{id}}_{F}$ for all $x \in U_i \cap U_j \cap U_k$.
\end{itemize}
These transition functions in fact completely characterize a fibre bundle. To elaborate this point, start with a pair of manifolds $M,F$, an open cover $\mc{U}:= \cup_{i \in \mc{I}} U_i$ of $M$ and a family of diffeomorphisms $\phi_{ji} \colon (U_i \cap U_j) \times F \ra (U_i \cap U_j) \times F$ on every non-empty intersections $U \cap U_j, i,j \in I$, such that ${\rm{pr}}_1 \circ \phi_{ji}= {\rm{pr}}_1$, where ${\rm{pr}}_1$ is the first projection map. If the family of maps $ \phi_{ji}(x):=\phi_{ji}(x,-) \colon F \ra F$, $i,j \in I$, $x\ \in U_i \cap U_j$ satisfy the the cocycle condition defined above, then we can construct a fibre bundle $\bar{\pi} \colon \bar{E} \ra M$ with fibre $F$, where $\bar{E}$ is the quotient space of $\sqcup_{i \in I} (U_i \times F)$ by the equivalence relation defined as $(i,x, \epsilon) \sim (j,x', \epsilon')$, if $x=x'$ and $\epsilon'= \phi_{ji}(x) \big(\epsilon \big)$. Moreover, the transition functions of $\bar{\pi} \colon \bar{E} \ra M$ coincide with $\phi_{ji}, i,j \in I$, we started with.

\begin{definition}\label{Definition: Morphism of fibre bundles}
	Let $\pi \colon E \ra M$  and $\pi' \colon E' \ra M'$ be fiber bundles over $M$ with fibres $F$ and $F'$ respectively. A \textit{morphism of fibre bundles} from  $\pi \colon E \ra M$ to $\pi' \colon E' \ra M'$ consists of a pair of smooth maps $$(f_E \colon E \ra E', f_M \colon M \ra M' )$$ such that the following diagram is commutative:
	\[
	\begin{tikzcd}
		E \arrow[d, "\pi"'] \arrow[r, "f_{E}"]                 & E' \arrow[d, "\pi'"] \\
		M \arrow[r, "f_M"']  & M'           
	\end{tikzcd}\]
\end{definition}
Mostly, we will restrict our attention to the case $M=M'$ and $f_M= {\rm{id}}$ and in that case, we say $f_E $ is a\textit{ morphism of fiber bundles over $M$}.  

We are specifically focused on two types of fiber bundles for our purposes.
\begin{enumerate}[(i)]
	\item A principal $G$-bundle, whose fibre is a Lie group $G$.
	\item A vector bundle, whose fibre is a finite-dimensional real vector space.
\end{enumerate}

\begin{definition}\label{Definition: Principal G-bundle}
	For a Lie group $G$, a \textit{principal $G$-bundle over a manifold $M$ with structure group G} is a fibre bundle $\pi \colon P \ra M$ with fibre $G$ satisfying the following conditions:
	\begin{enumerate}[(i)]
		\item $G$ has a smooth right action on $P$ such that the map $P \times G \ra P \times_{\pi,M, \pi} P$, defined as $(p,g) \mapsto (p,pg)$ is a diffeomorphism, redwhere $ P \times_{\pi,M, \pi} P = \lbrace (p,q) \in E \times E \colon \pi(p)=\pi(q) \rbrace $.
		\item There is a \textit{$G$-equivariant local trivialization} $(u_i, \phi_i)_{i \in \mc{I}}$ of $\pi \colon P \ra M$ i.e for each $i \in \mc{I}$,  $(u_i, \phi_i)$ is a local trivialization and satisfies $\phi_i(pg)= \phi_i(p)g$ for all $p \in \pi^{-1}(U_i)$, $g \in G$, where $G$ acts on the right-hand side as group multiplication.
	\end{enumerate}
\end{definition}
For each $x \in M$, one can check that $\pi^{-1}(x)$ is a $G$-torsor and is a closed submanifold of $P$, called the \textit{fibre over $x$}. Note that the first condition of the above definition also tell that every fiber is diffeomorphic to the structure group $G$. Same follows from the second condition too, by the restriction of the local trivialization maps on the fibres.

\begin{remark}\label{Automorphism group of the fibre}
	Given a Lie group $G$, a $G$-torsor $E$ (see \Cref{Notation conventions}) and a point $z \in E$, we have a group isomorphism defined as
	\begin{equation}\label{Canonical Lie group structure on Aut}
		\begin{split}
			\psi_z \colon & {\rm{Aut}}(E):= {\rm{Hom}}_{G {\rm{-}}{\rm{Tor}}} ( E, E) \ra G\\
			& f \mapsto \delta(z,f(z)),
		\end{split}
	\end{equation}
	where $\delta \colon E \times E \ra G$ is a smooth map defined implicitly as $x \cdot \delta(x,y)=y$, whose well-definedness follows from the freeness of the $G$-action.  Hence, $\rm{Aut}(E)$ is canonically a Lie group and in fact does not depend on the choice of $z$, (see \textbf{Lemma 3.4}, \cite{MR3521476}). Consequently, we see that given a principal $G$-bundle $\pi \colon P \ra M$, the automorphism group ${\rm{Aut}}(\pi^{-1}(x))$ is a Lie group and is isomorphic to $G$ for each $x \in M$, as every fibre of a principal $G$-bundle is a $G$-torsor.
	\end{remark}

\begin{definition}\label{Definition: Morphism of principal G-bundles}
	Let $\pi \colon P \ra M$  and $\pi' \colon P' \ra M'$ be  principal $G$-bundles over $M$ and principal $G'$ bundle over $M'$ respectively. A \textit{morphism of principal bundles} from  $\pi \colon P \ra M$ to $\pi' \colon P' \ra M'$ consists of a triple $$(f_P \colon P \ra P', f_G \colon G \ra G',f_M \colon M \ra M' )$$ such that 
	\begin{itemize}
		\item[(i)] $f_P(pg)=f_P(p)f_G(g)$ for all $p \in P$ and $g \in G$ and
		\item[(ii)] the following diagram is commutative:
		\[
		\begin{tikzcd}
			P \arrow[d, "\pi"'] \arrow[r, "f_{P}"]                 & P' \arrow[d, "\pi'"] \\
			M \arrow[r, "f_M"']  & M'           
		\end{tikzcd}.\]
	\end{itemize}
\end{definition}
We will primarily focus on situations where $G=G'$, $M=M'$, and $f_M$ is the identity. These mappings of principal bundles will be referred to as \textit{morphisms of principal $G$-bundles over $M$}. It is konown that any such morphism of principal $G$-bundles over a manifold $M$ is, in fact, an isomorphism. Thus, given a Lie group $G$ and a manifold $M$, the collection of principal $G$-bundles over $M$ forms a groupoid. We denote it by ${\rm{Bun}}(M, G)$. In particular, the automorphism group of an object $\pi \colon E \ra M$ in the groupoid ${\rm{Bun}}(M, G)$ holds a special name in the literature, which we will discuss in \Cref{subsection: gauge group of a principal bundle}.
\begin{definition}\label{Definition: real vector bundle}
	A \textit{(real) vector bundle of rank $r$ }is a fibre bundle $\pi \colon E \ra M$ with fibre $\mb{R}^{r}$ such that 
	\begin{enumerate}[(i)]
		\item[(i)] Every fibre $\pi^{-1}(x)$ over $x \in M$ is a real vector space of rank $r$;
		\item[(ii)] There is a local trivialization $(U_i, \phi_i)_{i \in I}$ such that for every $i \in \mc{I}$ and $x \in U_i$, the diffeomorphism $\phi_i \colon  \pi^{-1}(U_i) \ra U_i \times \mb{R}^{r}$ restricts to a linear isomorphism $\pi^{-1}(x) \ra \lbrace x \rbrace \times \mb{R}^{r}$.
	\end{enumerate}
	
\end{definition}

\begin{definition}\label{Definition: Morphism of vector bundles}
	Let $\pi \colon E \ra M$  and $\pi' \colon E' \ra M'$ be vector bundles over $M$ and $M'$ respectively. A \textit{morphism of vector bundles} from  $\pi \colon E \ra M$ to $\pi' \colon E' \ra M'$ consists of a pair of smooth maps$$(f_E \colon E \ra E',f_M \colon M \ra M' )$$ such that 
	\begin{itemize}
		\item[(i)] the following diagram commutes:
		\[
		\begin{tikzcd}
			E \arrow[d, "\pi"'] \arrow[r, "f_{E}"]                 & E' \arrow[d, "\pi'"] \\
			M \arrow[r, "f_M"']  & M'        
		\end{tikzcd}\]
		\item[(ii)] for each $x \in M$,  $f_E$ restricts to a linear map $f_x \colon \pi^{-1}(x) \ra \pi'^{-1}(f_M(x))$.
	\end{itemize}
\end{definition}
For the special case, when $M=M'$ and $f_M=id$, we will call these morphisms of principal bundles as \textit{morphisms of vector bundles over $M$}. 

We denote the category of vector bundles over $M$ by the notation ${\rm{Vect}}(M)$.

\begin{example}\label{Example: Product G-bundle}
	For a pair of manifolds $M$ and $F$, the projection map $M \times F \ra M$ given by $(m,f) \mapsto m$ is a fibre bundle over $F$ with fibre $F$ and we call it a \textit{product fibre bundle over $M$.} 
	
	\begin{enumerate}[(i)]
		\item When $F$ is a Lie group $G$ then it is principal $G$-bundle over $M$ with structure group $G$ and we call it a \textit{product $G$-bundle over M}.
		\item When $F$ is a finite-dimensional vector space of rank $r$ then it is a vector bundle over $M$ of rank $r$ and we call it a \textit{product vector bundle over $M$.}
	\end{enumerate}
\end{example}
Any fibre bundle over $M$ isomorphic to a product $G$-bundle over $M$ is called a\textit{ trivial fibre bundle over $M$.}
The following observation characterizes the triviality in principal bundles.

\begin{lemma}
	For a Lie group $G$, a principal $G$-bundle $\pi \colon P \ra M$ is trivial if and only if it admits a global section, i.e., a smooth map $\sigma \colon M \ra E$ such that $\pi \circ \sigma= {\rm{id}}$.
\end{lemma}
Unfortunately, a similar characterization of triviality does not exist for a vector bundle.  For example, a Möbius strip is the real tautological line bundle over the Real projective line, which is not trivial but always admits a global section.
\begin{example}
	For any manifold $M$ of dimension $n$, the tangent bundle $TM \ra M$ is a vector bundle of rank $2n$.
\end{example}

\begin{example}\label{Definition: pull-back principal G-bundle}
	Given a fiber bundle $\pi \colon E \ra M$ with fiber $F$ and a smooth map $f \colon N \ra M$, there is a fiber bundle $\pi^{*} \colon N \times_{f, M,\pi} E  \ra N$ with fiber $F$ over the manifold $M$ which we call the \textit{pull-back bundle of $\pi \colon E \ra M$ along $f$}.
	
	\[\begin{tikzcd}
		N \times_{f,M,\pi} E  \arrow[d, "\pi^{*}"'] \arrow[r, "{\rm{pr}}_{E}"]                 & E \arrow[d, "\pi'"] \\
		N \arrow[r, "f"']  & M           
	\end{tikzcd}\]
	
	The special cases of principal bundles and vector bundles are called \textit{pull-back principal bundles} and \textit{pull-back vector bundles along $f$}, respectively.	
\end{example}
\subsection*{Associated fibre bundle of a principal bundle}\label{Associate bundle}
Let $\pi \colon P \ra M$ be a principal $G$-bundle and $F$ be a manifold endowed with a smooth left action of $G$. Then consider the quotient space $\frac{P \times F}{G}$ consisting of the orbits of the right action of $G$ on $P \times F$ given by $\big( (p,f) ,g \big) \mapsto (pg, g^{-1}f)$ for all $g \in G$, $p \in P$, $f \in F$. One can check that $\pi^{F} \colon \frac{P \times F}{G} \ra M$, $[p,f] \mapsto \pi(p)$ defines a fibre bundle over $M$ with fibre $F$ and is known as an\textit{ associated fibre bundle of $\pi \colon P \ra M$ with fibre $F$.} 

In particular, when $F$ is a finite dimensional vector space $V$ and the action of $G$ on $V$ is a liner representation of $G$ on $V$, then we get a \textit{vector bundle over $M$}, which we call an \textit{associated vector bundle of $\pi \colon P \ra M$ with respect to the given linear representation of $G$ on $V$.}

%
\subsection{Atiyah sequence associated to a principal bundle}\label{Atiyah sequence associated to a principal bundle}
For a Lie group $G$, given a principal $G$-bundle $\pi \colon P \ra M$, one can associate a short exact sequence of vector bundles over $M$. 

\begin{equation}\label{Equation: Atiyah}
	\begin{tikzcd}
		0 \arrow[r, ""]  & {\rm{Ad}}(P) \arrow[r, "\delta^{/G}"] \arrow[d, ""'] & {\rm{At}}(P) \arrow[r, "\pi_{*}^{/G}"] \arrow[d, ""'] & TM \arrow[d, ""'] \arrow[r] & 0 \\
		0 \arrow[r, ""'] & M \arrow[r, "{\rm{id}}"']                & M \arrow[r, "{\rm{id}}"']                & M \arrow[r]                 & 0
	\end{tikzcd}
\end{equation}
called the \textit{Atiyah sequence} of $\pi \colon P \ra M$, whose `splittings' contain the data of `\textit{connection structures}', see \Cref{Section: Connection structures on a principal bundle}, on the principal bundle $\pi \colon P \ra M$. In \Cref{Equation: Atiyah}, $TM \ra M$ is the tangent bundle of $M$, whereas ${\rm{At}}(P) \ra M$ and ${\rm{Ad}}(P) \ra M$ are the \textit{Atiyah bundle }and the \textit{adjoint bundle }of $\pi \colon P \ra M$ and the maps $\delta^{/G}$ and $\pi_{*}^{/G}$ are canonical maps between them. The explicit details of the  Atiyah sequence (\Cref{Equation: Atiyah}) (i.e., defining Atiyah bundle, adjoint bundle, the maps $\delta^{/G}$ and $\pi_{*}^{/G}$) follows from the theory of \textit{quotient vector bundles} which we are going to recall next briefly. The reference for this subsection is mostly Appendix A of \cite{MR896907}

\subsection*{Quotient vector bundles}  Let $\pi \colon P \ra M$ be a principal $G$-bundle over a manifold $M$ and $\pi^{E} \colon E \ra P$ be a vector bundle over the manifold $P$  equipped with a smooth right action of the Lie group $G$ on $E$, such that the following conditions are staisfied:
\begin{enumerate}[(i)]
	\item[(i)] For each $g \in G$, the pair of right translation maps $(\delta_g^{E} \colon E \ra E ,  \delta_g^{P} \colon P \ra P)$ is an automorphism of vector bundles.
	\item[(ii)] There exists a local trivialization $(U_i, \phi_i)_{i \in I}$ of the vector bundle $\pi^{E} \colon  E \ra P$ such that 
	\begin{itemize}
		\item[(a)] each $U_{i}$ is $\pi$-saturated open set and is of the form $\pi^{-1}(V_i)$ for some open set $V_i \subseteq M$;
		\item[(b)] Diffeomorphisms $\phi_i$ are $G$-equivariant for each $i \in I$.
	\end{itemize}.
\end{enumerate}
Then, the quotient space $ \pi^{E/G} \colon E/G \ra M$ has a unique vector bundle structure over $M$ such that the quotient map $Q^{E} \colon E \ra E/G$ is a surjective submersion and the pair $(Q^{E}, \pi)$ is a morphism of vector bundles from $\pi^{E} \colon E \ra P$ to $\pi^{E/G} \colon E/G \ra M$ (\Cref{Definition: Morphism of vector bundles}). The following is a pull-back diagram:

\begin{equation}\nonumber
	\begin{tikzcd}
		E  \arrow[d, "\pi^{E}"'] \arrow[r, "Q^{E}"]                 & E/G \arrow[d, "\pi^{E/G}"] \\
		P \arrow[r, "\pi"']  & M           
	\end{tikzcd}
\end{equation}

Further, for each $p \in P$, the linear map $Q^{E}|_{(\pi^{E})^{-1}(p)} \colon (\pi^{E})^{-1}(p) \ra (\pi^{E/G})^{-1}(\pi(p))$ is an isomorphism of vector spaces. We will call the vector bundle $\pi^{E/G} \colon E/G \ra M$ as the \textit{quotient vector bundle of $\pi^{E} \colon E \ra P$ by the action of $G$.}

As a particular case of the above construction, we get the \textit{Atiyah bundle}  and the \textit{adjoint bundle} associated to a principal bundle:

\begin{example}\label{Example: Atiyah bundle}
	The \textit{Atiyah bundle of a principal $G$-bundle $\pi \colon P \ra M$} is defined as the quotient vector bundle of the tangent bundle $TP \ra P$ by the action of $G$ given by the right translation, i.e., $\big((p,v),g \big) \mapsto (pg,vg)$ for $(p,v) \in T_pP$ and $g \in G$. We denote it by $\pi^{{\rm{At}}(P)} \colon {\rm{At}}(P) \ra M$, see \Cref{Equation: Atiyah}.
\end{example}
\begin{example}\label{Example: Adjoint bundle}
	The \textit{adjoint bundle of a principal $G$-bundle} $\pi \colon P \ra M$ is defined as the quotient vector bundle of the trivial bundle $P \times L(G) \ra P $ by the adjoint action of $G$ i.e. $\big((p, A),g \big) \mapsto (pg, {\rm{ad}}_g^{-1}(A))$ for $p \in P$, $A \in L(G), g \in G$. We denote it by $\pi^{{\rm{Ad}}(P)} \colon {\rm{Ad}}(P) \ra M$, see \Cref{Equation: Atiyah}.
\end{example}

\begin{remark}
	Any associated vector bundle can be constructed as a quotient vector bundle. Precisely, given a principal $G$-bundle $\pi \colon P \ra M$ and linear representation $\rho$ of $G$ on a finite-dimensional vector space $V$, the associated bundle $\frac{E \times V}{G} \ra M$ is same as the quotient vector bundle of the product vector bundle $P \times V \ra P$ by the action of $G$ given by $ \big( (p,v),g \big) \mapsto (pg,  \rho(g^{-1})(v))$.		
\end{remark}

Next, we will construct the map $\delta^{/G} \colon {\rm{Ad}}(P) \ra {\rm{At}}(P)$ in \Cref{Equation: Atiyah}.  

Let $ \lbrace \pi_i \colon P_i \ra M_i \rbrace_{i=1,2}$ and $\lbrace \pi_i^{E_i} \colon E_i \ra P_i \rbrace_{i=1,2}$ be principal $G$-bundles and vector bundles with right action of $G$ respectively, such that we can construct their quotient bundles $\lbrace \pi^{E_i/G} \colon E_i/G \ra M_i \rbrace_{i=1,2}$ in the way as discussed above. Then given 
\begin{itemize}
	\item a morphism of principal $G$-bundles $(f^{\pi}_{P}, F^{\pi}_{M})$ from $\pi_1 \colon P_1 \ra M_1$ to $\pi_2 \colon P_2 \ra M_2$ and 
	\item a morphism of vector bundles $(f^{\pi^{E}}_{E}, f_{P}^{\pi^{E}})$ from $\pi_1^{E_1} \colon E_1 \ra P_1 $ to $\pi_2^{E_2} \colon E_2 \ra P_2 $
\end{itemize}
such that $f^{\pi^{E}}_{E}$ is $G$-equivariant,  there is a unique morphism of vector bundles $(f_{E/G}^{\pi^{E/G}}, f_M^{\pi})$ between the corresponding quotient vector bundles, such that the following diagram commutes: 
\[\begin{tikzcd}
	E_1  \arrow[d, "Q^{E_1}"'] \arrow[r, "f_E^{\pi^{E}}"]                 & E_2 \arrow[d, "Q^{E_2}"] \\
	E_1/G \arrow[r, "f_{E/G}^{\pi^{E/G}}"']  & E_2/G           
\end{tikzcd}\]

As a consequence, for a principal $G$-bundle $\pi \colon P \ra M$, we get the map 
\begin{equation}\nonumber
	\delta^{/G} \colon {\rm{Ad}}(P) \ra {\rm{At}}(P)
\end{equation}
in \Cref{Equation: Atiyah}, induced from the morphism of vector bundles $(\delta, 1_p) $ from $\pi^{P \times L(G)} \colon P \times L(G) \ra P$ to $\pi^{TP} \colon TP \ra P$, defined by $\delta(p,A):= \delta_p(A)$, where $\delta$ is the fundamental vector field on $P$ (see \Cref{Notation conventions}).

To construct the map $\pi_{*}^{/G} \colon {\rm{At}}(P) \ra TM$ in \Cref{Equation: Atiyah}, we will first recall a general fact about morphisms \textit{from} quotient vector bundles.

Given a principal $G$-bundle $\pi \colon P \ra M$, let $\pi^{E} \colon E \ra P$ be a vector bundle equipped with a right action of $G$ such that we can construct the quotient vector bundle $\pi^{E/G} \colon E/G \ra M$ in a way discussed before. Then, 
\begin{itemize}
	\item for any vector bundle $\pi^{E'} \colon E' \ra M'$ and 
	\item for any morphism of vector bundles $(\zeta, \eta)$ from $\pi^{E} \colon E \ra P$ to $\pi^{E'} \colon E' \ra M'$ such that both $\zeta$ and $\delta$ are $G$-invariant
\end{itemize}

there exists a unique morphism of vector bundles $(\zeta^{/G}, \eta^{/G})$ from the quotient vector bundle  $\pi^{E/G} \colon E/G \ra M$ to $\pi^{E'} \colon E' \ra M'$ such that the following diagrams commute:

\[\begin{tikzcd}
	P \arrow[r, "\pi"] \arrow[d, "\eta"'] & M \arrow[ld, "\eta^{/G}"] \\
	M'                              &                  
\end{tikzcd}
\begin{tikzcd}
	E \arrow[r, "Q^{E}"] \arrow[d, "\zeta"'] & E/G \arrow[ld, "\zeta^{/G}"] \\
	E'                             &                  
\end{tikzcd}\]
Hence, we readily obtain the map 
\begin{equation}\nonumber
	\pi_{*}^{G} \colon {\rm{At}}(P) \ra TM
\end{equation}
in the \Cref{Equation: Atiyah}, induced from the vector bundle morphism $(\pi_*, \pi)$ from the tangent bundle $TP \ra P$ on $P$ to the tangent bundle $TM \ra M$ on $M$ .
Now we have all the ingredients required to fill up the details of \Cref{Equation: Atiyah} and define the Atiyah sequence associated to $\pi \colon P \ra M$.
\begin{definition}\label{Definition: Atiyah sequence of a principal bundle}
	Given a principal $G$-bundle $\pi \colon P \ra M$, the short exact sequence of vector bundles over $M$
	
	\begin{equation}\nonumber
		\begin{tikzcd}
			0 \arrow[r, ""]  & {\rm{Ad}}(P) \arrow[r, "\delta^{/G}"] \arrow[d, ""'] & {\rm{At}}(P) \arrow[r, "\pi_{*}^{/G}"] \arrow[d, ""'] & TM \arrow[d, ""'] \arrow[r] & 0 \\
			0 \arrow[r, ""'] & M \arrow[r, "{\rm{id}}"']                & M \arrow[r, "{\rm{id}}"']                & M \arrow[r]                 & 0
		\end{tikzcd}
	\end{equation}
	is called the \textit{Atiyah sequence associated to the principal $G$-bundle $\pi \colon E \ra M$}.
\end{definition}
We denote the Atiyah sequence associated to $\pi \colon P \ra M$ by the notation ${\rm{At}}(\pi)$.

\subsection{Gauge group of a principal bundle}\label{subsection: gauge group of a principal bundle}
Recall that for a Lie group $G$ and a manifold $M$,  the collection of principal $G$-bundles over $M$ forms a groupoid ${\rm{Bun}}(M, G)$. Automorphism groups ${\rm{Hom}}(\pi \colon P \ra M, \pi \colon P \ra M )$ at each object $\pi \colon P \ra M$  of ${\rm{Bun}}(M, G)$ play significant roles in gauge theory.

\begin{definition}\label{Definoition: Gauge group of a principal bundle}
	The \textit{gauge group of a principal $G$-bundle $\pi \colon P \ra M$ over a manifold $M$ }is defined as the automorphism group of the object $\pi \colon P \ra M$ in the groupoid ${\rm{Bun}}(M, G)$.
\end{definition}
We will denote the gauge group of $\pi \colon P \ra M$ by the notation $\mc{G}(P)$. We will call the elements of $\mc{G}(P)$ as the \textit{gauge transformation of the principal $G$-bundle $\pi \colon P \ra M$.} The collection of smooth maps $\sigma \colon E \ra G$ satisfying the properties $$\sigma(pg)= g^{-1} \sigma(p)g$$ for all $p \in P$ and $g \in G$, forms a group under the point wise multiplication. We will denote this group by the notation $C^{\infty}(P, G)^{G}$. There is a group isomorphism $\mc{G}(P) \ra C^{\infty}(E, G)^{G}$ defined by $f \ra \sigma_F$ where for each $p \in P$,  $\sigma_f(p)$ is the unique element in $G$ such that $f(p)=p \sigma_f(p)$. Hence, we have $\mc{G}(P) \cong C^{\infty}(P, G)^{G}$.
\begin{remark}
	It is interesting to note that $\mc{G}(P)$ is usually an infinite-dimensional manifold.
\end{remark}
%

%

\section{Connection structures on a principal bundle}\label{Section: Connection structures on a principal bundle}
Connection structures on principal $G$-bundles are central objects of study in gauge theory. We begin this section by recalling the definition of a \textit{connection on a principal bundle} as a splitting of the Atiyah sequence, followed by a brief discussion on its equivalent characterizations as (i) an $L(G)$-valued 1-form and (ii) as a horizontal distribution. Finally, we end the section with a brief treatment of the action of the gauge group (\Cref{Definoition: Gauge group of a principal bundle}) on connections. Materials in this section are mostly standard and can be found in any standard textbook on differential geometry. Books such as \cite{MR1393940}, \cite{MR3837560}, and appendix A of \cite{MR896907} can be considered as references for this section.
\subsection{Connection on a principal bundle and its characterizations}\label{Subsection: Connection on a principal bundle and its characterizations} 

\begin{definition}\label{Definition: Connection on a principal bundle}
Let $\pi \colon P \ra M$ be a principal $G$-bundle. A \textit{connection on $\pi \colon P \ra M$} is a splitting of the associated Atiyah sequence $\rm{At}(\pi)$,
\begin{equation}\nonumber	
\begin{tikzcd}
	0 \arrow[r] & {\rm{Ad}}(P) \arrow[d] \arrow[r, "\delta^{/G}"] & {\rm{At}}(P) \arrow[d] \arrow[l, "R", bend left=49] \arrow[r, "\pi_{*}^{/G}"] & TM \arrow[r] \arrow[d] \arrow[l, "S", bend left=49] & 0 \\
	0 \arrow[r] & M \arrow[r, "\rm{id}"']          & M \arrow[r, "\rm{id}"']                                       & M \arrow[r]                                        & 0
\end{tikzcd}
\end{equation}	
\end{definition}
In otherwords, a \textit{retraction} of $\delta^{/G}$  i.e a morphism of vector bundles $ R \colon {\rm{At}}(P) \ra \rm{Ad}(P)$ such that $R \circ \delta^{/G}= {\rm{id}}$ or equivalently, a section $S \colon TM \ra {\rm{At}}(P)$ of $\pi_*^{/G}$.	

\begin{remark}
The existence of a connection on a principal bundle follows from the fact that a fibrewise surjective morphism of vector bundles over a  manifold admits a right-inverse.
\end{remark}

Let $R \colon {\rm{At}}(P) \ra {\rm{Ad}}(P)$ be a connection on a principal $G$-bundle $\pi \colon P \ra M$. Since we have 
\begin{equation}\nonumber
\pi^{{\rm{Ad}}(P)} \circ (R \circ Q^{TP})= \pi  \circ \pi^{TP},
\end{equation}
there exists a unique smooth map $\omega \colon TP \ra P \times L(G)$, such that it fits the following pull-back diagram in the usual way:

\[
\begin{tikzcd}
TP \arrow[rd, "\omega"', dotted] \arrow[rdd, "\pi^{TP}"', , bend right=49] \arrow[rrd, "R \circ Q^{TP} ", , bend left=49] &                                  &                  \\
& P \times L(G) \arrow[d, "\pi^{P \times L(G)}"'] \arrow[r, "Q^{P \times L(G)}"] & {\rm{Ad}}(P) \arrow[d, "\pi^{{\rm{Ad}}(P)}"] \\
& P \arrow[r, "\pi"']                & M              
\end{tikzcd}\]

One can show that the $L(G)$-valued 1-form $\omega \colon TP \ra P \times L(G)$ satisfy the following properties:
\begin{enumerate}
\item[(i)]$\omega(p)(A^{*}(p))=A$ for all $p \in P$ and $A \in L(G)$, where $A^{*}(p)$ is the fundamental vector field on $P$ evaluated at $p \in P$, (see \Cref{Notation conventions}).
\item[(ii)] $\omega(pg)(v \cdot g)= \rm{ad}(g^{-1})(\omega(p)(v))$ for all $g \in G, p \in P,v \in T_pP$. 
\end{enumerate}
Conversely, given a principal $G$-bundle $\pi \colon P \ra M$,  any  $L(G)$-valued 1-form $\omega \colon TP \ra L(G)$ satisfying the properties (i) and (ii) above,  canonically defines a connection given as
\begin{equation}\nonumber
\begin{split}
R \colon & {\rm{At}}(P) \ra {\rm{Ad}}(P)\\
& [(p,v)] \mapsto [(p, \omega(p)(v))].
\end{split}
\end{equation}

On the other hand, instead of a retraction $R \colon  {\rm{At}}(P) \ra {\rm{Ad}}(P)$, if we start with a section $S \colon TM \ra {\rm{At}}(P)$ of $\pi_*^{/G}$ in \Cref{Definition: Connection on a principal bundle}, then we get a distribution  $\mc{H}$ on $P$, which maps $p \in P$ to the subspace $\mc{H}_pP:=(Q^{TP})^{-1}(S(T_{\pi(p)}(M))) \subseteq T_pP$. This distribution satisfies the properties below:
\begin{enumerate}[(i)]
\item $T_pP = V_pP \oplus \mc{H}_pP$ for each $p \in P$;
\item $(\mc{H}_pP) \cdot g = \mc{H}_{pg}P$ for each $p \in P, g \in G$,
\end{enumerate}
where the subspaces $V_pP:=\rm{ker}(\pi_{*,p})$ and $\mc{H}_pP$ are called respectively the \textit{vertical subspace} and the \textit{horizontal subspace} of  $T_pP$ associated to the distribution $\mc{H}$. For any vector $v \in T_pP$, we denote by ${\rm{ver}}(v)$ and ${\rm{hor}}(v)$ for the components of $v$ lying in $V_pP$ and $\mc{H}_pP$ respectively.
\begin{itemize}
\item[(iii)] A vector field $X \colon P \ra TP$ is smooth if and only if the pair of vector fields ${\rm{ver}}(X) \colon P \ra TP$ and ${\rm{hor}}(X) \colon P \ra TP$ (defined by restriction of $X$ on vertical subspaces and horizontal subspaces) are smooth.
\end{itemize}
Conversely, let  $\mc{H}$ be a distribution on $P$ satisfying the properties (i),(ii) and (iii) above. Hence, we have a decomposition $TP= VP \oplus HP$ of the tangent bundle over $P$ into vertical and horizontal subbundles. This quotients to the decomposition ${\rm{At}}(P)=VP/G \oplus HP/G$. Since $\pi_{*} \colon {\rm{At}}(P) \ra TM$ is surjective and $\ker(\pi_*)= VP/G$, we get an isomorphism of vector bundles $\tilde{S} \colon HP/G \ra TM$,  whose inverse $S \colon TM \ra HP/G \subseteq {\rm{At}}(P)$ defines a splitting of the Atiyah sequence ${\rm{At}}(\pi)$. 

Summarising the whole discussion, we get the following three equiavlent definitions of a connection on a principal bundle over a manifold:

\subsection*{Three equivalent description of a connection structure on a principal bundle}\label{Equivalent characterisations of classical connection}
\begin{enumerate}[(i)]
\item Let $\pi \colon P \ra M$ be a principal $G$-bundle. A \textit{connection on $\pi \colon P \ra M$} is defined as a \textit{splitting} of ${\rm{At}}(\pi)$.
\begin{equation}\nonumber	
\begin{tikzcd}
	0 \arrow[r] & {\rm{Ad}}(P) \arrow[d] \arrow[r, "\delta^{/G}"] & {\rm{At}}(P) \arrow[d] \arrow[l, "R", bend left=49] \arrow[r, "\pi_{*}^{/G}"] & TM \arrow[r] \arrow[d] \arrow[l, "S", bend left=49] & 0 \\
	0 \arrow[r] & M \arrow[r, "\rm{id}"']          & M \arrow[r, "{\rm{id}}"']                                       & M \arrow[r]                                        & 0
\end{tikzcd}
\end{equation}

\item Let $\pi \colon P \ra M$ be a principal $G$-bundle. A \textit{connection on $\pi \colon P \ra M$} is defined as a $L(G)$-valued 1-form $\omega \colon TP \ra L(G)$ on $P$, satisfying the following propoerties:
\end{enumerate}

\begin{enumerate}
\item[(a)]$\omega(p)(A^{*}(p))=A$ for all $p \in P$ and $A \in L(G)$,
\item[(b)] $\omega(pg)(v \cdot g)$ for all $g \in G, p \in P,v \in T_pP$, where $\delta_g \colon P \ra P$, $p \mapsto pg$ is the right translation by $g$.

\item[(iii)] Let $\pi \colon P \ra M$ be a principal $G$-bundle. A \textit{connection on $\pi \colon P \ra M$} is defined as a distribution $\mc{H}$ on $P$, satisfying the following properties:

\begin{enumerate}
\item[(a)] $T_pP = V_pP \oplus \mc{H}_pP$ for each $p \in P$;
\item[(b)] $(\mc{H}_pP) \cdot g = \mc{H}_{pg}P$ for each $p \in P, g \in G$;
\item[(c)] A vector field $X \colon P \ra TP$ is smooth if and only if the pair of vector fields ${\rm{ver}}(X) \colon P \ra TP$ and ${\rm{hor}}(X) \colon P \ra TP$ (defined by restriction of $X$ on vertical subspaces and horizontal subspaces) are smooth.
\end{enumerate}
\end{enumerate}
The $L(G)$-valued 1-form $\omega \colon TP \ra L(G)$ in (ii) and the distribution $\mc{H}$ in (iii)  are called a \textit{connection 1-form} and an \textit{Ehresmann connection} respectively. Typically, conventional textbooks focus solely on connection 1-forms and Ehresmann connections. In contrast to our approach, the proofs demonstrating their equivalence often do not employ the Atiyah sequence version of the definition. To ensure a comprehensive treatment, here we briefly discuss the equivalence between (ii) and (iii). Readers can refer to standard textbooks such as \cite{MR1393940} for further elaboration.

(ii) $\Rightarrow$ (iii): Given a connection 1-form $\omega \colon TP \ra L(G)$, the assignment $\mc{H}_pP:= \ker(\omega(p))$ for $p \in P$, gives an Ehresmann connection.

(iii) $\Rightarrow$(ii): Given an Ehresmann connection $\mc{H}$, the associated  conection 1-from $\omega \colon TP \ra L(G)$ is given by $\omega(p,v):= i_p^{-1}({\rm{ver}}(v))$ for $p \in P, v \in T_pP$, where $i_p \colon L(G) \ra V_pP$ is the standard isomorphism defined as $A \mapsto A^{*}(p), A \in L(G)$.  

\begin{example}[Pullback-connection]\label{Pullbackconn}
Let $\pi \colon P \ra M$ and $\pi' \colon P' \ra M$ be a pair of principal $G$-bundles over a manifold $M$. Let $\omega \colon TP' \ra L(G)$ be a connection 1-form on $P'$. Then for any morphism of principal $G$ -bundles $f \colon P \ra P'$ over $M$ , $f^{*} \omega$ is a connection  1-form on $P$.
\end{example}
\begin{remark}
 The connection structures on principal bundles are not unique. Even for trivial principal bundles $M \times G \ra M$, this is the case when the dimension of $M$ and $G$ are greater than equal to one.
\end{remark}
Given a Lie group $G$ and a manifold $M$, the collection of principal $G$-bundles with connection structures over  $M$  form a groupoid $B^{{\rm{\nabla}}}G(M)$. To be more precise, objects of $B^{{\rm{\nabla}}}G(M)$ are pairs $(\pi \colon P \ra M, \omega \colon TP \ra L(G) )$ and a morphism from an object $( \pi \colon P \ra M, \omega \colon TP \ra L(G) )$ to another object $( \pi' \colon P' \ra M, \omega' \colon TP' \ra L(G) )$ is a \textit{connection preserving morphism of principal $G$-bundles}, i.e. a morphism of principal $G$-bundle $f \colon P \ra P'$ over $M$, satisfying $f^{*}\omega'= \omega$. \label{Groupoid of principal bundles with connection}
\subsection{Induced Ehresmann connection on the associated fibre bundle}\label{Induced Ehresmann connection on the associated bundle}

Recall, in \Cref{Associate bundle}, we discussed the construction of associated fibered bundles. Given an Ehresmann connection on a principal bundle, there is a natural way to induce a \textit{Ehresmann connection on its associated fiber bundles} i.e., decomposition of the tangent bundle of the total space into vertical and horizontal subbundles. 

Let $\mc{H}$ be an Ehresmann connection on a principal $G$-bundle $\pi \colon P \ra M$. Suppose $G$ acts from the left on a manifold $F$. We define the vertical subspace $V_{[p,f]}^{F}$ of the associated bundle $\pi^{F} \colon \frac{P \times F}{G} \ra M$ at a point $[p,f] \in \frac{P \times F}{G}$ as $\ker(\pi^{F}_{*,[p,f]})$. The horizontal subspace $\mc{H}_{[p,f]}^F$ of $\pi^{F}$ at  $[p,f]$ is defined as the image of the horizontal subspace $\mc{H}_pP \subseteq T_pP$ under the  mapping $P \ra \frac{P \times F}{G}$ defined as $q \mapsto [q,f]$. Then, we have the required direct sum decomposition of tangent spaces at each point as follows:
\begin{equation}\nonumber
T_{[p,f]} \frac{P \times F}{G} = V_{[p,f]}^{F} \oplus H_{[p,f]}^F.
\end{equation}
Hence, given an Ehresmann connection $\mc{H}$ on $\pi \colon P \ra M$, the\textit{ induced Ehresmann connection} on the  associated fibered bundle $\pi^{F} \colon \frac{P \times F}{G} \ra M$ is defined as the distribution $\mc{H}_{[p,f]}^{F}, [p,f] \in \frac{P \times F}{G}$. 

In particular, if $F$ is a finite-dimensional vector space $V$, then an induced Ehresmann connection on an associated vector bundle $\pi^{V}$ with respect to a linear representation $\rho$ of $G$ on $V$ is called a \textit{linear conection} on $\pi^{V}$. The linear connection has the additional property of being compatible with the linear structure of $V$ in the usual way.

\subsection{Action of gauge group on connections}\label{Classical action of gauge group on connections}
Given a principal $G$-bundle $\pi \colon P \ra M$, we denote the set of all connections on $\pi \colon P \ra M$ by the notation $\Omega_{P}$. There is a natural action of gauge group $\mc{G}(P)$ on $\Omega_{P}$ given by
\begin{equation}\label{Classical gauge group action}
\begin{split}
\rho \colon & \mc{G}(P) \times \Omega(P) \ra \Omega(P)\\
&(f,\omega) \mapsto \omega \circ f^{-1}_{*} \colon TP \ra L(G)
\end{split}
\end{equation}
In \Cref{subsection: gauge group of a principal bundle}, recall that we showed $\mc{G}(P)$ is isomoprhic to $C^{\infty}(P,G)^{G}$, the group of $G$-valued smooth maps $\sigma \colon P \ra L(G)$ satisfying $\sigma(pg)=g^{-1}\sigma(p)g$ for all $p 
\in P, g \in G$. Equivalent action of $C^{\infty}(P,G)^{G}$ on $\Omega(P)$ is given as
\begin{equation}\nonumber
\begin{split}
\tilde{\rho} \colon & C^{\infty}(P,G)^{G} \times \Omega(P) \ra \Omega(P)\\
& (\sigma,\omega) \mapsto {\rm{Ad}}_{\sigma}(\omega) -(d \sigma)\sigma^{-1}.
\end{split}
\end{equation}

\section{Parallel transport on a principal bundle}\label{Parallel transport on a principal bundle}
In \Cref{Section: Connection structures on a principal bundle}, we discussed connection structures on a principal bundle, but we did not specify anything about what it does to the bundle!! This section aims to give a brief account of one of its key roles. In particular, a connection on a principal bundle defines a rule that  \textit{appropriately} identifies its fibers along paths in the base space. Such \textit{appropriate} identification is known as the \textit{parallel transport of the connection}. The meaning of the word `\textit{appropriate}' is basically the content of this section! We begin by briefly recalling the notion of \textit{unique horizontal path lifting property of a connection}; see \cite{MR1393940} for a more detailed treatment.

\subsection{Parallel transport of a connection along a path}\label{subsection:Parallel transport of a connection along a path}

Let $\pi \colon P \ra M$ be a principal $G$-bundle equipped with a connection $\omega \colon TP \ra L(G)$. Consider a smooth path $\alpha \colon [0,1] \ra M$ on the base space $M$ and let $p \in \pi^{-1}(\alpha(0))$. By the local triviality of $\pi: P \ra M$ there is a path $\bar{\alpha}: [0,1] \ra P$ in $P$, such that $\bar{\alpha}(0)=p$ and $\pi \circ \bar{\alpha}(t)=\alpha(t)$ for all $t \in [0,1]$.  Now let $\tilde{\alpha}: [0,1] \ra E_0$ be an arbitrary lift of $\alpha$, such that $\tilde{\alpha}(0)=p$. Since $G$ acts on $P$ \textit{nicely}(see (i) of \Cref{Definition: Principal G-bundle}), there is a unique path $a:[0,1] \ra G$, such that $\tilde{\alpha}(t)= \rho \big(\bar{\alpha}(t), a(t) \big)$ for all $t \in [0,1]$ and $a(0)=e$, where $ \rho : P \times G \ra P$ is the action map. Note that if we find such a curve $a: [0,1] \ra G$ such that $\tilde{\alpha}$ is horizontal i.e. $\tilde{\alpha}'(t)$ lies in $H_{\alpha(t)}P$ for all $t \in [0,1]$, then we are done. Applying Leibniz's formula (see \textbf{Proposition 1.4} of \cite{MR1393940}) to the action map $\rho : P \times G \ra P$, we get
$$\tilde{\alpha}'(t)=  \rho_{{a(t)}_{*, \bar{\alpha}(t)}}(\bar{\alpha}'(t)) +   \rho_{{\bar{\alpha}(t)}_{*, a(t)}(a'(t))}$$ where $\rho_{a(t)}: P \ra P$ is defined as $p \mapsto \rho(p,a(t))$ and $\rho_{\bar{\alpha}(t)}: G \ra P$ is defined as $g \mapsto \rho(\bar{\alpha}(t),g)$. Note that $\rho_{\bar{\alpha}(t)}$ is same  as $\delta_{\tilde{\alpha}(t)} \circ L_{a(t)^{-1}}$, where $\delta_{\tilde{\alpha}(t)}: G \ra P$ is defined as $g \mapsto \tilde{\alpha}(t) g$ and $L_{a(t)^{-1}} : G \ra G$ is given by $g \mapsto a(t)^{-1} g$. Hence,
$$\rho_{{\bar{\alpha}(t)}_{*, a(t)}(a'(t))}= \delta_{{\tilde{\alpha}(t)}_{*,e}}(L_{{a(t)^{-1}}_{*,a(t)}}(a'(t))).$$ Hence, we get 
\begin{equation}\label{Equation:1WH}
	\tilde{\alpha}'(t)= \rho_{{a(t)}_{*, \bar{\alpha}(t)}}(\bar{\alpha}'(t)) + \delta_{{\tilde{\alpha}(t)}_{*,e}}(L_{{a(t)^{-1}}_{*,a(t)}}(a'(t)))
\end{equation}
Applying $\omega$ on both sides of \Cref{Equation:1WH} we get,
\begin{equation}\nonumber
	\omega_{\tilde{\alpha}(t)}(\tilde{\alpha}'(t))= {\rm{ad}}_{{a(t)^{-1}}_{*,e}}(\omega_{\bar{\alpha}(t)}(\bar{\alpha}'(t))) + L_{{a(t)^{-1}}_{*,a(t)}}(a'(t)).
\end{equation}
Hence, $\tilde{\alpha} \colon [0,1] \ra P$ is horizontal (i.e $\tilde{\alpha}'(t) \in H_{\tilde{\alpha}(t)}P \subseteq T_{\tilde{\alpha}(t)}P$ for all $t \in [0,1]$) if and only if we have 
\begin{equation}\nonumber
	L_{{a(t)^{-1}}_{*,a(t)}} \big(R_{a(t)_{*,e}}(\omega_{\bar{\alpha}(t)}(\bar{\alpha}'(t)))+ a'(t) \big) \big) =0
\end{equation}
Using the notational convention mentioned in \Cref{Notation conventions}, $\tilde{\alpha} \colon [0,1] \ra P$ is horizontal if and only if 
\begin{equation}\label{Existence and uniqueness of horizontal lift}
	a'(t)a(t)^{-1}=- \omega_{\bar{\alpha}(t)}(\bar{\alpha}'(t))
\end{equation}
Uniqueness and the existence of the \Cref{Existence and uniqueness of horizontal lift} follows from the general fact (See \cite{MR1393940} for the proof) that if $Y \colon [0,1] \ra L(G)$ is a path in $L(G)$, then there exists a unique path $b \colon [0,1]  \ra G$,  such that $b(0)=e$ and $b'(t)b(t)^{-1}=Y(t)$. 

Summarising, we get the so-called \textit{unique horizontal path lifting property of a connection}, given as
\begin{itemize}
	\item Given a connection $\omega \colon TP \ra L(G)$ on a principal $G$-bundle $\pi \colon P \ra M$,
	\item  a smooth path $\alpha \colon[0,1] \ra M$ and
	\item a point $p \in \pi^{-1}(\alpha(0))$,
\end{itemize}
there exists a unique path $\tilde{\alpha}_{p,\omega} \colon [0,1] \ra P$ in $P$ such that 
\begin{itemize}
	\item  $\pi \circ \tilde{\alpha}_{\omega}^{p}=\alpha$,
	\item $\tilde{\alpha}_{\omega}(0)^{p}=p$ and
	\item $(\tilde{\alpha}_{\omega}^{p})'(t) \in H_{\tilde{\alpha}^{p}_{\omega}(t)}P$,
\end{itemize}
for all $t \in [0,1]$. 

More is true here! 

For each path $\alpha \colon [0,1] \ra M$, we get a $G$-equivariant diffeomorphism 
\begin{equation}\label{Parallel transport aloing a path}
	\begin{split}
		{\rm{Tr}}_{\omega}^{\alpha} \colon & \pi^{-1}(\alpha(0)) \ra \pi^{-1}(\alpha(1))\\
		& p \mapsto \tilde{\alpha}_{\omega}^{p}(1),
	\end{split}
\end{equation}
which is known as the \textit{parallel transport of $\omega$ along the path $\alpha$.}

Next, we observe the behaviour of \Cref{Parallel transport aloing a path} with respect to connection preserving morphims of principal $G$-bundles (see the end of \Cref{Section: Connection structures on a principal bundle}), or to be more precise, we have the following:
\begin{proposition}\label{Classical Parallel transport and connection preserving morphism}
	Let $\pi \colon P \ra M$ and $\pi' \colon P' \ra M$ be a pair of principal $G$-bundles, endowed with connection 1-forms $\omega$ and $\omega'$ respectively. Suppose $f \colon P \ra P'$ is a morphism of principal bundles over $M$, satisfying $\omega= f^{*} \omega'$ (see the end of \Cref{Section: Connection structures on a principal bundle}). Then, for any path $\alpha \colon [0,1] \ra M$, we have the following:
	$$f|_{\pi^{-1}(y)} \circ {\rm{Tr}}_{\omega}^{\alpha}=  {\rm{Tr}}_{\omega'}^{\alpha}  \circ f|_{\pi^{-1}(x)},$$
	where $\alpha(0)=x$ and $\alpha(1)=y$.
\end{proposition}
\begin{proof}
	See \textbf{Lemma 3.11}, \cite{MR3521476}.
	\end{proof}

\subsection{ Parallel transport functor of a connection}\label{Subsection: Parallel transport functor of a connection}
In the previous subsection, we observed how the parallel transport map of a connection gives a prescription to identify the fibers of the principal bundle along paths in the base space. In this subsection, we give a brief account of its behavior with respect to the following operations on such paths: 
\begin{itemize}
	\item[(i)] \textit{Thin homotopy} of paths;
	\item[(ii)] Concatenation of paths.
\end{itemize}
To study such behavior, we need to briefly recall the notion of a `thin homotopy groupoid of the manifold', a refined version of the fundamental groupoid of a manifold. This notion is well-established and extensively covered in the existing literature. We refer \cite{MR3521476}, \cite{MR2520993} for a detailed treatement.
\subsection*{Thin fundamental groupoid of a smooth manifold}\label{Thin homotopy groupoid of a manifold} 
Before making the definition, we need to ensure that the concatenation of smooth paths in a manifold is smooth. To achieve this,  we restrict ourselves to only \textit{paths with sitting instants} i.e a smooth map $\alpha :[0,1] \rightarrow M$ to a manifold $M$, such that there exists an $\epsilon \in (0,1/2)$ satisfying $\alpha(t)= \alpha(0)$ for $t \in [0,\epsilon)$ and $\alpha(t)=\alpha(1)$ for $t \in (1- \epsilon, 1]$. In literature, often these paths are known by the name \textit{lazy paths} (for example see \cite{MR2825807}). We denote the set of lazy paths on a manifold $M$ by the notation $PM$.
\begin{definition}
	Let $\alpha , \beta \colon [0,1]  \ra M$ be elements of $PM$. Then $\alpha$ is said to be \textit{thin homotopic} to $\beta$ if there exists a smooth map $H:[0,1]^2 \rightarrow M$ with the following properties:
	
	\begin{itemize}
		\item[(i)] $H$ has \textit{sitting instants} i.e there exists an $\epsilon \in (0,1/2)$ with 
		\begin{itemize}
			\item[(a)] $H(s,t)=x$ for $t \in [0, \epsilon)$ and $H(s,t)=y$ for $t \in (1- \epsilon,1])$ and 
			\item[(b)] $H(s,t)= \alpha(t)$ for $s \in [0,\epsilon)$ and $H(s,t)= \beta(t)$ for $s \in (1- \epsilon, 1]$;
		\end{itemize}
		\item[(ii)] the differential of $H$ has a rank atmost 1 at all points.
	\end{itemize}
	The smooth map $H$ will be called a \textit{thin homotopy }from $\alpha$ to $\beta$.
\end{definition}
\begin{remark}
	Thin homotopy defines an equivalence relation $\sim$ on the set $PM$. We denote the quotient set by $\frac{PM}{\sim}$.
\end{remark}
\begin{definition}\label{Definition:Thin homotopy groupoid of a manifold}
For any smooth manifold $M$, the pair of spaces $M$ and $\frac{PM}{\sim}$ combine to form a \textit{groupoid} $\Pi_{{\rm{thin}}}(M):=[\frac{PM}{\sim} \rra M]$, known as the \textit{thin fundamental groupoid of the manifold $M$}, whose structure maps are given as follows:
\begin{itemize}
	\item Source map $s \colon \frac{PM}{\sim} \ra M$, $ [\alpha] \mapsto \alpha(0)$;
	\item Target map $t \colon \frac{PM}{\sim} \ra M$, $ [\alpha] \mapsto \alpha(1)$;
	\item Composition is defined by concatenation of paths;
	\item Unit map $u \colon M \ra \frac{PM}{\sim}$, $m \mapsto [c_{m}]$, where $c_m$ is the constant map at $m$;
	\item Inverse map $ \mathfrak{i} \colon \frac{PM}{\sim} \ra \frac{PM}{\sim}$, $[\alpha] \mapsto [\alpha^{-1}]$ (see \Cref{Notation conventions}).
\end{itemize}
\end{definition}

\subsection*{Functor from the thin fundamental groupoid}\label{subsectionFunctor from the thin homotopy groupoid}
The parallel transport map \Cref{Parallel transport aloing a path} behaves well with both concatenation of paths and thin homotopy. More precisely, 
\begin{itemize}
	\item[(i)]if $\alpha_2$ and $\alpha_1$ are two composable paths with sitting instants in the base space $M$, then ${\rm{Tr}}_{\omega}^{\alpha_2 * \alpha_1}= {\rm{Tr}}_{\omega}^{\alpha_2} \circ {\rm{Tr}}_{\omega}^{\alpha_1} $, where $*$ is the concatenation of paths;
	\item[(ii)] If $\alpha$, $\beta \in PM$, such that $\alpha$ is thin homotopic to $\beta$, then  \begin{equation}\label{thin homotopy invariance of classical transport}
		{\rm{Tr}}_{\omega}^{\alpha}= {\rm{Tr}}_{\omega}^{\beta}.
	\end{equation}
\end{itemize}
In turn, we get a functor 
\begin{equation}\label{Transport functor}
	\begin{split}
		T_{\omega} \colon & \Pi_{{\rm{thin}}}(M) \ra G {\rm{-}} {\rm{Tor}}\\
		& x \mapsto \pi^{-1}(x), x \in M\\
		& [\alpha] \mapsto {\rm{Tr}}_{\omega}^{\alpha}, \alpha \in \frac{PM}{\sim},
	\end{split}
\end{equation}
from the thin homotopy groupoid of the base space $M$ to the groupoid of $G$-torsors (see \Cref{Notation conventions}). The functor $T_{\omega}$ is called \textit{the parallel transport functor of the connection} $\omega$. Since the quotient space $\frac{PM}{\sim}$ has no natural finite-dimensional smooth manifold structure, to discuss the smoothness of $T_{\omega}$ we need the notion of `diffeology', a certain class of 
`generalized smooth spaces'. However, to maintain continuity and circumvent technical issues, we defer the discussion of the `smoothness condition' on the parallel transport functor until \Cref{Smoothness of traditional parallel transport}. Next, we see how a parallel transport on a principal bundle induces a notion of parallel transport on its associated bundles.

\subsection{Induced parallel transport on associated fibre bundles}\label{Induced parallel transport on associated fibre bundles}
A direct consequence of the induced Ehresmann connection (\Cref{Induced Ehresmann connection on the associated bundle}) is a notion of parallel transport on the associated fibre bundle. Let $\mc{H}$ be an Ehresmann connection on a principal $G$-bundle $\pi \colon P \ra M$. Suppose $G$ acts from the left on a manifold $F$ and let $\mc{H}^{F}$ be the induced Ehresmann connection on the associated bundle $\pi^{F} \colon \frac{P \times F}{G} \ra M$. Given a path $\alpha \colon [0,1] \ra M$ in $M$ and a point $[p,f]$ in  $(\pi^{F})^{-1}(x)$, it follows from the unique horizontal path lifting property of the connection $\mc{H}$ that there is a unique horizontal lift $\tilde{\alpha}_{\omega, F}^{[p,f]}(t)$ defined by $t \mapsto  [\tilde{\alpha}_{\omega}^{p}(t),f]$ of $\alpha$ in $\frac{E \times F}{G}$, such that $\tilde{\alpha}_{\omega, F}^{[p,f]}(0)=[p,f]$. Hence, given a path $\alpha \colon [0,1] \ra M$ in base space $M$, there is a smooth map ${\rm{Tr}}_{\omega,F}^{\alpha} \colon (\pi^{F})^{-1}(\alpha(0)) \ra (\pi^{F})^{-1}(\alpha(1))$ defined by $[p,f] \mapsto \tilde{\alpha}_{\omega,F}^{[p,f]}(1)= [{\rm{Tr}}_{\omega}^{\alpha}(p), f]$. The map ${\rm{Tr}}_{\omega,F}^{\alpha}$ is called the \textit{parallel transport on the associated bundle $\pi^{F}$ along the path $\alpha$ induced by the connection $\omega$ on $P$.} 

In particular, when $F$ is a finite-dimensional vector space $V$ with underlying scaler field $K$, then the induced linear connection (\Cref{Induced Ehresmann connection on the associated bundle}) on the associated vector bundle $\pi^{V}$ gives an analogue of \Cref{Transport functor}, that is a functor 
\begin{equation}\nonumber
	T_{\omega, V} \colon \Pi_{{\rm{thin}}}(M) \ra {\rm{Vect}}_{K}, 
\end{equation}
where ${\rm{Vect}}_{K}$ is the category of finite dimensional vector spaces with underlying scalar field $K$. 

With this, we end the chapter here. In the next chapter, we will build the necessary machinery to generalize the discussion made here in a categorified framework.

.

\chapter{Preliminaries}\label{A brief review of bundles over Lie groupoids and their gauge theory}


\lhead{Chapter 3. \emph{Preliminaries}}

The word \textit{categorification} was coined by Crane \cite{MR1355904, MR1295461}. Baez and Dolan have extensively discussed this in \cite{MR1664990}. According to \cite{MR1664990}, categorification is a way to find appropriate category-theoretic analogs of the existing concepts phrased in set-theoretic language. Usually, this is achieved by replacing sets, functions, and equations between functions by categories, functors, and natural isomorphisms between functors, respectively. Moreover, the choices of natural isomorphisms should satisfy some sort of equations known as \textit{coherence laws}. This process often transforms basic set-theoretic ideas into a more intricate and sophisticated version of the original idea, coinciding with the classical notion only in the simplest case. 
However, our intention is modest here. The purpose of this chapter is to briefly recall some well-established notions in the literature that we need for a particular kind of categorification done in this thesis to the ideas discussed in \Cref{Chapter Classical setup}.

\section{Some topics in category theory}
In this section, we briefly recall the notion of a \textit{2-category}, in particular a strict 2-category, a \textit{fibered category}, and a \textit{pseudofunctor}. Also, we outline the classic one-one correspondence between fibered categories and pseudofunctors (due to Grothendieck, \cite{grothendieck2004revetements}). We suggest \cite{MR1291599, leinster1998basic} as references for a detailed account on 2-categories and  \cite{MR2223406, MR1291599, MR4261588} for materials on fibered categories and pseudofunctors.
\subsection{2-categories}\label{2-categories}

We start by recalling the notion of a `strict 2-category'.

\begin{definition}\label{strict 2-category}[\textbf{Defintion 7.1.1}, \cite{MR1291599}]
	A \textit{ strict $2$-category}  $\mc{D}$, consists of the following:
	\begin{itemize}
		\item a class $|\mc{D}|$, \footnote{Although we use the term `class' here, a precise definition of a `class' requires involvement of foundations of set theory, which we have avoided intentionally to remain aligned with our purpose in this thesis. Informally, classes serve us a way to incorporate important `set-like' collections, while differing from sets, for avoiding paradoxes, especially Russel's paradox. For example, in commonly used categories Set, Grp, Top, the collection of objects is not a set, but a `class'. Readers interested in the foundations can look at standard references like \cite{MR1712872, MR1291599, MR2240597} for an elaborate treatement. However, assuming that the reader is familar with the definition of a category, for the sake of completeness, one can view any class $C$ as a discrete category $[C \rra C]$, whose only morphisms are identity arrows.}
		\item a small category $\mc{D}(x,y)$ for each pair of elements $x,y$ in $|\mc{D}|$,
		\item a functor $u_{x} \colon 1 \ra \mc{D}(x,x)$ for each element $x$ in $|\mc{D}|$, where $1$ is the terminal object in \textbf{Cat}, the category of small categories,
		\item a functor $c_{xyz}:\mc{D}(x,y)\times\mc{D}(y,z)\ra \mc{D}(x,z)$ for each triple $x,y,z$ of elements in $|\mc{D}|$,
		
		such that the following axioms are satisfied:
		\item \textit{Associativity axiom:} Given four elements $x,y,z,w$ in $|\mc{D}|$, the following diagram commutes on the nose:
		\begin{equation}\label{associativity axiom of strict 2-category}
			\begin{tikzcd}
				\mc{D}(x,y) \times \mc{D}(y,z) \times \mc{D}(z,w) \arrow[d, "c_{xyz} \times {\rm{id}}"'] \arrow[rrr, "{\rm{id}} \times c_{yzw}"] &  &  & \mc{D}(x,y) \times \mc{D}(y,w) \arrow[d, "c_{xyw}"] \\
				\mc{D}(x,z) \times \mc{D}(z,w) \arrow[rrr, "c_{xzw}"']                &  &  & \mc{D}(x,w)               
			\end{tikzcd}
		\end{equation}
		\item \textit{Unit Axiom:} Given two elements $x,y \in |\mc{D}|$, the following diagram commutes on the nose:
		\begin{equation}\label{unit axiom of strict 2-category}
			\begin{tikzcd}
				&  & \mc{D}(x,y) \arrow[lld, "\cong"'] \arrow[rrd, "\cong"] \arrow[ddd] &  &                    \\
				1 \times \mc{D}(x,y) \arrow[d, "u_{x} \times {\rm{id}}"']   &  & {} \arrow[d, "{\rm{id}}"]                               &  & \mc{D}(x,y) \times 1 \arrow[d, "{\rm{id}} \times u_{y}"]   \\
				\mc{D}(x,x) \times \mc{D}(x,y) \arrow[rrd, "c_{xxy}"'] &  & {}                                               &  & \mc{D}(x,y) \times \mc{D}(y,y) \arrow[lld, "c_{xyy}"] \\
				&  & \mc{D}(x,y) \arrow[uu]                                     &  &                   
			\end{tikzcd}
		\end{equation}
	\end{itemize}
\end{definition}

\subsection*{Some terminologies:}

Given a strict 2-category $\mc{D}$,
\begin{itemize}
	\item the elements of $|\mc{D}|$ are called \textit{objects} of $D$;
	\item for each pair of elements $x,y$ in $|\mc{D}|$, the objects and morphisms of the category $D(x,y)$ are called \textit{1-morphisms} and \textit{2-morphisms} of $D$ respectively;
	\item for any pair of elements $x,y$ in $|\mc{D}|$, a composition of 2-morphisms in the category $D(x,y)$ is called a \textit{vertical composition} and is denoted as $\circ$, the usual notation of composition in a category;
	\item For $x,y,z \in |\mc{D}|$, if $\alpha$ and $\beta$  are 2-morphisms in  $\mc{D}(x,y)$ and $\mc{D}(y,z)$ respectively, then we call the image $c_{xyz}(\alpha, \beta) \in \mc{D}(x,z)$ as the \textit{horizontal composition} of $\alpha$ and $\beta$, denoted by the notation $\beta \circ_h \alpha$. If $f$ and $g$ are objects of $\mc{D}(x,y)$ and $\mc{D}(y,z)$ respectively, then we denote the image $c_{xyz}(f,g)$ by the notation $g \circ f$, the same notation as vertical composition and the distinction will be understood from the context;
	\item For each $x \in |\mc{D}|$, $1_x$ denote the image $u_x(1) \in \mc{D}(x,x)$. The notation ${\rm{id}}_{1_x}$ denote the unit of $1_x$ in the category $\mc{D}(x,x)$;
\end{itemize}
\begin{remark}[Underlying category of a strict 2-category]\label{underlying category of strict 2-category}
	The collection of objects and  1-morphisms of $\mc{D}$ forms a category with $1_x$ as the identity morphism for each $x \in |\mc{D}|$. It is called the \textit{underlying category of} $\mc{D}$, which will also be denoted by $\mc{D}$, and the distinction should be understood from the context.
\end{remark}
For $x,y,z \in |\mc{D}|$,  let $f,g,h$ and $f',g',h'$ are 1-morphisms in $ \mc{D}(x,y)$ and $\mc{D}(y,z)$ respectively. Let $\alpha \colon f \Rightarrow g$, $\beta \colon g \Rightarrow h$ and  $\alpha' \colon f' \Rightarrow g'$, $\beta'  \colon g' \Rightarrow h'$ be 2-morphisms in $\mc{D}(x,y)$ and $\mc{D}(y,z)$ respectively. Then by functoriality of $c_{xyz}$ we have
\begin{equation}\nonumber
	(\beta' \circ_h \beta) \circ (\alpha' \circ_h \alpha)= (\beta' \circ \alpha')o_h (\beta \circ \alpha),
\end{equation}
which is called the \textit{interchange law} of $\mc{D}$.
\begin{definition}\label{Strict 2-groupoid}[Strict 2-groupoid] A \textit{strict 2-groupoid} is a strict 2-category $D$ such that 
	\begin{itemize}
		\item the underlying 1-category $C$ is a groupoid;
		\item For each $x,y \in |D|$, $D(x,y)$ is a groupoid.
	\end{itemize} 
\end{definition}
\begin{remark}\label{Weak 2-category}
	If we weaken \Cref{strict 2-category}
	by replacing the `commutativity on the nose' of the diagrams \Cref{associativity axiom of strict 2-category} and \Cref{unit axiom of strict 2-category} by `commutativity up to a choice of natural isomorphism' which being a part of the data, satisfying certain coherence conditions, then we obtain the notion of a \textit{weak 2-category} or a \textit{bicategory}. For our objectives, we will primarily be involved in strict 2-categories, we leave the details and refer \textbf{Definition 7.7.1} of \cite{MR1291599}
	for the precise definition of a bicategory.
\end{remark}
\begin{example}
	The prototypical example of a strict 2-category is Cat, whose objects are small categories, 1-morphisms are functors, and 2-morphisms are natural transformations. The usual vertical and horizontal compositions of natural transformations give the vertical and horizontal compositions of 2-morphisms.
\end{example}

We will see several examples of strict 2-categories throughout the course of the thesis.

\subsection{Fibered categories, pseudofuntors and Grothendieck construction}\label{subsection Fibered categories}
In this subsection, we recall the notion of a \textit{fibered category over a category}, a \textit{pseudofunctor over a category} and outline the well-known one-one correspondence (due to  Grothendieck) between them.

A category $\mathcal{E}$ endowed with a functor $\pi: \mathcal{E} \ra \mathcal{X}$ is known as \textit{category over $\mathcal{X}$.} A morphism $\delta \colon p \ra q$ in $\mathcal{E}$ is said to be \textit{cartesian} if for any morphism $\delta' \colon p' \ra q$ in $\mathcal{E}$ and any morphism $\gamma \colon \pi(p') \ra \pi(p)$ in $\mathcal{X}$ with $\pi(\delta')=\pi(\delta) \circ \gamma$ , there is a unique morphism $\tilde{\gamma}:p' \ra p$ satisfying $\pi(\tilde{\gamma})=\gamma$ and $\delta \circ \tilde{\gamma}= \delta'$.

The following is a standard property of cartesian morphisms. (see \textbf{Proposition 3.4}, \cite{MR2223406})

\begin{lemma}\label{lemma: Groupoid Cartesian}
	Suppose $\pi \colon \mc{E} \ra \mc{X}$ is a catoegory over $\mc{X}$. Let $\gamma$ be an arrow in $\mc{E}$ such that $\pi_1(\gamma)$ is invertible. Then  $\gamma$ is cartesian if and only if $\gamma$ is invertible.
 	\end{lemma}

\begin{definition}\label{Pullback cartesian}
	For a category $\pi \colon \mc{E} \ra \mc{X}$ over $\mc{X}$, let  $\gamma \colon x \ra y$ be an arrow in $\mc{X}$ and $q \in \pi_0^{-1}(y)$. Then for any cartesian arrow $\delta$ in $\mc{E}$ such that $t(\delta)=p$ and $\pi_1(\delta)=\gamma$, we say $s(\delta)$ as the \textit{pullback of $q$ along $\gamma$} and we denote it by $\gamma_{\delta}^{*}(q)$.
\end{definition}

\begin{remark}\label{Uniqueness of cartesian pullback}
	From the definition of a cartesian morphism above, it is evident that in a category $\pi \colon \mc{E} \ra \mc{X}$ over $\mc{X}$, for a given arrow $\gamma \colon x \ra y$ in $\mc{X}$ and a given point $q \in \pi_0^{-1}(y)$, $\gamma_{\delta}^{*}(q)$ is unique upto a unique isomorphism, where $\delta$ is a cartesian morphism in $\mc{E}$ such that $\pi_1(\delta)=\gamma$ and $t(\delta)=q$.
\end{remark}

\begin{definition}\label{Definition of fibered categories}
	A category $\pi: \mathcal{E} \ra \mathcal{X}$ over $\mc{X}$ is called a \textit{ fibered category} or \textit{fibration over $\mc{X}$} if for any morphism $\gamma$ in $\mc{X}$ and an object $p$ in $\mc{E}$ satisfying $\pi(p)=t(\gamma)$, there exists a cartesian morphism $\tilde{\gamma} \in \mc{E}$ such that $\pi(\tilde{\gamma})= \gamma$ and $t(\tilde{\gamma})=p$.
\end{definition}
\begin{definition}\label{Definition of fibre}
	Let $\pi: \mathcal{E} \ra \mathcal{X}$ be a fibered category over $\mc{X}$. For each object $x$ in $\mc{X}$, there is a subcategory of $\mc{E}$, whose objects are the objects $p$ of $\mc{E}$ such that $\pi_0(p)=x$ and arrows consist of morphisms $\delta$ of $\mc{E}$ such that $\pi_1(\delta)=1_x$. Such a subcategory is called the \textit{fiber over $x$} and is denoted by $\pi^{-1}(x)$.
\end{definition}
\begin{definition}\label{Definition cleavage}
	A \textit{cleavage} on a fibered category $\pi: \mathcal{E} \ra \mathcal{X}$  is defined as a function $K \colon {\rm{Mor}}( \mc{X}) \times_{t, {\rm{Obj}}(\mc{X}), \pi_0} {\rm{Obj}}(\mc{E}) \ra {\rm{Cart}}(\mc{E})$ such that $\pi_1(K(\gamma,p)) = \gamma$ and $t(K(\gamma,p))=p$ for all $(\gamma,p) \in {\rm{Mor}}( \mc{X}) \times_{t, {\rm{Obj}}(\mc{X}), \pi_0}{\rm{Obj}}(\mc{E})$, where ${\rm{Cart}}(\mc{E})$ is the set of all cartesian morphisms in $\mc{E}$.
	
\end{definition}

\begin{definition}\label{Definition Colven fibration}
	A fibered category $\pi \colon \mc{E} \ra \mc{X}$ equipped with a cleavage $K$ is called a\textit{ cloven fibered category} or a \textit{cloven fibration over $\mc{X}$} and we denote it by a pair $(\pi \colon \mc{E} \ra \mc{X}, K)$.
\end{definition}	
\begin{definition}\label{Definition splitting cleavage}
	A cleavage $K$ on a fibered category $\pi: \mathcal{E} \ra \mathcal{X}$  is called a \textit{splitting cleavage} if it has all the identities and is closed under composition. More precisely, $K$ is a splitting cleavage if it satisfies the following two conditions:
 \begin{enumerate}[(i)]
\item $K(1_x,p)=1_p$  for any $x\in {\rm{Obj}}(\mc{X})$ and $p\in \pi_{0}^{-1}(x)$;
\item if $(\gamma_2, p_2), (\gamma_1, p_1) \in {\rm{Mor}}( \mc{X}) \times_{t, {\rm{Obj}}(\mc{X}), \pi_0} {\rm{Obj}}(\mc{E})$ such that ${s}(\gamma_2)={t}(\gamma_1)$ and $p_2=s\bigl(K(\gamma_1, p_1)\bigr),$ then $K(\gamma_2 \circ \gamma_1 , p_1)= K(\gamma_1, p_1) \circ K(\gamma_2, p_2)$.

 \end{enumerate}

\end{definition}
\begin{definition}
	A cloven fibration $(\pi \colon \mc{E} \ra \mc{X}, K)$ is said to be a \textit{split fibered category} or a \textit{split fibration over $\mc{X}$} if the cleavage $K$ is splitting.
\end{definition}

	\begin{definition}\label{Definition of morphism of fibered categories}
		Let $\pi \colon \mc{E} \ra \mc{X}$ and $\pi' \colon \mc{E}' \ra \mc{X}$ be two fibered categories over the category $\mc{X}$. Then a \textit{morphism from $\pi \colon \mc{E} \ra \mc{X}$  to $\pi' \colon \mc{E}' \ra \mc{X}$ }is defined as a functor $F \colon \mc{E} \ra \mc{E}'$ such that $\pi' \circ F= \pi$ and it maps cartesian moprphisms to cartesian morphims.
	\end{definition}

Next, we recall the notion of a pseudofunctor over a category.

\begin{definition}\label{pseudofunctor}
	A \textit{pseudofunctor $\mc{F} \colon \mc{X}^{\rm{op}} \ra \rm{Cat}$ over a category $\mc{X}$} consists of the  following data:
	\begin{itemize}
		\item[(i)] For each $x \in \mc{X}$, we have a category $\mc{F}(x)$;
		\item[(ii)] For each morphism $x \xrightarrow {\gamma} y$ in $\mc{X}$, we have a functor  $\gamma^{*}: \mc{F}(y) \ra \mc{F}(x)$;
		\item[(iii)] A natural isomorphism $I_x: id^{*}_x \Longrightarrow {\rm{id}}_{\mc{F}(x)}$ for each object  $x$ in $\mc{X}$;
		\item[(iv)] A natural isomorphism $\alpha_{\gamma_1,\gamma_2}: \gamma_1^{*} \gamma_2^{*} \Longrightarrow (\gamma_2 \gamma_1)^{*}$ for each pair of composable morphisms 
		\begin{tikzcd}
			x \arrow[r, "\gamma_{1}"] & y \arrow[r, "\gamma_2"] & z
		\end{tikzcd},
		where the adjacency of $\gamma_2, \gamma_1$ denotes the composition;
	\end{itemize}
	The above data satisfies the following \textit{coherence laws}:
	\begin{itemize}
		\item[(i)] For any morphism $x \xrightarrow {\gamma} y$ in $\mc{X}$, and an object $p$ in $F(y)$, we have
		\begin{equation}\label{Coherence diagram 1}
			\begin{split}
				& \alpha_{{\rm{id}}_x, \gamma}(p)= I_{x}(\gamma^{*}(p)) \colon {\rm{id}}^{*}_{x}(\gamma^{*}(p)) \Longrightarrow \gamma^{*}(p),\\
				& \alpha_{\gamma, {\rm{id}}_{y}}(p)= \gamma^{*}(I_{y}(p)) \colon \gamma^{*}{\rm{id}}^{*}_{y}(p) \Longrightarrow \gamma^{*}(p).
			\end{split}
		\end{equation}
		\item[(ii)] For any composable sequence of morphisms of the form
		\begin{tikzcd}
			x \arrow[r, "\gamma_3"] & y \arrow[r, "\gamma_2"] & z \arrow[r, "\gamma_1"] & w
		\end{tikzcd}
		and an object $p$ in $F(x)$, the following diagram commutes
		\begin{equation}\label{Coherence diagram 2}
			\begin{tikzcd}
				\gamma_3^{*}\gamma_2^{*}\gamma_1^{*}(p) \arrow[d, "\gamma_3^{*}\big( \alpha_{\gamma_2,\gamma_1}(p) \big)"'] \arrow[r, "\alpha_{\gamma_3,\gamma_2}(\gamma_1^{*}(p))"] & (\gamma_2  \gamma_3)^{*}\gamma_1^{*} (p) \arrow[d, "\alpha_{ \gamma_2 \gamma_3,\gamma_1}(p)"] \\
				\gamma_3^{*}(\gamma_1  \gamma_2)^{*} (p) \arrow[r, "\alpha_{\gamma_3,  \gamma_1  \gamma_2 }(p)"']                & (\gamma_1  \gamma_2  \gamma_3)^{*} (p)               
			\end{tikzcd}
		\end{equation}
	\end{itemize}
\end{definition}
An exciting feature of a pseudofunctor $\mc{F} \colon \mc{X}^{{\rm{op}}} \ra {\rm{Cat}}$ is the fact that one can encode its whole data in a fibered category over $\mc{X}$. This encoding, i.e., construction of the fibered category from a pseudofunctor, is often called  \textit{Grothendieck construction} (for example, see \textbf{Chapter 10} of \cite{MR4261588}). On the other hand, given a fibered category over $\mc{X}$, there is a canonical way of constructing a pseudofucntor over $\mc{X}$. This association, along with the Grothendieck construction, defines a one-one correspondence between fibered categories over $\mc{X}$ and pseudofunctors over $\mc{X}$. While the correspondence is entirely conventional, we will briefly outline it below. This presentation differs slightly from the existing ones and is tailored to the objectives of our thesis.

\subsection*{Construction of a pseudofunctor from a fibered category:}\label{Construction of a pseudofunctor from a fibered category}
	
Consider a fibered category $\pi \colon \mc{E} \ra \mc{X}$ over $\mc{X}$. Let us choose a cleavage $$K \colon {\rm{Mor}}( \mc{X}) \times_{t, {\rm{Obj}}(\mc{X}), \pi_0}{\rm{Obj}}(\mc{E}) \ra {\rm{Cart}}(\mc{E}).$$ Then the following data defines a pseudofunctor $\mc{F} \colon \mc{X}^{{\rm{op}}} \ra {\rm{Cat}}$ over $\mc{X}$:
\begin{itemize}
	\item[(i)] each $x \in {\rm{Obj}}(\mc{X})$ is assigned to the fibre $\pi^{-1}(x)$ over $x$;
	\item[(ii)] each arrow $x \xrightarrow {\gamma} y$ in $\mc{X}$ is assigned to the functor 
	\begin{equation}\nonumber
		\begin{split}
			\gamma^{*} \colon & \pi^{-1}(y) \ra \pi^{-1}(x)\\
			& q \mapsto \gamma_{K(\gamma,q)}^{*}(q),\\
			& (p\xrightarrow{\delta}q) \mapsto \gamma^{*}(\delta),
		\end{split}
	\end{equation}
	where $\gamma^{*}(\delta)$ is the unique arrow in $\pi^{-1}(x)$ such that the following diagram commutes:
	\begin{equation}\nonumber
		\begin{tikzcd}
			\gamma_{K(\gamma,p)}^{*}(p) \arrow[d, "\gamma^{*}(\delta)"'] \arrow[r, "K(\gamma {\rm{,}}p)"] & p \arrow[d, "\delta"] \\
			\gamma_{K(\gamma,q)}^{*}(q)  \arrow[r, "K(\gamma {\rm{,}}q)"']                & q           
		\end{tikzcd}
	\end{equation}
	\item[(iii)] For each $x \in {\rm{Obj}}(\mc{X})$, we have a natural isomorphism $I_x \colon {\rm{id}}^{*}_{x} \Longrightarrow {\rm{id}}_{\pi^{-1}(x)}, p \mapsto K(1_x,p)$ for all $p \in \pi^{-1}(x)$,
	\item[(iv)]  For each pair of composable morphisms 
	\begin{tikzcd}
		x \arrow[r, "\gamma_1"] & y \arrow[r, "\gamma_2"] & z
	\end{tikzcd} in $\mc{X}$ we have a natural isomorphism $\alpha_{\gamma_1,\gamma_2}: \gamma_1^{*} \gamma_2^{*} \Longrightarrow (\gamma_2 \circ \gamma_1)^{*}$, defined as $p \mapsto \alpha_{\gamma_1,\gamma_2}(p)$ where $\alpha_{\gamma_1,\gamma_2}(p)$ is the unique isomorphism (by \Cref{Uniqueness of cartesian pullback}) from $(\gamma_2 \circ \gamma_1)^{*}_{K \big(\gamma_2,p \big) \circ K\big(\gamma_1, \gamma_{2_{K(\gamma_2,p)}}^{*}(p)\big)}(p)$ to $(\gamma_2 \circ \gamma_1)_{K(\gamma_2 \circ \gamma_1,p)}^{*}(p)$. 
\end{itemize}			

\subsection*{Construction of a fibered category from a pseudofunctor (Grothendieck Construction):} \label{Classical Grothendieck construction}

Consider a pseudofunctor $\mc{F} \colon \mc{X}^{{\rm{op}}} \ra {\rm{Cat}}$ over a category $\mc{X}$. We construct a category $\mc{\mc{E}}:=[E_1 \rra E_0]$, whose object set is $$E_0:= \bigl\{ (p,x) : x \in {\rm{Obj}}(\mc{X}), p \in \mc{F}(x) \bigr\},$$ and the morphism set is $$E_1:= \bigl\{ \big( (p, x) \xrightarrow{(\Gamma, \gamma)} (q, y)  \big) : \, x \xrightarrow{\gamma} y, \, p \xrightarrow{\Gamma} \gamma^{*}(q) \bigr\}.$$
For a composable pair 
\begin{tikzcd}
	(p {\rm{,}} x) \arrow[r, "(\Gamma_1 {\rm{,}} \gamma_1)"] & (q {\rm{,}} y) \arrow[r, "(\Gamma_2 {\rm{,}} \gamma_2)"] & (r {\rm{,}} z)
\end{tikzcd}  of elements in $E_1$, the composition is defined as $(\Gamma_2, \gamma_2) \circ (\Gamma_1, \gamma_1):= \big( \alpha_{\gamma_1,\gamma_2}(r) \circ \gamma_{1}^{*}(\Gamma_2) \circ \Gamma_1, \gamma_2 \circ \gamma_1 \big)$ and the unital map is given by $(p,x) \mapsto (I_{x}^{-1}(p), 1_x)$. Verifying that $\mc{\mc{E}}$ is indeed a category is straightforward. 

The obvious projection map
\begin{equation}\nonumber
	\begin{split}
		\pi \colon & \mc{E} \ra \mc{X}\\
		& (p,x) \mapsto x, \\
		& (\Gamma, \gamma) \mapsto \gamma,
	\end{split}
\end{equation}
define a fibered category $\pi \colon \mc{E} \ra \mc{X}$ over $\mc{X}$ equipped with a cleavage
\begin{equation}\nonumber
	\begin{split}
		K \colon & {\rm{Mor}}( \mc{X}) \times_{t, {\rm{Obj}}(\mc{X}), \pi_0} {\rm{Obj}}(\mc{E}) \ra {\rm{Cart}}(\mc{E})\\ & (\gamma, (p,x)) \mapsto ({\rm{id}}_{{\gamma}^{*}(p)}, \gamma).
	\end{split}
\end{equation}

The two constructions above are inverses of each other up to an isomorphism, which leads to the following proposition:

\begin{proposition}\label{Fibered and Pseudo}
	Given a category $\mc{X}$, there is a one-one correspondence between fibered categories over $\mc{X}$ and pseudofunctors over $\mc{X}$ up to isomorphisms.
\end{proposition}

In particular, if $K$ above is a splitting cleavage, the association above restricts between split fibered categories and Cat-valued contravariant functors over the category $\mc{X}$. To be more specific, we have the following characterization:
\begin{proposition}\label{Split and functor}
	The pseudofunctor associated to a cloven fibration is a functor if and only if the cleavage is splitting.
\end{proposition}

\begin{remark}
Respective collection of fibered categories and pseudofunctors naturally form strict 2-categories (\Cref{strict 2-category}) and the above one-one correspondence (\Cref{Fibered and Pseudo}) infact extends to a 2-equivalence of 2-categories. We recommend \textbf{Theorem 8.3.1}, \cite{MR1313497} for readers interested in such a general treatement.
	\end{remark}

\begin{example}[Sujective group homomorphism]
	Considering groups as one object categories, every surjective group homomorphism $\phi \colon H \ra G$ is a fibered category over $G$. Any section $\psi \colon G \ra H$ of $\phi$ is a cleavage and is a splitting cleavage if and only in addition $\psi$ is a group homomorphism. But, one should note that such a splitting cleavage may not always exist.
\end{example}

\begin{example}[Arrow category over a category]
	Let $\mc{C}$ be a category with pullbacks. Consider the \textit{arrow category ${\rm{Arr}}(C)$ of $C$}, whose objects are morphisms of $C$ and an arrow from $a \xrightarrow{f} b$ to $a' \xrightarrow{f'} b'$ is a commutative diagram
	\begin{equation}\label{mor Arr}
		\begin{tikzcd}
			a  \arrow[d, "f"'] \arrow[r, ""]                 & a^{{\rm{'}}} \arrow[d, "f^{{\rm{'}}}"] \\
			b \arrow[r, ""']  & b^{{\rm{'}}}           
		\end{tikzcd}
	\end{equation}
	with obvious composition law. Then the functor $\pi _{C} \colon {\rm{Arr}}(C) \ra C$ sending an object $(a \xrightarrow{f} b)$ to its target $b$ and a morphism \Cref{mor Arr} to its bottom arrow $b \ra b'$ defines a fibered category, whose cartesian morphisms are precisely the cartesian squares.					
\end{example}
\begin{example}[Classifying stack of a Lie group]\label{Classifying stack of a Lie group}
	For a Lie group $G$, there is a category $BG$, whose objects are principal $G$-bundles  $P \ra M$ and arrows are morphisms of principal bundles. The projection functor $BG \ra {\rm{Man}}$, that sends a principal bundle $P \ra M$ to $M$ and a morphism 
	\[
	\begin{tikzcd}
		P \arrow[d, "\pi"'] \arrow[r, "f_{P}"]                 & P' \arrow[d, "\pi'"] \\
		M \arrow[r, "f_M"']  & M'           
	\end{tikzcd}\]
	to $f_M$, defines a fibered category over Man, and is known by the name \textit{classifying stack of the Lie group $G$}.
\end{example}

	\begin{remark}\label{Remark on classifying stack}
		A \textit{generalized version} of the above example in particular, provides the association between a \textit{Lie groupoid} and a \textit{differentiable stack}. A `stack on a Grothendieck site' categorify the traditional notion of a sheaf on a topological space, while it must have already been noticed that a pseudofunctor over a category (hence, a fibered category with a cleavage) categorifies the notion of a presheaf. In this sense, a \textit{morphism of stacks} is a just a morphism of fibered categories as defined in \Cref{Definition of morphism of fibered categories}. We suggest \cite{MR2817778} and \cite{MR2778793} for readers interested in a general treatement of stacks.	 
\end{remark}

\section{Lie groupoids}\label{Lie groupoids}
The concept of a Lie groupoid, introduced in the works of Charles Ehresmann (\cite{MR0116360}, \cite{MR0197529}) in the 1960s, represents a categorification of the notion of a smooth manifold.
In particular, these objects encompass smooth manifolds and Lie groups as their unified generalizations. From a more abstract point of view, it serves as a geometric object representing a differentiable stack, \cite{MR2817778}. Specifically, a differentiable stack is just a Morita equivalent class of Lie groupoids. Nowadays, these objects are integral parts of higher gauge theory literature, as evident from the works like \cite{MR2805195}, \cite{MR2270285}, \cite{MR3894086}, \cite{MR3917427}, \cite{MR3521476}, \cite{MR4592876} and many others.
The breadth of these entities is quite extensive, as demonstrated by their intimate relations with Poisson geometry (through the works of A. Weinstein et al (for example see \cite{MR0866024}, \cite{MR1103911})), quantization (\cite{MR1103911}, \cite{MR2417440}), generalized complex geometry (\cite{MR4612595}), Dirac structures (\cite{MR2068969}), to name a few.

Nevertheless, our approach to Lie groupoids in this section is modest and tailored to our specific objectives. The content covered here is mainly conventional, and we recommend consulting \cite{MR2157566}, \cite{MR2012261}, and \cite{MR2778793} for further details.

. 

\subsection{Basic definitions, properties and examples}\label{Basic definitions}
\begin{definition}[Lie groupoid, Definition 1.1.3, \cite{MR2157566} ]\label{Lie groupoids definition}
	A \textit{Lie groupoid} is a groupoid $\mb{X}:=[X_1 \rra X_0]$, where $X_1$ and $X_0$ are smooth manifolds, the source and target maps $s,t \colon X_1 \ra X_0$ are submersions and all other structure maps, unit $u \colon X_0 \ra X_1$, composition $m \colon X_1 \times_{X_0} X_1 \ra X_1$ and the inverse map $\mathfrak{i} \colon X_1 \ra X_1$ are smooth maps.
\end{definition}
\begin{remark}
	One can show that the inverse map of a Lie groupoid is a diffeomorphism; see \textbf{Proposition 1.1.5}, \cite{MR2157566}.
\end{remark}
Next, we will construct some natural examples of Lie groupoids arising from the classical manifold theory.
\begin{example}\label{Lie grouppoid Example: Lie group}[Lie group]
	The groupoid $[G \rra *]$ associated to a Lie group $G$ is a Lie groupoid.
\end{example}
\begin{example}\label{Lie groupoid example: manifold}[Manifold]
	For any manifold $M$, the groupoid $[M \rra M]$, whose morphisms are only identity arrows, is a Lie groupoid. These Lie groupoids are often called \textit{discrete Lie groupoids}.
\end{example}
\begin{example}\label{Lie groupoid example:Pair groupoid}[Pair groupoid]
	For any manifold $M$, there is a Lie groupoid ${\rm{Pair}}(M):=[M \times M \rra M]$, called the \textit{pair groupoid of }$M$, whose source, target maps are first, second projection respectively, and other structure maps are obvious.
\end{example}
\begin{example}\label{Lie groupoid example action groupoid}[Action groupoid]
	A  left action of a Lie group $G$ on a manifold $M$ gives rise to a Lie groupoid $[G \times M \rra M]$ called \textit{left translation}(or \textit{action}) \textit{groupoid} of the left action, whose structure maps are given by 
	\begin{itemize}
		\item $s \colon G \times M \ra M,(g,m) \mapsto m$,
		\item $t \colon G \times M \ra M,(g,m) \mapsto gm$, 
		\item Composition: if $s(g',m')=t(g,m)$, then $(g',m') \circ (g,m):=(gg',m)$,
		\item $u \colon M \ra G \times M, m \mapsto (e,m)$,
		\item $\mathfrak{i} \colon G \times M \ra G \times M, (g,m) \mapsto (g^{-1},gm)$.
	\end{itemize}
\end{example}
Similarly for a right action of a Lie groupoid there is a \textit{right translation}(or \textit{action}) \textit{groupoid}  $[M \times G \rra M]$.

\begin{example}\label{Fundamental groupoid of a manifold}[Fundamental groupoid of a manifold] For any manifold $M$, \textit{the fundamental groupoid} $\Pi_1(M)$ is a Lie groupoid, whose objects are points of $M$, morphisms are smooth homotopy classes of paths relative to endpoints. For more details on its smooth structure see \textbf{Example 5.1 (6)} of \cite{MR2012261}.
\end{example}
\begin{example}\label{Frame groupoid}[Frame groupoid or linear groupoid]
	For any vector bundle $ \pi \colon E \ra M$ over a manifold $M$, the \textit{frame groupoid} ${\rm{GL}}(\pi)$ is a Lie groupoid, whose objects are elements of $M$ and morphisms are the linear isomorphisms between the fibers, see \textbf{Example 5.1 (7)} \cite{MR2012261}.
\end{example}
\begin{example}\label{Atiyah groupoid}[Atiyah groupoid/Gauge groupoid]
	For any principal $G$-bundle $ \pi \colon P \ra M$, the \textit{gauge groupoid} or \textit{Atiyah groupoid} $\mc{G}(\pi)$ is a Lie groupoid whose objects are the element of $M$ and arrows are morphism of $G$-torsors between fibers. The set of morphisms can also be identified by the quotient manifold $\frac{P \times P}{G}$ by the diagonal action of $G$. We suggest \textbf{Example 2.35} of \cite{koushik2021geometric} for readers interested in the detailed description of its structure maps.
\end{example}
\begin{example}\label{Cover groupoid}[Cover Lie groupoid or Čech groupoid]
	For any manifold $M$ and an open cover $\mc{U}:= \lbrace U_{i} \rbrace_{i}$, there is a Lie groupoid $C( \mc{U}) :=\big[ \bigsqcup_{i,j }U_{ij} \rra \bigsqcup_{i }U_i \big]$, where $U_{ij}:= U_{i} \cap U_{j}$, whose source and target maps are given as  :
	\begin{itemize}
		\item $s \colon \bigsqcup_{i,j }U_{ij} \ra \bigsqcup_{i }U_i, (x,i,j) \mapsto (x,i)$,
		\item $t \colon \bigsqcup_{i,j }U_{ij} \ra \bigsqcup_{i }U_i, (x,i,j) \mapsto (x,j)$.
	\end{itemize}
	Composition and other structure maps are evident.					
\end{example}
\begin{example}\label{Tangent Lie groupoid}[Tangent Lie groupoid]
	Given a Lie groupoid $\mb{X}=[X_1\rra X_0],$ the tangent bundle over $X_1$ and $X_0$ combine to form a Lie groupoid, called the   \textit{tangent Lie groupoid} of $\mb{X}$, denoted $T\mb{X}:=[TX_1\rra TX_0]$. The differentials of the respective structure maps of $\mb{X}$ induce the structure maps of $T\mb{X}$. Particularly, if $(\gamma_2, \mc{X}_2)$, $(\gamma_1, \mc{X}_1)$ are composable morphisms in $T\mb{X}$, then we denote $(\gamma_2, \delta_2) \circ (\gamma_1,\delta_1)= \Big( m(\gamma_2, \gamma_1), m_{*,(\gamma_2,\gamma_1)}(\mc{X}_2,\mc{X}_1) \Big)$ as $(\gamma_2 \circ \gamma_1, \mc{X}_2 \circ \mc{X}_1)$, where $m$ is the composition map of the Lie groupoid $\mb{X}$.
\end{example}

\begin{definition}[Lie subgroupoid, Appendix A.1, \cite{MR3451921}]\label{Lie subgroupoid}
A \textit{Lie subgroupoid of a Lie groupoid} $\mb{X}=[X_1 \rra X_0]$	is a subgroupoid $\mb{Y}=[Y_1 \rra Y_0]$ of $\mb{X}$, defined by the restriction of the structure maps such that the following three conditions hold:
\begin{enumerate}[(i)]
	\item $Y_1 \subseteq X_1$ and  $Y_0 \subseteq X_0$ are submanifolds;
	\item The restriction of the source map in $\mb{X}$, to $Y_1$ is a submersion;
	\item The structure maps of the groupoid $\mb{Y}$ are smooth.
	\end{enumerate}
	\end{definition}

\begin{definition}[Morphism of Lie groupoids, Section 5.1, \cite{MR2012261}]\label{Definition: Morphism of Lie groupoids}
	Let $\mb{X}$ and $\mb{Y}$ be Lie groupoids. A \textit{morphism of Lie groupoids} \textit{from} $\mb{X}$ to $\mb{Y}$ is defined as a functor $F \colon \mb{X} \ra \mb{Y}$, such that $F_0 \colon X_0 \ra Y_0$ and $F_1 \colon X_1 \ra Y_1$ are smooth. 
\end{definition}
\begin{example}
	Let $G$ and $H$ be Lie groups. Any Lie group homomorphism $\phi \colon G \ra H$ induces an obvious morphism of Lie groupoids $[G \rra *] \ra [H \rra *]$.
\end{example}
\begin{example}
	Any smooth map $f \colon M \ra N$ of manifolds induces a morphism of Lie groupoids between the associated discrete Lie groupoids.
\end{example}
\begin{example}
	Any smooth map $f \colon M \ra N$ induces a morphism of Lie groupoids between the associated pair groupoids $(f \times f, f) \colon {\rm{Pair}}(M) \ra {\rm{Pair}}(N)$.
\end{example}

\begin{example}
	Let $\phi: G \ra G'$ be a Lie group homomorphism and  $G, G'$ acts on manifolds $M, M'$ respectively from the left. Then any smooth map $f \colon M \ra M'$ preserving the said Lie group actions induces a morphism of Lie groupoids $ \big((f \times \phi), f \big) \colon [G \times M \rra M] \ra [G' \times M' \rra M']$.
\end{example}
\begin{example}\label{Induced morphism between fundamental groupoid of manifolds}
	Any smooth map $f \colon M \ra N$ induces a morphism of Lie groupoids $$(\bar{f},f) \colon \Pi_1(M) \ra \Pi_1(N)$$ between the associated fundamental groupoids where $\bar{f}$ takes a homotopy class of paths $[\alpha]$ in $M$ to the homotopy class of paths $[f \circ \alpha]$ in $N$.
\end{example}
\begin{example}
	Let $\pi \colon E \ra M$ and $ \pi' \colon E' \ra M'$ be two vector bundles. Any morphism of vector bundles $(f_{E}, f_{M})$ from $\pi \colon E \ra M$ to   $ \pi' \colon E' \ra M'$ induces an obvious morphism of Lie groupoids ${\rm{GL}}(\pi) \ra {\rm{GL}}(\pi')$ between the associated frame Lie groupoids.
\end{example}
\begin{example}
	For a Lie group $G$, let $\pi \colon P \ra M$ and $\pi' \colon P' \ra M'$  be two principal $G$-bundles. Then any morphism of principal $G$-bundles $(f_P, f_M)$ from $\pi \colon P \ra M$ to $\pi' \colon P' \ra M'$ induces a morphism of Lie groupoids $(\bar{f}, f_M) \colon \mc{G}(\pi) \ra \mc{G}(\pi')$  between the associated Atiyah groupoids, where $\bar{f}$ takes an element $[p,q]$ to $[f_{E}(p),f_E(q)]$.
\end{example}
\begin{example}\label{induced morphism from Cech groupoid}
	Given a manifold $M$ with an open cover $\mc{U}$, there is an obvious morphism of Lie groupoids from the corresponding Čech groupoid $C(\mc{U})$ (\Cref{Cover groupoid}) to the discrete Lie groupoid $[M \rra M]$, defined as
	\begin{equation}\nonumber
		\begin{split}
	 \mc{U}^{*} \colon & C(\mc{U}) \ra [M \rra M] \\
	 & (x,i) \mapsto x,\\
	 & (x,i,j) \mapsto x.
	 \end{split}
	 \end{equation}
	\end{example}
\begin{example}
	Any morphism of Lie groupoids $F \colon \mb{X} \ra \mb{Y}$ induces a morphism of Lie groupoids $dF:=(dF_1, dF_0) \colon T\mb{X} \ra T\mb{Y}$ between the associated tangent Lie groupoids where $dF_1$ and $dF_0$ are differentials of $F_1$ and $F_0$ respectively.
\end{example}
\begin{definition}[Smooth natural transformation, Section 5.3, \cite{MR2012261}]
	Given two morphisms of Lie groupoids $\phi, \psi \colon \mb{X} \ra \mb{Y}$, a \textit{smooth natural transformation } from $\phi$ to $\psi$ is defined as a natural transformation $\eta \colon \phi \Rightarrow \psi$  such that the map $\eta \colon X_0 \ra Y_1$ is smooth.
\end{definition}

\begin{proposition}\label{Category and 2-category of Lie groupoids}
	The collection of Lie groupoids forms a strict 2-category, whose 1-morphisms are morphisms of Lie groupoids and 2-morphisms are smooth natural transformations.			
\end{proposition}
\begin{proof}
	See \textbf{Section 5.3}, \cite{MR2012261}.
	\end{proof}
We denote the strict 2-category of Lie groupoids by 2-LieGpd and its underlying category (\Cref{underlying category of strict 2-category}) by LieGpd. 

Below, we list some basic properties of Lie groupoids:

\begin{proposition}\label{Basic properties}
	Suppose $\mb{X}=[X_1 \rra X_0]$ is a Lie gropoid, and let $x,y \in X_0$, then the following holds:
	\begin{enumerate}[(i)]
		\item  ${\rm{Hom}}(x,y)$ is a closed submanifold of $X_1$.
		\item ${\rm{Aut}}_{\mb{X}}(x):= {\rm{Hom}}(x,x)$ is a Lie group.
		\item $t \big(s^{-1}(x) \big)$ is an immersed submanifold of $X_0$.
		\item $t_{x} \colon t|_{s^{-1}(x)} \ra t \big(s^{-1}(x) \big)$ is a principal ${\rm{Aut}}_{\mb{X}}(x)$-bundle over $ t \big(s^{-1}(x) \big)$.
	\end{enumerate}
	
\end{proposition}
\begin{proof}
	See \textbf{Theorem 5.4}, \cite{MR2012261}.
\end{proof}

\subsection*{Some important classes of Lie groupoids:}

\begin{definition}[Proper groupoid, Definition 2.18, \cite{MR2778793}]
	A Lie groupoid $\mb{X}$ is said to be \textit{proper}, if the map $(s,t) \colon X_1 \ra X_1 \times X_1,\gamma \mapsto \big(s(\gamma), t(\gamma) \big)$, is a proper map. 
\end{definition}

\begin{definition}[Etale groupoid, Definition 2.19, \cite{MR2778793}]
	A Lie groupoid $\mb{X}$ is said to be \textit{étale}, if the source and target maps are local diffeomorphisms
\end{definition}
\begin{definition}\label{proper etale Lie groupoid}
	A Lie groupoid $\mb{X}=[X_1\rra X_0]$  is called  \textit{proper \'etale} if it is both proper and étale. 	
\end{definition}
\begin{example}
	A cover Lie groupoid or Čech groupoid \Cref{Cover groupoid} is a proper \'etale groupoid.
\end{example}
\begin{example}
	An action groupoid \Cref{Lie groupoid example action groupoid} is proper if and only if the action of the Lie group is proper.
\end{example}
\begin{remark}
	A proper étale Lie groupoid is same as an \textit{orbifold} (see \cite{MR1466622}, \cite{MR1950948})
\end{remark} 
\begin{definition}[Transitive Lie groupoid]\label{Definition: Transitive Lie groupoid}
	A Lie groupoid $\mb{X}$ is said to be \textit{transitve} if for any pair of elements $x,y \in X_0$, there is an element $\gamma \in X_1$, such that $s(\gamma)=x$ and  $t(\gamma)=y$.
	\end{definition}
	\begin{example}
		Pair groupoid ${\rm{Pair}}(M)$ (\Cref{Lie groupoid example:Pair groupoid}) of any manifold $M$ is a transitive Lie groupoid.
		\end{example}
\begin{example}
	The fundamental groupoid $\Pi_1(M)$ (\Cref{Fundamental groupoid of a manifold}) of a manofold $M$ is transitive if and only if $M$ is connected.
	\end{example}
	
	\begin{example}
		Let $\rho$ be an action of a Lie group $G$ on a manifold $M$. Then the corresponding action Lie groupoid $[G \times M \rra M]$ is transitive if and only if the action $\rho$ is transitive.
		\end{example}
		\begin{example}
Given a Lie group $G$, for any principal $G$-budle $\pi \colon P \ra M$ the corresponding Atiyah groupoid $\mc{G}(\pi)$ (\Cref{Atiyah groupoid}) is a transitive Lie groupoid. 
			\end{example}
			
			Conversely, it is also known (see \cite{MR1001474}) that any transitive Lie groupoid $\mb{X}=[X_1 \rra X_0]$ is isomorphic to the Atiyah groupoid $\mc{G}(t|_{s^{-1}(x)})$ (\Cref{Atiyah groupoid}) of the prinicpal $\rm{Aut}(x)$-bundle $t_x \colon s^{-1}(x) \ra X_0$ over $X_0$ (\Cref{Basic properties}) for any $x \in X_0$ .

 	\subsection*{Nerve of a Lie groupoid:}
		\begin{definition}\label{Definition: Nerve of a Lie groupoid}[\cite{MR2270285}, \textbf{Section 2.1}]
 		The \textit{nerve} of a Lie groupoid $\mb{X}$ is  a simplicial manifold $N(\mb{X})$, defined by the following data:
 		\begin{equation}\nonumber
 			\begin{split}
 					& N(\mb{X})_0=X_0\\
 					& N(\mb{X})_1=X_1\\
 					& N(\mb{X})_n= \lbrace (\gamma_1,\gamma_2,...,\gamma_n) \in X_1^{n}: s(\gamma_i)=t(\gamma_{i+1}),i=1,2,...,n-1 \rbrace, n >1,
 				\end{split}			
 		\end{equation}
 		whose face maps $s_i^{n}$ and the degeneracy maps $d_i^{n}$,  are given as	
 		\begin{equation}\nonumber
 			\begin{split}
 					&d_0^{n}(\gamma_1,\gamma_2,...,\gamma_n)=(\gamma_2,...,\gamma_n), n>1, \\
 					&d_{n}^{n}(\gamma, \gamma_2,...,\gamma_n)=(\gamma_1,...,\gamma_{n-1}), n>1,  \\
 					&d_i^{n}(\gamma_1, \gamma_2,...,\gamma_n)=(\gamma_1,..\gamma_{i} \gamma_{i+1},...,\gamma_n), n>1, \\
 					&s_0^{n}(\gamma_1,..,\gamma_n)=(1_{t(\gamma_1)},\gamma_1,..,\gamma_n), n>0, \\
 					&s_i^{n}(\gamma_1,..,\gamma_n)=(\gamma_1,..\gamma_i,1_{s(\gamma_i)},\gamma_{i+1},..\gamma_n), 1 \leq i \leq n, n>0, 
 				\end{split}
 		\end{equation}
 		and 			
 		\begin{equation}\nonumber
 			\begin{split}
 					&d_0^{1}(\gamma)=s(\gamma) \\
 					&d_1^{1}(\gamma)=t(\gamma) \\
 					&s_0^{0}(x)=1_x, x \in X_0.
 				\end{split}.
 		\end{equation}

 		\end{definition}

 	\subsection{Fibred products in Lie groupoids}\label{subsection pullbacks in Lie groupoids}
 	This subsection quickly recalls two notions of fibered products in Lie groupoids, namely,
 	\begin{itemize}
 		\item strong fibered products;
 		\item weak fibered products.
 	\end{itemize}
 	We suggest \textbf{Section 5.3}, \cite{MR2012261} for a detailed treatment. In particular, our reference for the portion on strong fibered products is \textbf{Appendix A} of \cite{MR3451921}.

 	\subsection*{Strong fibred products}\label{strong fibered products}
 	
 	\begin{definition}
 		A pair of smooth maps $f_1 : M_1 \ra M$ and $f_2: M_2 \ra M$ is said to form a \textit{good pair} if they satify the following two conditions:
 		\begin{enumerate}[(i)]
 			\item The set-theoretic pullback $M_{12} := M_1 \times_{M} M_2$ is an embedded submanifold of $M_1 \times M_2$ and 
 			\item for all  $(p_1,p_2) \in M_{12}$ with $p=f_1(p_1)=f_2(p_2)$, the following is an exact sequence
 			
 			\begin{tikzcd}
 				0 \arrow[r] & T_{(p_1,p_2)}M_{12} \arrow[r] & T_{p_1}M_1 \times T_{p_2}M_2 \arrow[r, "df_1 - df_2"] & T_{p}M
 			\end{tikzcd}
 			
 		\end{enumerate}
 		
 	\end{definition}
 	\begin{example}
 		A pair of submersion and any smooth map is always a good pair.
 	\end{example}

 	\begin{lemma}\label{Proposition: Strict pullback of Lie groupoids}
 		
 		Let $\phi: \mb{X} \ra \mb{Z}$ and $\psi: \mb{Y} \ra \mb{Z}$ be a pair of morphisms of Lie groupoids such that $\phi_1, \psi_1$ form a good pair. Then the pullback manifolds  $X_0 \times_{\phi_0, Z_0, \psi_0} Y_0$ and $X_1 \times_{\phi_1, Z_1, \psi_1} Y_1$ forms a Lie subgroupoid (see \Cref{Lie subgroupoid}) $\mb{X} \times_{\phi, \mb{Z}, \psi} \mb{Y}$ of the product Lie groupoid $\mb{X} \times \mb{Y}$  with evident structure maps. Moreover, the Lie groupoid $\mb{X} \times_{\phi, \mb{Z}, \psi} \mb{Y}$ satisfies the universal property of the pullback in LieGpd,  \Cref{Category and 2-category of Lie groupoids}.
 	\end{lemma}
 	\[ \begin{tikzcd}
 		\mb{X} \times{_{\phi,\mb{Z}, \psi}} \mb{Y} \arrow[d, "{\rm{pr}}_1"'] \arrow[r, "{\rm{pr}}_2"]                 & \mb{X} \arrow[d, "\phi"] \\
 		\mb{Y} \arrow[r, "\psi"']  & \mb{Z}               
 	\end{tikzcd}\]
 	\begin{proof}
 		See \textbf{Proposition A.1.4} in \cite{MR3451921}.
 	\end{proof}
 	
 	The Lie groupoid $\mb{X} \times_{\phi, \mb{Z}, \psi} \mb{Y}$  is called the \textit{strong fibered product with respect to the pair of maps $\phi: \mb{X} \ra \mb{Z}$ and $\psi: \mb{Y} \ra \mb{Z}$}.
 	\subsection*{Weak Fibred products}
 	Given a pair of  morphisms of Lie groupoids $\phi \colon \mb{Y} \ra \mb{X}$ and $\psi \colon \mb{Z} \ra \mb{X}$, there exists a topological groupoid  $\mb{Y} \times_{\phi, \mb{X}, \psi}^{h} \mb{Z}$. An object of this groupoid is given by a triple $(y, \psi(z) \xrightarrow{\gamma} \phi(y) ,z )$ for $y \in Y_0, z \in Z_0$, and an arrow from $(y, \psi(z) \xrightarrow {\gamma} \phi(y), z )$ to $(y', \psi(z') \xrightarrow {\gamma'} \phi(y') ,z')$ is given by a pair $(y \xrightarrow{\zeta} y', z \xrightarrow{\delta} z')$
 	such that the following diagram commutes:
 	\[\begin{tikzcd}
 		\psi(z) \arrow[d, "\psi(\delta)"'] \arrow[r, "\gamma"]                 & \phi(y) \arrow[d, "\phi(\zeta)"] \\
 		\psi(z') \arrow[r, "\gamma'"']  & \phi(y')               
 	\end{tikzcd}\]
 	That is 
  \begin{equation}\label{Weakfibreproductmorphims}
  \phi(\zeta) \circ \gamma = \gamma' \circ \psi(\delta).	
  \end{equation}
 	
 	The groupoid $\mb{Y} \times_{\phi, \mb{X}, \psi}^{h} \mb{Z}$ has the usual universal property of a fiber product (but up to an isomorphism). If one of $\phi_0 \colon Y_0 \ra X_0$ or $\psi_0 \colon Z_0 \ra X_0$ is a submersion, then  $\mb{Y} \times_{\phi, \mb{X}, \psi}^{h} \mb{Z}$ is Lie groupoid and is called the  \textit{weak fibered product with respect to the pair of maps $\phi \colon \mb{Y} \ra \mb{X}$ and $\psi \colon \mb{Z} \ra \mb{X}$}. For a rigorous treatment, readers can look at
 	\textbf{Section 5.3}, \cite{MR2012261}. Below, we outline an interesting property these weak fibered products enjoy, which we will use in later chapters.
 	
 	Any morphism of Lie groupoids $F \colon \mb{X} \ra \mb{Y}$ has a canonical factorization (see \cite{MR3968895}) through 
 	
 	$\mb{Y} \times_{\mb{Y},F}^{h} \mb{X}:= \mb{Y} \times_{{\rm{id}}_{\mb{Y}}, \mb{Y}, F}^{h} \mb{X}$, 
 	\begin{equation}\label{factorization of morphisms}
 	\begin{tikzcd}
 		\mb{X} \arrow[d, "F"'] \arrow[r, "F_{\mb{X}}"] & \mb{Y} \times_{\mb{Y},F}^{h} \mb{X} \arrow[ld, "F_{\mb{Y}}"] \\
 		\mb{Y}                                           &                                                           
 	\end{tikzcd}
 	\end{equation}
 	and is given by
 	\begin{equation}\label{piE}
 		\begin{split}
 			F_{\mb{X}} \colon & \mb{X} \ra \mb{Y} \times_{\mb{Y},F}^{h} \mb{X}\\
 			& x \mapsto \big( F(x), {\rm{id}}_{F(x)}, x \big)\\
 			&(x \xrightarrow{\gamma} y) \mapsto \big(F(\gamma),\gamma\big) \in {\rm{Hom}}\Big( \big(F(x), {\rm{id}}_{F(x)}, x\big), \big(F(y), {\rm{id}}_{F(y)}, y \big)\Big).
 		\end{split}
 	\end{equation}
 	and 
 	\begin{equation}\label{piX}
 		\begin{split}
 		  F_{\mb{Y}} \colon & \mb{Y} \times_{\mb{Y},F}^{h} \mb{X} \ra \mb{Y} \\
 			& (y, \zeta,p) \mapsto y \\
 			&( y \xrightarrow{\eta}  y', p \xrightarrow{\delta}  p') \mapsto \eta.
 		\end{split}
 	\end{equation}

 	\subsection{Lie groupoid $G$-extensions}\label{subsection Lie groupoid G extension}
 	Here, we quickly recall the definition of a Lie groupoid $G$-extension \cite{MR3480061}, for a Lie group $G$ over the identity map on a manifold $M$. 
 	
 	Consider the action groupoid $[M\times G\rra M]$ as given in \Cref{Lie groupoid example action groupoid}. 	A \textit{Lie groupoid $G$-extension} is a short exact sequence of Lie groupoids  
 	of the following form
 	\begin{equation}\label{Dia:Gextension}
 		\begin{tikzcd}[sep=small]
 			1 \arrow[r] & M\times G  \arrow[dd,xshift=0.75ex]\arrow[dd,xshift=-0.75ex]  \arrow[rr, "i"] &  & \Gamma_2 \arrow[rr, "\phi"]  \arrow[dd,xshift=0.75ex]\arrow[dd,xshift=-0.75ex]   &  & \Gamma_1 \arrow[r] \arrow[dd,xshift=0.75ex]\arrow[dd,xshift=-0.75ex]  & 1 \\
 			&                                                                    &  &                                                       &  &                           &   \\
 			1 \arrow[r] & M \arrow[rr, "{\rm Id}"]                                                     &  & M \arrow[rr, "{\rm Id}"]                                        &  & M \arrow[r]             & 1
 		\end{tikzcd},
 	\end{equation} where $\phi$ is a surjective submersion and $i$ is an embedding (\cite[Chapter $4$]{Moerdijk2}). 
 	

 	\subsection{Action and quasi-action of a Lie groupoid}\label{Action and quasiaction of a Lie groupoid}
 	In \Cref{Lie groupoid example action groupoid}, we demonstrated how an action of a Lie group on a manifold defines a Lie groupoid called the action groupoid. In this section, we construct a generalized version of such action groupoid arising from an `action of a Lie groupoid on a manifold'. This notion of action is standard in the current body of literature; for example, see \cite{MR2012261}, \cite{MR2270285}, \cite{MR2778793}. Furthermore, we also mention a weaker notion of action, called `quasi-action of a Lie groupoid on a manifold', comparatively a less standard notion.

 	
 	\begin{definition}\label{Left action of a Lie groupoid on a manifold}[\cite{MR2270285}, \textbf{Section 2.3}]
 		A \textit{left action of a Lie groupoid $\mb{X}$ on a manifold} $E$ consists of a pair of smooth maps $a \colon E \ra X_0, \mu_{L} \colon X_1 \times_{s,X_0,a} E  \ra E$ such that 
 		\begin{itemize}
 			\item[(i)] $\mu_{L}(1_{p},p)=p$ for all $p \in E$,
 			\item[(ii)] $a(\mu_{L}(\gamma,p))=t(\gamma)$ for all $(\gamma,p) \in X_1 \times_{a,X_0,s} E$,
 			\item[(iii)] $\mu_{L}(\gamma_2, \mu_{L}(\gamma_1,p))= \mu_{L}(\gamma_2 \circ \gamma_1,p)$ for suitable $\gamma_2, \gamma_1,p$.
 		\end{itemize}
 	\end{definition}
 	Analogously, one can define a notion of a right action of a Lie groupoid on a manifold. \label{right action of a Lie groupoid}.
 	
 	The maps $a$, $\mu_{L}$ and $\mu_{R}$ (for the right action) are called the \textit{anchor map}, the \textit{left action map} and the \textit{right action map}, respectively. We will often use the notations $\gamma p$ and $p \gamma$  for $\mu_{L}(\gamma,p)$ and  $\mu_{R}(\gamma,p)$ respectively.
 	\begin{remark}\label{transport category}
 		
 		Consider a left action of a Lie groupoid $\mb{X}$ on a manifold $E$ given by  $a \colon E \ra X_0, \mu \colon X_1 \times_{s, X_0, a} E  \ra E$. There is an obvious category $C_{a}$ whose objects are fibers $a^{-1}(x)$ of $a$ over $x \in X_0$, and morphisms are functions between such fibers. The composition of such arrows is simply the composition of functions. Now, observe that the action of $\mb{X}$ on $E$ induces a functor
 		\begin{equation}\nonumber
 			\begin{split}
 				T_{a, \mu} \colon & \mb{X} \ra C_a\\
 				& x \mapsto a^{-1}(x)\\
 				& (x \xrightarrow{\gamma}y) \mapsto T_{\gamma} \colon a^{-1}(x) \ra a^{-1}(y),
 			\end{split}
 		\end{equation}
 		where the function $T_{\gamma} \colon a^{-1}(x) \ra a^{-1}(y)$ is given by $p \mapsto \mu(\gamma,p)$. On the other hand, observe that if $\mb{X}$ acts on $E$ from right, then we get an analogous functor $\mb{X}^{\rm{op}} \ra C_a$. We will call the category $C_a$ as the \textit{associated transport category of the action }and the functor $T_{a,\mu}$ as the \textit{associated transport functor of the action}.
 	\end{remark}	
 	\begin{example}[Lie group action]\label{Lie group action}
 		Let $G$ be a Lie group and $M$ be manifold. Any left (right) action of $G$ on $M$ defines an obvious left (right) action of the Lie groupoid $[G \rra *]$ (\Cref{Lie grouppoid Example: Lie group}) on $M$. Hence, any representation of $G$ on a vector space $V$ determines an action of $[G \rra *]$ on $V$.
 	\end{example}
 	\begin{example}[Smooth map]
 		Any smooth map $f  \colon N \ra M$ of manifolds induces both left and right action of the discrete Lie groupoid $[M \rra M]$ (\Cref{Lie groupoid example: manifold}) on $N$, where the anchor map is $f$ and both the left and right action maps are 2nd projection maps.
 	\end{example}
 	\begin{example}[Composition of arrows]
 		Given a Lie groupoid $\mb{X}$, the pair of smooth maps $(m,t)$ and $(m,s)$ defines a left action and a right action of $\mb{X}$ on $X_1$ respectively.
 	\end{example}
 	
 	\begin{example}[Trivializations of a fibre bundle]
 		Let $\pi \colon E \ra M$ be a fiber bundle over a manifold $M$ with fiber $F$. Any trivialization $\phi \colon E \ra M \times F$  induces an obvious left (right) action of the pair groupoid $[M \times M \rra M]$ (\Cref{Lie groupoid example:Pair groupoid}) on $E$, whose anchor map is $\pi$, left action is given by $\big((m,n),p\big) \mapsto \phi^{-1}(n, {\rm{pr}}_2(\phi(p)))$ and the right action map is defined by $\big((m,n),(n,f)  \mapsto \phi^{-1}(m, {\rm{pr}}_2(\phi(p))) $, where ${\rm{pr}}_2$ is the 2nd projection map.
 	\end{example}
 	\begin{example}[Flat connections]\label{Flat connection action}
 		For a Lie group $G$, let $\pi \colon P \ra M$ be a principal $G$-bundle over a manifold $M$. Any flat connection $A$ on $\pi \colon P \ra M$ induces a left action of the fundamental groupoid $\Pi_1(M)$ (\Cref{Fundamental groupoid of a manifold}) on $P$, whose anchor map is $\pi$ and the  left action map is given by $([\alpha],p) \mapsto {\rm{Tr}}^{\alpha}_{\omega}(p)$(\Cref{Parallel transport aloing a path}), the parallel transport of $p$ along $\alpha$.
 	\end{example}
 	\begin{example}[Representation of a Lie groupoid]\label{Representation of a Lie groupoid}
 		Given a Lie groupoid $\mb{X}$, let $\pi \colon E \ra X_0$ be a vector bundle over the manifold $X_0$. Then, any morphism of Lie groupoids $\mb{X} \ra {\rm{GL}}(\pi)$, where ${\rm{GL}}(\pi)$ is the frame groupoid (\Cref{Frame groupoid}), induces an obvious left action of $\mb{X}$ on $E$. This action is called the \textit{represntation of $\mb{X}$ on the vector bundle $\pi \colon E \ra X_0$.} Note that when $\mb{X}$, is a Lie groupoid, we recover the usual represntation of Lie groups on a vector space, \Cref{Lie group action}.
 	\end{example}
 	\begin{example}\label{Example Functor to the Atiyah groupoid}[Functor to the Atiyah groupoid \Cref{Atiyah groupoid}]
 		Given a Lie groupoid $\mb{X}$ and a Lie group $G$, let $\pi \colon P \ra X_0$ be a principal $G$-bundle over the manifold $X_0$. Then, any morphism of Lie groupoids $\mb{X} \ra \mc{G}(\pi)$ induces an obvious left action of $\mb{X}$ on $P$.		
 	\end{example}

 	Next, we obtain a generalization of \Cref{Lie groupoid example action groupoid}:  		
 	
 	\begin{example}\label{Semi-direct product groupoid}[Semi-direct product groupoid]
 		Any left action $(a, \mu_{L})$ of a Lie groupoid $\mb{X}$ on a manifold $E$ defines a Lie groupoid $\mb{X}\rtimes E:=[s^{*}E \rra E]$ called \textit{semi-direct product groupoid of the $\mb{X}$-action}, where $s^{*}E:= X_1 \times_{s,X_0, a} E $. The structure maps are given as 
 		\begin{itemize}
 			\item $s  \colon  s^{*}E \ra E, \quad (\gamma,p) \mapsto p$;
 			\item $t  \colon s^{*}E \ra E, \quad (\gamma,p) \mapsto \mu_{L}(\gamma,p)= \gamma p$;
 			\item $u \colon E \ra s^{*}E, \quad p \mapsto (1_{a(p)},p)$;
 			\item Let $(\gamma_2,p_2), (\gamma_1,p_1) \in s^{*}E$ such that $s(\gamma_2,p_2)=t(\gamma_1,p_1)$, then $(\gamma_2,p_2) \circ (\gamma_1,p_1):=(\gamma_2 \circ \gamma_1,p_1)$;
 			\item $\mathfrak{i} \colon s^{*}E \ra s^{*}E, \quad (\gamma,p) \mapsto (\gamma^{-1}, \gamma p)$.
 		\end{itemize}
 		Similarly, a right action of $\mb{X}$ on $E$ gives a semi-direct product groupoid. Note that in the particular case when $\mb{X}$ is a Lie group \Cref{Lie grouppoid Example: Lie group}, the semi-direct product groupoid coincides with that of action groupoid \Cref{Lie groupoid example action groupoid}. 
 	\end{example}
 	\begin{example}[Connection on a Lie groupoid]
 	Any connection on a Lie groupoid $\mb{X}$ (Definition 3.1, \cite{MR3150770}) induces an action of $\mb{X}$ on $TX_0$.
 		\end{example}
 	There is also a weaker notion of action called a `quasi action of a Lie groupoid on a manifold' (\cite{MR3968895}, \cite{MR3757473}), which we recall below:
 	\begin{definition}\label{Definition: quasi action on a manifold}
 		A \textit{left quasi-action of a Lie groupoid $\mb{X}$ on a manifold} $E$ consists of a pair of smooth maps $a_{L} \colon E \ra X_0, \mu_{L} \colon X_1 \times_{s,X_0,a_{L}} E  \ra E$ such that 
 		$a_{L}(\mu_{L}(\gamma,p))=t(\gamma)$  for all $(\gamma,p) \in X_1 \times_{s,X_0,a_{L}} E$.
   
 	\end{definition}
 	\begin{example}[Ehresmann connection on a Lie groupoid]
 		Any \textit{Ehresmann connection} $\sigma$ (\textbf{Definition 2.8}, \cite{MR3107517}) on a Lie groupoid $\mb{X}$ induces a quasi-action of $\mb{X}$ on $TX_0$ defined by the \textbf{Definition 2.11}, \cite{MR3107517}. Furthermore, $\sigma$ is an action  if the basic curvature of the connection (\textbf{Definition 2.12}, \cite{MR3107517}), vanishes.
 		\end{example}
 	We will see several examples of quasi actions in later chapters. In fact, a notion of semi-direct product groupoid arising from certain quasi-action of a Lie groupoid on a manifold will play a pivotal role in \Cref{chapter 2-bundles}.
 	\subsection{Anafunctors and Morita equivalence of Lie groupoids}\label{subsection Generalized homomorphisms and Morita equivalence of Lie groupoids}
 	Recall in \Cref{Category and 2-category of Lie groupoids} we saw that the collection of Lie groupoids, morphisms of Lie groupoids, and smooth natural transformations form a strict 2-category. However, as a consequence of the failure of the Axiom of Choice in Man, the category of smooth manifolds (see \cite{MR2778793} for details), a fully faithful essentially surjective morphism of Lie groupoids $F \colon \mb{X} 
 	\ra \mb{Y}$ may not induce the existence of a smooth map $\bar{F} \colon \mb{Y} \ra \mb{X}$ such that $F \circ \bar{F}$ and $\bar{F} \circ F$ are naturally isomorphic (smooth) to identities, (for example, this is case for \Cref{induced morphism from Cech groupoid}). As a result, this strict 2-category is often considered too strict. Nonetheless, there is a natural way to embed this strict 2-category 2-LieGpd into a bicategory (\Cref{Weak 2-category}), whose objects are Lie groupoids, 1-morphisms are `bibundles' or `anafunctors' and 2-morphisms are `transformations between such anafunctors'. Under this embedding, these `smooth versions of fully faithful essentially surjective morphism of Lie groupoids' (often called `Morita maps') are mapped to isomorphisms (often called `weak equivalences') in the said bicategory. This subsection briefly discusses this bicategory. For a detailed description, we recommend the readers to look at \cite{MR3894086}, \cite{MR2778793}, \cite{MR2270285}, and \cite{MR2012261}. 
 	
 	To define the notion of an anafucntor, we recall the definition of a `right principal Lie groupoid bundle over a manifold' below:

 	\begin{definition}[\cite{MR2270285}, \textbf{Section 2.3}]\label{principal Lie groupoid bundle}
 		Given a Lie groupoid $\mb{X}$ and a manifold $M$, a \textit{right principal $\mb{X}$-bundle over $M$} consists of
 		\begin{itemize}
 			\item[(i)] a smooth manifold $M$
 			\item[(ii)] a surjective submersion $\pi \colon E \ra M$
 			\item[(iii)] a right action of $\mb{X}$ on $E$ (\Cref{right action of a Lie groupoid}) such that
 			$\pi(p \gamma)= \pi(p)$ for all $(\gamma,p) \in X_1 \times_{a,X_0,t} E$  and the map $X_1 \times_{a,X_0,t} E \ra E \times_{M} E$ defined as $(\gamma,p) \mapsto (p \gamma,p)$ is a diffeomorphism.
 		\end{itemize}
 	\end{definition}
 	\begin{example}
 		For any Lie groupoid $\mb{X}$, the target map $t \colon X_1 \ra X_0$ is a right principal $\mb{X}$-bundle over $X_0$, whose underlying right action of $\mb{X}$ on $X_1$ are given by the pair of maps $s \colon X_1 \ra X_0,$ and $m \colon X_1 \times_{s, X_0, t} X_1 \ra X_1$.
 	\end{example}
 	
 	\begin{definition}\label{generalized homomorphism}[\cite{MR2270285}, \textbf{Definition 2.8}]
 		An \textit{anafunctor $(a_{\mb{X}}, F, a_{\mb{Y}})$ from a Lie groupoid $\mb{X}$ to a Lie groupoid $\mb{Y}$} consists of 
 		\begin{enumerate}[(i)]
 			\item a smooth manifold $F$, called the \textit{total space},
 			\item a left action of $\mb{X}$ on $F$ with the anchor map $a_{\mb{X}} \colon F \ra X_0$,
 			\item a right action of $\mb{Y}$ on $F$ with the anchor map $a_{\mb{Y}} \colon F \ra Y_0$,
 		\end{enumerate}
 		such that the following holds
 		\begin{itemize}
 			\item[(a)] the anchor map $a_{\mb{X}} \colon F \ra X_0$ is a right principal $\mb{Y}$-bundle over $X_0$, \Cref{principal Lie groupoid bundle},
 			\item[(b)] the left action of $\mb{X}$ and the right action of $\mb{Y}$ on $F$ commutes; that is $$(\gamma f) \delta= \gamma (f \delta)$$ with $s(\gamma)= a_{\mb{X}}(f)$ and $t(\delta)=a_{\mb{Y}}(f)$,
 			\item[(c)] the anchor map $a_{\mb{Y}} \colon F \ra Y_0$ is a $\mb{X}$-invariant map; that is $a_{\mb{Y}}(\gamma f)= a_{\mb{Y}}(f)$ for $f \in F$  with $s(\gamma)=a_{\mb{X}}(f)$, where $\gamma \in X_1$.
 		\end{itemize}
 	\end{definition}
 	\begin{definition}[Definition 2.3.1(b), \cite{MR3894086}]
 		A \textit{transformation from an anafunctor $(a_{\mb{X}}, F, a_{\mb{Y}})$ to an anafunctor $(a'_{\mb{X}}, F', a'_{\mb{Y}})$} is defined as a smooth map $\eta \colon F \ra F'$ such that they preserve the anchors and satisfy the following equivariancy condition: $$\eta(\gamma f h)=\gamma \eta(f) h,$$ for all appropriate triples $(\gamma, f,h)$.
 	\end{definition}
 	The collection of Lie groupoids forms a bicategory, whose 1-morphisms are anafunctors and 2-morphisms are transformations between anafunctors (see \textbf{Section 2.3}, \cite{MR3894086} for a detailed description of this bicategory).
 	\begin{example}\label{Functor induced anafunctor}
 		Every morphism of Lie groupoids $ \psi \colon \mb{X} \ra \mb{Y}$ induces an anafunctor $(a_{\mb{X}}, \bar{\psi},  a_{\mb{Y}})$, with total space $\bar{\psi}:= X_0 \times_{\psi, Y_0, t } Y_1$, anchors $a_{\mb{X}}$ and $a_{\mb{Y}}$ defined as $(x, \gamma) \mapsto x$ and $(x,\gamma) \mapsto s(\gamma)$, respectively. The underlying left action of $\mb{X}$ and the right action of $\mb{Y}$ on $\bar{\psi}$ are given as $\delta (x,\gamma) := (t(\delta), \psi(\delta) \circ \gamma )$ and  $ (x,\gamma) \delta:= (x, \gamma \circ \delta)$, respectively. Similarly, a smooth natural transformation $(\eta \colon \phi \Rightarrow \psi) \colon \mb{X} \ra \mb{Y}$ gives a transformation $\bar{\eta} \colon \bar{\phi} \Rightarrow \bar{\psi}$ between the corresponding induced anafunctors, defined by the smooth map $(x,\gamma) \mapsto (x, \eta(x) \circ \gamma)$. 
 	\end{example}
 	Given a pair of Lie groupoids $\mb{X}$ and $\mb{Y}$, let ${\rm{Hom}}(\mb{X}, \mb{Y})$ and ${\rm{Ana}}(\mb{X}, \mb{Y})$ denote the functor category of morphisms of Lie groupoids and the category of anafunctors, respectively. Then under the embedding mentioned  in \Cref{Functor induced anafunctor},  only a particular class of morphisms of Lie groupoids $F \colon \mb{X} \ra \mb{Y}$ maps to a \textit{weak equivalence} $(a_{\mb{X}}, \bar{F}, a_{\mb{Y}})$ i.e there exists another anafunctor $\mc{G}$ from $\mb{Y}$  to $\mb{X}$ such that $\bar{F} \circ G \cong {\rm{id}}$ and $G \circ \bar{F} \cong {\rm{id}}$, where the composition $\circ$ is as defined in the \textbf{Remark 2.3.2(a)} of \cite{MR3894086}. This particular class of morphisms of Lie groupoids is usually called \textit{Morita maps}, defined as follows:
 	
 	\begin{definition}[\cite{MR2012261},\textbf{Section 5.4}]
 		A morphism of Lie groupoids $F:\mb{X} \ra \mb{Y}$ is said to be a \textit{Morita map}, if it satisfies the following two conditions:
 		\begin{enumerate}[(i)]
 			\item the map $t \circ {\rm{pr}}_1:Y_1 \times{_{Y_0}} X_0 \ra Y_0$ is a surjective submersion
 			\[\begin{tikzcd} 
 				Y_1 \times{_{s,Y_0, F_0}} X_0 \arrow[d, "{\rm{pr}}_1"'] \arrow[r, "{\rm{pr}}_2"]                 & X_0 \arrow[d, "F_0"] \\
 				Y_1 \arrow[r, "s"'] \arrow[r, "t"', bend right=49] & Y_0               
 			\end{tikzcd};\]
 			\item
 			\[\begin{tikzcd}
 				X_1 \arrow[d, "{(s,t)}"'] \arrow[r, "F_1"] & Y_1 \arrow[d, "{(s,t)}"] \\
 				X_0 \times X_0 \arrow[r, "F_0 \times F_0"']                & Y_0 \times Y_0               
 			\end{tikzcd} \]
 			is a pullback square.
 		\end{enumerate}
 	\end{definition}
 	
 	\begin{definition}[\cite{MR2012261},\textbf{Section 5.4}]\label{Morita equivalence of Lie groupoids}
 		Two Lie groupoids $\mb{X}$ and $\mb{Y}$ are said to be \textit{Morita equivalent} if there is a third Lie groupoid $\mb{Z}$ and a pair of morphisms of Lie groupoids
 		\begin{tikzcd}
 			\mb{X} & \mb{Z} \arrow[l, "\epsilon_{\mb{X}}"'] \arrow[r, "\epsilon_{\mb{Y}}"] & \mb{Y}
 		\end{tikzcd} such that $\epsilon_{\mb{X}}$ and $\epsilon_{\mb{Y}}$ are Morita maps.
 	\end{definition}
 	
 	\begin{remark}\label{Morita equivalent imply stack}
 		It is well known that the bicategory of Lie groupoids defined above is equivalent to the 2-category of differentiable stacks as bicategories. In particular, under the correspondence, Morita equivalent Lie groupoids present the same differentiable stack; see \cite{MR2817778} for a detailed discussion. Given a Lie groupoid $\mb{X}$, the associated differentiable stack is its classifying stack, which maps a manifold $M$ to the category of principal $\mb{X}$-bundles over $M$, whose modest version has already been mentioned in the case of Lie groups, see \Cref{Remark on classifying stack}.
 	\end{remark}

\section{Principal bundles over Lie groupoids and their connection structures}\label{Principal bundles over Lie groupoids and their connection structures}
In this segment, we encounter a significant instance of a Lie groupoid's action on a manifold (\Cref{Left action of a Lie groupoid on a manifold}). Specifically, the action is such that the arrows within the associated transport category (see \Cref{transport category}) now are morphisms of Lie group torsors. This gives rise to a well-established concept in higher differential geometry, known as a `principal Lie group bundle over a Lie groupoid' (see \cite{MR2270285, MR2119241, MR1950948, MR2748598, MR4592876}). Also, there is an alternative way of seeing these objects as 1-morphisms of stacks (see \cite{MR3521476}).
These notions of principal bundles extend the conventional definition outlined in \Cref{Definition: Principal G-bundle}. We outline an existing notion of connection structure on these objects, especially the one studied in \cite{MR2270285}. There are also other notions of connections explored, particularly in \cite{MR4592876, MR3150770, biswas2023connections} through the splitting of its Atiyah sequence, with the assumption of a connection structure on the base Lie groupoid defined in \cite{MR3150770}. Moreover, as shown in \cite{MR3521476}, connection structures can also be described as morphisms of stacks. However, for our purpose in the thesis, here we restrict  ourselves to the one mentioned in \cite{MR2270285}.

\subsection{Principal bundles over Lie groupoids}
We start the subsection by recalling the definition of a principal bundle over a Lie groupoid.
\begin{definition}[Definition 2.2, \cite{MR2270285}]\label{Definition: Principal Lie group bundle over a Lie groupoid}
	Given a Lie group $G$, a \textit{principal $G$-bundle over a Lie groupoid} $\mb{X}$ consists of a principal $G$-bundle $\pi\colon E_G \rightarrow X_0$ endowed with a smooth 
	map $\mu\colon s^{*}E_G:= (X_1 \times_{s, X_0, \pi} E_G) \rightarrow E_G$, satisfying the following conditions:
	\begin{enumerate}[(i)]
		\item $\mu$ defines a left action of $\mb{X}$ on $E_G$ (\Cref{Left action of a Lie groupoid on a manifold}), that is
		\begin{itemize}
			\item[(a)] $\mu(1_{\pi(p)}, p)=p$ for all $p \in E_G$,
			\item[(b)] $\bigl(\gamma, \mu(\gamma,p)\bigr) \in t^{*}E_G$ for each $(\gamma, p) \in s^{*}E_G $,
			\item[(c)] $\mu(\gamma_2 \circ \gamma_1,p)=\mu(\gamma_2,\mu(\gamma_1,p))$ for all $\gamma_2 , \gamma_1 \in X_1$ satisfying $t(\gamma_1)=s(\gamma_2)$ and $(\gamma_1,p) \in s^{*}E_G$.
		\end{itemize}
		\item $\mu$ commutes with the right action of $G$ on $E_G$, that is $\mu(\gamma, p)g=\mu(\gamma, pg).$ for all $p \in E_G, g \in G$ and $\gamma \in X_1$.
	\end{enumerate}
\end{definition}
A principal $G$-bundle over the Lie groupoid $\mathbb{X}= [X_1 \rightrightarrows X_0]$ defined above, will be denoted by $\bigl(\pi\colon E_G \rightarrow X_0, \mu, \mb{X} \bigr)$. 

\begin{definition}\label{morphism of principal bundles}
	Given a Lie group $G$ and a Lie groupoid $\mb{X}$, \textit{a morphism of principal $G$-bundles over $\mb{X}$ from $(\pi\colon E_G \rightarrow X_0, \mu, \mb{X})$ to $(\pi'\colon E'_G \rightarrow X_0, \mu', \mb{X})$ }is defined as a morphism of principal $G$-bundles $f \colon E_G \ra E'_G$ over $X_0$ such that it is compatible with the Lie groupoid actions i.e  $f(\mu(\gamma,p))= \mu'(\gamma,f(p))$ for all $(\gamma,p) \in s^{*}E_G$.
\end{definition}
The collection of principal $G$-bundles over $\mb{X}$ along with the morphisms \Cref{morphism of principal bundles} naturally defines a groupoid, denoted as ${\rm{Bun}}(\mb{X}, G)$. 

\begin{example}[Traditional principal bundle]\label{Tradition principal bundle as bundle over  Lie groupoid}
	For a Lie group $G$ and a manifold $M$, a principal $G$-bundle $\pi \colon P \ra M$ is a principal $G$-bundle over the discrete Lie groupoid $[M \rra M]$, whose action map is given by $(1_x, p) \mapsto p$, for all $x,p$ satisfying $\pi(p)=x$.
\end{example}
\begin{example}[Principal bundle with a flat connection]
	Given a Lie group $G$, consider a principal $G$-bundle $\pi \colon P \ra M$ with a flat connection $\omega$, then the action of the fundamental groupoid $\Pi(M)$ on $P$ as given in \Cref{Flat connection action}, defines a principal $G$-bundle over $\Pi(M)$.
\end{example}
\begin{example}[Example 3.6, \cite{MR4592876}]
	For a pair of Lie group $G$ and $H$, a principal $G$-bundle over the Lie groupoid $[H \rra *]$ is the same as a left action of $H$ on $G$ such that it commutes with the right translation of $G$ on itself.
\end{example}

The following characterization that is equivalent to the notion defined in \Cref{Definition: Principal Lie group bundle over a Lie groupoid} holds particular importance for our objective.
\begin{definition}\label{Definition: Principal G-groupoids}[\cite{MR2270285}, \textbf{Section 2}]
	For a Lie group $G$ and Lie groupoid $\mb{X}$, a \textit{principal $G$-groupoid over the Lie groupoid $\mb{X}$} is defined as a Lie groupoid $\mb{E}$ together with a morphism of Lie groupoids $\pi \colon \mb{E} \ra \mb{X}$, such that both $\pi_1 \colon E_1 \ra X_1$ and $\pi_0 \colon E_0 \ra X_0$ are principal $G$-bundles and the source-target maps are morphisms of principal $G$-bundles.
\end{definition}

\begin{definition}\label{Morphism of principal G groupoids}
	For a Lie groupoid $\mb{X}$ and a Lie group $G$, let $\pi \colon \mb{E} \ra \mb{X}$ and $\pi' \colon \mb{E}' \ra \mb{X}$ be a pair of principal $G$-groupoids over the Lie groupoid $\mb{X}$. A \textit{morphism of principal $G$-groupoids over $\mb{X}$} is a morphism of Lie groupoids $F \colon \mb{E} \ra \mb{E}'$ such that $F_0$ and $F_1$ are morphisms of principal $G$-bundles over $X_0$ and $X_1$, respectively.
\end{definition}
Given a Lie group $G$ and a Lie groupoid $\mb{X}$, the collection of principal $G$-groupoids over $\mb{X}$ and the corresponding morphisms (\Cref{Morphism of principal G groupoids}) define a groupoid denoted as ${\rm{Bun}}(\mb{X}, [G \rra G])$.

This notion of principal bundle \Cref{Definition: Principal Lie group bundle over a Lie groupoid} can be extended over the differentiable stack presented by the base Lie groupoid. Specifically, we have the following:

\begin{proposition}[Corollary 2.12, \Cref{Definition: Principal Lie group bundle over a Lie groupoid}]\label{principal Lie group bundles over stacks}
	If $\mb{X}$ and $\mb{Y}$ are Morita equivalent (\Cref{Morita equivalence of Lie groupoids}), then ${\rm{Bun}}(\mb{X}, G)$ and ${\rm{Bun}}(\mb{Y}, G)$ are equivalent as categories.
	
\end{proposition}

\subsection*{Correspondence between principal G-bundles and principal G-groupoids over a Lie groupoid}\label{Correspondence between principal G-bundles and principal G-groupoids over a Lie groupoid}

For a Lie group $G$ and a Lie groupoid $\mb{X}$, let $(\pi_G \colon E_G \rightarrow X_0, \mu, \mb{X})$ be a principal $G$-bundle over $\mb{X}$. Let $[s^{*}E_G \rra E_G]$ denote the associated  semi-direct product groupoid (\Cref{Semi-direct product groupoid}). There is an obvious morphism of Lie groupoids $\pi \colon [s^{*}E_G \rra E_G] \ra [X_1 \rra X_0]$ defined as $\pi_1= {\rm{pr}}_1$  and $\pi_{0}= \pi_G$, where ${\rm{pr_1}}$ is the 1st projection map. Observe that we get a pair of principal $G$-bundles $\pi_1 \colon s^{*}E_G \ra X_1$ and $\pi_0 \colon E_G \ra X_0,$ such that source and target maps are morphisms of principal $G$-bundles. Conversely, given a principal $G$-groupoid $\pi \colon  \mb{E} \ra \mb{X}$ over a Lie groupoid $\mb{X}$, consider the pair of maps $\pi_0 \colon E_0 \ra X_0$ and the map $s^{*}E_0 \ra E_0$ defined by $(\gamma,p) \mapsto t(\delta)$, where $\delta$ is the unique element in $\pi_1^{-1}(\gamma)$ such that $s(\delta)=p$. They define a principal $G$-bundle over $\mb{X}$. The above correspondence is inverse of each other and can be extended to define an equivalence of categories between $\rm{Bun}(\mb{X}, G)$ and $\rm{Bun}(\mb{X}, [G \rra G])$.\label{Equivalence of principal G bundles and principal G groupoids}

\begin{remark}\label{Different characterizations of principal Lie bundles over Lie groupoids}
	Apart from the characterization discussed above, a principal $G$-bundle over a Lie groupoid $\mb{X}$ can also be characterized in other ways, such as 
	\begin{enumerate}[(i)]
		\item a \textit{principal $G$-bundle over the nerve $N(\mb{X})$ of the Lie groupoid $\mb{X}$} ( \Cref{Definition: Nerve of a Lie groupoid}) i.e a simplicial manifold $\mc{E}:=(E_n)_{n \in \mb{N} \cup \lbrace 0 \rbrace}$ such that 
		\begin{itemize}
			\item for every $n \in \mb{N} \cup \lbrace 0 \rbrace$, $E_n$ is a principal $G$-bundle over the manifold $N(\mb{X})_n$;
			\item The degeneracy and the face maps are morphisms of principal $G$-bundles.
		\end{itemize}
		\item An anafunctor from $\mb{X}$ to the Lie group $[G \rra *]$, (\Cref{generalized homomorphism}).
		\item A morphism of stacks $\mc{X} \ra BG$ (see \Cref{Remark on classifying stack}), where $\mc{X}$ is the stack representing the Lie groupoid $\mb{X}$ (see \Cref{Morita equivalent imply stack}) and $BG$ is the classifying stack of the Lie group $G$ (\Cref{Classifying stack of a Lie group}).
	\end{enumerate}
	Despite having their independent significances in higher differential geometry, we will not delve into the detailed exploration of the descriptions mentioned above, as they do not exactly align with the current objectives of this thesis. However, they are expected to play important roles in some future projects (\Cref{Future}) based on the research presented in this thesis. We suggest \cite{MR2270285} for readers interested in a detailed treatment  for (i) and (ii). For (iii), see \cite{MR3521476}.
\end{remark}

\subsection{Connections on principal bundles over Lie groupoids}\label{Section: Connection on principal bundles over Lie groupoids}

As proposed at the beginning of this section, here we recall the definition of `connection on a principal bundle over a Lie groupoid' as given in \cite{MR2270285}. Readers interested in definitions involving connections on the base Lie groupoids can look at \cite{MR4592876, MR3150770, biswas2023connections} and for a definition as a morphism of stacks, see \cite{MR3521476}.

\begin{definition}\label{Connections on a principal bundle over a Lie groupoid}[Definition 3.5, \cite{MR2270285}]
	For a Lie group $G$ and a Lie groupoid $\mb{X}$, the \textit{connection on a principal $G$-bundle $\bigl(\pi\colon E_G \rightarrow X_0, \mu, \mb{X} \bigr)$ over $\mb{X}$} is defined as a connection $\omega$ on $\pi \colon E_G \ra X_0$ such that $s^{*}\omega= t^{*}\omega$.
\end{definition}
Note that the above definition coincides with the traditional one in \Cref{Chapter Classical setup} for the particular case discussed in \Cref{Tradition principal bundle as bundle over  Lie groupoid}.

The following gives criteria for the existence of connection on a principal Lie group bundle over a Lie groupoid.

\begin{proposition}[Theorem 3.16, \cite{MR2270285}]\label{Existence of connection on a Lie group bundle over a groupoid}
	For a Lie group $G$, any principal $G$-bundle over a proper étale Lie groupoid (\Cref{proper etale Lie groupoid}) admits a connection as given in \Cref{Connections on a principal bundle over a Lie groupoid}.
\end{proposition}

\section{Lie 2-group and its Lie 2-algebra}
The notion of a categorified version of a group, namely a `2-group', first appeared in the 1970s through the works of Solian \cite{MR0316529}, Sính's Ph.D. thesis (supervised by Grothendieck) and Brown-Spencer's \cite{MR0419643}. A smooth version of the same (in terms of a crossed module description of the automorphism 2-group of a Lie group)  found its importance in the geometry of non-abelian gerbes through Breen-Messing's \cite{MR2183393}. However, the terminology, as a `Lie 2-group' and its explicit description as a group object in the category of Lie groupoids, first appeared in Baez's \cite{baez2002higher}. Since then the notion gathered a significant amount of importance in the Higher gauge theory community as evident from the works of Bartel's \cite{MR2709030}, Baez-Schreiber's \cite{baez2004higher, MR2342821}, Wockel's \cite{MR2805195}, Picken-Martin's \cite{MR2661492}, Baez-Lauda's \cite{MR2825807}, Jurco's \cite{MR3351282}, Chatterjee-Lahiri-Sengupta's \cite{MR3126940}, Waldorf's \cite{MR3894086, MR3917427} and many others, to name a few. There is also a corresponding notion of the Lie 2-algebra of a Lie 2-group, categorifying the traditional one. This notion has appeared in almost all references mentioned above concerning Lie 2-groups. In the present day, the Lie 2-group and its Lie 2-algebra are considered standard notions in Higher gauge theory literature.

This section overviews the concept of a Lie 2-group, its Lie 2-algebra, and their correspondence. Also, we discuss the actions of Lie 2-groups on Lie groupoids. Most of the materials in this section are standard, and we recommend \cite{MR2342821, MR2068521, MR2068522, baez2004higher, MR3126940, MR3504595} for further reading on these topics.

\begin{definition}\label{Definition: Lie 2 group}
	A \textit{Lie $2$-group} is a Lie groupoid $\mb{G}$ equipped with a morphism of Lie groupoids $\otimes: \mb{G} \times \mb{G} \rightarrow \mb{G}$ such that $\otimes$ induces Lie group structures on both ${\rm{Obj}}(\mb{G})$ and ${\rm{Mor}}(\mb{G})$.
\end{definition}

\label{Group object}Let $\mb{C}$ be a category with finite products and the terminal object 1. Recall,  a \textit{group object} in $\mb{C}$ is an object $G$ with morphisms
\begin{itemize}
	\item $m:G \times G \ra G$
	\item $e:1 \ra G$
	\item ${\rm{inv}}:G \ra G$
\end{itemize}
such that $m, e$ and ${\rm{inv}}$ satisfy the standard properties of the composition, the identity, and the inverse in a group, respectively. Hence, alternatively but equivalently, one can define a Lie 2-group  as a group object in LieGpd, the category of Lie groupoids. Also, it is evident that all its structure maps are Lie group homomorphisms. 

\begin{lemma}\label{Lemma: Interchange} Let the notations ${}^{-1}$ and $\circ$ denote the group inverse functor and the composition, respectively. Then we have the following:
	\begin{equation}\nonumber
		\begin{split}
			&(k_2\otimes k_1)\circ (k'_2\otimes k'_1)=(k_2\circ k'_2)\otimes (k_1\circ k'_1),\\
			&(k_2\circ k_1)^{-1}={k_2}^{-1}\circ {k_1}^{-1}.		
		\end{split}
	\end{equation}
\end{lemma}
\begin{proof}
 Follows directly from the functoriality of ${}^{-1}$ and $\otimes$.
\end{proof}

\begin{remark}
	In the existing literature, there are more general versions of the concept of a Lie $2$-group, such as a semistrict Lie 2-group, a coherent Lie 2-group, a weak Lie 2-group, etc. obtained by weakening the axioms of a group object in varying degrees. However, \Cref{Definition: Lie 2 group} is sufficient for our purpose in this thesis. Readers interested in these more general notions of Lie 2-groups are referred to \cite{MR2068521}, \cite{MR3351282} and \cite{MR2709030}.
\end{remark}

An equivalent description of a Lie $2$-group is given by a Lie crossed module, defined below:

\begin{definition}\label{Definition:Lie crossed module}
	A \textit{Lie crossed module} is defined as a $4$-tuple 
	$(G,H, \tau, \alpha)$ such that
	\begin{enumerate}[(i)]
		\item $G$ and $H$ are Lie groups,
		\item $\alpha:G\times H\ra H$ is a smooth action of $G$ on $H$ such that $\alpha(g,-):H \ra H$ is a Lie group homomorphism for each $g \in G$, 
		\item $\tau:H\ra G$ is a homomorphism of Lie groups, 
	\end{enumerate}
	satisfying the \textit{Peiffer identites}:
	\begin{equation}\label{E:Peiffer}
		\begin{split}
			& \tau(\alpha(g,h))=g\tau(h)g^{-1} \, \textit {\rm for all} \, (g,h) \in G\times H,\\
			& \alpha(\tau(h),h')=hh'h^{-1} \, \textit {\rm for all}\, h,h'\in H.
		\end{split}
	\end{equation}

\end{definition}
\subsection{Correspondence between Lie $2$-groups and Lie crossed modules}\label{Lie 2 group from Lie crossed module}
Given a Lie crossed module $(G, H, \tau, \alpha)$, there is a canonical way to associate a  Lie $2$-group and is given by the Lie groupoid $ \mb{G}=[H \rtimes_{\alpha} G \rightrightarrows G],$ where $\rtimes$ denotes the semidirect product of groups $H$ and $G$ with respect to the action $\alpha$ of $G$ on $H.$ Structure maps on this Lie $2$-group are given below:
\begin{itemize}
	\item $s(h,g) :=g$, for all $(h,g) \in H \rtimes_{\alpha} G$;
	\item  $t(h,g) :=\tau(h)g$, for all $(h,g) \in H \rtimes_{\alpha} G$;
	\item The composition is defined as $m((h_2, g_2),(h_1, g_1)) :=(h_2 h_1,g_1)$, for all appropriate $h_2,h_1,g_2,g_1$;
	\item $1_g :=(e,  g),$ for all $g \in G$;
	\item $\mathfrak{i}(h,g) :=(h^{-1},\tau(h)g)$ for all $(h,g) \in  H \rtimes_{\alpha} G$;
	\item The group structure $H\rtimes_{\alpha} G$ is the standard semidirect product of the groups, that is 
	the morphism of Lie groupoids $\otimes: \mb{G} \times\mb{G} \rightarrow \mb{G}$ is given by
	\begin{equation}\label{E:grouppro}
		\begin{split}
			&\otimes_0: (g_1, g_2) \mapsto g_1g_2,
			\\
			&\otimes_{1}:  ((h_2, g_2), (h_1, g_1)) \mapsto (h_2 \alpha(g_2,h_1), g_2 g_1),
		\end{split}
	\end{equation}
\end{itemize}
while the group inverse and the identity elements are  given by $(h,g)^{-1}=(\alpha(g^{-1}, h^{-1}), g^{-1})$ and  $(e,e)$ respectively.
It is a straightforward verification that  $\mb{G}$ is a Lie $2$-group. We refer to $\mb{G}=[H \rtimes_{\alpha} G \rightrightarrows G]$ as the \textit{Lie $2$-group associated  to the Lie crossed module} $(G,H,\tau, \alpha)$.

Conversely, given a Lie 2-group  $\mb{G}=[G_1\rra G_0]$, there is a Lie crossed module given as \[(G_0, \ker(s), t|_{\ker(s)}:\ker(s)\ra G_0,\alpha:G_0\times \ker(s)\ra \ker(s)),\] where,
\begin{itemize}
	\item $\ker(s)=\{\gamma\in G_1: s(\gamma)=1_{G_0}\}$;
	\item the smooth map $\alpha:G_0\times \ker(s)\ra \ker(s)$ is defined as $(a,\gamma)\mapsto 1_a  \gamma 1_{a^{-1}}$.
\end{itemize}
The above Lie crossed module is said to be the \textit{Lie crossed module associated to the Lie $2$-group} $\mb{G}$.

The association above defines a one-one correspondence between Lie 2-groups and Lie crossed modules. We suggest the readers look at  \textbf{Section 2}, \cite{baez2002higher} for a detailed proof of the correspondence.

\begin{example}\label{Ex:Discretecrossedmodule}
	For any Lie group $G$, the associated discrete Lie groupoid $[G\rra G]$ is a Lie 2-group. The Lie crossed module associated to it is given by $(G, \{e\})$ with trivial $\tau$ and $\alpha.$
\end{example}

\begin{example}\label{Ex:ordinary}
	For an abelian Lie group $H$, the Lie groupoid $[H\rra {e}]$ is a Lie $2$-group. The Lie crossed module associated to it is given as $(\{e\}, H)$ with trivial $\tau$ and $\alpha :={\rm Id}\colon H\to H$.  		
\end{example}
\begin{example}\label{Ex:isocrossedmodule}
	For any simply connected Lie group $G$, there is a Lie crossed module  $\bigl({\rm Aut}(G), G, \tau, \alpha\bigr),$ where $\tau\colon G\ra {\rm Aut}(G)$ maps an element $g\in G$
	to its inner-automorphism, and $\alpha\colon {\rm Aut}(G)\ra {\rm Aut}(G)$ is the identity map. The Lie $2$-group associated to it is called the \textit{automorphism $2$-group} of $G$.	\end{example}

\begin{example}\label{Ex:singleobjlie2iso}
	Consider the Lie crossed module $(G, G, \tau, \alpha)$ where $\tau\colon G\ra G$ is the identity map, and $\alpha\colon G\times G\ra G$ is the action of $G$ on itself by conjugation. The Lie $2$-group associated to it is the pair groupoid $[G\times G\rra G]$, \Cref{Lie groupoid example:Pair groupoid}.
\end{example}

\subsection{The Lie 2-algebra of a Lie 2-group}\label{Lie 2-algebra as Lie crossed module}

Given a Lie $2$-group $\mb{G}=[G_1\rra G_0]$, there is a Lie groupoid $[L(G_1) \rra L(G_0)]$ whose structure maps are obtained by taking differentials of the structure maps of $\mb{G}$ at the identity. Particularly,  for the Lie $2$-group $\mb{G}=[G_1 \rra G_0]$ associated to the Lie crossed module  $(G,H,\tau:H\ra G, \alpha:G\times H\ra H)$, the Lie groupoid  $[L(G_1) \rra L(G_0)]$ is described as   $$[L(H)\oplus L(G)\rra L(G)],$$ whose structure maps are as follows:
\begin{itemize}
	\item $s(A,B) :=B$, for all $(A,B) \in L(H)\oplus L(G)$;
	\item  $t(A,B) :=\tau(A)+B$, for all $(A,B) \in L(H)\oplus L(G)$;
	\item The composition is given as $m((A_2,B_2),(A_1,B_1)) :=(A_2+A_1,B_1)$, where 
	$s(A_2,B_2)=t(A_1,B_1)$,
	\item $1_B :=(0, B)$, for all $B \in L(G)$;
	\item  $\mathfrak{i}(A,B)=(-A,\tau(A)+B)$, for all $(A,B) \in L(H)\oplus L(G)$.
\end{itemize}

\begin{remark}
	One must note that $L(G_1)$ splits into the direct sum $L(H)\oplus L(G)$ only as a vector space and is not a direct sum of Lie algebras.
\end{remark} 
Differential of the action map $\alpha$ induces an action $\alpha_{*,( e, e)}\colon L(G)\times L(H)\ra L(H)$ of $L(G)$ on $L(H)$ as a Lie algebra derivation, and the commutators on $L(G)$ and $L(H\rtimes G)$ are respectively given by
\begin{equation}\label{E:Commutator}
	\begin{split}
		(B_1, B_2)&\mapsto [B_1, B_2],\\
		\bigl((A_1, B_1), (A_2, B_2)  \bigr)&\mapsto [(A_1, B_1), (A_2, B_2)]\\
		& :=\bigl([A_1,A_2]+\alpha_{*,( e, e)} (A_1, B_2)-\alpha_{*,( e, e)} (A_2, B_1), [B_1, B_2]\bigr),
	\end{split}
\end{equation}
for all $A_1, A_2\in L(H),  B_1, B_2\in L(G)$. The Lie groupoid $[L(G_1)\rra L(G_0)]$ defined above, is the standard strict Lie $2$-algebra of the Lie $2$-group $[G_1\rra G_0]$. 
Although we won't explicitly state the general definition of a Lie $2$-algebra here, for an in-depth exploration of Lie 2-algebras and the corresponding Lie 2-algebra of a Lie 2-group, which also include some weaker notions of the same, we refer to \cite{MR2068522}.

For a Lie crossed module $(G,H,\tau, \alpha)$, consider the smooth map 
$\alpha: G\times H\ra H.$ Fixing $g\in G$, the map $\alpha(g)\colon H\to H$ defined by $h\mapsto \alpha(g, h)$, is a smooth map and its differential at the identity element of $H$ gives the following Lie algebra homomorphism
\begin{equation}\label{E:alphadifffixedg}
	\alpha(g)_{*, e_H}\colon L(H)\to L(H).
\end{equation}
While for a fixed $h\in H,$ we get the smooth map ${\bar \alpha}(h)\colon G\to H, g\mapsto \alpha(g, h)$, and its differential at the identity element of $G$
is given by the following linear map
\begin{equation}\label{E:alphadifffixedh}
	{\bar \alpha}(h)_{*, e_G}\colon L(G)\to T_hH.
\end{equation}
Another way to interpret  the map ${\bar \alpha}(h)_{*, e_G}$	is by considering the right action of $(h, e_G)$ on 
$H \rtimes_{\alpha}G,$ $(h',  g')\mapsto (h'\alpha_{g'}(h), g').$ Then differential of this map at the identity $(e_H, e_G),$ for an element
$B\in L(G) \subset L(H\rtimes G)$ gives $(0, B)\mapsto \bigl({\bar \alpha}(h)_{*, e_G}(B), B\bigr)\in T_hH\oplus L(G)=T_{(h,\, e_G)}(H \rtimes_{\alpha}G).$ On the other hand the left action of $(h, e)$ on $H \rtimes_{\alpha}G$ gives the  map $(0, B)\mapsto (0, B).$

\begin{remark} \label{Remark:notations}
	From this point forward, we will often adhere to the following conventions to simplify our notations and reduce the complexity of symbols.
	\begin{itemize}
		
		\item We will denote $\tau_{*, e_H}$  as $\tau.$ Likewise, we will denote $\alpha(g)_{*, h}$, $\alpha(g)_{*, e_H}$ and ${\bar \alpha}(h)_{*, e_G}$ as $\alpha(g)_h, \alpha(g)$ and ${\bar \alpha}(h),$ respectively.
		\item To avoid the cluttering of parentheses, many times we write $\alpha(g)$ etc as $\alpha_g$ etc.
		\item Unless necessary, we do not make a notational distinction between the identity elements of various groups. The typical notation for these will be $e.$

	\end{itemize} 
\end{remark}

The differential versions of the  Peiffer identities in \eqref{E:Peiffer} then read as follows,
\begin{equation}\label{E:Peifferdiff}
	\begin{split}
		&\tau(\alpha(g)(A))={\rm ad}_{g}\tau(A), \textit {\rm for all} \, g\in G, A\in L(H),\\
		&\tau({\bar \alpha}(h)(B))=B\cdot \tau(h)-\tau(h)\cdot B,\textit {\rm for all} \, h\in G, B\in L(G),\\
		&\alpha(\tau(h))(A)={\rm ad}_{h}A,  \textit {\rm for all} \, h\in H, A\in L(H),\\
		&{\bar \alpha}(h')(\tau(A))=h'\cdot A-A\cdot h',\, \textit {\rm for all} \, h'\in H, A\in L(H),
	\end{split}
\end{equation}
Note in the left-hand side of the second equation $\tau$ means $\tau_{*,  h}.$

\subsection{Adjoint actions of a Lie $2$-group }\label{SS:adjointLie 2 group}\
For the convenience in later calculations, we express the adjoint actions in terms of a Lie crossed module.

Consider the Lie 2-group $[H \rtimes_{\alpha}G \rra G]$ associated to a Lie crossed module $(G, H, \tau, \alpha)$.
Now, observe that identifying $H$ and $G$ respectively with $H\times\{e\}$ and $\{e\}\times G,$ we have 
\begin{equation}\label{E:product}
	H \rtimes_{\alpha}G=HG.
\end{equation}
Action on the $G$ is obvious and is given by the usual adjoint action ${{\rm Ad}}_g\colon g' \mapsto gg'g^{-1}$, $g, g' \in G$. While, to compute the adjoint action for $H \rtimes_{\alpha}G,$ we simplify our computation using \Cref{E:product}, by writing ${{\rm Ad}}_{(h, g)}(h', g')$ as ${{\rm Ad}}_{(h, g)}(h', g')=[{{\rm Ad}}_{(h, g)}(h', e)][{{\rm Ad}}_{(h, g)}(e, g')]$, and then calculate the bracketed terms on the right-hand side individually.  Then, \eqref{E:grouppro} yields
\begin{equation}\label{E:Adjongroups}
	\begin{split}
		&{{\rm Ad}}_{(h,\, g)}(h', e)=\biggl({\rm Ad}_{h}\bigl(\alpha_g (h ')\bigr),\, e\biggr), \\
		&{{\rm Ad}}_{(h,\, g)}(e, g')=\biggl(h\, \bigl({\alpha}_{({\rm Ad}_g(g '))}(h^{-1})\bigr) , \, {\rm Ad}_g (g')\biggr)=\biggl(h\, \bigl({\bar \alpha}_{h^{-1}}{({\rm Ad}_g(g '))}\bigr) , \, {\rm Ad}_g (g')\biggr),\\
		&{{\rm Ad}}_{(h, g)}(h', g')=[{{\rm Ad}}_{(h, g)}(h', e)][{{\rm Ad}}_{(h, g)}(e, g')].
	\end{split}
\end{equation}

For an element $(A, B)\in L(H)\oplus L(G)$, we write $(A, B)=(A, 0)+(0, B).$ The  adjoint actions of $H\rtimes G$ on $A\in L(H)$ and $B\in L(G)$ are then respectively calculated by  ${\rm{ad}}_{(h, g)}(A):=\frac{{\rm{d}}}{{\rm{d}}t}{\rm{Ad}}_{(h, g)}{{\rm e}}^{t A}\big|_{t=0}$ and  ${\rm{ad}}_{(h, g)}(B):=\frac{{\rm{d}}}{{\rm{d}}t}{\rm{Ad}}_{(h, g)}{{\rm e}}^{t B}\big|_{t=0},$	
\begin{equation}\label{E:Adjonalgebras}
	\begin{split}
		&{\rm{ad}}_{(h,\, g)}(A, 0)=\biggl({\rm ad}_{h}\bigl(\alpha_g (A)\bigr),\, 0\biggr), \\
		&{\rm{ad}}_{(h,\, g)}(0, B)=\biggl(h\,\cdot \bigl({\bar \alpha}_{h^{-1}}{({\rm ad}_g(B))}\bigr) , \, {\rm ad}_g (B)\biggr),\\
		&{\rm{ad}}_{(h, g)}(A, B)=[{\rm{ad}}_{(h, g)}(A, 0)]+[{\rm{ad}}_{(h, g)}(0, B)].
	\end{split}
\end{equation}

The symbols $\alpha_g, \bar \alpha_{h^{-1}}$ in \eqref{E:Adjonalgebras} should be understood in accordance with the \Cref{Remark:notations} (1st and 2nd bullet points). 

Using \eqref{Lemma: Interchange} one can give a functorial description of the adjoint actions.
\begin{lemma}
	Let $\mb{G}=[G_1\rra G_0]$ be a Lie $2$-group and $L(\mb{G})=[L(G_1)\rra L(G_0)]$ the associated Lie $2$-algebra. Then we have the following morphisms of Lie groupoids:
	\begin{equation}\nonumber
		\begin{split}
			{{\rm Ad}}\colon &\mb{G}\times  \mb{G}\ra  \mb{G},\\
			& (g, g')\mapsto {{\rm Ad}}_g(g'),  \, \forall g, g'\in G_0,\\
			& (k, k')\mapsto {{\rm Ad}}_k(k'), \, \forall k, k'\in G_1\\
			& \hskip 3 cm {\rm and} \\
			{\rm{ad}}\colon &\mb{G}\times L(\mb{G})  \ra  L(\mb{G}),\\
			& (g, B)\mapsto {{\rm ad}}_g(B),  \, \forall g\in G_0, B\in L(G_0),\\
			&(k, D)\mapsto {\rm{ad}}_k(K), \, \forall k\in G_1, K\in L(G_1).
		\end{split}
	\end{equation}
	
\end{lemma}

\subsection{Action of a Lie 2-group on a Lie groupoid}\label{Action of a Lie 2-group on a Lie grouipoid}
This subsection recalls the notion of an action of a Lie 2-group on a Lie groupoid. This notion categorifies the action of a Lie group on a manifold. The concept is well known in the current body of Higher gauge theory literature and has appeared in various forms; see \cite{MR2342821} \cite{MR3894086}, \cite{MR2805195}, \cite{MR3126940},\cite{MR3917427}, to name a few.

\begin{definition}[Action of a Lie $2$-group on a Lie groupoid]\label{Definition:Lie 2 groupaction}
	For a Lie 2 group $\mb{G}$, an \textit{action of $\mb{G}$ on a Lie groupoid $\mb{X}$} is defined as a morphism of Lie groupoids $\rho: \mb{X} \times \mb{G} \rightarrow \mb{X}$ such that following pair of maps
	\begin{itemize}
		\item $\rho_0: X_0 \times G_0 \rightarrow X_0$,
		\item $\rho_1:X_1 \times G_1 \rightarrow X_1$,
	\end{itemize}
	define Lie group actions on manifolds $X_0$ and $X_1$ respectively.
\end{definition} 
Functoriality of the above action immediately induces the following property:
\begin{lemma}
	Given an action of a Lie 2-group $\mb{G}=[G_1 \rra G_0]$ on a Lie groupoid $\mb{X}$, we have the following identity:
	\begin{equation}\label{E:Identitiesechangeinverse}
		(\gamma_2\phi_2)(\gamma_1\phi_1)=(\gamma_2 \circ \gamma_1)(\phi_2 \circ \phi_1),
		\end{equation}
		for such that $\gamma_2, \gamma_1 \in X_1$ and $\phi_2, \phi_1 \in G_1$ are respectively composable.
	\end{lemma}
\begin{example}\label{Adj action}
	For any Lie 2-group $\mb{G}$, there is an action of the Lie $2$-group $\mb{G}$ on the
	Lie groupoid  $L(\mb{G})=[L(G_1) \rightrightarrows L(G_0)]$ (\Cref{Lie 2-algebra as Lie crossed module}), defined by the adjoint actions of $G_1$ and $G_0$.
\end{example}
\begin{example}\label{Tangent action}
	Any action of a Lie 2-group $\mb{G}$ on a Lie groupoid $\mb{X}$ induces an action of $\mb{G}$ on the tangent Lie groupoid $T\mb{X}=[TX_1 \rightrightarrows TX_0]$ (\Cref{Tangent Lie groupoid}), given by the differential of respective actions.
\end{example}
\begin{definition} \label{Equivariant morphism of Lie gorupoid}
	Let $\mb{G}$ be a Lie 2-group acting on a pair of Lie groupoids $\mb{X}$ and  $\mb{Y}.$ Then a morphism of Lie groupoids $F:=(F_1, F_0): \mb{X} \rightarrow \mb{Y}$ is defined as \textit{$\mb{G}$-equivariant morphism of Lie groupoids} if $ F_i $ is $G_i$ equivariant for each $i=0,1$.
\end{definition}

\begin{definition}\label{Equivariant 2-morphism of Lie groupoid}
	Suppose a Lie $2$-group $\mb{G}$ acts on a pair of Lie groupoids $\mb{X}$ and  $\mb{Y}$ and let $F, F' \colon \mb{E} \ra \mb{E}'$ be two $\mb{G}$-equivariant morphisms of Lie groupoids. A smooth natural transformation $\eta \colon F \longrightarrow F'$  from $F$ to $F'$
	is called a \textit{$\mb{G}$-equivariant smooth natural transformation} if $\eta(x g)=\eta(x)1_g$ for all $x \in E_0$ and $g \in G_0$.
\end{definition}
\begin{remark}\label{Lie 2-group equivariant anafunctor}
	As we have seen in \Cref{Category and 2-category of Lie groupoids}, for a fixed Lie 2-group $\mb{G}$, the collection of Lie groupoids (equipped with an action of $\mb{G}$), $\mb{G}$-equivariant morphisms of Lie groupoids and $\mb{G}$-equivariant smooth natural transformations, naturally forms a strict 2-category. Also, the bicategory of  Lie groupoids, anafunctors, and transformations (\Cref{subsection Generalized homomorphisms and Morita equivalence of Lie groupoids}) has a $\mb{G}$-equivariant version, (see \cite{MR3894086} for details) and the corresponding anafunctors (\Cref{generalized homomorphism}) are called `$\mb{G}$-equivariant anafunctors'. Although the study conducted in the thesis does not involve such generality, we believe some ideas discussed here may have some interesting consequences when appropriately weakened to the set up of anafunctors.

\end{remark}
\begin{remark}
	Compared to the one given in \Cref{Definition:Lie 2 groupaction}, a weaker version of the action of a Lie 2-group on a Lie groupoid also exists in literature; for example, in  \cite{MR2805195}, the identity and compatibility axioms of a group action at the object level hold up to an isomorphism. The action can be generalized by replacing the category $\mb{X}\times \mb{G}$ 
	by a `twisted product  category' $\mb{X}\rtimes_{\eta}\mb{G}$ of $\mb{X}$ and $\mb{G}$ with respect to a specific map $\eta$, introduced in \cite{MR3213404}. 
	Although we will stick to the definition of the action given in \Cref{Definition:Lie 2 groupaction} for this thesis and will not delve deeply into either of the other two generalizations, however, we recall the definition of a twisted product to relate with the notion of Lie groupoid $G$-extensions (\Cref{subsection Lie groupoid G extension}), later in this thesis.
\end{remark}

Consider a Lie 2-group $\mb{G}$ and a Lie groupoid $\mb{X}$. Let $\eta\colon {\rm{Mor}}(\mb{X})\times {\rm{Mor}}(\mb{G})\ra {\rm{Mor}}(\mb{X})$ be a smooth map satisfying,
\begin{equation}\label{E:Contwist}
	\begin{split}
		\eta(\gamma, k)\in {\rm Hom}_{\mb{X}}(x, y), 	& \,\,\forall \gamma\in {\rm Hom}_{\mb{X}}(x, y), \\
		\eta(\gamma_2\circ \gamma_1, k)&=\eta(\gamma_2, k)\circ \eta(\gamma_1, k),\\
		\eta(1_x, k)&=1_x,\\
		\eta(\gamma, k_2\circ k_1)&=\eta(\eta(\gamma, k_2), k_1), \\ 
		\eta(\gamma, 1_g)&=\gamma.
	\end{split}
\end{equation}
Then  $\mb{X}\rtimes_{\eta}\mb{G}$ is a category internal to Man, the category of smooth manifolds, with the  description below, ( \cite[Proposition $5.1$]{MR3213404}):
\begin{equation}\label{EquationE:twisted category}
	\begin{split}
		&\rm{Obj}(\mb{X}\rtimes_{\eta}\mb{G})= \rm{Obj}(\mb{X})\times \rm{Obj}(\mb{G}),\\
		&\rm{Mor}(\mb{X}\rtimes_{\eta}\mb{G})= \rm{Mor}(\mb{X})\times \rm{Mor}(\mb{G}),\\
		&s(\gamma, k)=\bigl(s(\gamma), s(k)\bigr), t(\gamma, k)=\bigl(t(\gamma), t(k)\bigr)\\
		&(\gamma_2, k_2)\circ_{\eta} (\gamma_1, k_1)=\bigl(\gamma_2\circ \eta(\gamma_1, k_2), k_2\circ k_1\bigr).
	\end{split} 
\end{equation}
The category $\mb{X}\rtimes_{\eta}\mb{G}$ is called by the name \textit{$\eta$-twisted category}.

\begin{definition}[Twisted action of a Lie $2$-group on a Lie groupoid]\label{Definition:Lie 2 twistgroupaction}
	For a Lie 2-group $\mb{G}$, a Lie groupoid $\mb{X}$ and a smooth map $\eta\colon {\rm{Mor}}(\mb{X})\times {\rm{Mor}}(\mb{G})\ra {\rm{Mor}}(\mb{X})$ satisfying conditions in \Cref{E:Contwist}, an \textit{$\eta$-twisted action of  $\mb{G}$ on the Lie groupoid $\mb{X}$ } is defined as a smooth functor  $\rho\colon \mb{X} \rtimes_{\eta} \mb{G} \rightarrow \mb{X}$  such that 
	\begin{itemize}
		\item $\rho_0: X_0 \times G_0 \rightarrow X_0$,
		\item $\rho_1:X_1 \times G_1 \rightarrow X_1$,
	\end{itemize}
	are respectively Lie group actions on manifolds $X_0$ and $X_1$.
\end{definition}

\begin{remark}\label{eta twisted and usual} We observe that the product Lie groupoid $\mb{X} \times G$ is a particualar case of the $\eta$-twisted category $\mb{X}\rtimes_{\eta}\mb{G}$. More precisely, if we take $\eta={\rm pr}_1$ in \Cref{Definition:Lie 2 twistgroupaction}, we recover the standard direct product of categories.
\end{remark}
\begin{remark}\label{Remark:functorialitytwisted}
	Functoriality of $\rho: \mb{X} \rtimes_{\eta} \mb{G} \rightarrow \mb{X}$	implies $\bigl(\gamma_2\circ \eta(\gamma_1, k_2)\bigr) (k_2\circ k_1)=(\gamma_2 k_2)\circ (\gamma_1 k_1)$.
	

\end{remark}

	\section{VB-groupoids}\label{subsection VB groupoids}
	The notion of `VB-groupoid' is a categorification of traditional vector bundles (\Cref{Definition: real vector bundle}) and a resident in the world of Lie groupoids. 
	During the 1980s, Pradines \cite{MR0941624} introduced these objects in relation to the study of symplectic groupoids \cite{MR0866024, MR0854594, MR0996653}. In the later years, Mackenzie et al. discovered their crucial roles in their study of double structures \cite{MR2103015, MR2831518}, and Mehta et al. demonstrated their significance in the representation theory of Lie groupoids \cite{MR3696590}. Our interest lies in seeing VB-groupoids as Lie groupoid objects in the category of vector bundles and their underlying Lie groupoid fibration structures. In this section, we recall the definition and some basic facts about them. For a detailed treatment, we suggest \cite{MR3451921, MR3696590, MR3744376}.
	
\begin{definition}[VB-groupoid {\cite[Definition $3.1.$]{MR3696590}}]\label{Definition of VB groupoids}
	A \textit{VB-groupoid} over  a Lie groupoid $\mb{X}$ is defined as a morphism of  Lie groupoids $\pi \colon \mb{D} \ra \mb{X}$
	\[
	\begin{tikzcd}[sep=small]
		D_1 \arrow[rr,"\pi_1"] \arrow[dd,xshift=0.75ex,"t_D"]
		\arrow[dd,xshift=-0.75ex,"s_D"'] &  & X_1 \arrow[dd,xshift=0.75ex,"t_X"].  		\arrow[dd,xshift=-0.75ex,"s_X"'] \\
		&  &                \\
		D_0 \arrow[rr,"\pi_0"]            &  & X_0           
	\end{tikzcd},\]
	satisfying the following conditions
	\begin{enumerate}
		\item[(i)] $\pi_1\colon D_1\ra X_1$ and $\pi_0\colon D_0\ra X_0$ are vector bundles,
		\item[(ii)]  $(s_D,s_X)$ and $(t_D, t_X)$ are morphisms of vector bundles,
		\item[(iii)]  for suitable $\gamma_1,\gamma_2,\gamma_3,\gamma_4\in V_1$, we have 
		$(\gamma_3\circ \gamma_1)+(\gamma_4\circ \gamma_2)=
		(\gamma_3+\gamma_4)\circ (\gamma_1+\gamma_2)$.
	\end{enumerate}
	\begin{proposition}
		A VB-groupoid is a Lie groupoid object in the category of vector bundles and, equivalently, a vector bundle object in the category of Lie groupoids.
	\end{proposition}
	\begin{proof}
		See \textbf{Proposition 3.5}, \cite{MR3696590}.
	\end{proof}
\end{definition}
\begin{example}\label{Tangent VB groupoid}
	For any Lie groupoid $\mb{X}:=[X_1 \rra X_0]$, the pair of tangent bundles $TX_1 \ra X_1$ and $TX_0 \ra X_0$ combine to form a VB-groupoid $T\mb{X} \ra \mb{X}$ over $\mb{X}$, where $T\mb{X}$ is the tangent Lie groupoid of $\mb{X}$, (see \Cref{Tangent Lie groupoid}.)  The VB -groupoid $T\mb{X} \ra \mb{X}$ is called the \textit{tangent VB-groupoid of the Lie groupoid } $\mb{X}$.
\end{example}
\begin{example}
	For any representation of a Lie groupoid $\mb{X}$ on a vector bundle $\pi \colon E_0 \ra X_0$	(\Cref{Representation of a Lie groupoid}), the corresponding semi-direct product groupoid $\mb{X}\rtimes E:=[s^{*}E \rra E]$ (\Cref{Semi-direct product groupoid}), naturally defines a VB-groupoid over $\mb{X}$.
\end{example}
Recall we have seen a notion of cleavage \Cref{Definition cleavage} on a fibered category. Here, we see a similar notion exists in the framework of VB-groupoids:
\begin{definition}[\cite{MR4126305}]\label{Linear Cleavage}
	A \textit{(linear) cleavage} on a VB-groupiod $\pi \colon \mb{D} \ra \mb{X}$ is a smooth section $\mc{C}$ of the map $P^{\mb{D}} \colon D_1 \ra X_1 \times_{s,X_0, \pi_0} D_0$ given as $\delta \mapsto (\pi_1(\delta), s(\delta))$, such that $\mc{C}$ is a morphism of vector bundles.
\end{definition}
A linear cleavage satisfying the condition $\mc{C}(1_{\pi(p)},p)=1_p$ for all $p \in D_0$ is sometimes called by the name \textit{unital} (see \cite{MR4126305}) and sometimes by \textit{right-horizontal lifts} (for example see \cite{MR3696590}). A linear cleavage is known as \textit{flat} (see \cite{MR4126305}) if it satisfies the condition that if $(\gamma_2, p_2), (\gamma_1, p_1) \in X_1 \times_{s,X_0, \pi_0} D_0$ such that ${s}(\gamma_2)={t}(\gamma_1)$ and $p_2=t\bigl({\mathcal C}(\gamma_1, p_1)\bigr),$ then $\mathcal{C}(\gamma_2 \circ \gamma_1 , p_1)= \mathcal{C}(\gamma_2, p_2) \circ \mathcal{C}(\gamma_1, p_1)$.

\begin{proposition}\label{Definition:Lie groupoid fibration}
	Any VB-groupoid $\pi \colon \mb{D} \ra \mb{X}$ satifies the following two conditions:
	\begin{itemize}
		\item[(i)]  $\pi_0 \colon D_0 \ra X_0$ is a surjective submersion;
		\item[(ii)] the map $P \colon D_1 \ra s^{*}E_0, \delta \mapsto \big(\pi_1(
		\delta), s(\delta) \big)$, is a surjective submersion.
	\end{itemize}
\end{proposition}
\begin{proof}
	A direct consequence of \textbf{Lemma 2}, \textbf{Appendix A}, \cite{MR2806566}.
\end{proof}

\begin{remark}\label{Remark on Lie groupoid fibration}
	From the surjectivity of the map $P \colon D_1 \ra s^{*}E_0, \delta \mapsto \big(\pi_1(
	\delta), s(\delta) \big)$, it follows that the underlying functor of any VB-groupoid $\pi \colon \mb{D} \ra \mb{X}$ is a fibered category (\Cref{Definition of fibered categories}) over the underlying category of $\mb{X}$. In fact, a VB-groupoid is a specific kind of \textit{Lie groupoid fibration} (any morphism of Lie groupoids satisfying the conditions (i) and (ii) of \Cref{Definition:Lie groupoid fibration}, see \textbf{Section 2}, \cite{MR3968895} for details).
\end{remark}

\begin{definition}[Section 2.4, \cite{MR3744376}]\label{1morphism of VB groupids}
	A \textit{$1$-morphism of VB-groupoids} 
	$[V_1\rra V_0]\ra [V_1'\rra V_0']$ over the base Lie groupoid $\mb{X}$ is definied as a morphism of Lie groupoids $\Phi:=(\Phi_1, \Phi_0)\colon [V_1\rra V_0]\ra [V_1'\rra V_0']$ such that $(\Phi_1, {\rm{id}}_{X_1})$  and $(\Phi_0, {\rm{id}}_{X_0})$, are morphisms of vector bundles. 
\end{definition}
\begin{definition}
	Let $\Phi,\Phi' \colon [V_1\rra V_0]\ra [V_1'\rra V_0']$  be a pair of  $1$-morphisms of VB-groupoids over $\mb{X}$. Then, a \textit{$2$-morphsim}  $\eta\colon \Phi \longrightarrow  \Phi '$ is defined as a smooth natural transformation such that $(\eta, 1)$ is a morphism of vector bundles from $V_0\ra X_0$ to $V'_1\ra X_1$.
\end{definition}

The following is obvious:

\begin{proposition}	\label{Prop:naturaltransformationinVBgroupoids}
	The collection of VB-groupoids forms a strict 2-category, with $1$-morphisms and $2$-morphisms as defined above.  
\end{proposition}
We denote the strict 2-category of VB-groupoids by $2$-${\rm {VBGpd}}(\mb{X})$.
\begin{definition}\label{Short exact sequence of VB-groupoids}[Section 2.4, \cite{MR3744376}]
	A \textit{short exact sequence of VB-groupoids over a Lie groupoid $\mb{X}$ }consists of three VB-groupoids say $\pi \colon \mb{D} \ra \mb{X}$, $\pi' \colon \mb{D}' \ra \mb{X}$ and $\pi'' \colon \mb{D}'' \ra \mb{X}$  and two connecting 1-morphisms say $F \colon \mb{D} \ra \mb{D}'$ and  $F' \colon \mb{D}' \ra \mb{D}''$ of VB-groupoids, as given in the diagram below: 	
	\begin{equation}\nonumber
		\begin{tikzcd}
			0 \arrow[r, ""]  & \mb{D} \arrow[r, "F"] \arrow[d, ""'] & \mb{D}' \arrow[r, "F'"] \arrow[d, ""'] & \mb{D}'' \arrow[d, ""'] \arrow[r] & 0 \\
			0 \arrow[r, ""'] & \mb{X} \arrow[r, "\rm{id}"']                & \mb{X} \arrow[r, "\rm{id}"']                & \mb{X} \arrow[r]                 & 0
		\end{tikzcd}
	\end{equation}
	such that the following is a short exact sequence of vector bundles over $X_1$
	\begin{equation}\nonumber
		\begin{tikzcd}
			0 \arrow[r, ""]  & D_1 \arrow[r, "F_1"] \arrow[d, ""'] & D_1' \arrow[r, "F_1'"] \arrow[d, ""'] & D_1'' \arrow[d, ""'] \arrow[r] & 0 \\
			0 \arrow[r, ""'] & X_1 \arrow[r, "\rm{id}"']                & X_1 \arrow[r, "\rm{id}"']                & X_1 \arrow[r]                 & 0
		\end{tikzcd},
	\end{equation}
	
\end{definition}		
A consequence of the above definition is the intuitive one, which we state below:
\begin{proposition}
	If the following is a short exact sequence of VB-groupoids over a Lie groupoid $\mb{X}$,
	\begin{equation}\nonumber
		\begin{tikzcd}
			0 \arrow[r, ""]  & \mb{D} \arrow[r, "F"] \arrow[d, ""'] & \mb{D}' \arrow[r, "F'"] \arrow[d, ""'] & \mb{D}'' \arrow[d, ""'] \arrow[r] & 0 \\
			0 \arrow[r, ""'] & \mb{X} \arrow[r, "\rm{id}"']                & \mb{X} \arrow[r, "\rm{id}"']                & \mb{X} \arrow[r]                 & 0
		\end{tikzcd}
	\end{equation}
	then the following is a short exact sequence of vector bundles over $X_0$
	\begin{equation}\nonumber
		\begin{tikzcd}
			0 \arrow[r, ""]  & D_0 \arrow[r, "F_1"] \arrow[d, ""'] & D_0' \arrow[r, "F_1'"] \arrow[d, ""'] & D_0'' \arrow[d, ""'] \arrow[r] & 0 \\
			0 \arrow[r, ""'] & X_0 \arrow[r, "\rm{id}"']                & X_0 \arrow[r, "\rm{id}"']                & X_0 \arrow[r]                 & 0
		\end{tikzcd}.
	\end{equation}
\end{proposition}
\begin{proof}
	See \textbf{Proposition 2.6(i)}, \cite{MR3744376}.
\end{proof}

\section{2-Vector spaces}\label{Section: 2-vector space}
In the current literature, there is no particular standard way to categorify the notion of a vector space (see \cite{nlab:2-vector_space} for a detailed discussion on the issue). We restrict our attention to the concept of a categorified vector space introduced by Baez and Crans in \cite{MR2068522}. Usually, they are known by the name `Baez-Crans 2-vector spaces'. Among other existing notions of 2-vector spaces,  \textit{Kapranov-Voevodsky 2-vector spaces} \cite{MR1278735} is a significant one; however, we will not pursue them here.

\begin{definition}[Definition 3.1, \cite{MR2068522}]\label{Vector 2-spaces}
	A \textit{2-vector space} is defined as a category such that both $V_1$ and $V_0$ are vector spaces and all structure maps are linear. 	
\end{definition}
In other words, a 2-vector space is a category $\mb{V}:=[V_1 \rra V_0]$ internal to Vect, the category of finite dimensional vector spaces over a field $\mb{K}$. Similarly, there is a notion of a functor internal to Vect between a pair of vector 2-spaces and a natural transformation internal to Vect between such a pair of functors internal to Vect. These data form a strict 2-category naturally, and we denote it by 2Vect.
We suggest \textbf{Section 3} of \cite{MR2068522} for readers interested in a detailed account on 2Vect. 
\begin{example}
	For any VB-groupiod $\pi \colon \mb{D} \ra \mb{X}$, the groupoid $\pi^{-1}(x):=[\pi_1^{-1}(1_x) \rra \pi_0^{-1}(x)]$ is a 2-vector space for each $x \in X_0$.
\end{example}
\begin{example}\label{Lie 2-algebra}
	For any Lie 2-group $\mb{G}$, the Lie groupoid $L(\mb{G}):=[L(G_1) \rra L(G_0)]$ is a 2-vector space.
\end{example}
Next, we define an action of Lie 2-group on a 2-vector space. This notion is adapted from the one in \textbf{Section 11}, \cite{MR3126940}.
\begin{definition}[Section 11, \cite{MR3126940}]\label{Action of a Lie 2-group on a vector 2-space}
	Let $\mb{G}:=[G_1 \rra G_0]$ be a Lie 2-group and $\mb{V}:=[V_1 \rra V_0]$ a 2-vector space. An \textit{action of $\mb{G}$ on  $\mb{V}$} is given by a functor $\rho \colon \mb{G} \times \mb{V} \ra \mb{V}$, such that the maps $\rho_1 \colon G_1 \times V_1 \ra V_1$ and $\rho_0 \colon G_0 \times V_0 \ra V_0$ are traditional Lie group actions on $V_1$ and $V_0$ respectively, inducing linear representations of $G_1$ and $G_0$ on $V_1$ and $V_0$ respectively.
\end{definition}
We denote $\rho_i(g,v)$ by $gv$ for all $g \in G_i, v \in V_i$ and $i=0,1$, (assuming $\mb{G}$ is acting form the left).

A weaker version of this action was explored in \cite{MR3556124, MR2978538}, and for the representation theory of 2-groups we suggest
\cite{MR2331238, MR3213404, huan20222representations}.

\section{Haefliger paths and the fundamental groupoid of a Lie groupoid}\label{Haefliger paths and fundamental groupoid of a Lie groupoid}
This section recalls a notion of a path in a Lie groupoid, introduced by Haefliger (\cite{MR0285027, MR1086659, MR2218759}). Also, we quickly review the existing notion of homotopy between such paths that naturally extends the notion of fundamental groupoid of a manifold to the setting of Lie groupoids (see \cite{MR2772614, MR2861783, MR2166083, MR2166453}).

\begin{definition}[Section 4.1, \cite{MR2772614}]\label{Haefliger definition}
	Suppose $\mb{X}$ is a Lie groupoid and $x, y \in X_0$. An \textit{$\mb{X}$-path} or a \textit{Haefliger path from $x$ to $y$ over a subdivision $0=t_0 \leq t_1 \leq \cdots \leq t_n=1$}, is defined as a sequence $(\gamma_0, \alpha_1, \gamma_1, \cdots , \alpha_n, \gamma_n)$, where
	\begin{enumerate}[(i)]
		\item $\alpha_i \colon [t_{i-1}, t_i] \ra X_0$ is a path for all $1 \leq i \leq n$ and 
		\item $\gamma_i \in X_1$ for each $i$, such that
		\begin{itemize}
			\item $s(\gamma_0)=x$ and $t(\gamma_n)=y$;
			\item $s(\gamma_i)= \alpha_i(t_i)$ for all $0 < i \leq n$;
			\item $t(\gamma_i)= \alpha_{i+1}(t_i)$ for all $0 \leq i < n$,
		\end{itemize}
		as illustrated in the diagram below:
		\[\begin{tikzcd}
			x \arrow[r, "\gamma_0"] & \cdot \arrow[r, "\alpha_1", dotted] & \cdot \arrow[r, "\gamma_1"] & \cdot \arrow[r, no head, dotted] & {} \arrow[r, no head, dotted] & \cdot \arrow[r, "\alpha_n", dotted] & \cdot \arrow[r, "\gamma_n"] & y .
		\end{tikzcd}\]
	\end{enumerate}
\end{definition}
The following operations define an equivalence relation on the set of all $\mb{X}$-paths, (\textbf{Section 4.1}, \cite{MR2772614}) : \label{Equivalence Haefliger}
\begin{enumerate}[(i)]
	\item Adding a new point $z \in [t_{i-1}, t_i]$ to the subdivision, followed by taking the restrictions $\alpha'_i, \alpha''_i$ of the associated $\alpha_i$ to the newly formed intervals $[t_{i-1},z]$ and $[z, t_i]$ and then adding the identity arrow $1_{\alpha(z)}$, as illustrated below:
	\begin{equation}\label{classical equivalence 1}
		\begin{tikzcd}
			\cdot \arrow[r, "\alpha^{{\rm{'}}}_i", dotted] & \cdot \arrow["1_{\alpha(z)}"', loop, distance=2em, in=305, out=235] \arrow[r, "\alpha^{{\rm{''}}}_i", dotted] & \cdot
		\end{tikzcd}
	\end{equation}
	\item Given a smooth map $\zeta_i \colon [t_{i-1}, t_i] \ra X_1$, by making the following replacements:
	\begin{itemize}
		\item $\alpha_i$ by $t \circ \zeta_i$,
		\item $\gamma_{i-1}$ by $\zeta_{i}(t_{i-1}) \circ \gamma_{i-1}$ and
		\item $\gamma_i$ by $\gamma_i \circ \big(\zeta_i(t_i) \big)^{-1}$,
	\end{itemize}
	as illustrated below:
	\begin{equation}\label{classical equivalence 2}
		\begin{tikzcd}
			\cdot \arrow[r, "\gamma_{i-1}"] & \cdot \arrow[r, "\alpha_i", dotted] \arrow[d, "\zeta_i(t_{i-1})"'] & \cdot \arrow[r, "\gamma_i"]  & . \\
			& \cdot \arrow[r, "t \circ \zeta_i"', dotted]                & \cdot \arrow[u, "\zeta_i(t_i)^{-1}"'] &  
		\end{tikzcd}
	\end{equation}
\end{enumerate}
Next, we recall the notion of \textit{deformation between $\mb{X}$-paths}, as given in \textbf{Section 4.1}, \cite{MR2772614}.

\begin{definition}[Section 4.1, \cite{MR2772614}]\label{deformation}
	A \textit{deformation from an $\mb{X}$-path $(\gamma_0, \alpha_1, \gamma_1,..., \alpha_n, \gamma_n)$ to another one $(\gamma'_0, \alpha'_1, \gamma'_1, \cdots , \alpha'_n, \gamma'_n)$} \textit{from $x$ to $y$} is given by 
	\begin{itemize}
		\item a sequence of homotopies $H_i \colon [t_{i-1}, t_i] \times [0,1] \ra X_0$, with $H_i(0)= \alpha_i$ and $H_i(1)= \alpha'_i$ for $i=1,2,...,n$, and 
		\item a sequence of smooth paths $\zeta_i \colon [0,1] \ra X_1$ with $\zeta_i(0)= \gamma_i$ and $\zeta_i(1)=\gamma'_i$ for $i=1,...,n-1$,
	\end{itemize}
	such that $(\gamma_0, H_1(s), \zeta_1(s), \cdots, \zeta_{n-1}(s), H_n(s), \gamma_n)$ is an $\mb{X}$-path for each $s \in [0,1]$.
\end{definition}

\begin{definition}[Definition 4.2, \cite{MR2772614}]\label{homotopy of Xpaths}
	A pair of $\mb{X}$ paths between $x$ and $y$ is said to be \textit{homotopic} if one can be obtained from the other by a finite sequence of the following operations:
	\begin{itemize}
		\item \Cref{classical equivalence 1}, 
		\item \Cref{classical equivalence 2} and 
		\item deformations.
	\end{itemize}
\end{definition}
The above notion of homotopy between $\mb{X}$-paths naturally yields a groupoid $\Pi_1(\mb{X})$, whose objects are the elements of $X_0$ and arrows are the homotopy class of $\mb{X}$-paths, and is called the \textit{fundamental groupoid of the Lie groupoid $\mb{X}$}, \cite{MR2772614, MR2861783, MR2166083, MR2166453}. Furthermore, one can show that $\Pi_1(\mb{X})$ has a natural Lie groupoid structure, (see \cite{MR2166453}). We refer to \textbf{Section 4.1}, \cite{MR2772614} for a detailed description of its structure maps. The following proposition is a generalisation of the \Cref{Induced morphism between fundamental groupoid of manifolds}:
\begin{proposition}[Section 4.1, \cite{MR2772614}]\label{Induced morphism between haefliger groupoids}
	Any morphism of Lie groupoids $F \colon \mb{X} \ra \mb{Y}$ induces a morphisms of Lie groupoids $F^{*} \colon \Pi_1(\mb{X}) \ra \Pi_1(\mb{Y})$ between the corresponding fundamental groupoids, defined as $F^{*}_{0}: =\phi$ and $F^{*}_{1}([\gamma_0, \alpha_1, \cdots \alpha_n, \gamma_n]):=[F_1 (\gamma_0), F_0 \circ \alpha_1, \cdots, F_0 \circ \alpha_n, F_1(\gamma_n)]$.
\end{proposition}

\section{Diffeology}
Recall, we mentioned in \Cref{subsectionFunctor from the thin homotopy groupoid} that the quotient space $\frac{PM}{\sim}$, in general, has no finite-dimensional smooth manifold structure. Despite this, we can still talk about its smoothness through \textit{generalized smooth spaces} such as \textit{diffeological spaces}.
They can be thought of as spaces that subsume the notion of smooth manifolds but also encompass spaces like path spaces, spaces of smooth maps between two manifolds, quotient spaces, pullback spaces, etc. In other words, they capture the smoothness of many naturally occurring spaces that one would like to think of as smooth but do not possess finite dimensional manifold structures. 

This section contains some basic notions in diffeology and a brief discussion on the smoothness of parallel transport functor on a traditional principal bundle \Cref{Transport functor}. Readers interested in various kinds of generalized smooth spaces can look at \cite{MR2817410}. Our references for this portion are \cite{MR3025051} and the  \textbf{Appendix A} of \cite{MR3521476}.
\subsection{Definitions, basic properties and examples}\label{Definitionsbasic properties and examples}
\begin{definition}
	A \textit{diffeology} on a set $S$ is a collection of functions $D_{S} \subseteq \lbrace p \colon U \ra S : U \subseteq \mb{R}^{n}$, where $U$ is an open subset of $\mb{R}^{n}, n \in \mb{N} \rbrace$ satisfying the following conditions:
	\begin{itemize}
		\item[(i)]Every constant function lies in $D_S$;
		\item[(ii)] If $V \subseteq \mb{R}^{n}$ is open, $p \colon U \ra S$ is in $D_S$ and $f \colon V \ra U$ is a smooth map, then we have $p  \circ f \colon V \ra S$ is in $D_S$;
		\item[(iii)] If $\lbrace U_i \rbrace_{i \in I}$ is an open cover of $U \subseteq \mb{R}^{n}$ and $p \colon U \ra S$ is a function satisfying $p |_{U_i} \colon U_i \ra S$ is in $D_S$ for all  $i \in I$, then $p  \colon U \ra X$ is in $D_S$.
	\end{itemize}
\end{definition}
The pair $(S, D_S)$ is known as \textit{diffeological space} and the elements of $D_S$ are called \textit{plots}. 
\begin{definition}
	A \textit{map of diffeological spaces} from a diffeological space $(X, D_X) $ to a diffeological space  $(Y, D_Y)$ is defined as a map of sets $f \colon X \ra Y$,  such that for any  $p \in D_X$, $f \circ p \in D_Y$.
\end{definition}
The collection of diffeological spaces, along with the maps of diffeological spaces between them, form a category naturally and is denoted by \textbf{Diffeol}.
\begin{example}[Smooth manifiold]\label{manifold diffeology}
	Every smooth manifold $M$ is a diffeological space, with diffeology $D_M:= \lbrace p \colon U \ra M : U$ is an open subset of $\sqcup^{\infty}_{n=0}\mb{R}^{n}$ and $p$ is smooth$\rbrace$. Hence, every smooth map $f \colon M \ra N$ between manifolds is a map of diffelogical spaces $f \colon (M, D_M) \ra (N, D_N)$.
\end{example}
\begin{example}[Path space diffeology]\label{Path space diffeology}
	Given a smooth manifold $M$, the set of paths with sitting instances $PM$ (\Cref{Subsection: Parallel transport functor of a connection}) is a diffeological space, with diffeology $D_{PM}:= \lbrace p \colon U \ra PM : \bar{p} \colon U \times [0,1] \ra M, (u,x) \mapsto p(u)(x)$, is smooth.$\rbrace$. The diffeology $D_{PM}$ is known as the \textit{path space diffeology}. Furthermore, the evaluation maps $ev_0,ev_1 \colon PM \ra M$ at $0$ and $1$ are maps of diffeological spaces.
\end{example}
\begin{example}[Fibre product diffeology]\label{Fibre product diffeology}
	Given a pair of maps of diffeological spaces $f \colon (Y,D_Y) \ra (X,D_X)$ and $g \colon (Z,D_Z) \mapsto (X,D_X)$, the set theoretic fibre product $Y \times_{f,X,g} Z$ is a diffeological space with \textit{fiber product diffeology} $D_{Y \times_{f,X,g} Z}:= \lbrace(p_Y,p_Z) \in D_Y \times D_Z : f \circ p_Y= g \circ p_Z \rbrace$.  Observe that the projections are maps of diffeological spaces.
\end{example}
\begin{example}\label{subsapce diffeology}
	For any diffeological $(X, D_{X})$ and a subset $S \subseteq X$, $D_{S}:= \lbrace (p \colon U \ra X) \in D_{X} : p(U) \subseteq S \rbrace$ defines the \textit{subspace diffeology} on $S$.
\end{example}
\begin{example}[Quotient diffeology ]\label{quotient diffeology}
	For a diffeological space $(X, D_{X})$ and an equivalence relation $\sim$ on $X$, the quotient $q \colon X \ra \frac{X}{\sim}$ defines a diffeological structure with the following diffeology, (\textbf{Construction A.15}, \cite{MR3521476}):
	
	$D_{\frac{X}{\sim}}:= \lbrace p \colon U \ra \frac{X}{\sim}$: $U \subseteq  \mb{R}^{n}$ is ${\rm{open}}, n \in \mb{N}, p $ is a function such that for every $u \in U$, there exists an open neighbourhood $V$ of $u$ in $U$ and a plot $\bar{p} \colon V \ra X$ with $q \circ \bar{p}=p|_{V}\rbrace$.
	
	$D_{\frac{X}{\sim}}$ is known as the \textit{quotient diffeology}. Then, the quotient map naturally becomes a map of diffeological spaces.
\end{example}
\begin{example}\label{sum diffeology}
	Suppose $(S_i, D_{S_i})_{i \in I}$ is an arbitrary family of diffeological spaces. Then the disjoint union $S=\sqcup_{i \in I} S_i$ is a diffeological space with the diffeology given by $D:= \lbrace p \colon U \ra  S: U \subseteq  \mb{R}^{n}$ is ${\rm{open}}, n \in \mb{N}, p $ is a function such that for any $x \in U$ there exists an open neighborhood $U_x$ of $x$ and an index $i \in I$, with $P|_{U_x} \in D_{S_i}.  \rbrace$. The diffeology $D$ is known as the \textit{sum diffeology} on the family $\lbrace S_{i} \rbrace_{i 
		\in I}$, (see \textbf{Section 1.39} of \cite{MR3025051}).
\end{example}
\begin{proposition}\label{Technical 1}
	Suppose $q \colon (A, D_A) \ra (B, D_B)$ is a quotient map between two diffeological spaces and  $(C, D_C)$ is another diffeological space. Then a map $f \colon (B, D_B) \ra (C, D_C)$ is a map of diffeological spaces if and only if for any plot $p \colon U \ra A$, the composite $f \circ q \circ p$ is in $D_C$.
\end{proposition}
\begin{proof}
	See \textbf{Lemma A.16} in \cite{MR3521476}.
\end{proof}
\begin{definition}[Section 7.1, \cite{MR3025051}]\label{Definition:Diffeological group}
	A \textit{diffeological group} is a group $G$ endowed with a diffeology, such that the multiplication and the inversion are maps of diffeological spaces.
\end{definition}
\begin{example}\label{Lie group diffeology}
	Any Lie group is naturally a diffeological group with the diffeology defined in \Cref{manifold diffeology}.
\end{example}
\begin{example}\label{subgroup diffeology}
	Any subgroup of a diffeological group is a diffeological group equipped with the subspace diffeology, \Cref{subsapce diffeology}.
\end{example}
\begin{example}\label{quotient group diffeology}
	Any quotient group is a diffeological group with the quotient diffeology, \Cref{quotient diffeology}.
\end{example}
\begin{definition}[Section 8.3, \cite{MR3025051}]\label{Definition: Diffeological groupoid}
	A \textit{diffeological groupoid} is defined as a groupoid $\mb{X}=[X_1 \rra X_0]$, such that both $X_1$ and $X_0$ are diffeological spaces and all the structure maps of $\mb{X}$ are maps of diffeological spaces.
\end{definition}
\begin{example}
	For any manifold $M$, the corresponding thin fundamental groupoid $\Pi_{{\rm{thin}}}(M)=[\frac{PM}{\sim} \rra M]$ (\Cref{Definition:Thin homotopy groupoid of a manifold}) of $M$ is a diffeological groupoid. A detailed proof is available in \textbf{Proposition A.25}, \cite{MR3521476}.
\end{example}
\begin{example}\label{Automorphism diffeological group}
	Given a diffeological groupoid $\mb{X}=[X_1 \rra X_0]$ and any element $x \in X_0$, ${\rm{Aut}(x)}:={\rm{Hom}}(x,x)$ is a diffeological group (\Cref{Definition:Diffeological group}).
\end{example}	
\begin{definition}[Section 8.3, \cite{MR3025051}]\label{Definition:Morphism of diffeological grouopoids}
	A \textit{morphism of diffeological groupoids from a diffeological groupoid $\mb{X}=[X_1 \rra X_0]$ to a diffeological groupoid $\mb{Y}=[Y_1 \rra Y_0]$} is deifned as a functor $F \colon \mb{X} \ra \mb{Y}$ such that $F_1 \colon X_1 \ra Y_1$ and $F_0 \colon X_0 \ra Y_0$ are maps of diffeological spaces. 
\end{definition}
The collection of diffeological groupoids, along with the morphisms defined in \Cref{Definition:Morphism of diffeological grouopoids}, naturally define a category, denoted by \textbf{D-groupoids}, see Section 8.3, \cite{MR3025051}. It is worth mentioning that a theory concerning a bicategory of diffeological groupoids has been recently developed in  \cite{MR4535273}, which introduced notions analogous to the ones already discussed (\Cref{subsection Generalized homomorphisms and Morita equivalence of Lie groupoids}), in the framework of diffeological groupoids.

\subsection{On the smoothness of parallel transport functor of a traditional principal bundle}\label{Smoothness of traditional parallel transport}
This subsection briefly discusses the smoothness property of the parallel transport functor of a principal bundle over a manifold (\Cref{Transport functor}). In particular, we have the following smoothness condition:
\begin{proposition}\label{Proposition:Smoothness of traditional parallel transport}
	For a Lie group $G$, let $\pi \colon P \ra M$ be a principal $G$-bundle over a manifold $M$, equipped with a connection $\omega$. Then for every $x \in M$, the function $$T_{\omega}|_{\Pi_{\rm{thin}}(x)} \colon \Pi_{\rm{thin}}(x)  \ra {\rm{Aut}}(\pi^{-1}(x))$$ is a map of diffeological spaces from the diffeological group $\Pi_{\rm{thin}}(x)$, the automorphism group of the diffeological groupoid $\Pi_{\rm{thin}}(M)$ at $x$ (see \Cref{Automorphism diffeological group}), to the Lie group ${\rm{Aut}}(\pi^{-1}(x)) \cong G$ (see \Cref{Automorphism group of the fibre}), where $T_{\omega}$ is the corresponding parallel transport functor (\Cref{Transport functor}).
\end{proposition}
\begin{proof}
	See the proof of \textbf{Theorem 3.9}, \cite{MR3521476}.
\end{proof}
On the other hand, for a manifold $M$ and a Lie group $G$,  a functor $$T \colon \Pi_{\rm{thin}}(M) \ra  G {\rm{-}} {\rm{Tor}}$$ that satisfies the above \textit{smoothness property} i.e for each $x \in M$, $$T|_{\Pi_{\rm{thin}}(x)} \colon \Pi_{\rm{thin}}(x)  \ra {\rm{Aut}}(T(x))$$ is a map of diffeological spaces from $\Pi_{\rm{thin}}(x)$ to the Lie group $T(x) \cong G$, has been defined as a \textit{transport functor} in \cite{MR3521476}. Transport functors and natural isomorphisms between them form a groupoid ${\rm{Trans}}_{G}(M)$. Interestingly, such a transport functor contains the data of a principal bundle equipped with a connection structure. Particularly, it has been shown in \cite{MR3521476} that for any manifold $M$ and a Lie group $G$, the groupoid of transport functors ${\rm{Trans}}_{G}(M)$ is equivalent to the groupoid $B^{\rm{\nabla}}G(M)$ (see the end of \Cref{Groupoid of principal bundles with connection}). It is worth mentioning that a result similar to the one in \cite{MR3521476} has also been proved earlier by \textit{Schreiber} and \textit{Waldorf} in \cite{MR2520993} by introducing a notion called \textit{smooth descent data of a functor}.

\chapter{Principal 2-bundles over Lie groupoids and their characterizations}\label{chapter 2-bundles} 



\lhead{Chapter 4. \emph{Principal 2-bundles over Lie groupoids and their characterizations}} 


In this chapter, we introduce a `categorified version' of a traditional principal bundle (\Cref{Definition: Principal G-bundle}), whose structure group is now replaced by a Lie 2-group (\Cref{Definition: Lie 2 group}), its total space and the base are replaced by Lie groupoids (\Cref{Lie groupoids definition}) and the action map is replaced by an action functor (\Cref{Definition:Lie 2 groupaction}), resulting in an object that one can view as a groupoid object in the category of principal bundles. We call this object a `\textit{principal Lie 2-group-bundle over a Lie groupoid}'. One primary motivation for such a definition is to study `purely differential geometric' relationships between the theory of fibered categories/fibrations (\Cref{subsection Fibered categories}) and the concepts in classical gauge theory discussed in \Cref{Chapter Classical setup}. The purpose of this chapter is to investigate the underlying fibration structure of these objects in a way that characterizes certain classes of these Lie 2-group bundles with respect to the kind of underlying fibered category structure they possess. In particular, we introduce the following classes of principal Lie 2-group bundles over Lie groupoids, namely
\begin{enumerate}[(i)]
	\item a  `categorical-principal 2-bundle over a Lie groupoid'. 
	\item a `quasi-principal 2-bundle over a Lie groupoid' and
	\item a `unital-principal 2-bundle over a Lie groupoid',	\end{enumerate}
classified based on underlying cloven fibration structures (\Cref{Definition Colven fibration}). As a result of this investigation, we obtain the following interesting consequence that we consider our main achievement in this chapter:

`\textit{A statement and the proof of a Lie 2-group torsor version of the one-one correspondence between fibered categories and pseudofunctors' (\Cref{subsection Fibered categories})}.

The following two are important byproducts of the above correspondence:
\begin{enumerate}[(i)]
	\item We obtain a `weakened version' of a principal Lie group bundle over a Lie groupoid as mentioned in \Cref{Definition: Principal Lie group bundle over a Lie groupoid}. In this weakened form, the underlying left action of the base Lie groupoid on the total space is now replaced by a left quasi-action (\Cref{Definition: quasi action on a manifold}), and it deviates from being an action (a principal Lie group bundle over the base Lie groupoid) upto a factor coming from a Lie crossed module (\Cref{Definition:Lie crossed module}). We call this object a `pseudo-principal Lie crossed module bundle over a Lie groupoid'. It is a suitable analog of a psuedofunctor (\Cref{pseudofunctor}).
	\item We could extend the notion of a categorical-principal 2-bundle to be defined over the differentiable stack presented by the base Lie groupoid (see \Cref{Morita equivalent imply stack}).
\end{enumerate}

Apart from the main results stated above, in this chapter, we also related certain aspects of our 2-bundle theory to notions like connections on Lie groupoids and Lie groupoid $G$-extensions for a Lie group $G$, as side results.

The content of this chapter is mainly based on our papers \cite{chatterjee2023parallel} and \cite{MR4403617}.

\section{A principal 2-bundle over a Lie groupoid}\label{A principal 2-bundle over a Lie groupoid}
This section introduces the notion of a `principal Lie 2-group bundle over a Lie groupoid' and discusses some examples. The contents of this section are based on our paper \cite{MR4403617}. 
\begin{definition}[Principal $\mb{G}$-bundle over a Lie groupoid] \label{Definition:principal $2$-bundle over Liegroupoid}
	For a Lie 2-group $\mb{G}$, a \textit{principal $\mb{G}$-bundle over a Lie groupoid} $\mathbb{X}$ is defined as a morphism of Lie groupoids $\pi: \mb{E} \rightarrow \mb{X}$ along with a right action $\rho: \mb{E} \times \mb{G} \rightarrow \mb{E}$ of the Lie $2$-group $\mb{G}$ on the Lie groupoid $\mb{E}$ such that, 
	\begin{itemize}
		\item $\pi_0\colon E_0 \rightarrow X_0$ is a principal $G_0$-bundle over the manifold $X_0$,
		\item $\pi_1\colon E_1 \rightarrow X_1$ is a principal $G_1$-bundle over the manifold $X_1$.
	\end{itemize}
	We will call the Lie 2-group $\mb{G}$ as the \textit{structure $2$-group of $\pi\colon\mb{E} \rightarrow \mb{X}$}. We will denote the above principal $\mb{G}$-bundle over the Lie groupoid $\mb{X}$  by $\pi \colon \mb{E} \ra \mb{X}$.
\end{definition}
\begin{remark}\label{Our 2-bundle and Waldrorf}
Our definition of principal $2$-bundle over a Lie groupoid $[M \rra M]$ satisfies the one given in \textbf{Definition 3.1.1}, \cite{MR3894086}. To see this, note that given any principal $\mb{G}=[G_1 \rra G_0]$-bundle $\pi: \mb{E} \ra [M \rra M]$, the map $\mb{E} \times \mb{G} \ra \mb{E} \times_{[M \rra M]} \mb{E}$ sending $(p_i,g_i) \ra (p_i,p_ig_i)$ for $i=0,1$ is an isomorphism of Lie groupoids and hence a \textit{weak equivalence} in the sense of \cite{MR3894086}.
\end{remark}

\begin{remark}\label{Remark:OtherdefofLie2grp bundles}
	A little weaker version of a principal $2$-bundle over a Lie groupoid has been studied by Ginot and Sti\'{e}non in \textbf{Definition 2.20}, \cite{MR3480061}. They defined a principal $\mb{G}$-bundle over a Lie groupoid $\mb{X}$ in terms of a Hilsum \& Skandalis generalized morphism of Lie $2$-groupoids $\mb{X}\ra \mb{G}$, where both $\mb{G}$ and $\mb{X}$ were treated as Lie $2$-groupoids and represented by double Lie groupoid structures.
\end{remark}	

\begin{definition}\label{Definition: Morphism of principal 2-bundles}[Morphism of principal $2$-bundles over a Lie groupoid]
	Let $\mb{G}$ be a  Lie $2$-group. Let $\pi \colon \mb{E} \ra \mb{X}$ and $\pi' \colon \mb{E} \ra \mb{X}$ be a pair of principal $\mb{G}$-bundles over a Lie groupoid $\mb{X}$. A \textit{morphism of principal $\mb{G}$-bundles from $\pi \colon \mb{E} \ra \mb{X}$ to $\pi' \colon \mb{E} \ra \mb{X}$ over $\mb{X}$} is defined as a smooth $\mb{G}$-equivariant morphism $F:\mb{E} \rightarrow \mb{E'}$ of Lie groupoids such that the following diagram commutes on the nose
	\[
	\begin{tikzcd}
		\mb{E} \arrow[r, "F"] \arrow[d, "\pi"'] & \mb{E'} \arrow[ld, "\pi'"] \\
		\mb{X}                               &                  
	\end{tikzcd}.\]
\end{definition}
\begin{definition}\label{2-morphism of principal 2-bundles}
	Let $F, F' \colon \mb{E} \ra \mb{E}'$ be a pair of morphisms of prinicpal $\mb{G}$-bundles from $\pi \colon \mb{E} \ra \mb{X}$ to $\pi' \colon \mb{E}' \ra \mb{X}$ over a Lie groupoid $\mb{X}$. A \textit{$2$-morphism from $F$ to $F'$} is defined as a smooth natural isomorphism $\eta\colon F \Longrightarrow F'$ satisfying  $\eta(pg)=\eta(p)1_g$ and $\pi_1(\eta(p))=1_{\pi_0(p)}$ for all $p \in E_0$ and $g \in G_0$. 
\end{definition}
\label{strict 2-groupoid of principal 2-bundles}Since a morphism of principal $\mb{G}$-bundles over a Lie groupoid $\mb{X}$ is, in particular, given by a pair of morphisms of traditional principal bundles, the collection of principal $\mb{G}$-bundles over $\mb{X}$ forms a strict $2$-groupoid (\Cref{Strict 2-groupoid}), whose 1-morphisms and 2-morphisms are as defined in \Cref{Definition: Morphism of principal 2-bundles} and \Cref{2-morphism of principal 2-bundles} respectively. We denote this 2-groupoid by $\rm{Bun}(\mb{X},\mb{G})$.


  \begin{example}
  	For a Lie group $G$, a classical principal $G$-bundle $\pi\colon P\to M$ (see \Cref{Definition: Principal G-bundle}) over a manifold $M$ is same as a principal $[G\rra G]$-bundle $[P \rra P]$ over the Lie groupoid $[M \rra M].$
  \end{example}
  \begin{example}\label{Principal 2-bundle with discrete Lie 2 group}
  	For a Lie group $G$, any principal $G$-groupoid \Cref{Definition: Principal G-groupoids} over a Lie groupoid $\mb{X}$ is a principal $[G \rra G]$-bundle over $\mb{X}$ and vice-versa, (see \Cref{Correspondence between principal G-bundles and principal G-groupoids over a Lie groupoid}).
  \end{example}
  \begin{remark}
  	It is straightforward to observe that given a Lie group $G$ and a Lie groupoid $\mb{X}$, the groupoid of principal $[G \rra G]$-bundles over $\mb{X}$  is equivalent to the groupoid of principal $G$-groupoids over $\mb{X}$ and hence is equivalent to the groupoid of principal $G$-bundles over $\mb{X}$, (see \Cref{Equivalence of principal G bundles and principal G groupoids}).
  \end{remark}
  
  \begin{example}\label{E:Example of product bundle}
  	Given a Lie groupoid $\mb{X}=[X_1\rra X_0]$ and a Lie 2-group $\mb{G}=[G_1\rra G_0]$, we have the product principal $\mb{G}$-bundle $\mb{X}\times \mb{G}=[X_1\times G_1\rra X_0\times G_0]$ over $\mb{X}$ consisting of a product $G_1$-bundle $X_1 \times G_1 \ra X_1$ over $X_1$ and product $G_0$-bundle $X_0 \times G_0 \ra X_0$ over $X_0$.
  \end{example}
  
  \begin{example}\label{E:Exampleprincipalpairlie2}
  	Given a Lie groupoid  $\mb{X}$ and a classical principal $G$-bundle $\pi\colon E_0\ra X_0,$ we define a Lie groupoid $[E_1=\big\{(p, \gamma, q)| \gamma\in X_1, p\in \pi^{-1}(s(\gamma)), q\in \pi^{-1}(t(\gamma))\big\}\rra E_0]$ whose source, target and composition maps are respectively given as $s(p, \gamma, q)=p, t(p, \gamma, q)=q$ and $(q, \gamma_2, r)\circ (p, \gamma_1, q)=(p, \gamma_2\circ \gamma_1, r)$. The Lie $2$-group $[G\times G\rra G]$ has a natural action on  $[E_1\rra E_0]$ given by $(p, g)\mapsto p g$
  	and $\bigl((p, \gamma, q), (g_1, g_2)\bigr)=(p g_1, \gamma, q g_2).$	
  	Then we have a principal $[G\times G\rra G]$-bundle $[E_1\rra E_0]$ over the Lie groupoid $\mb{X}$ with the obvious projection functor. 	
  \end{example}

  \begin{example}\label{E:Example of principal 2-bundle ordinary}
  	Consider the single object Lie $2$-group $[G\rra e]$ (see \Cref{Ex:ordinary}) of an abelian Lie group $G$.
  	Suppose $\mb{E}=[E_1\rra E_0]$ is a principal $[G\rra e]$-bundle over the Lie groupoid ${\mb X}$.	Then it is obvious from the definition of a principal $2$-bundle that  $E_0=X_0$ and $E_1$ is a principal $G$-bundle over $X_1$ such that the following diagram commutes.
  	\[
  	\begin{tikzcd}[sep=small]
  		E_1 \arrow[rr,"\pi_1"] \arrow[dd,xshift=0.75ex,"t"]
  		\arrow[dd,xshift=-0.75ex,"s"'] &  & X_1 \arrow[dd,xshift=0.75ex,"t"]
  		\arrow[dd,xshift=-0.75ex,"s"'] \\
  		&  &                \\
  		X_0 \arrow[rr,"{\rm Id}"]            &  & X_0           
  	\end{tikzcd},\]	
  	Also, the functoriality of the action of $[G\rra e]$ on $E_1\rra X_0$ implies that the action $E_1\times G\ra E_1$ preserves the hom sets; that is, the restriction gives 
  	\begin{equation}\label{E:restricthom}
  		{\rm Hom}_{\mb{E}}(x, y)\times G \ra {\rm Hom}_{\mb{E}}(x, y)
  	\end{equation}
  	for all $x, y\in X_0$  and for each pair of composable morphisms $\gamma_2, \gamma_1 \in E_1,$ and $g, g'\in G$ we have the following identity: \begin{equation}\label{E:compoequi}
  		(\gamma_2 g)\circ (\gamma_1 g')=(\gamma_2\circ \gamma_1) g g'.
  	\end{equation}

  	For the inverse operation, suppose $[E\rra X_0]$ is a Lie groupoid, and a Lie group $G$ has a smooth, free, and proper action on $E$. 
  	Let us assume that the source-target maps are $G$ invariant; that is, the condition in \Cref {E:restricthom} and the action satisfies \Cref{E:compoequi}. Thus we have a principal $G$-bundle $E\ra E/G =: X_1,$ which defines a  principal $[G\rra e]$-bundle $\mb{E}=[E\rra X_0]$   over the Lie groupoid ${\mb X}=[E/G\rra X_0].$

  \end{example}
  The principal $2$-bundle in \Cref{E:Example of principal 2-bundle ordinary} is related to Lie groupoid $G$-extensions ( \Cref{subsection Lie groupoid G extension}), that we see later in this chapter in \Cref{Twisted principal 2-bundles and Lie groupoid $G$-extensions}.

\section{Decorated principal 2-bundles and categorical connections}\label{Decorated principal 2-bundles and categorical connections}
Most of the content of this section is based on our paper \cite{MR4403617}. Here, we are going to construct an important example of a principal Lie 2-group bundle over a Lie groupoid from the data of a Lie crossed module (\Cref{Definition:Lie crossed module}) and a principal Lie group bundle over a Lie groupoid (\Cref{Definition: Principal Lie group bundle over a Lie groupoid}). We call this class of principal 2-bundles `decorated principal 2-bundles'. The intended construction is a generalization of the construction of a principal $G$-groupoid over a Lie groupoid from a principal $G$-bundle over the base Lie groupoid (\Cref{Equivalence of principal G bundles and principal G groupoids}). We characterize these decorated principal 2-bundles in terms of a notion that we call a `categorical connection'. On the one hand, a categorical connection can be considered as an abstraction of the notion of horizontal lifting of paths in traditional principal bundles (\Cref{subsection:Parallel transport of a connection along a path}) and on the other, is an adaptation of the concept of splitting cleavage (\Cref{Definition splitting cleavage}) in the framework of principal 2-bundles over Lie groupoids.

\subsection{Decorated principal $2$-bundles}\label{SS:Decorated}
Let $\mb{G}=[H \rtimes_{\alpha} G \rra G]$ be the Lie $2$-group associated to a Lie crossed module $(G, H, \tau, \alpha)$ (\Cref{Lie 2 group from Lie crossed module}). Now, given a principal $G$-bundle over a Lie groupoid $\mb{X}$ (\Cref{Definition: Principal Lie group bundle over a Lie groupoid}), we will construct a principal $\mb{G}$-bundle over $\mb{X}$, which we call a \textit{decorated principal $\mb{G}$-bundle over $\mb{X}$}. A similar notion for a principal $2$-bundle over a `path space groupoid' is already present in the literature, introduced in \cite{MR3126940} by decorating the space of ${\bar A}$-horizontal paths ${\mathcal P}_{\bar A}P,$ for a connection   $\bar A$ on a principal $G$-bundle $P\to M$.  
		
For brevity, in our paper \cite{MR4403617},  we skipped some technical details in constructing decorated principal 2-bundles. Here, we provide a detailed version of the construction:	
\begin{proposition}\label{Prop:Decoliegpd}
	Let $\mb{G}=[H \rtimes_{\alpha} G \rra G]$ be the Lie $2$-group associated to a Lie crossed module $(G, H, \tau, \alpha)$. Consider  $\bigl(\pi\colon E_G \rightarrow X_0, \mu \colon s^{*} E_G \ra X_0, \mb{X}\bigr)$, a principal $G$-bundle over the Lie groupoid $\mb{X}$. Let us denote $s^{*}E_G \times H= (X_1 \times_{s, X_0, \pi} E_G) \times H$ by $(s^{*}E_G)^{\rm{dec}}$. 
	\begin{enumerate}[(i)]
		\item	The manifolds $(s^{*}E_G)^{\rm{dec}}$ and $E_G$ determines a Lie groupoid $[(s^{*}E_G)^{\rm{dec}} \rightrightarrows E_G]$ whose structure maps are defined as follows: 
		\begin{itemize}
			\item source map $s \colon (\gamma, p, h) \mapsto p$,
			\item target map $t \colon (\gamma, p, h) \mapsto \mu(\gamma, p) \tau(h^{-1})$,
			\item composition map $m \colon \big((\gamma_2, p_2, h_2), (\gamma_1, p_1, h_1) \big) \mapsto (\gamma_2 \circ \gamma_1, p_1 ,h_2h_1)$, 
			\item unit map $u \colon  : p \mapsto (1_{\pi(p)},p,e)$,
			\item $\mathfrak{i} \colon \bigl(\gamma, p, h) \mapsto (\gamma^{-1}, \mu(\gamma,p)\tau(h^{-1}), h^{-1}\bigr)$.
		\end{itemize}
		\item The Lie groupoid  $\mb{E}^{\rm{dec}}:=[(s^{*}E_G)^{\rm{dec}} \rightrightarrows E_G]$ forms a principal $\mb{G}$-bundle over the Lie groupoid $\mb{X}$. The action of $[H \rtimes_{\alpha} G \rra G]$ on $\mb{E}^{\rm{dec}}$ and the bundle projection are  given respectively, by 
		\begin{equation}\label{E:Actionondeco}
			\begin{split}
				\rho\colon &\mb{E}^{\rm dec}\times \mb{G}\ra \mb{E}^{\rm dec}\\
				&(p, g) \mapsto p\, g\\
				\bigl((\gamma, p, h)&, (h', g)\bigr)\mapsto \bigl(\gamma, p g, \alpha_{g^{-1}}(h'^{-1}\, h)\bigr),
			\end{split}
		\end{equation}
		and
		\begin{equation}\label{E:Projondeco}
			\begin{split}
				\pi ^{\rm{dec}} \colon &\mb{E}^{\rm dec}\ra \mb{X}\\
				& p \mapsto \pi(p)\\
				\bigl(\gamma,&  p, h\bigr) \mapsto \gamma.
			\end{split}
		\end{equation}
	\end{enumerate}
\end{proposition}
\begin{proof}
	\begin{enumerate}[(i)]
		\item It is easy to verify that  $[(s^{*}E_G)^{\rm{dec}} \rightrightarrows E_G]$   is indeed a groupoid. As $ \pi \colon E_G\ra X_0$ is a surjective submersion,  $(s^{*}E_G)^{\rm{dec}}$ is a manifold. Note that the source of an element $(\gamma,p, h)$ can be computed as $(\gamma,p,h) \mapsto (\gamma,p) \mapsto p$, i.e., a composition of surjective submersions; thus $s$ is a surjective submersion (and hence the target $t$ is too a surjective submersion).
		The smoothness of other structure maps follows quickly from the smoothness of $\mu$ and the smoothness of the structure maps of the Lie groupoid $\mb{X}$. 
		\item Note that it is obvious from the definition itself that $(\gamma,p,h),(e,e) \mapsto (\gamma,p,h)$ for all $(\gamma,p,h) \in (s^{*}E_G)^{\rm{dec}}$. Now let $(h_2,g_2), (h_1,g_1) \in H \rtimes_{\alpha} G$. Consider the following:
		\begin{equation}\label{Decor action E1}
			\begin{split}
				& (\gamma,p,h) \big((h_2,g_2)(h_1,g_1) \big)\\
				&=(\gamma,p,h) \big(h_2\alpha_{g_2}(h_1),g_2g_1 \big)\\
				&=\Big( \gamma, pg_2g_1, \alpha_{g^{-1}_1g^{-1}_2} \big(\alpha_{g_2}(h_1) \big)^{-1}h^{-1}_2h \big) \Big)\\
				&=\Big(\gamma, pg_2g_1, \alpha_{g^{-1}_1g^{-1}_2} \big(\alpha_{g_2}(h^{-1}_1) \big)h^{-1}_2h \big)   \Big) \,\, [\alpha_{g_2} \colon H \ra H \,\textit{is a homomorphism}]\\
				&= \Big( \gamma, pg_2g_1, \alpha_{g^{-1}_1g^{-1}_2} \big(\alpha_{g_2}(h^{-1}_1) \big)\alpha_{g^{-1}_1g^{-1}_2}(h^{-1}_2h) \Big) \,\,[\alpha_{g_1^{-1}g_2^{-1}} \colon H \ra H \,\textit{is a homomorphism}]\\
				&=\Big( \gamma, pg_2g_1, \alpha_{g^{-1}_1}(h^{-1}_1) \alpha_{g^{-1}_1g^{-1}_2}(h^{-1}_2h) \Big)
			\end{split}
		\end{equation}
		On the other hand, consider 
		\begin{equation}\label{Decor action E2}
			\begin{split}
				&\big((\gamma,p,h)(h_2,g_2) \big)(h_1,g_1)\\
				&= \big(\gamma,pg_2, \alpha_{g^{-1}_2}(h^{-1}_2h) \big)(h_1,g_1)\\
				&=\Big(\gamma,pg_2g_1, \alpha_{g^{-1}_1}\big(h_1^{-1}\alpha_{g^{-1}_2}(h^{-1}_2h) \big) \Big)\\
				&= \underbrace{\Big(\gamma,pg_2g_1, \alpha_{g^{-1}_1}\big(h_1^{-1} \big)\alpha_{g_1^{-1}g^{-1}_2} \big( h^{-1}_2h) \big) \Big)}_{[\alpha_{g_1^{-1}} \colon H \ra H \,\, \textit{and}\,\, \alpha \colon G \ra {\rm{Aut}}(H)\,\, \textit{are homomorphisms}]}
			\end{split}
		\end{equation}			
		Comparing \Cref{Decor action E1} and \Cref{Decor action E2}, we get $$(\gamma,p,h) \big((h_2,g_2)(h_1,g_1) \big)=\big((\gamma,p,h)(h_2,g_2) \big)(h_1,g_1).$$ Hence, the map $$\rho_1 \colon (s^{*}E_G)^{\rm{dec}} \times (H \rtimes_{\alpha} G) \ra (s^{*}E_G)^{\rm{dec}},$$ given by $\bigl((\gamma, p, h), (h', g)\bigr)\mapsto \bigl(\gamma, p g, \alpha_{g^{-1}}(h'^{-1}\, h)\bigr)$ is indeed a Lie group action as the smoothness of the action is obvious from the definition itself. Observe that as a direct consequence of the free action of $G$ on $E_G$, the action of $H\rtimes_{\alpha}G$ on $(s^{*}E_G)^{\rm{dec}}$ is free. 
		
		\subsubsection*{Verification of the functoriality of $\rho$:}\label{verification of functoriality of decorated action}
		Source consistency is obvious. Consistency with the target map is shown below:
		\begin{equation}\nonumber
			\begin{split}
				&t \big( (\gamma,p,h)(h'g) \big)\\
				&=t \bigl(\gamma, p g, \alpha_{g^{-1}}(h'^{-1}\, h)\bigr)\\
				&= \mu(\gamma,p)g \Big(\tau \big( \alpha_{g^{-1}}(h'^{-1}\, h) \big) \Big)^{-1}\\
				&=\mu(\gamma,p)g  \underbrace{\Big(g^{-1}\tau(h'^{-1}h)g \Big)^{-1}}_{[\textit{by}\,\, \Cref{E:Peiffer}]} \\
				&=\mu(\gamma,p)\tau(h^{-1}h')g\\
				&=\mu(\gamma,p)\tau(h^{-1}) \tau(h')g\\
				&=t(\gamma,p,h)t(h',g).
			\end{split}
		\end{equation}
		Consistency with the unit map is straightforward. Now let $(\gamma_2,p_2,h_2), (\gamma_1,p_1,h_1) \in (s^{*}{E_G})^{\rm{dec}}$ and $(h'_2,g_2),(h'_1,g_1) \in H \rtimes_{\alpha}G$, such that we have the following:
		\begin{equation}\label{Decor Ac E3}
			\begin{split}
				& p_2= \mu(\gamma_1,p_1)\tau(h^{-1}_1),\\
				&g_2= \tau(h'_1)g_1.
			\end{split}
		\end{equation}
 By direct computation, we have 
\begin{equation}\label{Decor Ac E4}
\rho \Big( \big((\gamma_2,p_2,h_2),(h'_2,g_2) \big) \circ \big((\gamma_1,p_1,h_1),(h'_1,g_1)  \big) \Big)= \Big( \gamma_2 \circ \gamma_1, p_1g_1, \alpha_{g^{-1}_1}(h'^{-1}_1h'^{-1}_2h_2h_1) \Big)
\end{equation}
On the other hand we have,
		\begin{equation}\label{Decor Ac E5}
			\begin{split}
				& \rho \Big((\gamma_2,p_2,h_2),(h'_2,g_2) \Big) \circ \rho \Big((\gamma_1,p_1,h_1),(h'_1,g_1) \Big)\\
				&= \Big(\big( \gamma_2,p_2g_2, \alpha^{-1}_{g_2}(h'^{-1}_2h_2) \big) \Big) \circ  \Big(\big( \gamma_1,p_1g_1, \alpha_{g^{-1}_1}(h'^{-1}_1h_1) \big) \Big)\\
				&=\Big( \gamma_2 \circ \gamma_1, p_1g_1, \alpha_{g^{-1}_2}(h'^{-1}_2h_2) \big)\alpha_{g^{-1}_1}(h'^{-1}_1h_1) \big) \Big)\\
				&= \Big( \gamma_2 \circ \gamma_1, p_1g_1, \alpha_{g^{-1}_1} \big(\alpha_{\tau(h'^{-1}_1)}(h'^{-1}_2h_2) \big)\alpha_{g^{-1}_1}(h'^{-1}_1h_1) \big)\Big) \quad [{\rm{using}} \,\, \Cref{Decor Ac E3}]\\
				&= \Big( \gamma_2 \circ \gamma_1, p_1g_1, \alpha_{g^{-1}_1}(h'^{-1}_1h'^{-1}_2h_2h_1) \big) \Big) \quad [{\rm{using}} \,\, \Cref{E:Peiffer}].
			\end{split}
		\end{equation}
		Comparing \Cref{Decor Ac E4} and \Cref{Decor Ac E5}, it follows that $\rho$ is consistent with the composition. Hence, we proved that	$\rho\colon \mb{E}^{\rm dec}\times \mb{G}\ra \mb{E}^{{\rm{dec}}}$ defines a Lie 2-group action of $[H \rtimes_{\alpha}G \rra G]$ on the Lie groupoid $\mb{E}^{\rm{dec}}$.
		\subsubsection*{To show that $\pi^{\rm{dec}}_1 \colon s^{*}E_{G}^{\rm{dec}} \ra X_1$ is a principal $H \rtimes_{\alpha}G$-bundle over $X_1$:}
		Note that according to \Cref{Definition: Principal G-bundle}, it is sufficient to prove the following:
		\begin{itemize}

  \item[(a)] Existence of an $H \rtimes_{\alpha}G$-equivariant local  trivialization of $\pi^{\rm{dec}}_1 \colon s^{*}E_{G}^{\rm{dec}} \ra X_1$.
  
	\item[(b)] the map $\psi \colon s^{*}E_G^{\rm{dec}} \times (H \rtimes_{\alpha}G) \ra s^{*}E_G^{\rm{dec}} \times_{X_1} s^{*}E_G^{\rm{dec}}$ defined by  $$\big((\gamma,p,h), (h',g) \big) \mapsto \big((\gamma,p,h), (\gamma,pg,\alpha_{g^{-1}}(h'^{-1}h)) \big)$$ is a diffeomorphism;
			
		\end{itemize}
  	\subsubsection*{Proof of (a):}
		In order to prove (a), we claim that any $G$-equivariant local trivialization $(u_i, \phi_i)_{i}$ on the principal $G$-bundle $\pi\colon E_G\ra X_0$ induces a $H \rtimes_{\alpha}G$-equivariant local  trivialization $(s^{-1}(u_i), \tilde{\phi}_{i})_{i}$ on $\pi_1^{\rm{dec}}\colon (s^{*}E_G)^{\rm{dec}}\ra \mb{X}_1$ given by
		\begin{equation}\nonumber
			\begin{split}
				\tilde{\phi}_i \colon &(\pi_1^{\rm{dec}})^{-1}\bigl(s^{-1}(u_i)\bigr)\ra s^{-1}(u_i)\times (H\rtimes_{\alpha} G)\\
				&(\gamma, p, h)\mapsto \big(\gamma,  \alpha_{g_{p}}(h^{-1}), g_{p} \big),
			\end{split}
		\end{equation}
		where $g_{p}={\rm{pr}}_2\circ \phi_i(p)\in G.$
		
		Observe that the injectivity of $\tilde{\phi}_i$ is a direct consequence of the injectivity of $\phi_i$. To see $\tilde{\phi}_i$ is surjective, consider an element $ \big(\gamma, (h, g) \big) \in s^{-1}(u_i)\times (H\rtimes_{\alpha} G)$ and then it follows that
		\begin{equation}\nonumber
			\tilde{\phi}_i \big(\gamma, \phi^{-1}_i(s(\gamma),g), \alpha_{g^{-1}}(h^{-1}) \big)=\big( \gamma,(h,g) \big).
		\end{equation}
		To verify the $H \rtimes_{\alpha} G$-equivariance of $\tilde{\phi_i}$, consider $(\gamma,p,h) \in (\pi_1^{\rm{dec}})^{-1}\bigl(s^{-1}(u_i)\bigr)$ and $(h',g) \in H \rtimes_{\alpha} G$. Then, we have
		\begin{equation}\label{Decor Ac7}
			\begin{split}
				& \tilde{\phi}_i\big((\gamma,p,h)(h',g) \big)\\
				&= \tilde{\phi}_i \big( \gamma, pg, \alpha_{g^{-1}}(h'^{-1}h) \big)\\
				&= \Bigg( \gamma, \Big(\alpha_{g_pg}\big((\alpha_{g^{-1}}(h'^{-1}h))^{-1} \big), g_pg \Big) \Bigg)\\
				&=\Bigg(\gamma, \alpha_{g_p}(h^{-1}h'),g_pg \Big) \Bigg)
			\end{split}
		\end{equation}
		On the other hand, we have 
		\begin{equation}\label{Decor Ac8}
			\begin{split}
				& \big(\tilde{\phi}_i\big((\gamma,p,h) \big)(h',g)\\
				&=\Big(\gamma, \alpha_{g_p}(h^{-1}),g_p  \Big)(h',g)\\
				&= \Bigg(\gamma, \alpha_{g_p}(h^{-1}) \alpha_{g_p}(h'),g_pg  \Bigg)\\
				&= \Bigg(\gamma,  \alpha_{g_p}(h^{-1}h'),g_pg  \Bigg).
			\end{split}
		\end{equation}
		
		From \Cref{Decor Ac7} and \Cref{Decor Ac8}, it follows that $\tilde{\phi}_i$ is $H \rtimes_{\alpha}G$-equivariant. The smoothness of $\tilde{\phi_i}$ is obvious from the definition. Moreover, inverse $\tilde{\phi}^{-1}$ can be easily seen to be as follows:
		\begin{equation}\nonumber
			\begin{split}
				\tilde{\phi}^{-1}_i \colon &  s^{-1}(u_i)\times (H\rtimes_{\alpha} G) \ra (\pi_1^{\rm{dec}})^{-1}\bigl(s^{-1}(u_i))\\
				& \Big(\gamma, (h,g) \Big) \mapsto \Big( \gamma, \phi^{-1}_{i}(s(\gamma),g), \alpha_{g^{-1}}(h^{-1}) \Big),
			\end{split}
		\end{equation}
		which is clearly smooth. So, we showed that  $(s^{-1}(u_i), \tilde{\phi}_{i})_{i}$  is an $H \rtimes_{\alpha}G$-equivariant local trivialization on $\pi^{\rm{dec}}_1 \colon s^{*}E_{G}^{\rm{dec}} \ra X_1$, and thus completing the proof of condition (a). 
		\subsubsection*{Proof of (b):}
      Injectivity of $\psi$ follows from the freeness of the action of  $H \rtimes_{\alpha}G$ on $s^{*}E_{G}^{\rm{dec}}$. To show the surjectivity of $\psi$, consider an element $\big((\gamma,p,h),(\gamma,q, \bar{h}) \big) \in s^{*}E_G^{\rm{dec}} \times_{X_1} s^{*}E_G^{\rm{dec}}$. Now, by direct computation one can verify that $\psi \big((\gamma, p,h), (h\alpha_g(\bar{h}^{-1}),g) \big)= ((\gamma,p,h),(\gamma,q, \bar{h}) \big)$, where $g$ is the unique element in $G$ such that $q=pg$. Then, the smoothness  the action map $\rho_1 \colon (s^{*}E_G)^{\rm{dec}} \times (H \rtimes_{\alpha} G) \ra (s^{*}E_G)^{\rm{dec}}$ and \textbf{(a)} show that $\psi$ is a bijective local diffeomorphism and hence, a diffeomorphism.

		Hence, $\pi^{\rm{dec}}_1 \colon s^{*}E_{G}^{\rm{dec}} \ra X_1$ is a principal $H \rtimes_{\alpha}G$-bundle over $X_1$. Moreover, since the functoriality of $\pi^{\rm{dec}}$ is obvious, we proved that $\pi^{\rm{dec}} \colon \mb{E}^{\rm{dec}}\ra \mb{X}$ is a principal $[H \rtimes_{\alpha} G \rra G]$-bundle over $\mb{X}$ and thus completing the proof of the main proposition.
	\end{enumerate}
\end{proof}
We call $\pi^{\rm{dec}} \colon \mb{E}^{\rm{dec}}\ra \mb{X}$ as the \textit{decorated $[H \rtimes_{\alpha} G \rra G]$-bundle associated to $\bigl(\pi\colon E_G \rightarrow X_0, \mu \colon s^{*} E_G \ra X_0, \mb{X}\bigr)$ and the Lie crossed module $(G, H, \tau, \alpha)$.}	
\begin{example}
	The product  Lie $2$-group $\mb{G}=[G_1\rra G_0]$-bundle $\mb{X}\times \mb{G}$ over a Lie groupoid $\mb{X}$
	in \Cref{E:Example of product bundle} can be seen as a decorated bundle. Observe that $s^{*}(X_0\times G_0)=X_1\times G_0.$ The action of $\mb{X}$ on $X_0\times G_0$ is given as $s^* (X_0\times G_0)\ra X_0\times G_0$, $(\gamma, g)\mapsto (t(\gamma), g).$ 
	Thus, the Lie groupoid $[X_1\times G_0\rra X_0\times G_0]$ is the pullback Lie groupoid, and by decoration (i.e., \Cref{Prop:Decoliegpd}), we get back $\mb{X}\times \mb{G}.$
\end{example}

\begin{example}\label{Ex:EoXodecobundle}
	Consider the Lie $2$-group $\mb{G}:=[H \rtimes_{\alpha}G \rra G]$ associated to the Lie crossed module $(G, H, \tau, \alpha)$. Let $ \pi \colon E\ra M$  be a principal $G$-bundle over a manifold $M$. The discrete Lie groupoid $[M \rra M]$ trivially acts on $E$ by $(x, p)\mapsto p.$ Then the  Lie  groupoid ${E}^{\rm dec}:=[E\times H\rra E]$ defines the associated decorated principal $\mb{G}$-bundle. 
\end{example}

	\subsection{Categorical connections}\label{SS:Catconnection}
	In this subsection, we introduce the notion of a `categorical connection' on a principal 2-bundle over a Lie groupoid. It serves as a tool that prescribes a way to lift morphisms in the base Lie groupoid of the bundle to the total Lie groupoid. Moreover, the said prescription behaves well with the underlying categorical structure of the base Lie groupoid and the action of the structure Lie 2-group on the total Lie groupoid. While the initial inspiration for the definition of categorical connection stemmed from the concept of the horizontal lifting of a path by a connection in the traditional set-up of principal bundles (\Cref{subsection:Parallel transport of a connection along a path}), in the subsequent section \Cref{Section: A Quasi-principal 2-bundle over a Lie groupoid}, we will see that one can view this notion as an analog of a splitting cleavage \Cref{Definition splitting cleavage} in a fibered category. 
	\begin{definition}\label{Def:categorical connection}
		For a Lie 2-group $\mb{G}$, a \textit{categorical connection} $\mathcal{C}$ on a principal $\mb{G}$-bundle $\pi\colon \mb{E} \ra \mb{X}$ over $\mb{X}$ is defined as a smooth map $\mathcal{C}\colon {s}^{*}E_0 \ra E_1$, satisfying the following propoerties:
		\begin{enumerate}[(i)]
			\item $s(\mathcal{C}(\gamma,p))=p$, for all $(\gamma,p) \in s^{*}E_0$,
			\item $\pi_1(\mathcal{C}(\gamma,p))= \gamma$, for all $(\gamma,p) \in s^{*}E_0$,
			\item $\mathcal{C}(\gamma, p. g)= \mathcal{C}(\gamma, p) \cdot 1_g$, for all $(\gamma, p) \in {s}^{*}E_0$ and $g \in G_0$,
			\item $\mathcal{C}(1_x,p)=1_p$, for any $x\in X_0$ and $p\in \pi^{-1}(x)$,
			\item if $(\gamma_2, p_2), (\gamma_1, p_1) \in {s}^{*}E_0$, such that ${s}(\gamma_2)={t}(\gamma_1)$ and $p_2=t\bigl({\mathcal C}(\gamma_1, p_1)\bigr),$ then we have $$\mathcal{C}(\gamma_2 \circ \gamma_1 , p_1)= \mathcal{C}(\gamma_2, p_2) \circ \mathcal{C}(\gamma_1, p_1).$$
		\end{enumerate}
	\end{definition}

	\begin{remark}
		An idea akin to categorical connection has been investigated earlier within the context of path space groupoids in \cite{MR3126940}. There is also a variant of this concept in the setup of Lie groupoid fibrations as \textit{flat cleavages} \cite{MR3968895}, and in the VB-groupoids framework as flat linear cleavages (see \Cref{subsection VB groupoids}). Moreover, in a different context, \textit{Martins} and \textit{Picken} used the terminology `categorical connection' in \cite{MR2661492}. In particular, their notion of categorical connection consists of a Lie crossed module $(G, H, \tau, \alpha)$ and 
		pair $(\omega, \Omega)$, where $\omega$ is a usual $L(G)$-valued connection 1-form on a traditional principal $G$-bundle $P \ra M$ and $\Omega$ is a $L(H)$-valued 2-form on $P$, satisfying certain conditions.
	\end{remark}
	\begin{remark}\label{Categorical connection is a smooth embedding}
		Every categorical connection on a principal 2-bundle over a Lie groupoid is a smooth embedding, i.e., an immersion such that it is also a topological embedding.
	\end{remark}

	\begin{example}\label{Unique categorical connection of [GG] bundle}
		A unique categorical connection exists on a  principal $[G\rra G]$-bundle over a Lie groupoid.  To see this, observe that one can identify a principal $[G \rra G]$-bundle $\pi \colon \mb{E} \ra \mb{X}$ with the decorated principal $[G \rra G]$-bundle $ \pi^{{\rm{dec}}} \colon s^{*}E_0 \ra \mb{X}$ (\Cref{Prop:Decoliegpd}), with a unique categorical connection given by the identity map $s^{*}E_0 \ra s^{*}E_0$.
	\end{example}	
	\begin{example}\label{Ex:CatconnDeco}
		Let  $\pi^{\rm{dec}} \colon s^{*}\mb{E}^{\rm{dec}} \ra \mb{X}$ be  the decorated principal $\mb{G}$-bundle over a Lie groupoid $\mb{X}$  constructed in \Cref{Prop:Decoliegpd}. Then, any smooth map $\beta\colon E_G\ra H$ that satisfies the condition $$\beta(p g)=\alpha_{g^{-1}}(\beta(p)),$$ for $p\in E_G, g\in G$, defines  a categorical connection
		${\mathcal C}\colon (\gamma, p)\mapsto \bigl(\gamma, p, 
		\beta(p)\beta(\mu(\gamma,p))^{-1} \bigr)$. 
		In particular, for the trivial map $\beta\colon p\mapsto e$, we call the categorical connection $(\gamma, p)\mapsto (\gamma, p, e)$ as the \textit{canonical categorical connection} on the decorated principal $\mb{G}$-bundle  $\pi^{\rm{dec}} \colon s^{*}\mb{E}^{\rm{dec}} \ra \mb{X}$.
	\end{example}
	
	On the otherhand, for any categorical connection  $\mathcal C$ on a principal $\mb{G}:=[H \rtimes_{\alpha}G \rra G]$-bundle  $\pi\colon \mb{E} \ra \mb{X}$  over a Lie groupoid $\mb{X}$, we can associate a principal $G$-bundle $(\pi_0 \colon E_0 \ra X_0, \mu_{\mc{C}}, \mb{X})$ over $\mb{X}$, where the action $\mu_{\mc{C}}$ is defined as
	\begin{equation}\label{actionwithcatcon}
		\begin{split}
			\mu\colon s^{*}E_0&\ra E_0,\\
			(\gamma, p)&\mapsto t\bigl({\mathcal C}(\gamma, p)\bigr).
		\end{split}
	\end{equation}
	Now, given any ${\widetilde \gamma}\in E_1$, since we have $\pi(\widetilde \gamma)=\pi \bigl({\mathcal C}\bigl(\pi(\widetilde \gamma), s(\widetilde \gamma)\bigr)\bigr)$ and $s(\widetilde \gamma)=s\bigl({\mathcal C}\bigl(\pi(\widetilde \gamma), s(\widetilde \gamma)\bigr)$,
	there is a unique $h\in H$ such that $[{\mathcal C}\bigl(\pi(\widetilde \gamma), s(\widetilde \gamma)\bigr)](h, e)=\widetilde \gamma$. This induces an isomorphism of principal $\mb{G}$-bundles over $\mb{X}$ defined as
	\begin{equation}\nonumber
		\begin{split}
			&\theta\colon s^*E_0\times H \ra E_1\\
			&((\gamma, p), h\bigr)\ra {\mathcal C}(\gamma, p) (h^{-1}, e).
		\end{split}
	\end{equation}
	\begin{remark}
		Observe that $\mc{C}$ idenitifies the semi-direct product groupoid $[s^{*}E_0 \rra E_0]$ (see \Cref{Semi-direct product groupoid}) as the Lie subgroupoid (\Cref{Lie subgroupoid}) of the Lie groupoid $\mb{E}=[E_1\rra E_0]$. 
		\begin{equation}\nonumber
			\begin{tikzcd}[sep=small]
				s^*E_0 \arrow[rr,"\mathcal C"] \arrow[dd,xshift=0.75ex]
				\arrow[dd,xshift=-0.75ex] &  & E_1 \arrow[dd,xshift=0.75ex]
				\arrow[dd,xshift=-0.75ex] \\
				&  &                \\
				E_0 \arrow[rr,"{\rm Id}"]            &  & E_0          
			\end{tikzcd}.	
		\end{equation}
	\end{remark}
	Thus, after identifying $G$ as a subgroup of $G_1=H\rtimes G,$ we see that \Cref{Categorical connection is a smooth embedding} implies that  $[s^*E_0\rra E_0]\simeq [{\mathcal C}(s^*E_0)\rra E_0]$ is a sub bundle over $\mb{X}$ with a reduced structure Lie $2$-group $[G\rra G]$.

	\begin{lemma}\label{Lem:isodecgeneral} 
		The map $(\theta,\rm{Id})$ defines an isomorphism of principal $[H \rtimes_{\alpha}G \rra G]$-bundles over $\mb{X}$, from  $\pi^{\rm{dec}} \colon \mb{E}^{\rm{dec}} \ra \mb{X}$ to $\pi \colon \mb{E} \ra \mb{X}$.
		\begin{proof}
			The composition law in the Lie groupoid $s^{*}\mb{E}^{\rm{dec}}$ (\Cref{Prop:Decoliegpd}) and the functoriality of the action of $\mb{G}$, together imply $\theta$ is a functor. Below, we give a detailed verification
			\begin{equation}\nonumber
				\begin{split}
					&\theta\bigl((\gamma_2, p_2), h_2)\bigr)\circ \theta\bigl((\gamma_1, p_1), h_1\bigr)\\
					&=\bigl({\mathcal C}(\gamma_2, p_2) (h_2^{-1}, e)\bigr)\circ\bigl({\mathcal C}(\gamma_1, p_1) (h_1^{-1}, e)\bigr)\\
					&=\bigl({\mathcal C}(\gamma_2, p_2) (h_2^{-1}, e)\bigr)\circ\bigl({\mathcal C}(\gamma_1, p_1)(e, \tau(h_1^{-1})) (e, \tau(h_1)) (h_1^{-1}, e)\bigr)\\
					&=\bigl({\mathcal C}(\gamma_2, p_2) (h_2^{-1}, e)\bigr)\circ\bigl({\mathcal C}(\gamma_1, p_1\tau(h_1)^{-1}) (e, \tau(h_1)) (h_1^{-1}, e)\bigr)\\
					&=\bigl({\mathcal C}(\gamma_2, p_2) (h_2^{-1}, e)\bigr)\circ\bigl({\mathcal C}(\gamma_1, p_1\tau(h_1)^{-1}) (e, \tau(h_1)) (h_1^{-1}, e)\bigr)\\
					&=\underbrace{\bigl({\mathcal C}(\gamma_2, p_2) \circ {\mathcal C}(\gamma_1, p_1\tau(h_1)^{-1})    \bigr)\bigl((h_2^{-1}, e) \circ \bigl((e, \tau(h_1)) (h_1^{-1}, e)\bigr)\bigr)}_{[\textit{using} \,\, \Cref{E:Identitiesechangeinverse}]}\\
					&={\mathcal C}\bigl(\gamma_2\circ \gamma_1, p_1\tau(h_1)^{-1}   \bigr)\bigl((h_2^{-1}, e) \circ (h_1^{-1}, \tau(h_1))\bigr)\,\, [\textit{using}\,\, \textit{condition}\,\, (v) \,\, \textit{in}\,\, \Cref{Def:categorical connection}]\\
					&={\mathcal C}\bigl(\gamma_2\circ \gamma_1, p_1\tau(h_1)^{-1}\bigr) \bigl(h_2^{-1} h_1^{-1}, \tau(h_1)\bigr)\\
					&={\mathcal C}\bigl(\gamma_2\circ \gamma_1, p_1\bigr) \bigl(e, \tau(h_1^{-1})\bigr)   \bigl(h_2^{-1} h_1^{-1}, \tau(h_1)\bigr)\\
					&={\mathcal C}\bigl(\gamma_2\circ \gamma_1, p_1\bigr) \bigl(h_1^{-1} h_2^{-1}, e \bigr) \\
					&=\theta\bigl((\gamma_2\circ \gamma_1, p_1), h_2 h_1\bigr).
				\end{split}
			\end{equation}
			To see the $H \rtimes_{\alpha} G$ -equivariancy of $\theta$, note that by the action defined in \Cref{Prop:Decoliegpd}, we have 
			\begin{equation}\nonumber
				\begin{split}
					&\theta\bigl((\gamma, p, h)(h', g')\bigr)\\
					&={\mathcal C}(\gamma, p g')\bigl(\alpha_{g'^{-1}}(h^{-1}h'), e\bigr)\\
					&={\mathcal C}(\gamma, p)(e, g')\bigl(\alpha_{g'^{-1}}(h^{-1}h'), e\bigr)\\
					&={\mathcal C}(\gamma, p)(h^{-1}, e)(h, e)(e, g')\bigl(\alpha_{g'^{-1}}(h^{-1}h'), e\bigr)\\
					&=\theta \bigl((\gamma, p), h\bigr)(h', g').
				\end{split}
			\end{equation}
			Hence, obviously  $(\theta, \rm Id)$ is a morphism of principal $[H \rtimes_{\alpha} G] \rra G$-bundles over $\mb{X}$
		\end{proof}
	\end{lemma}
	In conclusion, we obtain a characterization of principal $2$-bundles over Lie groupoids in terms of the existence of categorical connections, stated below:
	\begin{proposition}\label{prop:Characterisdecorated}
		A principal $2$-bundle over a Lie groupoid is a decorated principal 2-bundle if and only if it admits a categorical connection.
	\end{proposition}
	
	\begin{corollary}\label{Corollary:discreteisdecorated}
		For a Lie crossed module $(G,H, \tau, \alpha)$, any principal $\mb{G}=[H \rtimes_{\alpha}G \rra G]$-bundle over a discrete Lie groupoid is a decorated principal $\mb{G}$-bundle as given in \Cref{Ex:EoXodecobundle}.
	\end{corollary}
	\begin{proof}
		The map $(1_x, p) \mapsto 1_p$ for $p\in E, x=\pi(p)$, defines a categorical connection on such a principal $\mb{G}$-bundle.
	\end{proof}

Thr following is obvious from the proof of above corollary.	
\begin{corollary}\label{unique Cat connection on bundle over discrete space}
For a Lie 2-group $\mb{G}$, any principal $\mb{G}$-bundle over a discrete Lie groupoid (\Cref{Lie groupoid example: manifold}) admits a unique categorical connection.
	\end{corollary}

	\begin{example}\label{E:Exampleprincipalpairlie2deco}
		The principal $[G\times G\rra G]$-bundle $[E_1\rra E_0]$ over the Lie groupoid $\mb{X}=[X_1\rra X_0]$ in \Cref{E:Exampleprincipalpairlie2} admits a categorical connection if and only if $E_0\ra X_0$ is a principal $G$-bundle over the Lie groupoid $\mb{X}$.	Let $E_0\ra X_0$ be a principal $G$-bundle over the Lie groupoid $\mb{X}$ with respect to the action map $\mu\colon X_1\times_s E_0\ra E_0.$ Then the isomorphism with the corresponding decorated principal bundle is given by $(p, \gamma, q)\mapsto\bigl (( \gamma,p), g\bigr),$ where $\mu(\gamma, p)=q\, g.$
	\end{example}
	
	\begin{example}	
		Consider a Lie crossed module $(G,H, \tau, \alpha)$, a principal $G$-bundle  $\pi: P \ra M$ over a smooth manifold $M$ and the Čech groupoid $C(\mc{U})=\big[ \bigsqcup_{i,j }U_{ij} \rra \bigsqcup_{i }U_i \big]$ (\Cref{Cover groupoid}) associated to an open cover $\mc{U}:= \lbrace U_{i} \rbrace_{i \in I }$ of $M$, where $U_{ij}:= U_{i} \cap U_{j}$. Suppose $ \lbrace P_i \ra U_i \rbrace_{i \in I}$ and $ \lbrace P_{ij} \ra U_{ij}\rbrace_{i,j \in I}$ be the families of restricted principal $G$-bundles  of $\pi: P \ra M$. Then it is easy to see that $ \Big( \big[ \bigsqcup_{i,j} P_{ij} \rra \bigsqcup_{i}P_i \big] \longrightarrow \big[ \bigsqcup_{i,j}U_{ij} \rra \bigsqcup_{i}U_i \big] \Big)$ is a principal  $[G \rra G]$-bundle with respect to the obvious projection. Observe that the action   $\big( (i, j, x), (i, p) \big) \mapsto (j, p)$ of $\big[ \bigsqcup_{i,j }U_{ij} \rra \bigsqcup_{i }U_i \big]$ on $\bigsqcup_{i}P_i $	turns $\bigsqcup_{i}P_i $ into a principal $G$-bundle over  $\big[ \bigsqcup_{i,j }U_{ij} \rra \bigsqcup_{i }U_i \big]$. Thus, $\big[ \bigsqcup_{i,j}P_{ij} \rra \bigsqcup_{i} P_i \big]$  is the principal  $[G \rra G]$-bundle associated to the principal $G$-bundle $\bigsqcup_{i }P_i $ over the Čech groupoid. We construct the corresponding decorated $[H \rtimes_{\alpha} G \rra G]$-bundle $\big[ \bigsqcup_{i,j}P_{ij} \times H \rra \bigsqcup_{i}P_i \big]$  over $\big[ \bigsqcup_{i,j}U_{ij} \rra \bigsqcup_{i}U_i \big]$ using  \Cref{Prop:Decoliegpd}.	 
	\end{example}
	The following were not presented in either of our papers \cite{chatterjee2023parallel} or \cite{MR4403617} arising out of this thesis.
	\subsection*{Relation between the existence of categorical connections and the triviality of traditional principal bundles}
	For a Lie 2-group $\mb{G}$, let $\pi \colon \mb{E} \ra \mb{X}$ be a principal $\mb{G}$-bundle over a Lie groupoid $\mb{X}$. Here, we will see how the existence of a categorical connection on $\pi \colon \mb{E} \ra \mb{X}$  is related to the triviality of the underlying pair of classical principal bundles $\pi_{1} \colon E_1 \ra X_1$ and $\pi_0 \colon  E_0 \ra X_0$.
	
	\begin{proposition}
		For a Lie 2-group $\mb{G}$, let $\pi \colon \mb{E} \ra \mb{X}$ be a principal $\mb{G}$-bundle over a Lie groupoid $\mb{X}$ such that $\pi_1 \colon E_1 \ra X_1$ is a trivial principal $G_1$-bundle, then $\pi_0 \colon E_0 \ra X_0$ is also a trivial principal $G_0$-bundle and $\pi \colon \mb{E} \ra \mb{X}$ admits a categorical connection.
	\end{proposition}
	\begin{proof}
		Since $\pi_1 \colon E_1 \ra X_1$ is trivial, it has a smooth global section $\sigma_1 \colon X_1 \ra E_1$. Define, $\sigma_0 \colon X_0 \ra E_0$ as $x \mapsto s(\sigma_1(1_x))$. Note that by definition, $\sigma_0$ is smooth. Also, $\pi_0(s(\sigma_1(1_x)))=s(\pi_1(\sigma_1(1_x)))=s(1_x)=x$. Since $\sigma_0$ is a smooth global section of $\pi_0 \colon E_0 \ra X_0$, and therefore $\pi_0 \colon E_0 \ra X_0$ is trivial. So, $E_1 \cong X_1 \times G_1$ and $E_0 \cong X_0 \times G_0$. It is easy to verify that the map $\mc{C} \colon s^{*} (X_0 \times G_0) \ra X_1 \times G_1$ defined by $(\gamma, x,g) \mapsto (\gamma, 1_g)$ is a categorical connection.
	\end{proof}

	\begin{proposition}
		For a Lie 2-group $\mb{G}$, let $\pi \colon \mb{E} \ra \mb{X}$ be a principal $\mb{G}$-bundle over a Lie groupoid $\mb{X}$ such that $\pi_0 \colon E_0 \ra X_0$ is a trivial principal $G_0$-bundle and there is a smooth map $\mc{C} \colon s^*{E_0} \ra E_1$ such that $\pi_1\big(\mc{C}(\gamma,p)\big)= \gamma$ for all $(\gamma,p) \in s^{*}E_0$. Then $\pi_1 \colon E_1 \ra X_1$ is also a trivial principal $G_1$-bundle and hence, $\pi \colon \mb{E} \ra \mb{X}$ admits a categorical connection.
	\end{proposition}
	\begin{proof}
		As $\pi_0 \colon E_0 \ra X_0$ is trivial, the pull-back principal $G_0$-bundle $\pi_o^* \colon s^{*}E_0 \ra X_1$ (\Cref{Definition: pull-back principal G-bundle}) is trivial. Suppose $\sigma \colon X_1 \ra s^{*}E_0$ is a smooth global section of $\pi_0^{*}$. Now, consider the map $\mc{C}  \circ \sigma \colon X_1 \ra E_1$. Then, the triviality of $\pi_1 \colon E_1 \ra X_1$ follows from the observation that for $\gamma \in X_1$, we have $\pi_1 \circ \mc{C}  \circ \sigma(\gamma)= \gamma$.
	\end{proof}
	The following corollary is an immediate consequence of the above two propositions:
	\begin{corollary}
		For a Lie 2-group $\mb{G}$, let $\pi \colon \mb{E} \ra \mb{X}$ be a principal $\mb{G}$-bundle over a Lie groupoid $\mb{X}$ such that there is a smooth map $\mc{C} \colon s^*{E_0} \ra E_1$ satisfying $\pi_1\big(\mc{C}(\gamma,p)\big)= \gamma$ for all $(\gamma,p) \in s^{*}E_0$. If either of  $\pi_0 \colon E_0 \ra X_0$ or $\pi_1 \colon E_1 \ra X_1$ is trivial, then $\pi \colon \mb{E} \ra \mb{X}$ admits a categorical conection.
	\end{corollary}

	\section{Quasi-principal 2-bundles over Lie groupoids and their characterizations}\label{Section: A Quasi-principal 2-bundle over a Lie groupoid}
This section introduces the notion of a \textit{quasi-principal 2-bundle over a Lie groupoid} and a \textit{pseudo-principal Lie crossed module-bundle over a Lie groupoid}. The main result of the section (\Cref{Main Theorem 1}) shows that the respective categories are equivalent via a proof of Lie 2-group torsor version of the classical Grothendieck construction (\Cref{subsection Fibered categories}). The content of this section is mainly adapted from our paper \cite{chatterjee2023parallel}.

\subsection{A quasi-principal 2-bundle over a Lie groupoid}\label{A quasi-principal 2-bundle over a Lie groupoid}
Consider a Lie 2-group $\mb{G}:=[G_1 \rra G_0]$. Now, given a principal $\mb{G}$-bundle $\pi \colon \mb{E} \ra \mb{X}$ over a Lie groupoid $\mb{X}$, there is a canonical morphism $P \colon E_1 \ra s^{*}E_0$ of principal bundles, from the pull-back principal $G_0$-bundle $\pi_0^{*} \colon s^{*}E_0 \ra X_1$ (\Cref{Definition: pull-back principal G-bundle}) to the principal $G_1$-bundle $\pi_1 \colon E_1 \ra X_1$, defined as $\delta \mapsto (\pi_1(\delta), s(\delta)).$ Adhering to the same notations as above, we define the following:
\begin{definition}\label{Definition:Quasicategorical Connection}
	For a Lie 2-group $\mb{G}$, a \textit{quasi connection} on a principal $\mb{G}$-bundle $\pi: \mb{E} \ra \mb{X}$ over a Lie groupoid $\mb{X}$ is defined as a smooth section $\mc{C}: s^{*}E_0 \ra E_1$ of the morphism of principal bundles $P : E_1 \ra s^{*}E_0$, such that $\mc{C}$ is a morphism of principal bundles over $X_1$ along the unit map $u \colon G_0 \ra G_1$. 
	\[
	\begin{tikzcd}
		E_1 \arrow[d, "\pi_1"'] \arrow[r, "P", bend left] & s^{*}E_0 \arrow[ld, "\pi_0^{*}", bend left] \arrow[l, "\mc{C}"', bend left=49] \\
		X_1                                          &                                                           
	\end{tikzcd}\]
	The pair $(\pi: \mb{E} \ra \mb{X}, \mc{C})$, will be called as a \textit{quasi-principal $\mb{G}$-bundle over $\mb{X}$}.
\end{definition}
\begin{remark}
	A notion analogous to a quasi connection in VB-groupoid setup has already been discussed in \Cref{Linear Cleavage}. A key distinctive characteristic of our setup is the Lie 2-group equivariance.
\end{remark}
The following observation is obvious:
\begin{proposition}\label{Proposition: Quasi-Cat}
	For a Lie 2-group $\mb{G}$, let $\pi \colon \mb{E} \ra \mb{X}$ be a principal $\mb{G}$-bundle over a Lie groupoid $\mb{X}$. Then, every categorical connection $\mc{C} \colon s^*E_0 \ra E_1$ is a quasi connection and conversely, any quasi connection $\mc{C} \colon s^{*}E_0 \ra E_1$, satisfying the following two properties
	\begin{itemize}
		\item[(i)] $\mathcal{C}(1_x,p)=1_p$  for any $x\in X_0$ and $p\in \pi^{-1}(x)$,
		\item[(ii)]  if $(\gamma_2, p_2), (\gamma_1, p_1) \in {s}^{*}E_0$ such that ${s}(\gamma_2)={t}(\gamma_1)$ and $p_2=t\bigl({\mathcal C}(\gamma_1, p_1)\bigr),$ then $\mathcal{C}(\gamma_2 \circ \gamma_1 , p_1)= \mathcal{C}(\gamma_2, p_2) \circ \mathcal{C}(\gamma_1, p_1)$,
	\end{itemize}
	is a categorical connection (\Cref{Def:categorical connection}).
\end{proposition}
\begin{definition}\label{Unital connection}
	A quasi connection $\mc{C} \colon s^{*}E_0 \ra E_1$ 
	is said to be a \textit{unital connection} if it satisfies the condition (i) of \Cref{Proposition: Quasi-Cat}. We will call a principal 2-bundle equipped with a unital connection a \textit{unital-principal 2-bundle}. Likewise, we will call a principal 2-bundle equipped with a categorical connection a \textit{categorical-principal 2-bundle}. \\ Notationally, we will not distinguish between quasi, unital, or categorical-principal 2-bundles. 
\end{definition}
\begin{remark}
	Note that a	quasi, unital, or a categorical-principal 2-bundle can be viewed as a suitable adaptation of a  cloven fibration on a fibered category (\Cref{Definition Colven fibration}) in the framework of principal Lie 2-group bundles over Lie groupoids. In \Cref{Chapter: Parallel transport on quasi-principal 2-bundles}, we will touch upon this aspect in a little detail. 
\end{remark}
Given a Lie 2-group $\mb{G}$ and a Lie groupoid $\mb{X}$, the collection of quasi-principal $\mb{G}$-bundles over $\mb{X}$ natutally defines a groupoid, as we see below:
\begin{proposition}\label{Groupoid of quasi principal 2-bundles}
	For a Lie 2-group $\mb{G}$ and a Lie groupoid $\mb{X}$, the category $\rm{Bun}_{\rm{quasi}}(\mb{X}, \mb{G})$, whose objects are quasi-principal $\mb{G}$-bundles $(\pi \colon \mb{E} \ra \mb{X}, \mc{C})$ over $\mb{X}$ and an arrow from $(\pi \colon \mb{E} \ra \mb{X}, \mc{C})$ to $(\pi' \colon \mb{E}' \ra \mb{X}, \mc{C}')$ is a morphism of principal $\mb{G}$-bundles $F \colon \mb{E} \ra \mb{E}'$ that satisfies the condition $$F_1 \big(\mc{C}(\gamma,p) \big)= \mc{C}' \big(\gamma, F_0(p) \big),$$ for all $(\gamma,p) \in s^{*}E_0$, forms a groupoid. In a similar way, the collection of unital principal $\mb{G}$-bundles and categorical principal $\mb{G}$-bundles over $\mb{X}$ forms the respective groupoids $\rm{Bun}_{\rm{unital}}(\mb{X}, \mb{G})$ and $\rm{Bun}_{\rm{Cat}}(\mb{X}, \mb{G})$.
\end{proposition}
With a similar intention, we propose a weaker version of \Cref{Definition: Principal Lie group bundle over a Lie groupoid}:
\begin{definition}\label{quasi-principal G-bun dle}
	For a Lie group $G$, a \textit{quasi-principal $G$-bundle over a Lie groupoid} $\mb{X}$ is defined as a principal $G$-bundle $\pi\colon E_G \rightarrow X_0$ equipped with a  smooth 
	map $\mu\colon s^{*}E_G \rightarrow E_G$ that satisfies the following conditions:
	\begin{enumerate}[(i)]
		\item for each $(\gamma, p) \in s^{*}E_G$, we have $\bigl(\gamma, \mu(\gamma,p)\bigr) \in X_1 \times_{t, X_0, \pi}E_G$,
		\item for all $p \in E_G, g \in G$ and $\gamma \in X_1$ we have $\mu(\gamma, p)g=\mu(\gamma, pg).$			
	\end{enumerate}
\end{definition}
The notation $\bigl(\pi\colon E_G \rightarrow X_0, \mu, \mb{X} \bigr)$ may either denote a quasi-principal $\mb{G}$-bundle or a principal $G$-bundle (\Cref{Definition: Principal Lie group bundle over a Lie groupoid}) and should be understood from the context.

Observe that the condition (i) in the definition above says that $\mu$ defines a left quasi-action of $\mb{X}$ on $E_G$ (see \Cref{Definition: quasi action on a manifold}), and the condition (ii) ensures that this quasi-action commutes with the right action of $G$ on $E$.

\begin{example}\label{underlying quasi-principal bundle}
	For any quasi-principal $\mb{G}$-bundle $(\pi \colon \mb{E} \ra \mb{X},\mc{C})$ over a Lie groupoid $\mb{X}$,  $(\pi_0 \colon E_0 \ra X_0, \mu_{\mc{C}}:=t \circ \mc{C}, \mb{X})$ is a quasi-principal $G_0$-bundle over $\mb{X}$, which we call the \textit{underlying quasi-principal $G_0$-bundle of the quasi-principal $\mb{G}$-bundle $\pi \colon \mb{E} \ra \mb{X}$.}
\end{example}
\begin{remark}\label{underlying principal G-bundle}
	Observe that if $\mc{C}$ is a categorical connection in \Cref{underlying quasi-principal bundle}, then  $(\pi_0 \colon E_0 \ra X_0, \mu_{\mc{C}}:=t \circ \mc{C}, \mb{X})$ is a principal $G$-bundle over $\mb{X}$, see \Cref{actionwithcatcon}.
\end{remark}

\subsection{Examples of quasi-principal 2-bundles}\label{Examples of quasiprincipal 2-bundles}
\Cref{Proposition: Quasi-Cat} says that any categorical principal 2-bundle is a quasi-principal 2-bundle. In this subsection, we construct some non-trivial examples of quasi-principal 2-bundles which fail to be categorical-principal 2-bundles.

\begin{lemma}\label{detailed QuasiExamples}
	For a Lie crossed module $(G,H,\tau, \alpha)$, let $(\pi \colon \mb{E} \ra \mb{X}, \mc{C})$ be a categorical principal $[H \rtimes_{\alpha} G \rra G]$-bundle over a Lie groupoid $\mb{X}$. If there exists a smooth map $\mc{H} \colon s^{*}E_0 \ra H$ satisfying $\alpha_{g}(\mc{H}(\gamma,pg))= \mc{H}(\gamma,p)$ for all $(\gamma,p) \in s^{*}E_0$ and $g \in G$, then for $\mc{C}_{\mc{H}}(\gamma,p):= \mc{C}(\gamma,p) \big( \mc{H}(\gamma,p) ,e\big)$, the pair $(\pi \colon \mb{E} \ra \mb{X},\mc{C}_{\mc{H}})$ defines a quasi-principal $[H \rtimes_{\alpha} G \rra G]$-bundle over $\mb{X}$. Furthermore, $\mc{C}_{\mc{H}}$ is a categorical connection if and only if the following two conditions are satisfied:
	\begin{enumerate}[(i)]
		\item $\mc{H}\big(1_{\pi(p)},p \big)=e_H$ for all $p \in E_0$ and 
		\item $\mc{H}(\gamma_2 \circ \gamma_1,p)= \mc{H}(\gamma_2,t(\mc{C}(\gamma_1,p))) \mc{H}(\gamma_1,p_1)$ for all $\gamma_2,\gamma_1 \in X_1$, such that $s(\gamma_2)=t(\gamma_1)$ and $(\gamma_1,p) \in s^{*}E_0$.
	\end{enumerate}
\end{lemma}
\begin{proof}
	Observe that  for any $(\gamma,p) \in s^{*}E_0$, we have
	\begin{equation}\nonumber
		\Bigg(\pi \Big(\mc{C}(\gamma,p) \big( \mc{H}(\gamma,p) ,e\big), s \Big( \mc{C}(\gamma,p) \big( \mc{H}(\gamma,p) ,e\big) \Big) \Big) \Bigg)= (\gamma,p).
	\end{equation}
	Now, for $(\gamma,p) \in s^{*}E_0$ and $g \in G$, 
	\begin{equation}\nonumber
		\begin{split}
			& \mc{C}_{\mc{H}}(\gamma,pg)\\
			&= \mc{C}(\gamma,pg) \big( \mc{H}(\gamma,pg) ,e\big)\\
			&= \mc{C}(\gamma,p)(e_H,g) \big( \mc{H}(\gamma,pg) ,e\big)\\
			&=\mc{C}(\gamma,p)\Big(\alpha_{g}\big(\mc{H}(\gamma,pg)  \big),g\Big).
		\end{split}
	\end{equation}
	On the other hand, consider
	\begin{equation}\nonumber
		\begin{split}
			& \mc{C}_{\mc{H}}(\gamma,p)(e_H,g)\\
			&=\mc{C}(\gamma,p) \big( \mc{H}(\gamma,p) ,e\big) (e_H,g)\\
			&= \mc{C}(\gamma,p) \Big( \mc{H}(\gamma,p),g\Big).	
		\end{split}
	\end{equation}
	But since $\alpha_{g}(\mc{H}(\gamma,pg))= \mc{H}(\gamma,p)$ , we have 
	\begin{equation}\nonumber
		\mc{C}_{\mc{H}}(\gamma,pg)=\mc{C}_{\mc{H}}(\gamma,p)(e_H,g).
	\end{equation}
	Hence, $(\pi \colon \mb{E} \ra \mb{X}, \mc{C}_{\mc{H}})$ is a quasi-principal $[H \rtimes_{\alpha} G \rra G]$-bundle over $\mb{X}$. 
	
	For $p \in E_0$, consider,
	\begin{equation}\label{16}
		\begin{split}
			& \mc{C}_{\mc{H}}(1_{\pi(p)},p)\\
			&=\mc{C}(1_{\pi(p)},p) \big(\mc{H}(1_{\pi(p)},p)\,,e \big)\\
			&=1_p\big(\mc{H}(1_{\pi(p)},p)\,,e \big).
		\end{split}
	\end{equation}
	Let $\gamma_2,\gamma_1 \in X_1$, such that $s(\gamma_2)=t(\gamma_1)$ and $(\gamma_1,p) \in s^{*}E_0$. Consider,
	\begin{equation}\label{17}
		\begin{split}
			& \mc{C}_{\mc{H}}\Big(\gamma_2, t \big(\mc{C}_{\mc{H}}(\gamma_1,p) \big) \Big) \circ \mc{C}_{\mc{H}}(\gamma_1,p)\\			
			&= \mc{C}_{\mc{H}}\Big(\gamma_2, t \big(\mc{C}(\gamma_1,p) \big( \mc{H}(\gamma_1,p),e) \big)  \Big) \circ \mc{C}_{\mc{H}}(\gamma_1,p)\\
			&= \mc{C}_{\mc{H}}\Big(\gamma_2, t \big(\mc{C}(\gamma_1,p) \Big) \big(e,\tau(\mc{H}(\gamma_1,p)\big) \circ \mc{C}_{\mc{H}}(\gamma_1,p)\\
			&= \mc{C}\Big(\gamma_2, t \big(\mc{C}(\gamma_1,p) \Big) \Big(\mc{H}\big( \gamma_2,t(\mc{C}(\gamma_1,p_1)) \big), \tau(\mc{H}(\gamma_1,p)) \Big) \circ \mc{C}(\gamma_1,p) \big(\mc{H}(\gamma_1,p),e \big)\\
			&= \underbrace{\mc{C}(\gamma_2 \circ \gamma_1,p) \Big(\mc{H}\big( \gamma_2,t(\mc{C}(\gamma_1,p_1)) \big) \mc{H}(\gamma_1,p), e \Big)}_{[\text{by}\,\, \Cref{E:Identitiesechangeinverse} \,\, \text{and}\,\, \textit{the}\,\, \textit{condition}\,\, (v) \,\,\textit{in}\,\, \Cref{Def:categorical connection}]}
		\end{split}
	\end{equation}
	From \Cref{16} and \Cref{17} it is evident that $\mc{C}_{\mc{H}}$ is a categorical connection if and only if we have
	\begin{enumerate}[(i)]
		\item $\mc{H}\big(1_{\pi(p)},p \big)=e_H$ for all $p \in E_0$ and 
		\item $\mc{H}(\gamma_2 \circ \gamma_1,p)= \mc{H}(\gamma_2,t(\mc{C}(\gamma_1,p))) \mc{H}(\gamma_1,p_1)$ for all $\gamma_2,\gamma_1 \in X_1$, such that $s(\gamma_2)=t(\gamma_1)$ and $(\gamma_1,p) \in s^{*}E_0$.
	\end{enumerate}
\end{proof}
Note that for conciseness, only an outline of the above proof was provided in our paper \cite{chatterjee2023parallel}. Here, we have given its detailed version.

With the help of \Cref{detailed QuasiExamples}, next, we proceed to construct some concrete examples of quasi-principal 2-bundles, which are not categorical-principal 2-bundles. 
\begin{example}\label{concrete family of example}
	For a Lie crossed module $(G, H, \tau, \alpha)$, suppose $\pi^{\rm{dec}} \colon \mb{E}^{\rm{dec}} \ra \mb{X}$ is a decorated principal $[H \rtimes_{\alpha}G \rra G ]$-bundle over a Lie groupoid $\mb{X}$, constructed from a Lie crossed module  $(G,H,\tau,\alpha)$ and a principal $G$-bundle $(\pi_G \colon E_G \ra X_0, \mu, \mb{X})$. Let us assume that there exists a non-identity element $h$ in $H$ satisfying $\alpha(g)(h)=h$ for all $g \in G$. Now, define a map $\mc{H} \colon  s^{*}E_0 \ra H$ as $(\gamma,p) \mapsto h$ for all $(\gamma,p) \in s^{*}E_0$. As the assignment $(\gamma,p) \mapsto (\gamma,p,e)$ for all $(\gamma,p) \in s^{*}E_G$ defines a categorical connection on $\pi^{\rm{dec}} \colon \mb{E}^{\rm{dec}} \ra \mb{X}$, it immediately follows from \Cref{detailed QuasiExamples} that  $\mc{C}_h \colon s^{*}E_G \ra s^{*}E_G \times H$, $(\gamma,p) \mapsto (\gamma,p, e)(h,e)$ defines a quasi connection on $\pi^{\rm{dec}} \colon \mb{E}^{\rm{dec}} \ra \mb{X}$. Since $h \neq e$, $\mc{C}_h$ is not a categorical connection.
\end{example}
As a particular case of \Cref{concrete family of example}, we get the following example:
\begin{example}\label{Hquasi}
	Let $\mb{X}$ be a Lie groupoid. Note that the identity map $ {\rm{id}} \colon X_0 \ra X_0$ defines a principal $\lbrace e \rbrace$-bundle over $X_0$ under the natural action of the trivial Lie group $\lbrace e \rbrace $. Now, the following smooth map
	\begin{equation}\nonumber
		\begin{split}
			& \mu \colon s^{*} X_0 \ra X_0\\
			& (\gamma,p) \mapsto t(\gamma).
		\end{split}
	\end{equation}
	defines a principal $\lbrace e \rbrace$-bundle $({\rm{id}} \colon X_0 \ra X_0, \mu, \mb{X})$ over $\mb{X}$. Fix an abelian Lie group $H \neq \lbrace e \rbrace$. Consider the decorated principal $[H \rra \lbrace e \rbrace]$-bundle over $\mb{X}$ (see \Cref{E:Example of principal 2-bundle ordinary}), constructed from the Lie crossed module $(\lbrace e \rbrace, H, \tau ,\alpha)$ (where $\tau$ is trivial and $\alpha$ is ${\rm{id}}_{H}$) and the principal  $\lbrace e \rbrace$-bundle $({\rm{id}} \colon X_0 \ra X_0, \mu, \mb{X})$. As $H$ is not trivial and $\alpha$ is ${\rm{id}}_H$, it follows from \Cref{concrete family of example} that for any non-identity $h$ in $H$, the map $\mc{C}_{h} \colon s^{*}X_0 \ra X_0$ defined by $(\gamma,p) \mapsto (\gamma,p,e)(h,e)$ is a quasi connection, but not a categorical connection.
	
\end{example}

\begin{example}\label{qusi bundle over discrete groupoid}
	For a Lie crossed module $(G, H, \tau, \alpha)$, consider a principal $[H \rtimes_{\alpha}G \rra G]$-bundle $\pi \colon \mb{E} \ra [M \rra M]$ over a discrete Lie groupoid $[M \rra M]$, such that there exists $h \in H$, $h \neq e$ and $\alpha(g)(h)=h$ for all $g \in G$. Then \Cref{detailed QuasiExamples} implies that the map $\mc{C}_{h} \colon s^{*}E_0 \ra E_1, (1_x,p) \mapsto 1_p(h,e)$ defines a quasi connection, which is not a categorical connection. Hence, contrary to a unique categorical connection (see \Cref{unique Cat connection on bundle over discrete space}), it may admit many quasi connections. 	
\end{example}
	\subsection{A Lie 2-group torsor version of the Grothendieck construction}\label{Section Quasi-principal 2-bundle over a Lie groupoid as a Grothendieck construction}
In this subsection, we will obtain the first main result (\Cref{Main Theorem 1}) in this thesis. For that, we start by observing some properties of the underlying quasi-principal Lie group bundle (\Cref{underlying quasi-principal bundle}) of a quasi-principal Lie 2-group bundle.

\begin{proposition}\label{Lemma: cohenrence of canonical quasi action}
	For a Lie crossed module $(G, H, \tau, \alpha)$, let $(\pi \colon  \mb{E} \ra \mb{X}, \mc{C})$ be a quasi-principal $[H \rtimes_{\alpha}G \rra G]$-bundle over a Lie groupoid $\mb{X}$. Consider the underlying quasi-principal $G$-bundle $(\pi_0 \colon E_0 \ra X_0, \mu_{\mc{C}}:=t \circ \mc{C}, \mb{X})$ over $\mb{X}$.
	Then there exist smooth maps $\mc{H}_{u,\mc{C}} \colon E_0 \ra H$ and $\mc{H}_{m ,\mc{C}} \colon X_1 \times_{s,X_0,t} X_1 \ra H$ which satisfy the following properties:
	\begin{itemize}
		\item[(a)] $\mu_{\mc{C}} \big(1_{\pi(p)},p \big)=p \tau(\mc{H}_{u,\mc{C}}(p))$ for all $p \in E_0$.
		\item[(b)] $\mu_{\mc{C}}\big(\gamma_2, \mu_{\mc{C}}(\gamma_1,p)\big)= \mu_{\mc{C}}(\gamma_2 \circ \gamma_1,p) \tau \big(\mc{H}_{m,\mc{C}}(\gamma_2, \gamma_1) \big)$ for all appropriate $\gamma_2, \gamma_1 \in X_1, p \in E_0$.
		\item[(c)] [\textit{Right unitor}] $\mc{H}_{m,\mc{C}}(\gamma, 1_{\pi(p)})=  \mc{H}_{u,\mc{C}}(p)$ for all $\gamma \in X_1$ such that $s(\gamma)= \pi(p)$.
		\item[(d)][\textit{Left unitor}] $\mc{H}_{m,\mc{C}} \big( 1_{\pi(\mu_{\mc{C}}(\gamma,p))},\gamma \big)=  \mc{H}_{u, \mc{C}} \big( \mu_{\mc{C}}(\gamma,p) \big)$ for $(\gamma,p) \in s^{*}E_0$.
		\item[(e)] $\mc{H}_{u,\mc{C}}$ is $G$ invariant.
		\item[(f)] $\alpha_{g^{-1}} \big(\mc{H}_{u,\mc{C}}(p) \big)=\mc{H}_{u,\mc{C}}(p)$ for all $g \in G$ and $p \in E_0$.
		\item[(g)] $\mc{H}_{u,\mc{C}}(p) \in Z(H)$ for all $p \in E_0$, where $Z(H)$ is the centre of $H$.
		\item[(h)] $\alpha_{g^{-1}}(\mc{H}_{m,\mc{C}}^{-1}(\gamma_2, \gamma_1))= \mc{H}_{m,\mc{C}}^{-1}(\gamma_2, \gamma_1)$ for all composable $\gamma_2, \gamma_1 \in X_1$.
		\item[(i)]  $\mc{H}_{m,\mc{C}}(\gamma_2, \gamma_1) \in Z(H)$ for all $\gamma_2, \gamma_1 \in X_1 \times_{s,X_0,t} X_1$.
		\item[(j)] [\textit{Associator}] For $\gamma_3, \gamma_2, \gamma_1 \in X_1$ such that $s(\gamma_3)= t(\gamma_2)$ and $s(\gamma_2)= t(\gamma_1)$, we have 
		$$\mc{H}_{m,\mc{C}}^{-1}(\gamma_3,\gamma_2) \mc{H}_{m,\mc{C}}^{-1}(\gamma_3 \circ \gamma_2, \gamma_1)= \mc{H}_{m,\mc{C}}^{-1}(\gamma_2, \gamma_1) \mc{H}_{m,\mc{C}}^{-1}(\gamma_3, \gamma_2 \circ \gamma_1).$$
		\item[(k)][\textit{Invertor}] If $(\gamma,p) \in s^{*}E_0$, then we have
		$$\mc{H}_{m,\mc{C}}({\gamma^{-1}, \gamma}) \mc{H}_{m,\mc{C}}(\gamma,\gamma^{-1})^{-1}= \mc{H}_{u,\mc{C}}(p)^{-1}\mc{H}_{u,\mc{C}}(\mu_{\mc{C}}(\gamma,p)).$$
	\end{itemize}
\end{proposition}
\begin{proof}
	Let us define $\mc{H}_{u,\mc{C}} \colon E_0 \ra H$  as $p \mapsto h_p$ and $\mc{H}_{m,\mc{C}} \colon X_1 \times_{s,X_0,t} X_1 \ra H$ as $(\gamma_2, \gamma_1) \mapsto h_{\gamma_2, \gamma_1}$, where $h_p$ and  $h_{\gamma_2, \gamma_1}$  are respectively unique elements in $H$ which satisfy
	\begin{equation}\label{Equation: Construction 1}
		\mc{C}(1_{\pi(p)},p) = 1_p (h_p,e)
	\end{equation}
	for $p \in E_0$, and
	\begin{equation}\label{Equation: Construction 2}
		\mc{C} \big(\gamma_2 , \mu_{\mc{C}}(\gamma_1,p) \big) \circ \mc{C}(\gamma_1,p)= \mc{C}(\gamma_2 \circ \gamma_1, p) (h_{\gamma_2, \gamma_1}, e) 
	\end{equation}
	for composable $\gamma_2, \gamma_1$. Note that the uniqueness of $h_p$ and $h_{\gamma_2,\gamma_1}$ follow from the freeness of the action of $[H \rtimes_{\alpha}G \rra G]$ on $E_1$, as the source of $\mc{C}\big( 1_{\pi(p)}, p \big)$, $1_p$, $\mc{C}(\gamma_1,p)$ and $\mc{C}(\gamma_2 \circ \gamma_1, p)$ are all equal to $p$.

	\subsubsection*{Proof of (a) and (b):} 
	Properties (a) and (b) follow immediately by applying target map $t$ on both sides of \Cref{Equation: Construction 1} and \Cref{Equation: Construction 2} respectively.

	\subsubsection*{Proof of (c):}
	
	To prove (c), observe that from \Cref{Equation: Construction 2}, immediately we obtain $$\mc{C}(\gamma,p) \big( h_{\gamma, 1_{\pi(p)}}, e \big)=\mc{C} \big( \gamma, \mu_{\mc{C}}(1_{\pi(p)},p) \big) \circ \mc{C} \big(1_{\pi}(p),p \big)$$ for all $(\gamma,p) \in s^{*}E_0$. Then (c) follows straightforwardly from (a), \Cref{Equation: Construction 1} and the freeness of the action of $H \rtimes_{\alpha}G$ on $E_1$. 
	
	\subsubsection*{Proof of (d):}
	(d) can be proven using similar techniques as in (c), as we see in the calculation below:
	
	For $(\gamma,p) \in s^{*}E_0$, we have
	\begin{equation}\nonumber
		\begin{split}
			& \mc{C}(\gamma,p) \big( h_{1_{\pi(\mu_{\mc{C}}(\gamma,p))}, \gamma}, e \big)\\
			&=  \mc{C} \big(1_{\pi(\mu_{\mc{C}}(\gamma,p))}, \mu_{\mc{C}}(\gamma,p) \big) \circ \mc{C}(\gamma,p)  \\
			&= 1_{\mu_{\mc{C}}(\gamma,p)}(h_{\mu_{\mc{C}}(\gamma,p)},e) \circ \mc{C}(\gamma,p)(e, e) \\
			&= \mc{C}(\gamma,p)(h_{\mu_{\mc{C}}(\gamma,p)},e).
		\end{split}
	\end{equation}
	Hence, $\mc{H}_{m,\mc{C}}(1_{\pi(\mu_{\mc{C}}(\gamma,p))},\gamma)=  \mc{H}_{u,\mc{C}}(\mu_{\mc{C}}(\gamma,p))$.
	
	\subsubsection*{Proof of (e):}
	Follows from $(c)$ i.e $\mc{H}_{u,\mc{C}}(pg)= \mc{H}_{m,\mc{C}}(1_{\pi(pg)}, 1_{\pi(pg)})= \mc{H}_{m,\mc{C}}(1_{\pi(p)}, 1_{\pi(p)})= \mc{H}_{u,\mc{C}}(p)$ for all $p \in E_0, g \in G$.
	
	\subsubsection*{Proof of (f):}
	For $p \in E_0$ and $g \in G$, we have $\mc{C}(1_{\pi(pg^{-1})},pg^{-1})= 1_{pg^{-1}}(h_{pg^{-1}},e)$. Hence, using (e), we have
	\begin{equation}\nonumber
		1p(h_p,e) (e, g^{-1})=1_p(e, g^{-1})(h_p,e).
	\end{equation}
	Hence, $\alpha_{g^{-1}}(\mc{H}_{u,\mc{C}}(p))=\mc{H}_{u,\mc{C}}(p)$.

	\subsubsection*{Proof of (g):}
	From Peiffer identity in \Cref{E:Peiffer}, we have $\alpha_{\tau(h)}(h_ph)=hh_p$ for all $h \in H, p \in E_0$. Hence using (f), we get $h_ph=hh_p$.

	\subsubsection*{Proof of (h):}
	For composable $\gamma_2, \gamma_1 \in X_1$, $g \in G$ and  $(\gamma_1,p) \in s^{*}E_0$, 
	we have the following:
	$$\mc{C}(\gamma_2 , \mu_{\mc{C}} \big(\gamma_1,pg^{-1}) \big) \circ \mc{C}(\gamma_1,pg^{-1})= \mc{C}(\gamma_2 \circ \gamma_1, pg^{-1}) (h_{\gamma_2, \gamma_1}, e).$$ Then it is easy to see that (h) follows from \Cref{Equation: Construction 2}. 
	\subsection*{Proof of (i):}
	(i) follows from the Peiffer identity \Cref{E:Peiffer} and (h).

	\subsubsection*{Proof of  (j):}
	To prove (j), consider a sequence of composable morphisms $\gamma_3, \gamma_2, \gamma_1 \in X_1$,  such that $(\gamma_1,p) \in s^{*}E_0$. Then, we  have
	\begin{equation}\nonumber
		\begin{split}
			& \mc{C}(\gamma_3 \circ \gamma_2 \circ \gamma_1,p)(h_{\gamma_3 \circ \gamma_2, \gamma_1},e) \\
			&= \mc{C}(\gamma_3 \circ \gamma_2, \mu_{\mc{C}}(\gamma_1,p)) \circ \mc{C}(\gamma_1,p) [{\textit{by \Cref{Equation: Construction 2}}}].  \\
			&= \underbrace{\Big( \mc{C}(\gamma_3, \mu_{\mc{C}}(\gamma_2 \circ \gamma_1,p)\tau(h_{\gamma_2,\gamma_1})\Big) \circ \mc{C} \big(\gamma_2,\mu_{\mc{C}}(\gamma_1,p) \big)(h^{-1}_{\gamma_3,\gamma_2},e)}_{\mc{C}(\gamma_3 \circ \gamma_2, \mu_{\mc{C}}(\gamma_1,p))[{\textit{by \Cref{Equation: Construction 2}}}]} \circ \big( \mc{C}(\gamma_1,p)(e,e) \big) \\
			&= \Big( \mc{C}(\gamma_3, \mu_{\mc{C}}(\gamma_2 \circ \gamma_1,p)\tau(h_{\gamma_2,\gamma_1}) \Big) \circ \underbrace{ \Big( \mc{C}\big(\gamma_2,\mu_{\mc{C}}(\gamma_1,p) \big) \circ \mc{C}(\gamma_1,p) \Big) (h^{-1}_{\gamma_3,\gamma_2},e)}_{[\textit{using} \,\, \Cref{E:Identitiesechangeinverse}]}  \\
			&= \underbrace{\big( \mc{C}(\gamma_3, \mu_{\mc{C}}(\gamma_2 \circ \gamma_1,p)) 1_{\tau(h_{\gamma_2,\gamma_1})}}_{[{\textit{by (iii), \Cref{Def:categorical connection}}}]} \circ \underbrace{\mc{C}(\gamma_2 \circ \gamma_1,p)(h_{\gamma_2,\gamma_1},e) \circ \mc{C}(\gamma_1,p)^{-1}}_{\mc{C}(\gamma_2,\mu_{\mc{C}}(\gamma_1,p))[{\textit{by \Cref{Equation: Construction 2}}}]}  \circ \mc{C}(\gamma_1,p) \big) (h^{-1}_{\gamma_3,\gamma_2},e) \\
			&= \big( \mc{C}(\gamma_3, \mu_{\mc{C}}(\gamma_2 \circ \gamma_1,p)) 1_{\tau(h_{\gamma_2,\gamma_1})} \circ \mc{C}(\gamma_2 \circ \gamma_1,p)(h_{\gamma_2,\gamma_1},e) \big)(h^{-1}_{\gamma_3,\gamma_2},e)  \\ 
			&= \underbrace{\Big(\mc{C}\big(\gamma_3, \mu_{\mc{C}}(\gamma_2 \circ \gamma_1,p) \big) \circ \mc{C}(\gamma_2 \circ \gamma_1,p)(h_{\gamma_2,\gamma_1},e) \Big)}_{[\textit{using} \,\, \Cref{E:Identitiesechangeinverse}]}(h^{-1}_{\gamma_3,\gamma_2},e)  \\
			&= \underbrace{\bigg( \mc{C}(\gamma_3 \circ \gamma_2 \circ \gamma_1,p) (h_{\gamma_3,\gamma_2\circ \gamma_1 },e)\bigg)}_{[\textit{{by \Cref{Equation: Construction 2}}}]} (h_{\gamma_2,\gamma_1},e)(h^{-1}_{\gamma_3,\gamma_2},e),
		\end{split}
	\end{equation}
	which completes the proof of (j).
		
	\subsubsection*{Proof of (k):}
	To prove (k), observe that we have $\mc{C}\big(\gamma,\mu_{\mc{C}}(\gamma^{-1}, \mu_{\mc{C}}(\gamma, p)) \big)\circ \mc{C}(\gamma^{-1}, \mu_{\mc{C}}(\gamma, p))= 1_{\mu_{\mc{C}}(\gamma,p)}(h_{\mu_{\mc{C}}(\gamma,p)}h_{\gamma,\gamma^{-1}},e)$ for $(\gamma,p) \in s^{*}E_0$. On the other hand, consider 
	\begin{equation}\label{Equation inverse 2}\nonumber
		\begin{split}
			& \mc{C}\Big(\gamma,\mu_{\mc{C}}\big(\gamma^{-1}, \mu_{\mc{C}}(\gamma,p) \big) \Big) \circ \mc{C}\big(\gamma^{-1}, \mu_{\mc{C}}(\gamma,p)\big)\\
			&= \mc{C} \big( \gamma, \underbrace{\mu_{\mc{C}}(1_{\pi(p)},p)\tau(h_{\gamma^{-1},\gamma})}_{[{\textit{by (b), \Cref{Lemma: cohenrence of canonical quasi action}}}]} \big) \circ \mc{C}\big(\gamma^{-1}, \mu_{\mc{C}}(\gamma,p)\big) \\
			&= \underbrace{\mc{C} \big( \gamma, p.\tau(h_ph_{\gamma^{-1},\gamma}) \big)}_{[{\textit{by (a), \Cref{Lemma: cohenrence of canonical quasi action}}}]} \circ \mc{C}\big(\gamma^{-1}, \mu_{\mc{C}}(\gamma,p)\big)\\
			&= \underbrace{\mc{C}(\gamma,p)1_{\tau(h_ph_{\gamma^{-1},\gamma})}}_{[{\textit{by (iii), \Cref{Def:categorical connection}}}]} \circ \underbrace{\big(\mc{C}(1_{\pi(p)},p)(h_{\gamma^{-1},\gamma}, e) \circ \mc{C}(\gamma,p)^{-1} \big)}_{[\textit{by \Cref{Equation: Construction 2}}].}\\
			&= \mc{C}(\gamma,p)1_{\tau(h_ph_{\gamma^{-1},\gamma})} \circ \underbrace{\big(1_p(h_ph_{\gamma^{-1},\gamma},e) \big)}_{[{\textit{by \Cref{Equation: Construction 1}}}].} \circ \mc{C}(\gamma,p)^{-1}\\
			&= \underbrace{\big( \mc{C}(\gamma,p)(h_ph_{\gamma^{-1},\gamma},e) \big)}_{[\textit{using} \,\, \Cref{E:Identitiesechangeinverse}]}\circ  \mc{C}(\gamma,p)^{-1}\\
			&= 1_{(\mu_{\mc{C}}(\gamma,p))}(h_ph_{\gamma^{-1},\gamma},e)\,\, [\textit{using} \,\, \Cref{E:Identitiesechangeinverse}].
		\end{split}
	\end{equation}
	which concludes the proof of (k).
\end{proof}
We will call the \textbf{Properties (a)-(k) listed above} \textit{coherence properties}.

Although \Cref{Lemma: cohenrence of canonical quasi action} may look a little technical from the first look, it contains all the necessary vital ingredients to construct a quasi-principal 2-bundle over a Lie groupoid, by \textit{twisting} a decorated principal 2-bundle \Cref{Prop:Decoliegpd} in a suitable sense. To be more precise, we have the following significant result that played a crucial role in the proof of our first main result (\Cref{Main Theorem 1}).

\begin{proposition}\label{Theorem:quasi-bundle construction}
	For a Lie group $G$, let $(\pi \colon E_G \ra X_0, \mu, \mb{X})$ be a quasi-principal $G$-bundle over a Lie groupoid $\mb{X}$. Consider a Lie crossed module $(G,H, \tau, \alpha)$ and a pair of smooth maps $\mc{H}_u \colon E_G \ra H$ and $\mc{H}_{m} \colon X_1 \times_{s,X_0,t} X_1 \ra H$, satisfying the coherence properties in \Cref{Lemma: cohenrence of canonical quasi action}. Then we have the following:

	\begin{enumerate}
		\item	The manifolds $(s^{*}E_G)^{q-\rm{dec}}:= s^{*}E_G \times H $ and $E_G$ determines a Lie groupoid $[(s^{*}E_G)^{\rm{q-dec}} \rightrightarrows E_G]$, whose structure maps are defined as
		\begin{itemize}
			\item[(i)]  $s$: $(\gamma, p, h) \mapsto p$,
			\item[(ii)]  $t$: $(\gamma, p, h) \mapsto \mu(\gamma, p) \tau(h^{-1})$,
			\item[(iii)]  $m \colon \bigg((\gamma_2, p_2, h_2), (\gamma_1, p_1, h_1) \bigg) \mapsto \bigg( \gamma_2 \circ \gamma_1, p_1 ,h_2h_1\mc{H}^{-1}_m({\gamma_2,\gamma_1}) \bigg)$, 
			\item[(iv)]  $u : p \mapsto \big(1_{\pi(p)},p, \mc{H}_{u}(p) \big)$,
			\item[(v)]  $ \mathfrak{i} \colon \bigl(\gamma, p, h \bigr) \mapsto \bigg(\gamma^{-1}, \mu(\gamma,p)\tau(h^{-1}), \mc{H}_{u}(p)\mc{H}_m({\gamma^{-1}, \gamma})h^{-1}\biggr)$.
		\end{itemize}
		\item The Lie groupoid  $\mb{E}^{\rm{q-dec}}:=[(s^{*}E_G)^{\rm{q-dec}} \rightrightarrows E_G]$ defines a quasi-principal $[H \rtimes_{\alpha}G \rra G]$-bundle $\pi^{\rm{q-dec}}\colon \mb{E}^{\rm q-dec}\ra \mb{X}$, equipped with the quasi connection $$\mc{C}^{\rm{q-dec}} \colon s^{*}E_G \ra (s^{*}E_G)^{\rm{q-dec}}, \, \, (\gamma,p) \mapsto (\gamma,p, e).$$ The bundle projection $\pi^{\rm{q-dec}}$ and the action of $[H \rtimes_{\alpha}G \rra G]$ on $\mb{E}^{\rm{q-dec}}$, coincide with that of the decorated case (See \Cref{Prop:Decoliegpd}).
		\item The quasi connection $\mc{C}^{\rm{q-dec}}$ is a categorical connection if and only if the maps $\mc{H}_m$ and $\mc{H}_u$ are constant maps to $e$.
	\end{enumerate}		
	
\end{proposition}
\begin{proof}:	
	\subsubsection*{Proof of (1)}
	From \Cref{Prop:Decoliegpd}, it readily follows that the source and target maps are surjective submersions.
	Consider a pair of composable morphisms  $(\gamma_2,p_2,h_2),(\gamma_1,p_1,h_1)$. To show that the source is compatible with the composition, note that $$s((\gamma_2,p_2,h_2) \circ (\gamma_1,p_1,h_1))= s \big( \gamma_2 \circ \gamma_1, p_1, h_2h_1 \mc{H}^{-1}_m({\gamma_2,\gamma_1}) \big)= p_1= s(\gamma_1, p_1, h_1).$$ To check the target map consistency,  consider
	\begin{equation}\nonumber
		\begin{split}
			& t \big( \gamma_2 \circ \gamma_1, p_1, h_2h_1 \mc{H}^{-1}_m({\gamma_2,\gamma_1}) \big)\\
			&=  \mu(\gamma_2 \circ \gamma_1,p_1) \tau(H_m(\gamma_2,\gamma_1)) \tau(h_1^{-1}) \tau(h_2^{-1}) \\
			&= \underbrace{\mu(\gamma_2, \mu(\gamma_1,p_1))\tau(\mc{H}_m(\gamma_2,\gamma_1)^{-1})}_{ [\textit{from} \,\, (b), \,\, \Cref{Lemma: cohenrence of canonical quasi action}]} \tau(H_m(\gamma_2,\gamma_1)) \tau(h_1^{-1}) \tau(h_2^{-1}) \\
			&= \mu(\gamma_2, \mu(\gamma_1,p_1))\tau(h_1^{-1}) \tau(h_2^{-1})\\
			&= \mu(\gamma_2,p_2)\tau(h^{-1}_2)\\
			&= t(\gamma_2,p_2,h_2).
		\end{split}
	\end{equation}
	To make sense of the unit map, note that
	$$t(u(p))= t(1_{\pi(p)},p, \mc{H}_u(p))=\mu(1_{\pi(p)},p) \tau(\mc{H}_u(p)^{-1})=p \tau(\mc{H}_u(p))\tau(\mc{H}_u(p)^{-1})=p.$$ Then the compatibility of $u$ with the composition follows from the right and left unitor (conditions (c) and (d) in \Cref{Lemma: cohenrence of canonical quasi action}, respectively). In particular, we have $$\big(\gamma,p,h \big) \circ (1_{\pi(p)},p, \mc{H}_u(p))= \Big(\gamma, p, h \mc{H}_u(p) \mc{H}^{-1}_m(\gamma,1_{\pi(p)}) \Big).$$ Then, from the right unitor property (condition (c) in \Cref{Lemma: cohenrence of canonical quasi action}), we get $$\big((\gamma,p),h \big) \circ \big(1_{\pi(p)},p, \mc{H}_u(p) \big)= \big((\gamma,p),h \big).$$ Whereas, $G$-invariance of $\mc{H}_u$ (condition (e) in \Cref{Lemma: cohenrence of canonical quasi action} ) and the left unitor property show 
	\begin{equation}\nonumber
		\begin{split}
			& \Big( 1_{\pi \big(\mu(\gamma,p)\tau(h^{-1}) \big)}, \mu(\gamma,p)\tau(h^{-1}), \mc{H}_u \big( \mu(\gamma,p)\tau(h^{-1}) \big) \circ \big(\gamma,p,h \big) \Big)\\
			& = \big(\gamma,p, \mc{H}_u\big(\mu(\gamma,p)\big)h \mc{H}^{-1}_m(1_{\pi(\mu(\gamma,p)}, \gamma) \big)\\
			&= \big(\gamma,p,h \big).
		\end{split}
	\end{equation}
	To verify the associativity of composition, consider a sequence of composable morphisms $(\gamma_3,p_3,h_3), (\gamma_2,p_2,h_2), (\gamma_1,p_1,h_1) \in (s^{*}E_G)^{\rm{q-dec}}$. Now, we have
	\begin{equation}\nonumber
		\begin{split}
			& \Big( (\gamma_3,p_3,h_3) \circ(\gamma_2,p_2,h_2) \Big) \circ (\gamma_1,p_1,h_1)\\
			& = \big( \gamma_3 \circ \gamma_2, p_2, h_3h_2 \mc{H}_m(\gamma_3, \gamma_2)^{-1} \big) \circ(\gamma_1, p_1 ,h_1)\\
			&= \Big( \gamma_3 \circ \gamma_2 \circ \gamma_1, \, p_1, \, h_3h_2 \mc{H}_m(\gamma_3, \gamma_2)^{-1}h_1 \mc{H}_m^{-1}(\gamma_3 \circ  \gamma_2,\gamma_1) \Big).
		\end{split}
	\end{equation}
	On the other hand, 
	\begin{equation}\nonumber
		\begin{split}
			& (\gamma_3,p_3,h_3) \circ \big((\gamma_2,p_2,h_2)  \circ  (\gamma_1,p_1,h_1) \big)\\
			&= (\gamma_3,p_3,h_3) \circ \big( \gamma_2 \circ \gamma_1, p_1, h_2h_1 \mc{H}^{-1}_m(\gamma_2, \gamma_1) \big)\\
			&= \bigg( \gamma_3 \circ \gamma_2 \circ \gamma_1, \, p_1, \, h_3h_2h_1 \mc{H}^{-1}_m(\gamma_2, \gamma_1) \mc{H}^{-1}_m(\gamma_3,\gamma_2 \circ \gamma_1) \bigg).
		\end{split}
	\end{equation}
	Then the associativity follows from the associatior (condition (j)) and condition (i) in \Cref{Lemma: cohenrence of canonical quasi action}.

	The following straightforward calculation shows the compatibility of the inverse map $\mathfrak{i}$ with the target.
	\begin{equation}\nonumber
		\begin{split}
			&  t(\mathfrak{i}(\gamma, p,h)) \\
			&=  t \Big( \gamma^{-1}, \mu(\gamma,p)\tau(h^{-1}), \mc{H}_{u}(p)\mc{H}_m({\gamma^{-1}, \gamma})h^{-1}\Big)\\
			&= \mu \Big(\gamma^{-1},\mu(\gamma,p)\tau(h^{-1})\Big) \tau(h) \tau(\mc{H}^{-1}_m(\gamma^{-1},\gamma)) \tau(\mc{H}^{-1}_u(p))\\
			&= \mu(\gamma^{-1},\mu(\gamma,p))\tau(\mc{H}^{-1}_m(\gamma^{-1},\gamma)) \tau(\mc{H}^{-1}_u(p))\\
			&= \underbrace{\mu(1_{\pi(p)},p) \tau(\mc{H}_m(\gamma^{-1},\gamma))}_{[\textit{from}\,\, (b), \,\,\Cref{Lemma: cohenrence of canonical quasi action} ]}\tau(\mc{H}^{-1}_m(\gamma^{-1},\gamma)) \tau(\mc{H}^{-1}_u(p))\\
			&=p \,\, \,\,\, [\textit{from}\,\, (a), \,\,\Cref{Lemma: cohenrence of canonical quasi action}].
		\end{split}
	\end{equation}
	To verify $\mathfrak{i}$ is indeed the inverse, observe
	$$\biggl( \gamma^{-1}, \mu(\gamma,p)\tau(h^{-1}), \mc{H}_{u}(p)\mc{H}_m({\gamma^{-1}, \gamma})h^{-1}\biggr) \circ \big(\gamma,p,h \big)= \big(1_{\pi(p)},p, \mc{H}_u(p) \big)$$ and on the other hand,  the condition (k) in \Cref{Lemma: cohenrence of canonical quasi action} implies $\big(\gamma,p,h \big) \circ \mathfrak{i}(\gamma,p,h)= u\big({\pi(\mu(\gamma,p) \tau(h^{-1}))} \big)$. Finally, as its structure maps are smooth by definition, it follows $\mb{E}^{\rm{q-dec}}$ is indeed a Lie groupoid.
	
	\subsubsection*{Proof of (2)}
	As the action 
	\begin{equation}\nonumber
		\begin{split}
			\rho\colon &\mb{E}^{\rm q-dec}\times [H \rtimes_{\alpha}G \rra G] \ra \mb{E}^{\rm q-dec}\\
			&(p, g) \mapsto pg,\\
			& \bigl((\gamma, p, h), (h', g)\bigr)\mapsto \bigl(\gamma, p g, \alpha_{g^{-1}}(h'^{-1}\, h)\bigr),
		\end{split}
	\end{equation}
	coincides with the decorated case as defined in \Cref{E:Actionondeco}. To prove $\rho$ is a right action of $[H \rtimes_{\alpha}G \rra G]$ on $\mb{E}^{\rm q-dec}$, we only need to check the compatibility of $\rho$ with unit maps and the composition. However, these are immediate consequences of condition (f) and condition (h) in \Cref{Lemma: cohenrence of canonical quasi action} respectively, combined with the functoriality of the action in the decorated case (see \Cref{verification of functoriality of decorated action}). Moreover, as the bundle projection functor coincides with the decorated case as given in \Cref{E:Projondeco}, it directly follows that $\pi \colon \mb{E}^{\rm{q-dec}} \ra \mb{X}$ is a principal $[H \rtimes_{\alpha}G \rra G]$-bundle over $\mb{X}$.

	Now, since for each $p \in E_G$ we have
	\begin{equation}\label{F1}
		\mc{C}^{\rm{q-dec}}(1_{\pi(p)},p)= u(p)(\mc{H}_u(p),e),
	\end{equation}
	and for any composable pair of morphisms $\gamma_2, \gamma_1 \in X_1$ satisfying $(\gamma_1,p) \in s^{*}E_G$, we have
	\begin{equation}\label{E2}
		\mc{C}^{\rm{q- dec}}(\gamma_2,t(\mc{C}(\gamma_1,p))) \circ \mc{C}^{\rm{q- dec}}(\gamma_1,p)= \mc{C}^{\rm{q- dec}}(\gamma_2 \circ \gamma_1,p)(\mc{H}_m(\gamma_2,\gamma_1),e),
	\end{equation}
	it follows that $\mc{C}^{\rm{q-dec}}$ defines a quasi connection.
	
	\subsubsection*{Proof of (3)}

	Direct consequences of \Cref{F1} and \Cref{E2}.
\end{proof}
We will call the quasi-principal $[H \rtimes_{\alpha}G \rra G]$-bundle $( \pi^{\rm{q-dec}} \colon \mb{E}^{\rm{q-dec}}\ra \mb{X}, \mc{C}^{\rm{q-dec}})$ in \Cref{Theorem:quasi-bundle construction} as the \textit{quasi-decorated principal $[H \rtimes_{\alpha}G \rra G]$-bundle over $\mb{X}$ associated to a pseudo-principal $(G,H, \tau, \alpha)$-bundle $(\pi \colon E_G\ra X_0,\mu, \mc{H}_u, \mc{H}_m,\mb{X})$}.
\begin{definition}\label{Definition:PseudoprincipalLiecrossedmodulebundle}
	Given a Lie crossed module $(G,H, \tau, \alpha)$, a \textit{pseudo-principal $(G,H, \tau, \alpha)$-bundle over a Lie groupoid $\mb{X}$} is defined as a quasi-principal $G$-bundle $(\pi_G \colon E_G \ra X_0, \mu, \mb{X})$ over the Lie groupoid $\mb{X}$ (\Cref{quasi-principal G-bun dle}), equipped with a pair of smooth maps $\mc{H}_{u} \colon E_0 \ra H$   and $\mc{H}_{m} \colon X_1 \times_{s,X_0,t} X_1 \ra H$, which satisfies the following coherence properties: 
	\begin{itemize}
		\item[(a)] $\mu(1_{\pi(p)},p)=p \tau(\mc{H}_u(p))$ for all $p \in E_G$.
		\item[(b)] $\mu(\gamma_2, \mu(\gamma_1,p))= \mu(\gamma_2 \circ \gamma_1,p) \tau(\mc{H}_m(\gamma_2, \gamma_1))$ for all appropriate $\gamma_2, \gamma_1 \in X_1, p \in E_G$. 
		\item[(c)] [\textit{Right unitor}] $\mc{H}_m(\gamma, 1_{\pi(p)})=  \mc{H}_{u}(p)$ for all $\gamma \in X_1$ such that $s(\gamma)= \pi(p)$.
		\item[(d)][\textit{Left unitor}] $\mc{H}_m(1_{\pi(\mu(\gamma,p))},\gamma)=  \mc{H}_{u}(\mu(\gamma,p))$.
		\item[(e)] $\mc{H}_u$ is $G$ invariant.
		\item[(f)] $\alpha_{g^{-1}}(\mc{H}_u(p))=\mc{H}_u(p)$ for all $g \in G$ and $p \in E_G$.
		\item[(g)] $\mc{H}_u(p) \in Z(H)$ for all $p \in E_0$, where $Z(H)$ is the centre of $H$.
		\item[(h)] $\alpha_{g^{-1}}(\mc{H}_{m}^{-1}(\gamma_2, \gamma_1))= \mc{H}_{m}^{-1}(\gamma_2, \gamma_1)$ for composable $\gamma_2, \gamma_1$.
		\item[(i)]  $\mc{H}_m(\gamma_2, \gamma_1) \in Z(H)$ for all $\gamma_2, \gamma_1 \in X_1 \times_{s,X_0,t} X_1$.
		\item[(j)] [\textit{Associator}] For $\gamma_3, \gamma_2, \gamma_1 \in X_1$ such that $s(\gamma_3)= t(\gamma_2)$ and $s(\gamma_2)= t(\gamma_1)$, we have 
		$$\mc{H}_m^{-1}(\gamma_3,\gamma_2) \mc{H}_{m}^{-1}(\gamma_3 \circ \gamma_2, \gamma_1)= \mc{H}_{m}^{-1}(\gamma_2, \gamma_1) \mc{H}_{m}^{-1}(\gamma_3, \gamma_2 \circ \gamma_1).$$
		\item[(k)][\textit{Invertor}] If $(\gamma,p) \in s^{*}E_0$, then we have
		$\mc{H}_m({\gamma^{-1}, \gamma}) \mc{H}_{m}(\gamma,\gamma^{-1})^{-1}= H_{u}(p)^{-1}H_{u}(\mu(\gamma,p))$.
	\end{itemize}
\end{definition}
We use the notation  $(\pi \colon E_G\ra X_0,\mu, \mc{H}_u, \mc{H}_m,\mb{X})$ to denote a pseudo-principal $(G,H, \tau, \alpha)$-bundle, and call the smooth maps $\mc{H}_m$ and $\mc{H}_{u}$ as \textit{compositional deviation} and \textit{unital deviation} respectively. 

\begin{example}\label{Definition makes sense}
	A direct consequence of \Cref{Lemma: cohenrence of canonical quasi action} is that the underlying quasi-principal $G$-bundle $(\pi_0 \colon E_0 \ra X_0,  \mu_{\mc{C}}, \mb{X})$ (\Cref{underlying quasi-principal bundle})  of a quasi-principal-$[H \rtimes_{\alpha}G \rra G]$-bundle $(\pi \colon \mb{E} \ra \mb{X}, \mc{C})$, endowed with the pair of smooth maps $\mc{H}_{u,\mc{C}}$ and $\mc{H}_{m,\mc{C}}$ (as defined in \Cref{Lemma: cohenrence of canonical quasi action}) is a pseudo-principal $(G,H, \tau, \alpha)$-bundle over the Lie groupoid $\mb{X}$. \\ We call $(\pi_0 \colon E_0 \ra X_0,  \mu_{\mc{C}}, \mc{H}_{u,\mc{C}}, \mc{H}_{m,\mc{C}}, \mb{X})$ as the \textit{underlying  pseudo-principal $(G,H, \tau, \alpha)$-bundle of the quasi-principal $[H \rtimes_{\alpha}G \rra G]$-bundle $(\pi \colon \mb{E} \ra \mb{X}, \mc{C})$.} 
\end{example}
\begin{remark}
	Observe that in \Cref{Definition makes sense}, the unital deviation $\mc{H}_{u,\mc{C}}$ and compositional deviation $\mc{H}_{m,\mc{C}}$ combinedly evaluate the extent to which the quasi connection $\mc{C}$ differs from being a categorical connection.
\end{remark}
\begin{remark}
	One may view \Cref{Definition:PseudoprincipalLiecrossedmodulebundle} as a suitable adaptation of the idea presented in a usual pseudofunctor \Cref{pseudofunctor}.
\end{remark}

\begin{remark}
	In \Cref{Theorem:quasi-bundle construction}, the pair of smooth maps  $\mc{H}_{u}$ and $\mc{H}_m$ together precisely measure the amount by which a quasi-decorated principal $[H \rtimes_{\alpha}G \rra G]$-bundle deviates from being a decorated principal $[H \rtimes_{\alpha}G \rra G]$-bundle in \Cref{Prop:Decoliegpd}. 
\end{remark} 

The next corollary is an easy but important consequence of \Cref{Theorem:quasi-bundle construction}:		

\begin{corollary}\label{Corollary: Essential surjectivity}
	Every pseudo-principal $(G,H, \tau, \alpha)$-bundle $(\pi \colon E_G\ra X_0,\mu, \mc{H}_u, \mc{H}_m,\mb{X})$ is same as the underlying pseudo-principal-$(G,H, \tau, \alpha)$-bundle of the quasi-decorated $[H \rtimes_{\alpha}G \rra G]$-bundle $( \pi^{\rm{q-dec}} \colon \mb{E}^{\rm{q-dec}}\ra \mb{X}, \mc{C}^{\rm{q-dec}})$ associated to $(\pi \colon E_G\ra X_0,\mu, \mc{H}_u, \mc{H}_m,\mb{X})$. So, in particular we have $\mu_{\mc{C^{\rm{q-dec}}}}=\mu$, $\mc{H}_{u,\mc{C}^{\rm{q-dec}}}=\mc{H}_u$ and $\mc{H}_{m, \mc{C}^{\rm{q-dec}}}= \mc{H}_m$.
\end{corollary}
Given a Lie groupoid $\mb{X}$ and a Lie crossed module $(G, H, \tau, \alpha)$ the collection of pseudo-principal $(G, H, \tau, \alpha)$-bundles over $\mb{X}$ has a natural groupoid structure, as evident from the proposition below.

\begin{proposition}\label{Groupoid of pseudo crossed module principal bundle}
	For a Lie groupoid $\mb{X}$ and a Lie crossed module $(G,H, \tau, \alpha)$, there is a groupoid ${\rm{Pseudo}} \big(\mb{X}, (G,H,\tau, \alpha) \big)$, whose objects are pseudo-principal   $(G,H, \tau, \alpha)$-bundles over $\mb{X}$ and arrows are defined in the following way:
	
	For any pair of pseudo-principal $(G,H, \tau, \alpha)$-bundles $(\pi \colon E_G\ra X_0,\mu, \mc{H}_u, \mc{H}_m,\mb{X})$ and $(\pi' \colon E'_G\ra X_0,\mu', \mc{H}'_u, \mc{H}'_m,\mb{X})$,
	\begin{itemize}
		\item[(i)]if $\mc{H}_{m} \neq \mc{H}'_{m}$, then $${\rm{Hom}}\Big((\pi \colon E_G\ra X_0,\mu, \mc{H}_u, \mc{H}_m,\mb{X}),(\pi' \colon E'_G\ra X_0,\mu', \mc{H}'_u, \mc{H}'_m,\mb{X}) \Big) = \emptyset ,$$
		\item[(ii)] if $\mc{H}_{m}=\mc{H}'_{m}$, 
		
		then an element of $ {\rm{Hom}}\Big((\pi \colon E_G\ra X_0,\mu, \mc{H}_u, \mc{H}_m,\mb{X}),(\pi' \colon E'_G\ra X_0,\mu', \mc{H}'_u, \mc{H}'_m,\mb{X}) \Big)$ is defined as a morphism of principal $G$-bundles $f \colon E_G \ra E'_G$ over $X_0$ (\Cref{Definition: Morphism of principal G-bundles}), satisfying the following pair of conditions:
		\begin{itemize}
			\item[(a)] 
			\begin{equation}\label{fcompatible}
			f \big(\mu(\gamma,p) \big)= \mu' \big(\gamma,f(p) \big)
			\end{equation}
			and
			\item[(b)] 
			\begin{equation}\label{funitcompatible}
			\mc{H}_u= \mc{H}'_u \circ f.
			\end{equation}
		\end{itemize}
	\end{itemize}
\end{proposition}
At this point, we have equipped ourselves with the necessary machinery to state and prove the first main result of this thesis.
		
	\begin{theorem}\label{Main Theorem 1}
	Given a Lie crossed module $(G,H,\tau, \alpha)$ and a Lie groupoid $\mb{X}$, the groupoid $\rm{Bun}_{\rm{quasi}}(\mb{X}$, $[H \rtimes_{\alpha}G \rra G])$ is equivalent to the groupoid ${\rm{Pseudo}} \big(\mb{X}, (G,H,\tau, \alpha) \big)$.
\end{theorem}
\begin{proof}
	Define
	\begin{equation}\nonumber
		\begin{split}
			\mc{F}\colon &{\rm{Bun}}_{\rm{quasi}}(\mb{X}, [H \rtimes_{\alpha}G \rra G])\ra {\rm{Pseudo}} \big(\mb{X}, (G,H,\tau, \alpha) \big),\\
			& (\pi \colon \mb{E} \ra \mb{X}, \mc{C}) \mapsto (\pi_0 \colon E_0 \ra X_0,  \mu_{\mc{C}}, \mc{H}_{u,\mc{C}}, \mc{H}_{m,\mc{C}}, \mb{X}), \\
			&(F \colon \mb{E} \ra \mb{E}') \mapsto (F_0 \colon E_0 \ra E'_0).
		\end{split}
	\end{equation}
	We want to show that $\mc{F}$ is an essentially surjective, faithful, and full functor.

	\subsubsection*{Well-definedness of $\mc{F}$:}
 
	Let  $(\pi \colon \mb{E} \ra \mb{X}, \mc{C})$ be a quasi-principal $[H \rtimes_{\alpha}G \rra G])$-bundle over $\mb{X}$. Then the \Cref{Definition makes sense} implies that $(\pi_0 \colon E_0 \ra X_0,  \mu_{\mc{C}}, \mc{H}_{u,\mc{C}}, \mc{H}_{m,\mc{C}}, \mb{X})$ is indeed a pseudo-principal $(G,H, \tau, \alpha)$-bundle over $\mb{X}$.  Now, consider $F \in {\rm{Hom}}\Big((\pi \colon \mb{E} \ra \mb{X}, \mc{C}), (\pi' \colon \mb{E}' \ra \mb{X}), \mc{C}') \Big)$, then as an immediate consequence of the identities 
	$F_1\big(\mc{C}(\gamma,p) \big)= \mc{C}' \big(\gamma, F_0(p) \big)$ and  $F_1 \big( \mc{C}(1_{\pi(p)},p) \big)= 1_{F_0(p)} \Big(\mc{H}_{u,\mc{C}'}\big(F(p)\big),e \Big)$, we get a well-defined $\mc{F}(F)$. 
	
 \subsubsection*{Functoriality of $\mc{F}$:}
 
	A straightforward verification.
	
	\subsubsection*{Essential surjectivity of $\mc{F}$:}	 
	
	A direct consequence of \Cref{Corollary: Essential surjectivity}.

	\subsubsection*{Faithulness of $\mc{F}$:}
	
	Consider $F,\bar{F} \in {\rm{Hom}} \Big( (\pi \colon \mb{E} \ra \mb{X}, \mc{C}) ,(\pi' \colon \mb{E}' \ra \mb{X}, \mc{C}')  \Big)$. Let $\mc{F}(F)= \mc{F}(\bar{F})$, that is $F_0=\bar{F}_0$. Now, for  any $\delta \in E_1$, there is a unique $h_{\delta} \in H$, satisfying 
	\begin{equation}\label{hdelta}
		\delta= \mc{C}\big(\pi_1(\delta),s(\delta)\big)(h_{\delta},e).
	\end{equation}
	Then the equality $F_1(\delta)=\bar{F}_1(\delta)$ follows from the compatiblity condition  of $F$ and $\bar{F}$ with quasi connections $\mc{C}$ and $\mc{C}'$ (see \Cref{Groupoid of quasi principal 2-bundles}). Hence, $\mc{F}$ is a faithful functor.

	\subsubsection*{Fullness of $\mc{F}$:}

	Consider an element  $$f \in {\rm{Hom}}\Big((\pi_0 \colon E_0 \ra X_0,  \mu_{\mc{C}}, \mc{H}_{u,\mc{C}}, \mc{H}_{m,\mc{C}}, \mb{X}), (\pi'_0 \colon E'_0 \ra X_0,  \mu_{\mc{C}'}, \mc{H}_{u,\mc{C}'}, \mc{H}_{m,\mc{C}'}, \mb{X}) \Big),$$ for a pair of quasi-principal $[H \rtimes_{\alpha}G \rra G]$-bundles $(\pi \colon \mb{E} \ra \mb{X}, \mc{C})$ and $(\pi' \colon \mb{E}' \ra \mb{X}, \mc{C}')$ over $\mb{X}$, and define 
	\begin{equation}\label{Definition of F}
		\begin{split}
			&F \colon \mb{E} \ra \mb{E}'\\
			& p \mapsto f(p), \, \, p \in E_0,\\
			&\delta \mapsto \mc{C}'\Big(\pi_1(\delta), f(s(\delta)) \Big) \Big(h_{\delta},e \Big), \, \, \delta \in E_1,
		\end{split}
	\end{equation}
	where $h_{\delta}$ is the unique element in $H$, such that $\delta= \mc{C}\Big(\pi_1(\delta),s(\delta)\Big)(h_{\delta},e)$. We will show $F:=(F_1,F_0)$ is a morphism of quasi-principal $[H \rtimes_{\alpha}G \rra G])$-bundles over $\mb{X}$.

	Observe that $\pi'_0 \circ F_0= \pi'_0 \circ f=\pi_0$ and for any $\delta \in E_1$,  we have $$\pi'_1 \circ F_1(\delta)= \pi'_1 \Big(\mc{C}'\big(\pi_1(\delta), f(s(\delta)) \big)(h_{\delta},e) \Big)= \pi_1(\delta).$$ Now, as $f$ is $G$-equivariant,  $F_0$ is $G$-equivariant. To, verify the $H \rtimes_{\alpha}G$-equivaeriancy of $F_1$, consider an element $\delta \in E_1$ and an element $(h,g) \in H \rtimes_{\alpha} G$. Observe that $\pi'_0 \circ F_0= \pi'_0 \circ f=\pi_0$ and for any $\delta \in E_1$, $\pi'_1 \circ F_1(\delta)= \pi'_1 \Big(\mc{C}'\big(\pi_1(\delta), f(s(\delta)) \big)(h_{\delta},e) \Big)= \pi_1(\delta)$. $G$-equivariancy of $f$ induces the same for $F_0$. Observe that there exists a unique element $\bar{h} \in H$ such that $\mc{C}\Big(\pi_1(\delta),s(\delta)\Big)(h_{\delta},e)(h,g)= \mc{C} \Big(\pi_1(\delta), s(\delta) \Big)(e,g)(\bar{h},e)$ for $\delta \in E_1$ and $(h,g) \in H \rtimes_{\alpha} G$. Next, we see
	\begin{equation}\nonumber
		\begin{split}
			&F_1 \big( \delta(h,g) \big)\\
			&= F_1 \big( \mc{C}\big(\pi_1(\delta),s(\delta)\big)(h_{\delta},e)(h,g)  \big)\\
			&= \mc{C}' \big( \pi_1(\delta), f(s(\delta))\big)(e_H,g)(\bar{h},e) \\
			&=\mc{C}' \big( \pi_1(\delta), f(s(\delta))\big)(\bar{h},g),
		\end{split}
	\end{equation}
	From the observation $\bar{h}=h_{\delta}h$, we obtain
	$$F_1 \big( \delta(h,g) \big)= \mc{C}' \Big( \pi_1(\delta), f(s(\delta))\Big)(h_{\delta},e)(h,g)= F_1(\delta)(h,g).$$ Thus, $F_1$ is $H \rtimes_{\alpha}G$-equivariant. Note that from the definition itself, it follows that $$F_1\big(\mc{C}(\gamma,p) \big)= \mc{C}' \big(\pi_1(\mc{C}(\gamma,p)), f(s(\mc{C}(\gamma,p) \big) \big(h_{\mc{C}(\gamma,p)},e \big)= \mc{C}'\big(\gamma, F_0(p)\big),$$ for all $(\gamma,p) \in s^{*}E_0$. As both $F_0$ and $F_1$ are smooth by definition, in order to prove $\mc{F}$ is full, it suffices to show that $F \colon \mb{E} \ra \mb{E}'$ is a functor. Source map consistency is immediate. Let $\delta \in E_1$. Then, the following calculations establish the target consistency:
 \begin{equation}
 \begin{split}
    &F_0\big(t(\delta)\big)\\
    &= F_0 \Bigg(t\Big(\mc{C}\Big(\pi_1(\delta),s(\delta)\Big)(h_{\delta},e) \Big)\Bigg)\\
    &= F_0\Bigg(\mu_{C} \Big( \pi_1(\delta),s(\delta) \Big)\Bigg)\tau(h_{\delta})\\
    &=\mu_{\mc{C}'}\Big( \pi_1(\gamma), f(s(\gamma) \Big)\tau(h_{\delta})\,\, [\textit{using \Cref{fcompatible}}]\\
    &= t \Bigg( \mc{C}'\Big( \pi_1(\gamma), f(s(\gamma)\Big)(h_{\delta},e)  \Bigg)\\
    &= t\big(F_1(\delta)\big).
 \end{split}
 \end{equation}
  Now, for any $p \in E_0$, we have $F_1(1_p)= \mc{C}' \Big(1_{\pi'_0 \circ f(p)}, f(p) \Big) \Big((\mc{H}_{u,\mc{C}}(p))^{-1},e \Big)$. By suitably using \Cref{Equation: Construction 1} and \Cref{funitcompatible}, we get the compatibility of $\mc{F}$ with unit maps.
	
	Let $\delta_2, \delta_1 \in E_1$ such that $s(\delta_2)=t(\delta_1)$. 
	To verify that $F$ is compatible with the composition map, we need to workout the following lengthy but straightforward calculation:
	\begin{equation}\nonumber
		\begin{split}
			&F_1(\delta_2) \circ F_1(\delta_1)\\
			&= \underbrace{\mc{C}'\big(\pi_1(\delta_2), f(s(\delta_2)) \big)(h_{\delta_2},e)}_{[F_1(\delta_2)]} \circ \underbrace{\mc{C}'\big(\pi_1(\delta_1), f(s(\delta_1)) \big)(h_{\delta_1},e)}_{[F_1(\delta_1)]}\\
			&= \underbrace{\mc{C}'\big(\pi_1(\delta_2), f(\mu_{\mc{C}}(\pi_1(\delta_1),s(\delta_1)))\big)\big(e,\tau(h_{\delta_{1}})\big)}_{[\textit{using} \,\, \Cref{hdelta}]} (h_{\delta_2},e) \circ \mc{C}'\big(\pi_1(\delta_1), f(s(\delta_1)) \big)(h_{\delta_1},e) \\
			&= \underbrace{\mc{C}'\big(\pi_1(\delta_2), \mu_{\mc{C}'}(\pi_1(\delta_1),f(s(\delta_1)))\big)}_{[\textit{using} \,\, \Cref{fcompatible}]}(e,\tau(h_{\delta_{1}}))  (h_{\delta_2},e) \circ \mc{C}'\big(\pi_1(\delta_1), f(s(\delta_1)) \big)(h_{\delta_1},e)\\
			&= \mc{C}'\Big(\pi_1(\delta_2), \mu_{\mc{C}'}\big(\pi_1(\delta_1),f(s(\delta_1))\big)\Big)\underbrace{\big(\alpha_{\tau(h_{\delta_1})}(h_{\delta_2}), \tau(h_{\delta_1}) \big)}_{[\textit{using} \,\, \Cref{E:Peiffer}]} \circ \mc{C}'\big(\pi_1(\delta_1), f(s(\delta_1)) \big)(h_{\delta_1},e)\\
			&= \mc{C}' \Big( \pi_1(\delta_2 \circ \delta_1), f(s(\delta_{1}))\Big) \Big(\mc{H}_{m, \mc{C}'}(\gamma_2, \gamma_1)h_{\delta_1}h_{\delta_2},e \Big) \,\, [\textit{Using}\,\, \Cref{E:Identitiesechangeinverse} \textit{and}\,\, \Cref{Equation: Construction 2}]\\
			&= \mc{C}' \Big( \pi_1(\delta_2 \circ \delta_1), f(s(\delta_2 \circ \delta_{1}))\Big) \Big(\mc{H}_{m, \mc{C}}(\gamma_2, \gamma_1)h_{\delta_1}h_{\delta_2},e \Big)
		\end{split}
	\end{equation}
	Thus, \Cref{Definition of F} implies that in order to show $F_1(\delta_2) \circ F_1(\delta_1)=F(\delta_2 \circ \delta_1)$, it is suffcient to prove the follwoing:
	\begin{equation}\label{657}
		\mc{H}_{m, \mc{C}}(\gamma_2, \gamma_1)h_{\delta_1}h_{\delta_2}=h_{\delta_2 \circ \delta_1}.
	\end{equation}
	
	To prove the above identity, consider
	\begin{equation}\nonumber
		\begin{split}
			&\delta_2 \circ \delta_1\\
			&= \underbrace{\mc{C}\Big(\pi_1(\delta_2),s(\delta_2)\Big)(h_{\delta_2},e)}_{\delta_2} \circ \underbrace{\mc{C}\Big(\pi_1(\delta_1),s(\delta_1)\Big)(h_{\delta_1},e)}_{\delta_1} \\
			&= \underbrace{\mc{C}\Big(\pi_1(\delta_2), \underbrace{\mu_{\mc{C}} \big( \pi_1(\delta_1),s(\delta_1)}_{s(\delta_2)=t(\delta_1)} \big)\Big)\big(e, \tau(h_{\delta_1})\big)}_{[{\textit{by (iii), \Cref{Def:categorical connection}}}]}(h_{\delta_2},e) \circ \mc{C}\Big(\pi_1(\delta_1),s(\delta_1)\Big)(h_{\delta_1},e)\\
			&= \underbrace{ \bigg( \mc{C}\Big(\pi_1(\delta_2), \mu_{\mc{C}} \big( \pi_1(\delta_1),s(\delta_1) \big)\Big)\circ  \mc{C}\big(\pi_1(\delta_1),s(\delta_1)\big) \bigg) \bigg(\Big(\alpha_{\tau(h_{\delta_1})}(h_{\delta_2}),\tau(h_{\delta_1}) \Big) \circ(h_{\delta_1},e) \bigg)}_{[\textit{using} \,\, \Cref{E:Identitiesechangeinverse}]} \\
			&= \bigg( \mc{C}\Big(\pi_1(\delta_2), \mu_{\mc{C}} \big( \pi_1(\delta_1),s(\delta_1) \big)\Big)\circ  \mc{C}\big(\pi_1(\delta_1),s(\delta_1)\big) \bigg)  \bigg( \underbrace{\Big(h_{\delta_1}h_{\delta_2}h_{\delta_1}^{-1}, \tau(h_{\delta_1}) \Big)}_{[{\textit{by \Cref{E:Peiffer}}}]}  \circ \Big(h_{\delta_1}, e \Big) \bigg)\\
			&= \mc{C}\Big( (\pi_1(\delta_2 \circ \delta_1),s(\delta_2 \circ \delta_1) \big)\Big) \Big(\mc{H}_{m, \mc{C}}(\gamma_2, \gamma_1)h_{\delta_1}h_{\delta_2},e \Big) [{\textit{by \Cref{Equation: Construction 2}}}].
		\end{split}	
	\end{equation}
	Thus, $\mc{H}_{m, \mc{C}}(\gamma_2, \gamma_1)h_{\delta_1}h_{\delta_2}=h_{\delta_2 \circ \delta_1}$. 
	
	Hence, we proved that $\mc{F}$ is an equivalence of categories.
				
				\end{proof}
\begin{remark}
	One can regard \Cref{Main Theorem 1} as a  Lie 2-group torsor version of the classical correspondence between fibered categories with cleavage and pseudofunctors (\Cref{subsection Fibered categories}) where, in particular, the essential surjectivity of $\mc{F}$ in \Cref{Main Theorem 1}, is a Grothendieck construction in our framework.
\end{remark}
A crucial consequence of \Cref{Main Theorem 1} is a complete characterization of quasi-principal 2-bundles over Lie groupoids. In \cite{chatterjee2023parallel}, we skipped its proof for brevity. Here, we decided to include the proof.

\begin{corollary}\label{Detailed quasicharacterisation}
	For a Lie crossed module $(G,H, \tau, \alpha)$, any quasi-principal $[H \rtimes_{\alpha}G \rra G]$-bundle $(\pi \colon \mb{E} \ra \mb{X}, \mc{C})$ over a Lie groupoid $\mb{X}$ is canonically isomorphic to the quasi-decorated $[H \rtimes_{\alpha}G \rra G]$-bundle $(\pi^{\rm{q-dec}} \colon \mb{E}^{\rm{q-dec}} \ra \mb{X}, \mc{C}^{\rm{q-dec}})$ (\Cref{Theorem:quasi-bundle construction}) associated to the underlying pseudo-principal $(G,H,\tau,\alpha)$-bundle $(\pi_0 \colon E_0 \ra X_0,  \mu_{\mc{C}}, \mc{H}_{u,\mc{C}}, \mc{H}_{m,\mc{C}}, \mb{X})$ (\Cref{Definition makes sense}). The canonical  isomorphism is explicitly given as
	\begin{equation}\nonumber
		\begin{split}
			\theta_{\mb{E}} \colon & \mb{E}^{\rm{q-dec}} \ra \mb{E}\\
			& p \mapsto p, \, \, p \in E_0 \\
			&(\gamma,p,h ) \mapsto \mc{C}(\gamma,p)(h^{-1},e), \, \, (\gamma,p,h ) \in s^{*}E_0 \times H.
		\end{split}
	\end{equation}
	
\end{corollary}
\begin{proof}
	The existence of the required canonical natural isomorphism follows directly from the definition of an equivalence of categories. To show $\theta$ is a canonical natural isomorphism, observe that it is sufficient to prove the following:
	\begin{enumerate}[(i)]
		\item $\theta_{\mb{E}}$ is a morphism of quasi-principal $[H \rtimes_{\alpha}G \rra G ]$-bundles over $\mb{X}$ and
		\item the map $\eta \colon {\rm{Obj}(\rm{Bun}_{\rm{quasi}}}(\mb{X},[H \rtimes_{\alpha}G \rra G])) \ra {\rm{Mor}\big(\rm{Bun}_{\rm{quasi}}}(\mb{X},[H \rtimes_{\alpha}G \rra G])\big)$ given by 
		$(\pi \colon \mb{E} \ra \mb{X}, \mc{C}) \mapsto \theta_{\mb{E}}$, defines a natural transformation from the functor $\mc{G} \circ \mc{F}$ to the functor $\rm{id}_{\mb{E}}$, where $\mc{G}$ is the weak inverse of $\mc{F}$ in \Cref{Main Theorem 1}.
	\end{enumerate}
 
Proof of (i):

	By definition, $\pi_0 \circ \theta_{\mb{E}}(p)= \pi^{\rm{q-dec}}(p)$ and $\pi_1 \circ \theta_{\mb{E}}\big(\gamma,p,h \big)= \pi^{\rm{q-dec}}\big(\gamma,p,h \big)$ for all $p \in E_0$ and $\big((\gamma,p),h \big) \in s^{*}E_0 \times H$. Also $[H \rtimes_{\alpha}G \rra G ]$-equivariance of $\theta_{\mb{E}}$ is a straightforward verification. Moreover,  by definition, $\theta_{\mb{E}}$ is smooth at both object and morphism level. Thus, in order to show $\theta_{\mb{E}}$ is a moprhism of principal $[H \rtimes_{\alpha}G \rra G ]$-bundles over $\mb{X}$,  we just need to prove $\theta_{\mb{E}}$ is a functor. Source-target consistency of $\theta_{\mb{E}}$ is obvious. The compatibility with unit maps follows from the following observation
	$$\theta_{\mb{E}} \big((1_{\pi_0(p)},p), \mc{H}_{u,\mc{C}}(p) \big) =1_p \big(\mc{H}_{u,\mc{C}}(p),e \big)\big((\mc{H}_{u,\mc{C}}(p))^{-1},e \big)$$
	for each $p \in E_0$. 
	
	To show $\theta_{\mb{E}}$ is compatible with the composition, consider a composable pair $(\gamma_2,p_2,h_2\big)$, $(\gamma_1,p_1, h_1\big) \in s^{*}E_0 \times H$. Then, we have
	\begin{equation}\nonumber
		\begin{split}
			&\theta_{\mb{E}}\Big(\big(\gamma_2,\mu(\gamma_1,p_1)\tau(h_1^{-1}),h_2\Big) \circ \theta_{\mb{E}}\Big((\gamma_1,p_1),h_1\Big)  \\
			&= \Bigg( \underbrace{\mc{C}\Big(\gamma_2, \mu(\gamma_1,p_1)\Big) \big(e, \tau(h_1^{-1}) \big)}_{[{\textit{by (iii), \Cref{Def:categorical connection}}}]} (h_2^{-1},e) \Bigg) \circ \Bigg(\mc{C}(\gamma_1,p_1)(h_1^{-1},e) \Bigg)\\
			&=\underbrace{\Big(\mc{C}\big(\gamma_2, \mu(\gamma_1,p_1)\big)\circ \mc{C}(\gamma_1,p_1) \Big)(h_1^{-1}h_2^{-1},e)}_{[\textit{using} \,\, \Cref{E:Identitiesechangeinverse}]}\\
			&= \underbrace{\mc{C}(\gamma_2 \circ \gamma_1,p_1) \big(\mc{H}_{m,\mc{C}}(\gamma_2,
				\gamma_1)h_1^{-1}h_2^{-1},e \big)}_{\textit{by}\,\, [\Cref{Equation: Construction 2}]}\\
			&= \theta_{\mb{E}}\Big( \big( \gamma_2 \circ \gamma_1,p_1 \big),h_2h_1 \big( \mc{H}_{m,\mc{C}}(\gamma_2,\gamma_1)\big)^{-1} \Big)\\
			&= \theta_{\mb{E}}\Big(\big((\gamma_2,p_2),h_2\big) \circ \big((\gamma_1,p_1),h_1 \big)  \Big).
		\end{split}
	\end{equation}
	Thus, $\theta_{\mb{E}}$ is a functor. Note that it is obvious that $\theta_{\mb{E}}\big( \mc{C}^{\rm{q-dec}}(\gamma,p) \big)= \mc{C} \big(\gamma,\theta_{\mb{E}}(p) \big)$ for each $(\gamma,p) \in s^{*}E_0$. Hence, $\theta_{\mb{E}}$ is a morphism of quasi-$[H \rtimes_{\alpha}G \rra G ]$-bundles over $\mb{X}$.
 
	Proof of (ii):
 
	To show $\eta$ is a natural transformation, we need to prove that for any morphism $F \colon  \mb{E} \ra \mb{E}'$ of quasi-principal-$[H \rtimes_{\alpha}G \rra G ]$-bundles over $\mb{X}$, we have
	\begin{equation}\label{Canonical iso}
		F \circ \theta_{\mb{E}}= \theta_{\mb{E}'} \circ \big(\mc{G} \circ \mc{F}\big)(F)
	\end{equation}
	Note that for each $p \in E_0$, we have
	$F_0 \circ \theta_{\mb{E}'}(p)=F_0(p)= \big(\mc{G} \circ \mc{F}\big)(F)\big(p \big)= \theta_{\mb{E}'} \circ \big(\mc{G} \circ \mc{F}\big)(F)\big(p \big)$. Hence,  we get the equality of   \Cref{Canonical iso} at the object level. The equality at morphism level follows easily by observing $h_{\gamma,p,h}=h^{-1}$, where $h_{\gamma,p,h}$ is the unique element in $H$ such that $(\gamma,p, h )= \mc{C}^{\rm{q-dec}} \Big( \pi^{\rm{q-dec}}_1 (\gamma,p, h), s\big( \gamma,p, h \big) \Big)\Big(h_{(\gamma,p,h)},e \Big)$.
\end{proof}
An immediate consequence of \Cref{Detailed quasicharacterisation} is the following:

\begin{corollary}\label{Splitting fibration}
	For a Lie crossed module $(G, H,\tau,\alpha)$ and a Lie groupoid $\mb{X}$, the functor $\mc{F}$ in \Cref{Main Theorem 1} restricts to an essentially surjective, full and faithful functor to the subcategory ${\rm{Bun}}_{{\rm{Cat}}}(\mb{X}, [H \rtimes_{\alpha}G \rra G]) )$ of ${\rm{Bun}}_{{\rm{quasi}}}(\mb{X}, [H \rtimes_{\alpha}G \rra G])$ and hence yielding an equivalence of categories between ${\rm{Bun}}_{{\rm{Cat}}}(\mb{X}, [H \rtimes_{\alpha}G \rra G]) )$ and ${\rm{Bun}}(\mb{X}, G)$.
\end{corollary}
The following observation is partly motivated by the \textbf{Lemma 4.20}, \cite{MR3521476} and is a consequence of \Cref{Splitting fibration}. The construction below was not presented in either of our papers \cite{MR4403617} or \cite{chatterjee2023parallel} arising out of this thesis.
\subsection*{Construction of decorated principal 2-bundles (\Cref{Prop:Decoliegpd}) from functors}
Consider a functor $T \colon \mb{X} \ra G {\rm{-}} {\rm{Tor}}$, from a transitive Lie groupoid $\mb{X}$ (\Cref{Definition: Transitive Lie groupoid}) to the category of $G$-torsors, such that for each $x \in X_0$, the restriction $T|_{{\rm{Aut}_{\mb{X}}}(x)} \colon {\rm{Aut}_{\mb{X}}}(x) \ra T(x)$ is smooth.	Now, fix a point $z \in X_0$. Then by \Cref{Basic properties}, $t_z \colon s^{-1}(z) \ra X_0$ is a principal ${\rm{Aut}}_{\mb{X}}(z)$-bundle over $X_0$, where the Lie group ${\rm{Aut}}_{\mb{X}}(z)$ acts on $s^{-1}(z)$ by the composition, and the projection $t_x$ is given by the restriction of the target map $t|_{s^{-1}(z)}$.	 Observe that there is a left action of ${\rm{Aut}}_{\mb{X}}(z)$ on $T(z)$, given by $(\delta,p) \mapsto (T(\delta)(p))$ for $\delta \in {\rm{Aut}}_{\mb{X}}(z), p \in T(z)$. This induces a left action of ${\rm{Aut}}_{\mb{X}}(z)$ on $s^{-1}(z) \times T(z)$, $(\delta,(\gamma,p)) \mapsto (\gamma \circ \delta^{-1}, T(\delta)(p))$. Then consider the associated bundle $\pi_z \colon E_{z} := \frac{s^{-1}(z) \times T(z)}{{\rm{Aut}}_{\mb{X}}(z)} \ra X_0$, whose fibre is $T(z) \cong G$. Note that the right action of $G$ on $E_z$, $([\gamma,p],g) \mapsto (\gamma, pg)$, defines a principal $G$-bundle $\pi_z \colon E_z \ra X_0$ over $X_0$. The map $\mu_z \colon s^{*}E_z \ra E_z$ defined as $(\Gamma, [\gamma,p]) \mapsto [\Gamma \circ \gamma, p]$, is well defined and gives an action of the Lie groupoid $\mb{X}$ on the manifold $E_z$. Thus, we get a a principal $G$-bundle $(\pi_z \colon E_z \ra X_0, \mu_z, [X_1 \rra X_0])$ over $\mb{X}$, with respect to the point $z \in X_0$. Now let us fix another point $y \in X_0$, and consider the associated principal $G$-bundle $(\pi_y \colon E_y \ra X_0, \mu_y, [X_1 \rra X_0])$ over $\mb{X}$. Then, it is straighforward to show that for any $\theta \in {\rm{Hom}}(z,y)$, the map $\theta_{z,y} \colon E_z \ra E_y$, given by $[\gamma,p] \mapsto [\gamma \circ \theta^{-1}, T(\theta)(p)]$, is an isomorphism of principal $G$-bundles over the Lie groupoid $\mb{X}$ (\Cref{morphism of principal bundles}). Hence, summarizing the above discussion, we obtain the following result:
\begin{proposition}
	If a functor $T \colon \mb{X} \ra G {\rm{-}} {\rm{Tor}}$, from a transitive Lie groupoid $\mb{X}$ to the category of $G$-torsors, satisfies the property that for each $x \in X_0$, the restriction $$T|_{{\rm{Aut}_{\mb{X}}}(x)} \colon {\rm{Aut}_{\mb{X}}}(x) \ra T(x)$$ is smooth, then it determines a principal $G$-bundle $(\pi:E_G \ra X_0, \mu, \mb{X})$ over the Lie groupoid $\mb{X}$ and the association is unique upto an isomorphism.
\end{proposition}
Now, as an immediate consequence of \Cref{Splitting fibration}, we obtain the following result:

\begin{corollary}
	Let $(G, H, \tau, \alpha)$ be a Lie crossed module.	If a functor $T \colon \mb{X} \ra G {\rm{-}} {\rm{Tor}}$, from a transitive Lie groupoid $\mb{X}$ to the category of $G$-torsors, satisfies the property that for each $x \in X_0$, the restriction $$T|_{{\rm{Aut}_{\mb{X}}}(x)} \colon {\rm{Aut}_{\mb{X}}}(x) \ra T(x)$$ is smooth, then it determines a decorated principal  $[H \rtimes_{\alpha} G 
	\rra G]$-bundle (\Cref{Prop:Decoliegpd}) over $\mb{X}$ and the association is unique upto an isomorphism.
\end{corollary}

		\subsection{Quasi-connections as retractions}\label{Quasiconnections as retractions}
Recall, in \Cref{factorization of morphisms}, we saw that any morphism of Lie groupoids $F \colon \mb{X} \ra \mb{Y}$ has a canonical factorization through its weak fibered product $\mb{Y} \times_{\mb{Y},F}^{h} \mb{X}:= \mb{Y} \times_{{\rm{id}}_{\mb{Y}}, \mb{Y}, F}^{h} \mb{X}$. In this subsection, we will see one of its consequences. Precisely, for a Lie 2-group $\mb{G}$, given a principal $\mb{G}$-bundle $ \pi \colon \mb{E} \ra \mb{X}$, we will show a one-one correspondence between the following two:

`\textit{Set of all unital connections on $ \pi \colon \mb{E} \ra \mb{X}$' (see \Cref{Unital connection}}) $${\rm{and}}$$ `\textit{Set of all morphisms of Lie groupoids $r \colon \mb{X} \times_{\mb{X}, \pi}^{h} \mb{E} \ra \mb{E}$ (see \Cref{subsection pullbacks in Lie groupoids})
, satisfying }
\begin{enumerate}[(i)]
	\item \textit{$r \circ \pi_{\mb{E}}= \rm{id}_{\mb{E}}$ (i.e $r$ is a retraction of $\pi_{\mb{E}}$)},
	\item  \textit{$r$ is a morphism of $\mb{G}$-torsors and}
	\item  \textit{$\pi \circ r= \pi_{\mb{X}}$}.'
\end{enumerate}

Furthermore, we will also obtain a characterization of categorical connections in terms of such retractions. A similar result in the framework of general fibered categories (\Cref{subsection Fibered categories}) is available in \cite{del2008homotopy}, and a cursory mention of an analogous result in the set up of general Lie groupoid fibrations can be found in \cite{MR3968895}. However, to the best of our knowledge, the corresponding result in the framework of a Lie groupoid fibration equipped with an action of a Lie 2-group has not been explored in the literature. So, here we decide to provide the proof with some details, even though we have omitted the proof in either of our papers \cite{MR4403617} or \cite{chatterjee2023parallel} arising out of this thesis.

\begin{proposition}
	Given a Lie 2-group $\mb{G}$, let $\pi \colon \mb{E} \ra \mb{X}$ be a principal $\mb{G}$-bundle over a Lie groupoid $\mb{X}$. Then we have the following:
	\begin{itemize}
		\item[(a)] there is an induced right action of $\mb{G}$ on the Lie groupoid $\mb{X} \times_{\mb{X}, \pi}^{h} \mb{E}$, defined as
		\begin{equation}\nonumber
			\begin{split}
				\rho\colon &(\mb{X} \times_{\mb{X}, \pi}^{h} \mb{E}) \times \mb{G} \ra \mb{X} \times_{\mb{X}, \pi}^{h} \mb{E}\\
				&\big((x, \gamma, p), g \big) \mapsto \big(x, \gamma, pg \big), \\
				& \big( x \xrightarrow{\zeta} x', p \xrightarrow{\delta} p', \phi \big) \mapsto (\zeta, \delta \phi),
			\end{split}
		\end{equation}
		\item[(b)] the set of $\mb{G}$-equivariant morphisms of Lie groupoids $r \colon \mb{X} \times_{\rm{id}_{\mb{X}}, \mb{X}, \pi}^{h} \mb{E} \ra \mb{E}$ that satisfy $\pi \circ r= \pi_{\mb{X}}$, $r \circ \pi_{\mb{E}}= \rm{id}_{\mb{E}}$, is in one-one correspondence with the set of unital connections $\mc{C}$ on $\pi \colon  \mb{E} \ra \mb{X}$, where $\pi_{\mb{E}}$ and $\pi_{\mb{X}}$ are as defined in \Cref{piE} and \Cref{piX} respectively,
		\item[(c)] an unital connection $\mc{C}$ on $\pi \colon \mb{E} \ra \mb{X}$ is a categorical connection if and only if the image of morphisms of the form  $(\Gamma, 1_p)$ under the associated map $r_{\mc{C}} \colon \mb{X} \times_{\mb{X}, \pi}^{h} \mb{E} \ra \mb{E}$, lies in the image of $\mc{C}$.
	\end{itemize}
\end{proposition}

\begin{proof}

	\subsubsection*{Proof of (a):}
 
	A straightforward verification. 
 
	\subsubsection*{Proof of (b):}
 
	Suppose $r \colon \mb{X} \times_{\mb{X}, \pi}^{h} \mb{E}\ra \mb{E}$  is a $\mb{G}$-equivariant morphism of Lie grouipoids satisfying  $\pi \circ r= \pi_{\mb{X}}$ and $r \circ \pi_{\mb{E}}= \rm{id}_{\mb{E}}$.
	\[
	\begin{tikzcd}
		\mb{E} \arrow[d, "\pi"'] \arrow[r, "\pi_{\mb{E}}", bend left] & \mb{X} \times_{\mb{X}, \pi}^{h} \mb{E} \arrow[ld, "\pi_{\mb{X}}", bend left] \arrow[l, "r"', bend left=49] \\
		\mb{X}                                           &                                                           
	\end{tikzcd}.\]
	Define a map $\mc{C}_r \colon s^{*}E_0 \ra E_1$, taking $(\gamma,p) \in s^{*}E_0$ to the image of the element $(\gamma,1_p) \in {\rm{Hom}}\big( (s(\gamma),1_\pi(p),p), (t(\gamma),\gamma,p) \big)$ under $r$. Now, define $P \colon E_1 \ra s^{*}E_0$, $\gamma \mapsto (\pi_1(\gamma),s(\gamma))$. Then, we have	
	\begin{equation}\nonumber
		\begin{split}
			&P \circ r(\gamma,1_p)\\
			&= \big( \pi_{\mb{X}}(\gamma, 1_p), r \circ s(\gamma,1_p) \big)\\
			&= \Big( \gamma, r\big(s(\gamma),1_{\pi(p)},p \big) \Big)\\
			&= \big( \gamma, r \circ \pi_{\mb{E}}(p) \big)\\
			&= \big( \gamma,p \big).
		\end{split}
	\end{equation}	
	Hence, $\mc{C}_r$ is a section of $P$ . Now, as $r$ is $\mb{G}$-equivariant, it follows $\mc{C}_r(\gamma, pg)= \mc{C}_r(\gamma,p)1_g$ for $(\gamma,p) \in s^{*}E_0$, $g \in G_0$. Note that from the definition in \Cref{piE}, we have $\pi_{\mb{E}}(1_p)= (1_{\pi(p)},1_p) \in {\rm{Hom}} \big( (\pi(p), 1_{\pi(p)},p), (\pi(p),1_{\pi(p)},p) \big)$ for each $p \in E_0$. So, $\mc{C}_r(1_{\pi}(p),p)= r(1_{\pi(p)},1_p)=r \circ \pi_{\mb{E}}(1_p)=1_p$, which concludes that $\mc{C}_r$ is a unital connection.
	
	Conversely, let $\mc{C}$ be a unital  connection on $\pi \colon \mb{E} \ra \mb{X}$. Now, define a map $r_{\mc{C}} \colon \mb{X} \times_{\mb{X}, \pi}^{h} \mb{E} \ra \mb{E}$ as
	\begin{equation}\label{definition of rC}
		\begin{split}
			& (x, \gamma, p) \mapsto t(\mc{C}(\gamma,p))\\
			&\big((x, \gamma, p)\xrightarrow{(\Gamma, \delta)} (x', \gamma', p')  \big) \mapsto \mc{C}(\gamma',p') \circ \delta \circ (\mc{C}(\gamma,p))^{-1}.
		\end{split}
	\end{equation}
	It is clear from the definition that $r_{\mc{C}}$ is a morphism of Lie groupoids. We will show $r_C$ has all the desired properties, that is 
	\begin{itemize}
		\item[(i)]$\pi \circ r_{\mc{C}}= \pi_{\mb{X}}$,
		\item[(ii)]$r_{\mc{C}} \circ \pi_{\mb{E}}= \rm{id}_{\mb{E}}$ and
		\item[(iii)]$\mb{G}$-equivariance of $r_{\mc{C}}$.
	\end{itemize}
 
Proof of (i):

	To prove the equality at object level, consider $\pi \circ r_{\mc{C}}(x, \gamma, p)= t(\gamma)=x= \pi_{\mb{X}}(x, \gamma,p)$. For the morphism level, consider $\pi \circ r_{\mc{C}}\big( (x, \gamma, p)\xrightarrow{(\Gamma, \delta)} (x', \gamma', p') \big)= \pi(\mc{C}(\gamma',p') \circ \delta \circ (\mc{C}(\gamma,p))^{-1})= \Gamma= \pi_{\mb{X}}(\Gamma, \delta)$.	
	
Proof of (ii):

	Follows easily from the fact that $\mc{C}(\pi(p),p)=1_p$ for all $p \in E_0$.
 
	Proof of (iii):
 
	Since $\mu_{\mc{C}}(\gamma,pg)= \mu_{\mc{C}}(\gamma,p)g$ for each $(\gamma,p) \in s^{*}E_0$ and $g \in G_0$, we have the $G_0$-equivariance at the object level. Now, consider a morphism $\big( (x, \gamma, p)\xrightarrow{(\Gamma, \delta)} (x', \gamma', p') \big)$ in $\mb{X} \times_{\rm{id}_{\mb{X}}, \mb{X}, \pi}^{h} \mb{E}$ and a moprhism $\phi$ in $\mb{G}$. Then, we have
	\begin{equation}\nonumber
		\begin{split}
			& r_{\mc{C}} \Big((\Gamma, \delta)\phi \Big)  \\
			&= \mc{C}(\gamma', p't(\phi)) \circ \delta \phi \circ (\mc{C}(\gamma, ps(\phi)))^{-1}\\
			&=\mc{C}(\gamma',p') 1_{t(\phi)} \circ \delta \phi \circ (\mc{C}(\gamma,p))^{-1}1_{s(\phi)}\\
			&=r_{\mc{C}}\big( (\Gamma, \delta) \big)  \phi \,\, [\textit{using}\,\, \Cref{E:Identitiesechangeinverse}].
		\end{split}
	\end{equation}
 $\mc{C}_r$ and $r_{\mc{C}}$ are in fact mutual inverses of each other. To see this, consider a unital connection $\mc{C}$. Then by definition (see \Cref{definition of rC}), $\mc{C}_{r_{\mc{C}}} \colon s^{*}E_0 \ra E_1$ is given by $(\gamma,p) \mapsto r_{\mc{C}}(\gamma, 1_p)= \mc{C}(\gamma,p) \circ 1_p \circ  \big(\mc{C}(1_{\pi(p)},p) \big)^{-1}=\mc{C}(\gamma,p)$. On the other hand, consider a $\mb{G}$-equivariant morphism of Lie groupoids $r \colon \mb{X} \times_{\mb{X}, \pi}^{h} \mb{E}\ra \mb{E}$ satisfying  $\pi \circ r= \pi_{\mb{X}}$ and $r \circ \pi_{\mb{E}}= \rm{id}_{\mb{E}}$, where $\pi_{\mb{E}}$ and $\pi_{\mb{X}}$ are as defined in \Cref{piE} and \Cref{piX} respectively. Then by definition, $r_{\mc{C}_{r}} \colon \mb{X} \times_{\mb{X}, \pi}^{h} \mb{E}\ra \mb{E}$ is given by $(x, \gamma,p) \mapsto t \big( \mc{C}_{r}(\gamma,p)\big)$\\ $=t \big( r(\gamma, 1_p) \big)=r\big(t(\gamma,1_p) \big)=r(x,\gamma,p)$, for each object $(x, \gamma,p)$ in $\mb{X} \times_{\mb{X}, \pi}^{h} \mb{E}$. Now, for a morphism $\big( (x, \gamma, p)\xrightarrow{(\Gamma, \delta)} (x', \gamma', p') \big)$ in $\mb{X} \times_{\rm{id}_{\mb{X}}, \mb{X}, \pi}^{h} \mb{E}$, we have by defnition (see \Cref{definition of rC}), $r_{\mc{C}_r}(\Gamma, \delta)= \mc{C}_{r}(\gamma',p') \circ \delta \circ (\mc{C}_r(\gamma,p))^{-1}= r(\gamma',1_{p'}) \circ \delta \circ \big( r(\gamma,1_p) \big)^{-1}$. Since $\delta= r(\pi_{\mb{E}}(\delta))$, we get $r_{\mc{C}_r}(\Gamma, \delta)= r \Big( (\gamma',1_{p'}) \circ \pi_{\mb{E}}(\delta) \circ (\gamma, 1_p)^{-1}  \Big)$. As $\pi_{\mb{E}} (\delta)= \big(\pi_1(\delta), \delta \big)$ (see \Cref{piE}), we obtain $r_{\mc{C}_r}(\Gamma, \delta)= r \Big( \gamma' \circ \pi_1(\delta) \circ \gamma^{-1}, \delta \Big)$. Then, \Cref{Weakfibreproductmorphims} implies $r_{\mc{C}_r}(\Gamma, \delta)= r(\Gamma, \delta)$, which establishes the required one-one correspondence and concludes the proof of (b). 

\subsubsection*{Proof of (c):}

	Let $\mc{C} \colon s^{*}E_0 \ra E_1$ be a categorical connection and $r_{\mc{C}} \colon \mb{X} \times_{\mb{X}, \pi}^{h} \mb{E}\ra \mb{E}$ its associated map, as defined in \Cref{definition of rC}). Consider a morphism of the form $(x, \gamma, p) \xrightarrow{(\Gamma, 1_p)}(x', \gamma', p)$ in $\mb{X} \times_{\mb{X}, \pi}^{h} \mb{E}$. Then, we have
	\begin{equation}\nonumber
		\begin{split}
			& r_{\mc{C}}(\Gamma, 1_p) \\
			&=\mc{C}(\gamma',p) \circ 1_p \circ (\mc{C}(\gamma,p))^{-1}\\
			&=\mc{C}(\gamma',p) \circ \mc{C}(\gamma^{-1},p)\\
			&= \underbrace{\mc{C}(\gamma' \circ \gamma^{-1},p)}_{[\textit{by}\,\, \textit{the}\,\, \textit{condition}\,\,(v)\,\, \textit{in} \,\, \Cref{Def:categorical connection}]}.
		\end{split}
	\end{equation}
	Hence, $r_{\mc{C}}(\Gamma,p)$ lies in the image of $\mc{C}$.
	
	Converesely, let us assume that $r_{\mc{C}}(\Gamma, 1_p) \in \mc{C}(s^{*}E_0)$ for all morphisms of the form  $$(x, \gamma, p) \xrightarrow{(\Gamma, 1_p)}(x', \gamma', p)$$ in $\mb{X} \times_{\mb{X}, \pi}^{h} \mb{E}$, where $r_{\mc{C}}$ is the map associated to the unital categorical connection $\mc{C}$.  For composable $\gamma_2, \gamma_1 \in X_1$ such that $(\gamma_1,p) \in s^{*}E_0$, suppose $$r_{\mc{C}} \Big( (s(\gamma_2), \gamma_1, p) \xrightarrow{(\gamma_2, 1_p)} (t(\gamma_2),\gamma_2 \circ \gamma_1,p) \Big)= \mc{C}(\bar{\gamma},q)$$ for some $(\bar{\gamma},q) \in s^{*}E_0$. However, by definition of $r_{\mc{C}}$ in \Cref{definition of rC}, we have $r_{\mc{C}}(\gamma_2, 1_p)= \mc{C}(\gamma_2 \circ \gamma_1,p) \circ 1_p \circ \mc{C}(\gamma_1,p)^{-1}$.
	Thus, we have
	\begin{equation}\label{Equation:a}
		\mc{C}(\bar{\gamma},q)= \mc{C}(\gamma_2 \circ \gamma_1,p) \circ \mc{C}(\gamma_1,p)^{-1}.
	\end{equation}
	Applying $s$ and $\pi_1$ on both side of \Cref{Equation:a}, we get $q= t(\mc{C}(\gamma,p))$ and $\bar{\gamma}=\gamma_2$. Then, substituting $q$ and $\bar{\gamma}$ in \Cref{Equation:a}, we conclude
	\begin{equation}\nonumber
		\mathcal{C}(\gamma_2 \circ \gamma_1 , p_1)= \mathcal{C}(\gamma_2, t(\mc{C}(\gamma,p))) \circ \mathcal{C}(\gamma_1, p_1).
	\end{equation}
	Hence, $\mc{C}$ is a categorical connection, completing the proof of (c).
\end{proof}

\section{Towards a principal 2-bundle over a differentiable stack}\label{Towards a principal 2-bundle over a differentiable stack}
The material of this section is based on our paper \cite{	chatterjee2023parallel}.

Recall, according to  \Cref{principal Lie group bundles over stacks}, if the Lie groupoids $\mb{X}$ and $\mb{Y}$ are Morita equivalent, then for any Lie group $G$, the categories ${\rm{Bun}}(\mb{X},G)$ and ${\rm{Bun}}(\mb{Y},G)$ are equivalent. Then, \Cref{Splitting fibration} gives us the following result:
\begin{proposition}\label{stack}
	For a Lie 2-group $\mb{G}$, if Lie groupoids $\mb{X}$ and $\mb{Y}$ are Morita equivalent, then the category $\rm{Bun}_{\rm{Cat}}(\mb{X}, \mb{G})$ is equivalent to the category $\rm{Bun}_{\rm{Cat}}(\mb{Y}, \mb{G})$.
\end{proposition}
Hence, according to  \Cref{Morita equivalent imply stack}, the \Cref{stack} enables us to extend the notion of a principal 2-bundle over a Lie groupoid, equipped with a categorical connection, to be defined over the differentiable stack presented by its base Lie groupoid. In an upcoming project, we are trying to extend the notion of a quasi-principal 2-bundle to be defined over a differentiable stack.

\section{$\eta$-twisted principal 2-bundles over Lie groupoids}\label{Twisted principal 2-bundles and Lie groupoid $G$-extensions}
This section is based on our paper \cite{MR4403617}. Here, we discuss a correspondence between a weaker notion of principal 2-bundles over Lie groupoids and the Lie groupoid $G$-extensions (\Cref{subsection Lie groupoid G extension}) for a Lie group $G$. Such notion of principal 2-bundles 
are obtained by replacing the action functor $\rho\colon\mb{E} \times \mb{G} \rightarrow \mb{E}$ in \Cref{Definition:principal $2$-bundle over Liegroupoid}	by an $\eta$-twisted action $\rho\colon\mb{E} \rtimes_{\eta} \mb{G} \rightarrow \mb{E}$ (\Cref{Definition:Lie 2 twistgroupaction}) determined by the smooth map $\eta\colon  \rm{Mor}(\mb{E})\times \rm{Mor}(\mb{G})\ra \rm{Mor}(\mb{E})$. We call this weaker notion a `$\eta $-twisted principal $\mb{G}$-bundle over the Lie groupoid $\mb{X}$'.
\begin{definition}\label{eta twisted principal 2-bundle}
	Let there be an $\eta$-twisted action $\rho\colon\mb{E} \rtimes_{\eta} \mb{G} \rightarrow \mb{E}$ of a Lie 2-group $\mb{G}$ (\Cref{Definition:Lie 2 twistgroupaction}) on a Lie groupoid $\mb{E}$ for a smooth map $\eta\colon  \rm{Mor}(\mb{E})\times \rm{Mor}(\mb{G})\ra \rm{Mor}(\mb{E})$. An\textit{ $\eta $-twisted principal $\mb{G}$-bundle over the Lie groupoid $\mb{X}$} is given by a morphism of Lie groupoids $\pi: \mb{E} \rightarrow \mb{X}$ such that 
	\begin{itemize}
		\item $\pi_0\colon E_0 \rightarrow X_0$ is a principal $G_0$-bundle over $X_0$,
		\item $\pi_1\colon E_1 \rightarrow X_1$ is a principal $G_1$-bundle over $X_1$.
	\end{itemize}
	
\end{definition}
\begin{remark}
	\Cref{eta twisted and usual} implies that a prinicipal 2-bundle over a Lie groupoid (\Cref{Definition:principal $2$-bundle over Liegroupoid}) is a special case of the above definition.
\end{remark}

\subsection{Correspondence between $\eta$-twisted principal 2-bundles and Lie groupoid $G$-extensions}
\label{eta twisted to Lie groupoid extension}For an abelian Lie group $G$, let us consider a principal $[G\rra e]$-bundle $\mb{E}=[E_1\rra X_0]$ over the Lie groupoid $\mb{X}=[X_1\rra X_0]$, see \Cref{E:Example of principal 2-bundle ordinary} . Now, define the map $i\colon X_0\times G \ra E_1$ as
$$(x, g)\mapsto  1^{\mb{E}}_{x} \,g,$$	
where $1^{\mb{E}}$	is the identity assigning map in the Lie groupoid $\mb{E}=[E_1\rra X_0].$  As $1^{\mb E}\colon X_0\ra E_1$ is a diffeomorphism onto its image, by \Cref{E:restricthom}, the map $i$ defines an embedding. Since $E_1\ra X_1$ is a principal $G$-bundle, the following
\[
\begin{tikzcd}[sep=small]
	1 \arrow[r] & X_0\times G  \arrow[dd,xshift=0.75ex]\arrow[dd,xshift=-0.75ex]  \arrow[rr, "i"] &  & E_1 \arrow[rr, "\pi_1"]  \arrow[dd,xshift=0.75ex]\arrow[dd,xshift=-0.75ex]   &  & X_1 \arrow[r] \arrow[dd,xshift=0.75ex]\arrow[dd,xshift=-0.75ex]  & 1 \\
	&                                                                    &  &                                                       &  &                           &   \\
	1 \arrow[r] & X_0 \arrow[rr, "{\rm Id}"]                                                     &  & X_0 \arrow[rr, "{\rm Id}"]                                        &  & X_0 \arrow[r]             & 1
\end{tikzcd},\]
is a Lie groupoid $G$-extension by definition, see \Cref{subsection Lie groupoid G extension}. 

Conversely, consider a Lie groupoid $G$-extension as in Diagram \ref{Dia:Gextension}.  Note that as the second square from the left commutes, we have $$i(x, g)\in {\rm Hom}_{\Gamma_2}(x, x).$$

\begin{lemma}\label{L:Almost}
	
	\begin{itemize}
		\item[(i)] There is a smooth free action $\Gamma_2\times G\ra \Gamma_2$ defined as   \begin{equation}\label{E:ActionGamma2}
			(x \xrightarrow{\gamma}y, g)\mapsto \gamma\circ i(x, g).
		\end{equation} 
		
		\item[(ii)] The action is transitive on the fibers $\phi^{-1}(\lambda),$ for each $\lambda\in \Gamma_1$ 
		
		\item[(iii)]  $\phi$ is constant on the orbits of the action. 
		
		\item[(iv)] The action satisfies the condition in \Cref{E:restricthom};
		\begin{equation}\label{E:Restrichomgamma}
			{\rm Hom}_{\Gamma_2}(x, y)\times G \ra {\rm Hom}_{\Gamma_2}(x, y)
		\end{equation}
		for all $x, y\in M.$  
	\end{itemize}
\end{lemma}
\begin{proof}
	
	\begin{itemize}
		\item[(i)] As $i$ is a functor, we have
		\[	\gamma (gg')=\gamma \circ i(x, g\, g')
		=\gamma \circ \bigl(i(x, g)\circ i(x, g')\bigr)
		=\bigl(\gamma \circ i((x, g)\bigr)\circ i(x, g')
		=(\gamma g) g'.\]
		Observe that as $i$ is injective, it follows that the action is free.
		\item[(ii)] If $\phi(\gamma_2)=\phi(\gamma_1).$ Then $\phi(\gamma_2^{-1}\circ \gamma_1)=1^{\Gamma_1}_x.$ Hence, by the exactness of Diagram \ref{Dia:Gextension} we have $(x, g)\in M \times G$ such that $\gamma_2^{-1}\circ \gamma_1=i(x, g).$ Thus $\gamma_1=\gamma_2 g.$  
		
		\item[(iii)] $\phi(\gamma g)=\phi(\gamma\circ i(x, g))=\phi(\gamma)\circ \phi (i(x, g)).$ As \textbf{Diagram} \ref{Dia:Gextension} is a short exact sequence, $\phi (i(x, g))=1^{\Gamma_1}_x.$ Hence, $\phi(\gamma g)=\phi(\gamma).$
		\item[(iv)] Directly follows from the observation that $i(x, g)\in {\rm Hom}_{\Gamma_2}(x, x).$
	\end{itemize}
\end{proof}

Summarising,  since $\phi$ is a surjective submersion, the Lie groupoid $G$-extension in Diagram \ref{Dia:Gextension} yields a  principal $G$-bundle $\phi\colon \Gamma_2\ra \Gamma_1$ such that the diagram
\[
\begin{tikzcd}[sep=small]
	\Gamma_2 \arrow[rr,"\phi"] \arrow[dd,xshift=0.75ex,"t"]
	\arrow[dd,xshift=-0.75ex,"s"'] &  & \Gamma_1 \arrow[dd,xshift=0.75ex,"t"]
	\arrow[dd,xshift=-0.75ex,"s"'] \\
	&  &                \\
	M \arrow[rr,"{\rm Id}"]            &  & M          
\end{tikzcd},\]	
commutes, and the action satisfies the \Cref{E:Restrichomgamma}. In general, the action defined in \Cref{L:Almost}  \textit{does not} satisfy the other functoriality condition in \Cref{E:compoequi}! In order to see this, consider a pair of composable morphisms $y\xrightarrow{\gamma_2}z$ and 	$x\xrightarrow{\gamma_1}y$ and a pair of elements $g,g' \in G$. By \Cref{E:ActionGamma2}, we get  $(\gamma_2 \, g') \circ (\gamma_1 \, g)= \gamma_2 \circ i(y, g') \circ \gamma_1 \circ i(x, g).$ To remain consistent with  \Cref{E:compoequi}, we would require the right-hand side of the last equation to be equal to $ \big( \gamma_2 \circ \gamma_1 \circ i(x,g'g) \big)$, which is in general, of course, may not be true. Hence, although the principal $G$-bundle $\phi \colon \Gamma_2\ra \Gamma_1$ \textit{does not} define a  $[G\rra e]$-bundle $ [\Gamma_2\rra M]$  over  the Lie groupoid $[\Gamma_1\rra M],$ however it indeed defines a {twisted principal $[G\rra e]$-bundle over the Lie groupoid} $[\Gamma_1\rightrightarrows M]$. To see this, let us define
\begin{equation}\label{E:TwistGextension}
	\begin{split}
		\eta\colon & \rm{Mor}(\Gamma_2)\times G \ra \rm{Mor}(\Gamma_2)\\
		&(x\xrightarrow{\gamma} y, g)\mapsto i(y, g)\circ \gamma\circ i(x, g^{-1}).
	\end{split}
\end{equation}	
One can easily verify that $\eta$ satisfies all the necessary conditions in \Cref{E:Contwist}. So, now we have the twisted product category $\bigl([\Gamma_2\rra M]\rtimes_{\eta} [G\rra e]\bigr).$ Next, we verify the following:
\begin{equation}\label{E:Veriftwisfunc}
	\begin{split}
		\bigl(y\xrightarrow{\gamma_2}z\,  {g'}\bigr)\circ \bigl(x\xrightarrow{\gamma_1}y\,g\bigr)
		= &\bigl({\gamma_2}\circ  i(y, g')\bigr)\circ \bigl({\gamma_1}\circ i(x, g)\bigr)\\
		= &\bigl({\gamma_2}\circ  (i(y, g')\circ {\gamma_1}\circ i(x, g'^{-1})\bigr)\circ i(x, g' g)\\
		= & \bigl({\gamma_2}\circ  \eta (\gamma_1, g')\bigr)\circ i(x, g' g)
		= \bigl({\gamma_2}\circ  \eta (\gamma_1, g')\bigr)g' g, 
	\end{split}
\end{equation} 
which is exactly the second condition in \Cref{Remark:functorialitytwisted}, needed for the action in \Cref{E:ActionGamma2} to be an $\eta$-twisted action
$$\bigl([\Gamma_2\rra M]\rtimes_{\eta} [G\rra e]\bigr)\ra [\Gamma_2\rra M].$$
Also, note that in the other direction, the construction proceeds for a twisted $[G\rra e]$-bundle the same way as was for a  \textit{non-twisted} $[G\rra e]$-bundle (as we have seen in \Cref{eta twisted to Lie groupoid extension}) using the third property in  \Cref{E:Contwist}.

In conclusion, we obtain the following proposition:	
\begin{proposition}
	For a Lie group $G$, a twisted principal $[G\rra e]$-bundle defines a Lie groupoid $G$-extension and vice-versa. 
\end{proposition}
 The above correspondence is not one-to-one.	It is worth mentioning here that \cite{MR3480061} also discussed the relation between Morita equivalent classes of Lie groupoid $G$-extensions and Morita equivalence classes of what the authors called $[G\rra {\rm Aut} (G)]$-bundles (see \textbf{Theorem 3.4}, \cite{MR3480061}).


\chapter{Connection structures and gauge transformations on a principal 2-bundle over a Lie groupoid}\label{chapter 2-bundles copnnection} 


\lhead{Chapter 5. \emph{Connection structures and gauge transformations on a principal 2-bundle over a Lie groupoid}} 

The contents of this chapter are mainly based on our paper \cite{MR4403617}. 

This chapter develops a theory of connection structures and gauge transformations on our principal 2-bundles over Lie groupoids. Our theory encompasses both the classical theory discussed in \Cref{Chapter Classical setup} and the one available for a principal Lie group bundle over a Lie groupoid, recalled in \Cref{Principal bundles over Lie groupoids and their connection structures}.Our definition allows us to construct a categorified version of the Atiyah sequence (\Cref{Definition: Atiyah sequence of a principal bundle}). As a result, we obtain a weaker version of the differential geometric connection structure. Interestingly, this `weakened notion', has no traditional counterpart. More precisely, for a Lie 2-group $\mb{G}$, given a principal $\mb{G}$-bundle $\pi \colon \mb{E} \ra \mb{X}$ over a Lie groupoid $\mb{X}$, we associate a short exact sequence of VB groupoids (\Cref{Short exact sequence of VB-groupoids}) over $\mb{X}$, that we call the \textit{Atiyah sequence of VB-groupoids associated to the principal $\mb{G}$-bundle $\pi \colon \mb{E} \ra \mb{X}$}. We then introduce two notions of connection structures on $\pi \colon \mb{E} \ra \mb{X}$, one arising from a retraction of the Atiyah sequence and the other one arising from a retraction up to a natural isomorphism. We call them, respectively, a \textit{strict connection} and the `weaker one' a \textit{semi-strict connection}.

We obtain our first main result in this chapter by extending the traditional one-one correspondence between connections as splittings and connections as 1-forms (\Cref{Equivalent characterisations of classical connection}) to the level of isomorphism of categories by describing strict and semi-strict connections as $L(\mb{G})$-valued 1-forms on the Lie groupoid $\mb{E}$.

In the final part of this chapter, we investigate the action of the 2-group of gauge transformations on strict and semi-strict connections. As the second main result of this chapter, we discovered an extended symmetry in the category of semi-strict connections.

Moreover, as a part of the primary purpose of this thesis (studying `differential geometric relationships' between classical gauge theory and fibered categories), we studied
\begin{enumerate}[(i)]
	\item relation between connection structures on decorated principal 2-bundles and connections on principal Lie group bundles over Lie groupoids (\Cref{Connections on a principal bundle over a Lie groupoid}),
	\item proposed an existence criterion for the strict and semi-strict connections on a principal $2$-bundle over a proper, \'etale Lie groupoid,
	\item investigated an action of gauge transformations on categorical connections and 
	\item studied `extended gauge transformations' on decorated principal 2-bundles.
\end{enumerate}
Having said that, we are yet to explore the above four aspects in the general context of quasi-principal 2-bundles (\Cref{Definition:Quasicategorical Connection}) and pseudo-principal Lie crossed module-bundles (\Cref{Definition:PseudoprincipalLiecrossedmodulebundle}). However, in the \Cref{Chapter: Parallel transport on quasi-principal 2-bundles}, we will see a beautiful interplay between strict connections and quasi connections.

	\section{Connection structures on a principal 2-bundle over a Lie groupoid}\label{Connection structures on a principal 2-bundle over a Lie groupoid}

This section studies connection structures on our categorified principal bundles (\Cref{Definition:principal $2$-bundle over Liegroupoid}) in two equivalent ways:

\begin{enumerate}[(i)]
	\item as `splittings of an associated Atiyah sequence of VB-groupoids' and 
	\item as `differential 1-forms on Lie groupoids'.
\end{enumerate}
We begin with the Atiyah sequence approach.
\subsection{Atiyah sequence associated to a principal 2-bundle over a Lie groupoid}\label{Subsection: Atiyah sequence associated to a principal 2-bundle over a Lie groupoid}
For a Lie 2-group $\mb{G}$, let $\pi \colon \mb{E} \ra \mb{X}$ be a principal $\mb{G}$-bundle over a Lie groupoid $\mb{X}$. Consider the pair of Atiyah sequences (\Cref{Definition: Atiyah sequence of a principal bundle}) $\rm{At}(\pi_1)$ and $\rm{At}(\pi_0)$, associated respectively to the principal $G_1$-bundle $\pi_1 \colon E_1 \ra X_1$ and the principal $G_0$-bundle $\pi_0 \colon E_0 \ra X_0$ over the manifold $X_1$ and $X_0$ respectively:
\begin{equation}\label{Atiyah Objects}
	\begin{tikzcd}
		0 \arrow[r, ""]  & \rm{Ad}(E_i) \arrow[r, "\delta_i^{/G_i}"] \arrow[d, ""'] & \rm{At}(E_i) \arrow[r, "\pi_{{1}{*}}^{/G_i}"] \arrow[d, ""'] & TX_i \arrow[d, ""'] \arrow[r] & 0 \\
		0 \arrow[r, ""'] & X_i \arrow[r, "\rm{id}"']                & X_i \arrow[r, "\rm{id}"']                & X_i \arrow[r]                 & 0
	\end{tikzcd}
\end{equation}
for $i=0,1$.

Now, consider the following pair of vector bundles:
\begin{enumerate}
	\item Pair of tangent bundles $\{TX_i \ra X_i\}_{i\in\{0, 1\}}$;
	\item Pair of adjoint bundles $\{ \pi^{\rm{Ad}(E_i)} \colon {\rm Ad}(E_i) \ra X_i\}_{i\in\{0, 1\}} $ (\Cref{Example: Adjoint bundle});
	\item Pair of Atiyah bundles $\{\pi^{\rm{At}(E_i)} \colon {\rm At}(E_i) \ra X_i\}_{i\in\{0, 1\}} $ (\Cref{Example: Atiyah bundle}).
	 
\end{enumerate}
We will show that these three pairs of vector bundles combine appropriately to produce a short exact sequence of VB-groupoids over the Lie groupoid $\mb{X}$ (\Cref{Short exact sequence of VB-groupoids}).
\subsection*{Construction of the tangent VB-groupoid $T\mb{X} \ra \mb{X}$}

See \Cref{Tangent VB groupoid}.

\subsection*{Construction of the adjoint VB-groupoid ${\rm{Ad}}(\mb{E}) \ra \mb{X}$}\label{Adjoint VB construction}

Manifolds ${\rm{Ad}}(E_1) $ and ${\rm{Ad}}(E_0)$ combine to form a Lie groupoid ${\rm{Ad}(\mb{E})}:=[{\rm Ad}(E_1)\rra {\rm Ad}(E_0)]$, whose structure maps are defined below:
\begin{itemize}
	\item the source map $s^{/\sim}\colon {\rm Ad}(E_1)\ra {\rm Ad}(E_0)$  as \[[(\widetilde \gamma, K)]\mapsto [(s(\widetilde \gamma),s_{*,e}(K))],\]
	\item the target map $t^{/\sim}\colon {\rm Ad}(E_1)\ra {\rm Ad}(E_0)$  as \[[(\widetilde \gamma, K)]\mapsto [(t(\widetilde\gamma),t_{*,e}(K))],\]
\end{itemize}
where $\widetilde \gamma\in E_1$ and $K\in L(G_1).$
To define the composition, we need to make an observation. Let $[(\widetilde \gamma_2 ', K_2')], [(\widetilde \gamma_1, K_1)] \in {\rm{Ad}}(E_1)$, such that
\begin{equation}\label{E:Sourcetargetad}
	t^{/\sim}([(\widetilde \gamma_2 ', K_2')])=s^{/\sim}([(\widetilde \gamma_1, K_1)]).
\end{equation} 
Then there exists an element $\theta$ in $G_0$ such that $s(\widetilde \gamma_2')\theta=t(\widetilde \gamma_1)$ and ${\rm{ad}}_{\theta}(s_{*, e}(K_2'))=t_{*, e}(K_1).$ It means $(\widetilde \gamma_2', K_2')1_{\theta}=\bigl(\widetilde \gamma_2'\, 1_{\theta}, \rm{ad}_{1_{\theta}}(K_2')\bigr)$ is composable with $(\widetilde\gamma_1, K_1).$ It is easy to verify that  $(\widetilde \gamma_2', K_2')1_{\theta}\in [(\widetilde \gamma'_2, K'_2)].$ Hence, whenever the condition in \Cref{E:Sourcetargetad}	is satisfied, there exists a composable pair belonging to $ [(\widetilde \gamma'_2, K'_2)]$ and $[(\widetilde \gamma_1, K_1)]$ respectively. Choosing such a pair define the composition as follows:
\begin{equation}\label{E:Compoad}
	\bigl([(\widetilde \gamma_2, K_2)]\bigr)\circ \bigl([(\widetilde \gamma_1, K_1)]\bigr)=[\bigl(\widetilde \gamma_2\circ \widetilde \gamma_1, K_2\circ K_1\bigr)].
\end{equation} 
To check that the composition is well defined, observe that if $(\widetilde \gamma_2', K_2')$ and $(\widetilde \gamma_1', K_1')$ are	another pair of such elements, then there exist composable $k_2, k_1\in G_1$ such that
$(\widetilde \gamma_2', K_2')=\bigl(\widetilde \gamma_2 k_2, {\rm{ad}}_{k_2}(K_2)\bigr)$ and $(\widetilde \gamma_1', K_1')=\bigl(\widetilde \gamma_1 k_1, {\rm{ad}}_{k_1}(K_1)\bigr).$
Then  we have \begin{equation}\nonumber
	\begin{split}
		&(\widetilde \gamma_2' \circ \widetilde \gamma_1', K_2'\circ K_1')\\
		&=\bigl((\widetilde \gamma_2 k_2)\circ (\widetilde \gamma_1 k_1),   ({\rm{ad}}_{k_2}(K_2))\circ   ({\rm{Ad}}_{k_1}(K_1)) \bigr)\\
		&=\bigl((\widetilde \gamma_2\circ \widetilde \gamma_1)(k_2 k_1),   ({\rm{ad}}_{k_2 k_1}(K_2\circ K_1)) \bigr)\\
		&=(\widetilde \gamma_2 \circ \widetilde \gamma_1, K_2\circ K_1) (k_2 k_1),
	\end{split}
\end{equation}
where in the third step, the functoriality of the Lie $2$-group action is used. The inverse and unit maps are obvious.

Now consider the commutative diagram below:
\[\begin{tikzcd}
	E_1 \times L(G_1) \arrow[rr, "\delta_0"] \arrow[d, "s_{\mb{E}} \times s_{{\mb{G}}_{*,e}}"'] &  & {\rm Ad}(E_1) \arrow[d, "s^{/\sim}", dotted] \\
	E_0 \times L(G_0) \arrow[rr, "\delta_1"']                &  & {\rm Ad}(E_0)
\end{tikzcd},\]
Observe that as $s_{\mb E}$ and $s_{{\mb{G}}_{*,e}}$ are surjective submersions, $s^{/\sim} \circ \delta_0$  a surjective submersion. As $\delta_0$ is a surjective submersion, it follows immediately that $s^{/\sim}$ is also a surjective submersion.  $t^{/\sim}$ can also be shown to be a surjective submersion using a similar argument. Hence, ${\rm{Ad}}{(\mb{E})}=[{\rm{Ad}}(E_1)\rra {\rm{Ad}}(E_0)]$ is a Lie groupoid. Now, it is a straightforward but lengthy verification that the vector bundles ${\rm Ad}(E_1)\ra X_1$, and ${\rm Ad}(E_0)\ra X_0$ combine to form a  VB-groupoid $[{\rm Ad}(E_1)\rra {\rm Ad}(E_0)]\ra [X_1\rra X_0]$. We denote this VB-groupoid by $\pi^{\rm{Ad}}(\mb{E}) \colon {\rm{Ad}}(\mb{E}) \ra \mb{X}$.

\subsection*{Construction of the Atiyah VB-groupoid ${\rm At}(\mb{E}) \ra \mb{X}$}

Technical details of the construction of the Lie groupoid  ${\rm{At}(\mb{E})}:=[{\rm At}(E_1)\rra {\rm At}(E_0)]$ are almost similar to the construction  of ${\rm{Ad}}{(\mb{E})}$. The Lie groupoid structure of ${\rm{At}}(\mb{E})$ is provided below:
\begin{itemize}
	\item the source $s_*^{/\sim}\colon {\rm At}(E_1)\ra {\rm At}(E_0),$ \[[(\widetilde \gamma, \widetilde X)]\mapsto [(s(\widetilde \gamma),s_{*,\widetilde \gamma}(\widetilde X))],\]
	\item the target $t_*^{/\sim}\colon {\rm At}(E_1)\ra {\rm At}(E_0),$ \[[(\widetilde \gamma, \widetilde X)]\mapsto [(t(\widetilde \gamma),t_{*,\widetilde \gamma}(\widetilde X))]\]	
\end{itemize}
The composition  is defined as 
\begin{equation}\label{E:AtCompo}
	[(\widetilde \gamma_2, \widetilde X_2)]\circ [(\widetilde \gamma_1, \widetilde X_1)]=[(\widetilde \gamma_2\circ \widetilde \gamma_1, \widetilde X_2\circ \widetilde X_1)]
\end{equation}
for suitable choices of representative elements of the equivalence classes. The unit map and the inverse map are obvious. Similar to the case of adjoint VB-groupoid, the pair of vector bundles ${\rm At}(E_1)\ra X_1$ and ${\rm At}(E_0)\ra X_0$ combine to give a VB-groupoid $[{\rm At}(E_1)\rra {\rm At}(E_0)]\ra [X_1\rra X_0]$ over $\mb{X}$ and we denote it by $\pi^{\rm{At}}(\mb{E}) \colon {\rm{Ad}}(\mb{E}) \ra \mb{X}$

In \Cref{Atiyah Objects}, the pair of maps $\delta_1^{/G_1}$ and $\delta_0^{/G_0}$ combine to form a morphism of Lie groupoids $\delta^{/G}:=(\delta_1^{/G_1}, \delta_0^{/G_0}) \colon {\rm{Ad}(\mb{E})} \ra {\rm{At}(\mb{E})}$ and hence, a 1-morphism of VB-groupoids (\Cref{1morphism of VB groupids}) from the adjoint VB-groupoid to the Atiyah VB-groupoid over $\mb{X}$. To see it, we need to make the following observation:
\begin{lemma}\label{Vertical vector field generating functor}
	For a Lie 2-group $\mb{G}$, let $\pi \colon \mb{E} \ra \mb{X}$ be a principal $\mb{G}$-bundle over a Lie groupoid $\mb{X}$. The generating maps $\delta_p\colon L(G_0)\ra T_p E_0$ and $\delta_{(p\xra{\widetilde \gamma} q)}\colon L(G_1)\ra T_{\widetilde \gamma}E_1$ for vertical vector fields define a morphism of Lie groupoids $\delta\colon \mb{E}\times L(\mb{G})\ra T(\mb{E}).$ Moreover the 
	functor $\delta$ is $\mb{G}$-equivariant in the sense that, $\delta \bigl(p g, {\rm{ad}}_{g^{-1}}(B)\bigr)=\delta (p, B) \cdot g$ and $\delta \bigl(\widetilde \gamma k, {\rm{ad}}_{k^{-1}}(K)\bigr)=\delta (\widetilde \gamma, K) \cdot k,$
	for any $p\in E_0, g\in G_0, B\in L(G_0)$ and $\widetilde \gamma\in E_1, k\in G_1, K\in L(G_1).$
\end{lemma}
\begin{proof}
	A straightforward consequence of the functoriality of the Lie $2$-group action. 
\end{proof}
On the other hand, it is obvious that in \Cref{Atiyah Objects}, the pair of maps $\pi_{1_{*}}^{/G_1}$ and $\pi_{0_{*}}^{/G_0}$ combine to form a morphism of Lie groupoids  $\pi^{/G} \colon {\rm{At}}(\mb{E}) \ra T\mb{X}$ and hence induce a 1-morphism of VB-groupoids from the Atiyah VB-groupoid to the tangent VB-groupoid over $\mb{X}$.

Summarising the discussion above, we obtain the following: 
\begin{proposition}\label{Prop:AtiyahLie2gpd}
	For a Lie 2-group $\mb{G}$, let $\pi \colon \mb{E} \ra \mb{X}$ be a principal $\mb{G}$-bundle over a Lie groupoid $\mb{X}$. Then, there is a short exact sequence 
	\begin{equation}\label{Atiyah sequece of 2-bundles proposition}
		\begin{tikzcd}
			0 \arrow[r, ""]  & {\rm{Ad}}(\mb{E}) \arrow[r, "\delta^{/ \mb{G}}"] \arrow[d, ""'] & {\rm{At}}(\mb{E}) \arrow[r, "\pi_{*}^{/ \mb{G}}"] \arrow[d, ""'] & T\mb{X} \arrow[d, ""'] \arrow[r] & 0 \\
			0 \arrow[r, ""'] & \mb{X} \arrow[r, "\rm{id}"']                & \mb{X} \arrow[r, "\rm{id}"']                & \mb{X} \arrow[r]                 & 0
		\end{tikzcd}
	\end{equation}
	of VB-groupoids  over $\mb{X}$ (\Cref{Short exact sequence of VB-groupoids}).	
\end{proposition} 
Now, we are ready to define a higher analog of \Cref{Definition: Atiyah sequence of a principal bundle}.
\begin{definition}
	For a Lie 2-group $\mb{G}$, let $\pi \colon \mb{E} \ra \mb{X}$ be a principal $\mb{G}$-bundle over a Lie groupoid $\mb{X}$. The associated short exact sequence of VB-groupoids \Cref{Atiyah sequece of 2-bundles proposition} is defined as the \textit{Atiyah sequence associated to $\pi \colon \mb{E} \ra \mb{X}$}, which we denote by ${\rm{At}}(\pi)$.
\end{definition}
Note that the notation we used to denote an Atiyah sequence of a principal 2-bundle over a Lie groupoid is the same as we used to denote a classical one \Cref{Definition: Atiyah sequence of a principal bundle}. One should understand the distinction from the context.

\subsection{Strict and semi-strict connections as splittings of the Atiyah sequence}\label{subsectionStrict and semi-strict connections as splittings of the Atiyah sequence}
Recall in \Cref{Definition: Connection on a principal bundle}, we observed that in the setup of classical principal bundles, every splitting of the associated Atiyah sequence could be viewed as connection data on the principal bundle. In a similar spirit, we introduce a notion of connection on a principal 2-bundle over a Lie groupoid. We will see that the categorical structure involved in these bundles will naturally broaden the differential geometric connection structure.
\begin{definition}[strict connection and semi-strict connection] \label{strict and semi-strict connection definition}
	For a Lie 2-group $\mb{G}$, let  $\pi\colon \mb{E} \rightarrow \mb{X}$ be 
	a principal $\mb{G}$-bundle over a Lie groupoid $\mb{X}$, and 	
	\begin{equation}\label{Definition: strict and semistrict connection}
		\begin{tikzcd}
			0 \arrow[r] & {\rm{Ad}}(\mb{E}) \arrow[d] \arrow[r, "\delta^{/\mb{G}}"] & {\rm{At}}(\mb{E}) \arrow[d] \arrow[l, "R", bend left=49] \arrow[r, "\pi_{*}^{/\mb{G}}"] & T\mb{X} \arrow[r] \arrow[d] \arrow[, "", bend left=49] & 0 \\
			0 \arrow[r] & \mb{X} \arrow[r, "\rm{id}"']          & \mb{X} \arrow[r, "\rm{id}"']                                       & \mb{X} \arrow[r]                                        & 0
		\end{tikzcd}
	\end{equation}
	its associated Atiyah sequence. A 1-morphism $R: {\rm At}(\mb{E}) \ra {\rm Ad}(\mb{E})$ of VB-groupoids is said to be a \textit{strict connection on $\pi \colon \mb{E} \ra \mb{X}$} if
	$$ R \circ \delta^{/\mb{G}}=1_{{\rm Ad}(\mb{E})}.$$
	
	A 1-morphism $R\colon  {\rm At}(\mb{E}) \ra {\rm Ad}(\mb{E})$ of VB-groupoids is called a  \textit{ semi-strict connection on $\pi \colon \mb{E} \ra \mb{X}$}  if  $R \circ \delta^{/\mb{G}}$ is  $2$-isomorphic to $1_{{\rm Ad}(\mb{E})},$
	$$R \circ \delta^{/\mb{G}}\simeq 1_{{\rm Ad}(\mb{E})}$$
	in the strict 2-category 2-$\rm{VBGpd}(\mb{X})$  (\Cref{Prop:naturaltransformationinVBgroupoids}).
\end{definition}
An immediate consequence of the categorification is the existence of the following groupoid of connection structures:
\begin{definition}[Groupoid of strict and semi-strict connections] \label{category of strict connection}
	For a Lie 2-group $\mb{G}$, let $\pi\colon \mb{E} \rightarrow \mb{X}$ be a principal $\mb{G}$-bundle over a Lie groupoid $\mb{X}$. We define \textit{the groupoid of  strict (resp. semi-strict) connections  for  $\pi \colon \mb{E} \ra \mb{X}$}  as a category whose objects are $R:{\rm At}(\mb{E}) \ra {\rm Ad}(\mb{E})$ such that $R$ is a strict (resp. semi-strict) connections on $\pi \colon \mb{E} \ra \mb{X}$ and 
	morphisms are  $2$-morphisms  $\eta\colon R \Longrightarrow R'$ of 2-$\rm{VBGpd}(\mb{X}).$	
	We denote the groupoid of strict and semi-strict connections respectively as  $C_{\mb{E}}^{\rm{strict}}$ and $C_{\mb{E}}^{\rm{semi}}.$
\end{definition}
\begin{remark}
	It is almost evident that a strict connection induces a functorial section  $\begin{tikzcd}
		T \mb{X} 	   \arrow[r, "\Sigma"] & {\rm{At}}(\mb{E})		\end{tikzcd},$	
	with respect to the VB-groupoid morphism, $\begin{tikzcd}
		{\rm At}(\mb{E})  \arrow[r, "\pi_{*}^{/\mb{G}}"] & T \mb{X} 
	\end{tikzcd},$
	resulting in splitting of the tangent bundles $TE_i\ra E_i, i\in\{0, 1\}$ into horizontal and vertical subbundles. For a semi-strict connection, the natural isomorphism $R \circ j^{/\mb{G}} \Longrightarrow 1_{{\rm Ad}(\mb{E})}$ poses an obstruction for the corresponding VB-groupoid morphism  $\begin{tikzcd}
		T \mb{X} 	   \arrow[r, "\Sigma"] & {\rm At}(\mb{E})	\end{tikzcd}$ to be a section of $\begin{tikzcd}
		{\rm At}(\mb{E})  \arrow[r, "\pi_{*}^{/\mb{G}}"] & T \mb{X}
	\end{tikzcd}$
	in a strict sense. But, clearly the natural isomorphism $R \circ \delta^{/\mb{G}}\Rightarrow 1_{{\rm Ad}(\mb{E})}$ defines a natural isomorphism between $\pi_{*}^{/\mb{G}} \circ \Sigma $ and $1_{T\mb{X}}$. Hence,  $\Sigma$ can be viewed as a section in a weaker sense. The reviewer of our paper \cite{MR4403617} suggested the name \textit{homotopy section} for such $\Sigma$.
	However, this issue is not explored in this thesis and will be pursued in future work.
\end{remark}

\subsection{Strict and semi-strict connections as Lie 2-algebra valued 1-forms on Lie groupoids}\label{Strict and semi-strict connections as Lie 2-algebra valued 1-forms on Lie groupoids}

In the classical framework \Cref{Section: Connection structures on a principal bundle}, we already observed that for a traditional principal bundle over a manifold, the splittings of its associated Atiyah sequence are in one-one correspondence with the connection 1-forms. To extend this association in the framework of principal 2-bundles over Lie groupoids, we begin by defining a notion of a Lie 2-algebra valued differential form on a Lie groupoid, whose description is tailored to our set-up.
\begin{definition}\label{Definition:LGvaluedformOnLiegroupoi}
	Given a Lie 2-group $\mb{G}$ and a Lie groupoid $\mb{E}$, an \textit{$L(\mb{G})$-valued $1$-form on $\mb{E}$} is defined as a morphism of Lie groupoids 
	$\omega:=(\omega_1,\omega_0)\colon T\mb{E} \rightarrow L(\mb{G})$ such that  $\omega_i$ is an $L(G_i)$-valued differential $1$-form on $E_i,$ for $i\in\{0, 1\}.$
	If $\mb{G}$ acts on $\mb{E}$ and $\omega\colon T\mb{E} \ra L(\mb{G})$ is  $\mb{G}$-equivariant with respect to \Cref{Adj action} and \Cref{Tangent action}, then $\omega$ said to be a \textit{$\mb{G}$-equivariant $1$-form on $\mb{E}$}.
\end{definition}

The notion of a differential form 
on a Lie groupoid, called a \textit{multiplicative form}, does exists in literature (for example, see \cite{MR2565034}). It would be convenient for later calculations if we express an $L(\mb{G})$-valued differential $1$-form  in terms of a Lie crossed module data (\Cref{Lie 2-algebra as Lie crossed module}). Moreover, such a description will make the relation between our definition and the notion of multiplicative forms more apparent. To be more precise, we have the following result.
\begin{lemma}\label{Multiplicative form condition}
	For a Lie crossed module $(G, H, \tau, \alpha)$ and a Lie groupoid $\mb{E}$, let $\omega_0 \colon TE_0 \ra L(G) $ and $\omega_1 \colon TE_1 \ra L(H \rtimes_{\alpha}G)$ be an $L(G)$-valued 1-form on $E_0$ and an $L(H \rtimes_{\alpha}G )$-valued 1-form on $E_1$ respectively. Then the following two are equivalent:
	\begin{itemize}
		\item[(i)] The pair $\omega_1, \omega_0$ defines a morphism of Lie groupoids $$\omega:=(\omega_1, \omega_0)\colon T\mb{E} \ra L \big( [H \rtimes_{\alpha}G \rra G] \big);$$
		\item[(ii)] The pair $\omega_1, \omega_0$ satisfies the following conditions:
		\begin{equation}\label{E:Relmultipl}
			\begin{split}
				&\omega_{1 G}=s^{*}\omega_0,\\
				&t^{*}\omega_0=s^{*}\omega_0+\tau (\omega_{1 H}),\\
				&m^{*}\omega_{1 H}={\rm pr}_1^{*}\omega_{1 H}+{\rm pr}_2^{*}\omega_{1 H},
			\end{split},
		\end{equation}
		where the notations $\omega_{1G}$ and $\omega_{1H}$ are respectively the $L(G)$-valued and $L(H)$-valued components of $\omega_1$. Here, $m\colon E_1\times_{s, E_0, t}E_1 \ra E_1$ and $\pr_i\colon E_1\times_{s, E_0, t}E_1 \ra E_1, i\in \{1, 2\}$, are respectively the composition map and projection maps to the first and second components.
	\end{itemize}
\end{lemma}
\begin{proof}:
	\subsection*{(i) $\Rightarrow$ (ii)}
	1st condition in \Cref{E:Relmultipl} directly follows from the source consistency of $\omega$. The second condition follows from the first, and the target consistency of $\omega$. 3rd condition is a direct consequence of the compatibility of $\omega$ with the composition map.
	
	\subsection*{(ii) $\Rightarrow$ (i)} Source consistency is a direct consequence of the 1st condition in \Cref{E:Relmultipl}. Target consistency follows from plugging in the first condition in the second one. Compositional compatibility follows from the third condition combined with the first condition. Smoothness is evident from the definition of $\omega_0$ and $\omega_1$. 
\end{proof}

\begin{remark}\label{multiplicative 1-form}
	In particular the 3rd condition in \Cref{E:Relmultipl} implies 
	$\omega_{1 H}$ is an $L(H)$-valued multiplicative form on the Lie groupoid $\mb{E}$. When $G$ is trivial in the Lie crossed module set-up, we get back the definition of a multiplicative form mentioned in  \cite{MR2565034}.
\end{remark}

\begin{remark}\label{Remark::Relation to existing Lie 2-algebra valued differential forms}
	It is worth mentioning  \cite{MR3894086} for a different, but related notion of an  $L(\mb{G})$-valued differential form on Lie groupoid $\mb{E}$, where the author considers the nerve $N(\mb{E})$ of the Lie groupoid $\mb{E}$ (\Cref{Definition: Nerve of a Lie groupoid}), and defines an $L(\mb{G})$-valued differential form as a `suitably chosen' subcomplex of the double complex $T^p:=\bigoplus_{p=i+j+k}\Omega^i({E_j}, {\mathfrak G}_{k}),$ where ${\mathfrak G}_{-1}=L(G), {\mathfrak G}_0=L(H), {\mathfrak G}_i=0\, \forall\, i\neq -1\, \rm{or}\, 0.$ Although the choice of the subcomplex
	described above is particularly motivated by the connection structure on a  principal $2$-bundle over a manifold, the
	motivation for our definition of $L(\mb{G})$-valued differential forms (\Cref{Definition:LGvaluedformOnLiegroupoi}) is to find an infinitesimal representation of the strict and semi-strict connection arising out of splitting of the associated Atiyah sequence (\Cref{strict and semi-strict connection definition}). Furthermore, observe that if we attach an $L(H)$-valued differential $2$-form on $E_0$ with our differential form $\omega$ in 
	\Cref{Definition:LGvaluedformOnLiegroupoi}, we get a differential $1$-form as defined in \cite{MR3894086}. 
\end{remark}
\begin{example}\label{E:Hemultiplicative}
	Let $\mb{E}$ be a Lie groupoid. Suppose the Lie $2$-group $\widehat H:=[H\rra {e}]$, associated to an abelian Lie group $H$ (\Cref{Ex:ordinary}) acts on $\mb{E}.$   Then it is easy to see that an $L(\widehat H)$-valued 
	$1$-form on $\mb{E}$ is  same as an $L(H)$-valued multiplicative $1$-form on $\mb{E}$ (see \Cref{multiplicative 1-form}). If the $L(H)$-valued multiplicative $1$-form on $\mb{E}$ is $H$-equivariant, then  the corresponding $L(\widehat H)$-valued $1$-form will  be $\widehat{H}$-equivariant 1-form on $\mb{E}.$
\end{example}
\begin{example}
	Given an action of a Lie group $G$ on a manifold $M,$  an $L(G)$-valued  $1$-form is same as an $L([G\rra G])$-valued $1$-form on the discrete Lie groupoid $[M\rra M].$ Also, if the 1-form on $M$ is $G$-equivariant, it is immediate to see that the corresponding 
	$L([G\rra G])$-valued $1$-form on the Lie groupoid $[M\rra M]$ is $[G\rra G]$ (\Cref{Ex:Discretecrossedmodule}) equivariant as per the \Cref{Definition:LGvaluedformOnLiegroupoi}.
\end{example}

\begin{example}\label{Ex:LGGvalued}
	Let the Lie $2$-group $[G\rra G]$ acts on a Lie  groupoid $\mb{E}$. Then an $L(\mb{G})$-valued $1$-form on the Lie groupoid $\mb{E}$ is same as
	an $L(G)$-valued $1$-form on $E_0$ satisfying $t^*\omega=s^*\omega.$ It is obvious that the equivariancy of one implies the equivariancy of the other.
\end{example}

\begin{example}\label{special decorated connections}
	Consider a Lie crossed module $(G, H, \tau, \alpha)$ and a Lie groupoid $\mb{E}$. Then there is another Lie groupoid $[E_1\times H\rra E_0]$  with the following structure maps:
	\begin{itemize}
		\item source, $(\gamma,h) \mapsto s(\gamma)$,
		\item target, $(\gamma,h) \mapsto t(\gamma) \tau(h^{-1})$,
		\item for composable $(\gamma_2,h_2), (\gamma_1, h_1)$, define $(\gamma_2, h_2)\circ (\gamma_1, h_1)=\bigl((\gamma_2\tau(h_1)) \circ \gamma_1, h_2h_1\bigr)$,
		\item unit, $p \mapsto (1_p,e)$,
		\item inverse, $(\gamma,h) \mapsto (\gamma^{-1} \tau(h^{-1}), h^{-1}\bigr).$
	\end{itemize}
	Now, assume an action of the discrete Lie 2-group $[G \rra G]$ on $\mb{E}$ (\Cref{Ex:LGGvalued}). This induces an action of $[H \rtimes_{\alpha}G \rra G]$ on $[E_1\times H\rra E_0]$, defined as $(p, g)\mapsto p g$ on objects and $ \big((\gamma, h),(h', g') \big)=(\gamma g,  \alpha_{g^{-1}}(h'^{-1}\, h))$ on morphisms. Suppose $\omega$ is an 
	$L(G)$-valued $G$-equivariant $1$-form on $E_0$, satisfying $s^*\omega=t^*\omega.$ Then the $L(H \rtimes_{\alpha}G)$-valued 1-form $\widetilde \omega$,  defined as $${\widetilde \omega}_{\gamma,\, h}(X, \mathfrak K):={\rm{ad}}_{(h, e)}\omega(s_{*\gamma}(X))-\mathfrak{K}\cdot h^{-1},$$ and $\omega$ gives an $L(\mb{G})$-valued $\mb{G}$-equivariant $1$-form on  
	Lie  groupoid $[E_1\times H\rra E_0].$   Later in this chapter, we will see that this particular example stems from a broader and more general construction.
\end{example}
The following definition is natural.
\begin{definition}[Groupoid of $\mb{G}$-equivariant $L(\mb{G})$-valued $1$-forms]\label{Def:Equivardiffcateg}
	Given an action of a Lie 2-group $\mb{G}$ on a Lie groupoid $\mb{E}$, the \textit{ {groupoid} of $\mb{G}$-equivariant $L(\mb{G})$-valued $1$-forms} on $\mb{E}$, denoted $\Omega_{\mb{E}^{\mb{G}}},$ is defined below:
	\begin{itemize}
		\item objects are given by $\mb{G}$-equivariant $1$-forms on $\mb{E}$, 
		\item morphisms are given by smooth natural isomorphisms  $\eta\colon \omega \Longrightarrow \omega'$,  such that $\eta\bigl((p, v)\cdot g\bigr)=\eta(p, v)\cdot 1_g$ for all $(p, v) \in TE_0, g \in G_0$, and $\eta:TE_0\ra L(G_1)$ is an $L(G_1)$-valued $1$-form on $E_0$. 
	\end{itemize}
\end{definition}

Before presenting a categorified analog of the one-one correspondence between connections as splitting of the Atiyah sequence and connections as differential 1-forms (\Cref{Equivalent characterisations of classical connection}), we make the following observation:
\begin{remark}\label{Associated G-equivariant L(G)-valued 1-form}
	A morphism of VB-groupoids $R \colon {\rm At}(\mb{E})\ra {\rm Ad}(\mb{E})$ 
	\begin{equation}\nonumber
		\begin{tikzcd}[sep=small]
			{\rm At}(\mb{E}) \arrow[rr, "R"] \arrow[dd,xshift=0.75ex,"\pi"]&  & {\rm{Ad}}({\mb{E}}) \arrow[dd,xshift=0.75ex,"\pi"] \\
			&  &                \\
			\mb{X} \arrow[rr]            &  & \mb{X}          
		\end{tikzcd},
	\end{equation}
	is of the following form
	\begin{equation}\nonumber
		\begin{split}
			&[(p, v)]\mapsto [(p, \omega(p, v))],  \forall [p, v]\in {\rm At}(E_0),\\
			& [(\widetilde \gamma, \widetilde X)]\mapsto [(\widetilde \gamma, \omega(\widetilde \gamma, \widetilde X))], \forall [(\widetilde \gamma, \widetilde X)]\in  {\rm At}(E_1)
		\end{split}
	\end{equation}
	and hence, defines a $\mb{G}$-equivariant $L(\mb{G})$-valued $1$-form  $\omega\colon T\mb{E} \rightarrow L(\mb{G}).$ Of course, the converse also holds. 
\end{remark}

\begin{proposition}\label{Prop:Corresconndiffform}
	For a Lie 2-group $\mb{G}$, let $\pi\colon \mb{E} \rightarrow \mb{X}$ be a  principal $\mb{G}$-bundle over a Lie groupoid $\mb{X}.$ Let $R \colon {\rm At}(\mb{E}) \ra {\rm Ad}(\mb{E})$ be a morphism of VB-groupoids and $\omega\colon T\mb{E} \rightarrow L(\mb{G})$ the associated  $\mb{G}$-equivariant $L(\mb{G})$-valued differential $1$-form  on the Lie groupoid $\mb{E}.$ Then we have the following:
	
	\begin{enumerate}[(i)]
		\item  $R \colon {\rm At}(\mb{E}) \ra {\rm Ad}(\mb{E})$ is  a strict connection as in \Cref{strict and semi-strict connection definition} if and only if  the following diagram  of morphisms of Lie groupoids
		\begin{equation} \label{Dia:condver}
			\begin{tikzcd}
				T{\mb{E}} \arrow[r, "\omega"]                  & L(\mb{G}) \\
				\mb{E} \times L(\mb{G}) \arrow[u, "\delta"] \arrow[ru, "\rm{pr_2}"'] &  
			\end{tikzcd},
		\end{equation}
		commutes on the nose. 
		
		\item $R \colon {\rm At}(\mb{E}) \ra {\rm Ad}(\mb{E})$ is  a semi-strict connection if and only if 
		the  Diagram~\ref{Dia:condver}  commutes up to a  $\mb{G}$-equivariant, fiber-wise linear natural isomorphism.
	\end{enumerate}
	\begin{proof}
		\begin{itemize}
			\item[(i)] Let $R$ be a strict connection. By definition, we have $R \circ j^{/\mb{G}}=1_{{\rm Ad}(\mb{E})}.$ From \Cref{Associated G-equivariant L(G)-valued 1-form}, there is an associated   $\mb{G}$-invariant $L(\mb{G})$ valued 1-form $(\omega_1, \omega_0)=\omega\colon T\mb{E} \rightarrow L(\mb{G}) $, which defines a pair of classical connections on princiapl $G_0$-bundle $E_0\ra X_0$ and principal $G_1$-bundle $E_1\ra X_1$ respectively. 
			
			So, by \Cref{Equivalent characterisations of classical connection},
			$\omega_0, \omega_1$ satisfy
			\[\omega_i\bigl(x_i, \delta_{x_i}(A_i)\bigr)=A_i, \qquad i\in \{0, 1\},
			\]  
			for all $x_i\in E_i, A_i\in L(G_i),$ giving the commutation relation in Diagram  \ref{Dia:condver}. 
			
			The converse also follows straightforwardly.
			
			\item[(ii)]  Now, let $R$ be a semi-strict connection. Suppose  $\epsilon\colon R \circ j^{/\mb{G}} \Longrightarrow 1_{{\rm Ad}(\mb{E})}$ is a $2$ isomorphism in $2$-$\rm VBGpd(\mb{X})$: 
			\begin{equation}\nonumber
				\begin{split}
					&\epsilon_{([p, B])}\colon R([p, \delta_p(B)])\longrightarrow [(p, B)].
				\end{split}
			\end{equation}
			
			Since  $\pi(\epsilon([p, B])=1_{\pi(p)},$ we claim that there is a unique $\bigl(\omega(p, {\delta_p(B)})\xrightarrow{\kappa_{(p, B)}}B\bigr)\in L(G_1)$ such that $\eta([p, B])=[(1_p, \kappa_{(p, B)})],$ for a chosen $p\in E_0, B\in L(G_0).$ 
			To prove the claim, observe that if $\epsilon([p, B])=[(\gamma_{p, B}, \alpha_{p, B})],$ for $\gamma_{p, B}\in E_1, \alpha_{p, B}\in L(G_1),$ then $[s(\gamma_{p, B})]=[p]$ implies $\exists! g\in G_0$ such that $s(\gamma_{p, B})g=p.$ It means 
			$\epsilon([p, B])=[(\gamma_{p, B}1_g, {\rm{ad}}_{{1_g}^{-1}}(\alpha_{p, B}))],$ with $\gamma_{p, B}1_g$ (see \Cref{Notation conventions}) with source $p$, and it projects down to $1_{\pi(p)}=\pi(1_p)$ in the principal $G_1$-bundle $E_1\ra X_1.$ Thus, choosing the unique element of
			$G_1$ to translate $\gamma_{p, B}1_g$ to $1_p,$ we get the unique $\biggl(\omega(p, {\delta_p(B)})\xrightarrow{\kappa_{(p, B)}}B\biggr)\in L(G_1)$ which satisfy $\eta([p, B])=[(1_p, \kappa_{(p, B)})].$
			It can be easily seen that $\kappa\colon \omega\circ \delta \Longrightarrow \pr_2$ is a natural isomorphism. Also, since $\epsilon$ is smooth, so is $\kappa$. 
			Now suppose if we represent the equivalence class $[(p, B)]$ by some other element $(p', B')=\bigl(p g, {\rm{ad}}_{g^{-1}}(B)\bigr),$ we will get  $\epsilon([p', B'])=[(1_{p'}, \kappa_{(p', B')})]$ which implies $(1_{p'}, \kappa_{(p', B')})\sim (1_p, \kappa_{(p, B)}).$ This implies $\bigl(1_{p'}, \kappa_{(p', B')}\bigr)=\bigl(1_p, \kappa_{(p, B)}\bigr)\cdot 1_g$ [follows easily from the fact that $p'=p g$ and freeness of action on the fiber of $E_1\ra X_1$]. This provides us the equivariancy of $\kappa:$
			$$\kappa\bigl(p g, {\rm{ad}}_{g^{-1}}(B)\bigr)={\rm{ad}}_{{1_g}^{-1}}\bigl(\kappa(p, B)\bigr).$$
			
			On the converse direction, given a $\mb{G}$-equivariant, fiber-wise linear  natural transformation $\kappa\colon \omega\circ \delta \Longrightarrow \pr_2,$ we define $\epsilon\colon R \circ j^{/\mb{G}} \Longrightarrow 1_{{\rm Ad}(\mb{E})}$ as  $\epsilon([p, B]):=[(1_p, \kappa_{(p, B)})].$ 
			We claim this map is well-defined. To see this, note that if $(p', B')\in [(p, A)],$ i.e $p'=pg, B'={\rm{ad}}_{g}(B)$ for some $g\in G,$ then the equivariancy of $\kappa$ implies $(1_{p'}, \kappa_{(p', B')})=(1_p, \kappa_{(p, B)})1_g.$ Hence,
			$[(1_{p'}, \kappa_{(p', B')})]=[(1_p, \kappa_{(p, B)})].$ $\epsilon$ satisfies $\pi(\epsilon_{[p, B]})=1_{\pi(p)}.$
			The map is clearly smooth. It is a straightforward verification that $\eta$ is a natural transformation and a morphism in $2$-$\rm VBGpd(\mb{X})$.
		\end{itemize}
	\end{proof}
\end{proposition}
Now, we are ready to make the following definition:
\begin{definition}[Strict and semistrict connection 1-forms]\label{strict ans semi strict connetion 1-forms}
	For a Lie 2-group $\mb{G}$, let $\pi\colon \mb{E} \rightarrow \mb{X}$ be a principal $\mb{G}$-bundle over a Lie groupoid $\mb{X}.$ We will call a $\mb{G}$-equivariant $L(\mb{G})$-valued $1$-form  $\omega\colon T\mb{E} \rightarrow L(\mb{G})$ as a \textit{strict connection $1$-form} if it satisfies the property (i) in \Cref{Prop:Corresconndiffform} and a \textit{semi-strict connection $1$-form} if it satisfies the property (ii) in \Cref{Prop:Corresconndiffform}. 
\end{definition}

 Since strict (resp. semi-strict) connection 1-forms on a principal $\mb{G}$-bundle $\pi \colon \mb{E} \ra \mb{X}$ over a Lie groupoid $\mb{X}$ are defined as morphisms of Lie groupoids $T\mb{E} \ra L(\mb{G})$, their collection have natural groupoid structures as functor categories. More precisely, we have the following:
\begin{definition} [Groupoid of strict and semi-strict connection $1$-forms]\label{Def:Semisemistriccat}
	For a Lie 2-group $\mb{G}$, let $\pi\colon \mb{E} \rightarrow \mb{X}$ be a principal $\mb{G}$-bundle over a Lie groupoid $\mb{X}.$ The \textit{groupoid of strict and semi-strict connections} is defined below:
	\begin{itemize}
		\item   objects are respectively the strict connection $1$-forms and semi-strict 
		connection $1$-forms,
		\item  morphisms are the smooth natural isomorphisms  $\eta\colon \omega \Longrightarrow \omega'$ such that $\eta\bigl((p, v)\cdot g\bigr)=\eta(p, v)\cdot 1_g$ for all $(p, v) \in TE_0$ and $g \in G_0.$
	\end{itemize}
	We denote the groupoid of strict and semi-strict connections respectively by $\Omega_\mb{E}^{\rm{strict}}$ and $\Omega_\mb{E}^{\rm{semi}}.$ 
\end{definition}
It is clear that we have the following sequence of categories
$$\Omega_\mb{E}^{\rm{strict}} \subset \Omega_\mb{E}^{\rm{semi}}\subset \Omega_{\mb{E}^{\mb{G}}}.$$

For further investigating the natural isomorphism $\kappa\colon \omega\circ \delta \Longrightarrow \pr_2$ in the Diagram \ref{Dia:condver},
for a semi-strict connection $\omega\colon T\mb{E}\ra L(\mb{G})$, let us express the Lie 2-group $\mb{G}$ by its Lie crossed module $(G, H, \tau, \alpha)$, i.e $\mb{G}:=[H \rtimes_{\alpha}G \rra G]$.
So, for each $(p, B)\in E_0\times L(G_0)$ we have a ${\kappa{(p, B)}}$  of the form $\kappa(p, B)=(\kappa_{\omega}(p, B), \omega\bigl(p, {\delta_p(B)}\bigr),$ where $\kappa_{\omega}(p, B)\in L(H)$ and 
\begin{equation}\label{semi-strict connection: condition on omega_0}
	\omega\bigl(p, {\delta_p(B)}\bigr)-B=-\tau\bigl(\kappa_{\omega}(p, B)\bigr)
\end{equation}
Hence, $\kappa_{\omega}(p, B)$ measures the deviation  $\omega\bigl(p, {\delta_p(B)}\bigr)-B.$  Now, given a $\kappa$ there is a smooth fiber-wise linear map $\kappa_{\omega}\colon E_0\times L(G)\ra L(H)$. The following condition follows from the equivariancy of $\kappa$
\begin{equation}\label{E:equilambda}
	\kappa_{\omega}\bigl(p g, {\rm{ad}}_{g^{-1}} B\bigr)=\alpha_{g^{-1}}\bigl(\kappa_{\omega}(p, B)\bigr).
\end{equation}
As $\kappa$ is a natural transformation, for $\bigl(A\xrightarrow{K}B\bigr)\in L(G_1)$  and $\bigl(p\xrightarrow{\widetilde \gamma}q\bigr)\in L(E_1)$
the following diagram commutes,
\[
\begin{tikzcd}[sep=small]
	\omega_{0, p}\bigl(\delta_p(A)\bigr) \arrow[dd, "{\kappa(p, A)}"'] \arrow[rrr, "\omega_{1, \widetilde \gamma}\bigl((\delta_{\widetilde \gamma})(K) \bigr)"] &  &  & \omega_{0, p}\bigl(\delta_q(B)\bigr)  \arrow[dd, "{\kappa(q, B)}"] \\
	&  &  &                                                 \\
	A \arrow[rrr, "K"]                                                                            &  &  & B                                              
\end{tikzcd},\]
and hence, we arrive at the following condition
\begin{equation} \label{semi-strict connection: condition on omega_1}		 
	\omega_{1, \tilde{\gamma}}\bigl(\delta_{\tilde{\gamma}}(K)\bigr)-K=\bigl(\kappa_{\omega}(p, A)-\kappa_{\omega}(q, B),-\tau(\kappa_{\omega}(p, A))\bigr).
\end{equation}

As a consequence of \Cref{semi-strict connection: condition on omega_1}, we get the following uniqueness property:
\begin{lemma}
	For a Lie crossed module $(G, H,\tau, \alpha)$,	let $\omega$ be a semi-strict connection $1$-form on a principal $[H \rtimes_{\alpha}G \rra G]$-bundle over a Lie groupoid $\mb{X}$. Then the natural transformation $\kappa$ in Diagram \ref{Dia:condver} is unique.
\end{lemma}

Summarising the discussion on the role of $\kappa$, we conclude the following:
\begin{proposition}\label{Prop:Semisirictnatural}
	For a Lie crossed module $(G, H, \tau, \alpha)$, let $\pi \colon \mb{E} \ra \mb{X}$ be a principal $\mb{G}:=[H \rtimes_{\alpha}G \rra G]$-bundle over a Lie groupoid $\mb{X}$. A $\mb{G}$-equivariant $1$-form $\omega\colon T\mb{E}\ra L(\mb{G})$ is a semi-strict connection $1$-form if and only if the functor 
	$(\omega\circ \delta-\pr_2):={\widehat \kappa}_{\omega}$  is of the form $\widehat{\kappa}_{\omega}(p, B)=\tau(\kappa_\omega(p,B))$ and $\widehat{\kappa}_{\omega}(\gamma, K)=\bigl(\kappa_{\omega}(p, A)-\kappa_{\omega}(q, B),-\tau(\kappa_{\omega}(p, A))\bigr)$ for a smooth, fiber wise linear map $\kappa_{\omega}\colon E_0 \times L(\mb{G}) \ra L(H)$ satisfying the \Cref{E:equilambda}. Moreover, $\omega$ is a strict connection $1$-form if and only if  $\omega\circ \delta-\pr_2$ is the zero functor.
\end{proposition}
Before prescribing a systematic way to construct semi-strict connections, we need to prove a pair of technical identities. Although, in our paper \cite{MR4403617} we left out the technical details of the proof, here we decide to fill in the gaps. Adhering to the notational convention as in \Cref{Remark:notations}, we state the following lemma.
\begin{lemma}\label{Lemma:Identityfor later} 		
	Let $(G, H, \tau, \alpha)$ be a Lie crossed module. Then,
	\begin{enumerate}
		\item for any $h_2, h_1\in H$ and $B\in L(G),$ we have 
		$$h_2\cdot \biggl({\bar \alpha}(h_2^{-1})({\rm ad}_{\tau(h_1)}B )\biggr)+h_1\cdot \biggl({\bar \alpha}(h_1^{-1}) (B) \biggr)=h_2h_1 \cdot \biggl({\bar \alpha}(h_1^{-1}h_2^{-1}) (B)\biggr).$$
		\item For any $g\in G, h\in H, A\in L(H)$
		$$\alpha_{g^{-1}}(h^{-1})\,\cdot \bigl({\bar \alpha}_{\alpha_{g^{-1}}(h)}   {({{\rm{ad}}}_{g^{-1}}(\tau(A)\bigr)}
		+\alpha_{g^{-1}}\bigl(A\bigr) =\alpha_{g^{-1}}({\rm{ad}}_{h^{-1}}(A)).$$
		
	\end{enumerate}
	
\end{lemma}
\begin{proof}
	\begin{enumerate}
		\item	We will show that 
		for any $h_2, h_1\in H$ and $g\in G,$ we have 
		$$\bigg[h_2 \biggl({\bar \alpha}(h_2^{-1})({\rm ad}_{\tau(h_1)}(g) )\biggr)\bigg] \bigg[h_1 \biggl({\bar \alpha}(h_1^{-1}) (g) \biggr)\bigg]=h_2h_1 \biggl({\bar \alpha}(h_1^{-1}h_2^{-1}) (g)\biggr).$$
		Then, the required identity follows directly as the infinitesimal version of the above equation on both sides. 
		\begin{equation}\nonumber
			\begin{split}
				&h_2 \bigg[\biggl({\bar \alpha}(h_2^{-1})({\rm ad}_{\tau(h_1)}(g) )\biggr)\bigg] h_1\bigg[ \biggl({\bar \alpha}(h_1^{-1}) (g) \biggr)\bigg]\\
				&=h_2 \bigg[\biggl({\alpha}({\rm ad}_{\tau(h_1)}(g))(h_2^{-1}))\biggr)\bigg] h_1\bigg[ \biggl({\bar \alpha}(h_1^{-1}) (g) \biggr)\bigg]\\
				&=h_2 \bigg[\biggl({\alpha}\bigl(\tau(h_1)g\tau(h_1)^{-1}\bigr)(h_2^{-1}))\biggr)\bigg] h_1\bigg[ \biggl({\bar \alpha}(h_1^{-1}) (g) \biggr)\bigg]\\
				&=h_2 \bigg[\biggl({\alpha}\bigl(\tau(h_1)g)\underbrace{\bigl(\alpha(\tau(h_1)^{-1}\bigr)(h_2^{-1})}_{h_1^{-1}h_2^{-1}h_1 \textit{\rm by}\, \eqref{E:Peiffer}}\bigr)\biggr)\bigg] h_1\bigg[ \biggl({\bar \alpha}(h_1^{-1}) (g) \biggr)\bigg]\\
				&=h_2 \bigg[\biggl({\alpha}\bigl(\tau(h_1)g)\bigl((h_1^{-1}h_2^{-1}h_1)\bigr)\biggr)\bigg] h_1\bigg[ \biggl({\bar \alpha}(h_1^{-1}) (g) \biggr)\bigg]\\
				&=h_2 \bigg[\biggl({\alpha}(\tau(h_1))\bigl(\alpha(g)(h_1^{-1}h_2^{-1}h_1)\bigr)\biggr)\bigg] h_1\bigg[ \biggl({\bar \alpha}(h_1^{-1}) (g) \biggr)\bigg]\\
				&=h_2 \bigg[\underbrace{\biggl({\alpha}(\tau(h_1))\bigl(\alpha(g)(h_1^{-1}h_2^{-1}h_1)\bigr)\biggr)}_{h_1 \alpha(g)(h_1^{-1} h_2^{-1} h_1) h_1^{-1} \textit{\rm by}\,  \eqref{E:Peiffer}}\bigg] h_1\bigg[ \biggl({\bar \alpha}(h_1^{-1}) (g) \biggr)\bigg]\\
				&=h_2 \bigg[h_1 \alpha(g)(h_1^{-1} h_2^{-1} h_1) h_1^{-1}\bigg] h_1\bigg[ \biggl({\bar \alpha}(h_1^{-1}) (g) \biggr)\bigg]\\
				&=h_2 h_1 \alpha(g)(h_1^{-1} h_2^{-1} h_1) [h_1^{-1} h_1] {\alpha}(g)(h_1^{-1})\\
				&=h_2 h_1 \alpha(g)(h_1^{-1} h_2^{-1} h_1) {\alpha}(g)(h_1^{-1})\\
				&=h_2 h_1 \alpha(g)(h_1^{-1} h_2^{-1} )\\
				&=h_2h_1  \biggl({\bar \alpha}(h_1^{-1}h_2^{-1}) (g)\biggr). 
			\end{split}
		\end{equation}
		Note that in the fourth, sixth, and eleventh steps, we have used the fact that $\alpha$ is an action of $G$ on $H$.

		\item A straightforward verification gives  $$\big[\alpha_{g^{-1}}(h^{-1})\,\cdot \bigl({\bar \alpha}_{\alpha_{g^{-1}}(h)}   {({{\rm{ad}}}_{g^{-1}}(\tau(h')\bigr)}\bigr]
		\big[\alpha_{g^{-1}}\bigl(h'\bigr)\big] =\alpha_{g^{-1}}({\rm{ad}}_{h^{-1}}(h')),$$ for any $h, h'\in H, g\in G.$ Then the identity follows as the infinitesimal version of the above equation.

	\end{enumerate}
\end{proof}

\begin{corollary}\label{cor:stricconnto semi}
	For a Lie crossed module $(G, H, \tau, \alpha)$, let $\omega=(\omega_1,\, \omega_0)\colon T\mb{E}\ra L(\mb{G})$ be a strict connection $1$-form on a principal $\mb{G}:=[H \rtimes_{\alpha}G \rra G]$-bundle $\pi \colon \mb{E} \ra \mb{X}$ over a Lie groupoid $\mb{X}$.
	Suppose $\lambda\colon T E_0\ra L(H)$ is
	an $L(H)$-valued $1$-form on $E_0$ that satisfies the equivariance property
	$\lambda(p g, v\cdot g)=\alpha_{g^{-1}}\bigl(\lambda(p, v)\bigr)$ for all $(p,v) \in TE_0$ and $g \in G$.
	Then the pair $(\tilde{\omega}_1, \tilde{\omega}_0)$ defined by ${\widetilde \omega}_1=\omega_1 + \tau (s^{*}\lambda)+ t^{*}\lambda-s^{*}\lambda$ and ${\widetilde \omega}_0=\omega_0 + \tau (\lambda)$ defines a semi-strict connection $\widetilde \omega=({\widetilde \omega}_1, {\widetilde \omega}_0)\colon T\mb{E}\ra L(\mb{G})$ on $\pi \colon \mb{E} \ra \mb{X}$. 
\end{corollary}
\begin{proof}
	The functoriality of $({\widetilde \omega}_1, {\widetilde \omega}_0)$ is obvious. As $\omega_{0 p} (\delta_p(A))=A,$ for $A\in L(G_0),$ only thing that we need to verify   is the $\mb{G}$-equivariancy of the functor 
	\begin{equation}\nonumber
		\begin{split}
			&\Lambda\colon T{\mb{E}}\ra L(\mb{G}),\\
			&(p, v)\mapsto \tau\bigl(\lambda (p, v)\bigr),\\
			&(\widetilde \gamma, \widetilde X)\mapsto  \tau \bigl(\lambda(s(\widetilde \gamma), s_{*, \widetilde \gamma}(X) )\bigr)+\lambda\bigl(t(\widetilde \gamma), t_{*, \widetilde \gamma}(X)\bigr) -\lambda\bigl(s(\widetilde \gamma), s_{*, \widetilde \gamma}(X)\bigr).
		\end{split}
	\end{equation}
	Then $\Lambda\bigl((p, v)g\bigr)=\tau\bigl(\lambda(p g, v\cdot g)\bigr)={\rm{ad}}_{g^{-1}}\bigl(\Lambda(p, v)\bigr).$ In order to show the equivariancy at the morphism level, consider $s(\widetilde \gamma, \widetilde X)=(p, v), t(\widetilde \gamma, \widetilde X)=(q, w).$ Then, we have
	\begin{equation}\nonumber
		\begin{split}
			&\Lambda\bigl((\widetilde \gamma, \widetilde X) (h, g)\bigr)\\
			&=\tau \bigl(\lambda(p g, v g)\bigr)-\lambda(p g, v g)+\lambda(q \tau(h) g, w \tau(h) g)\\
			&={\rm{ad}}_{(e, g)^{-1}}\tau \bigl(\lambda(p , v)\bigr)-\alpha_{g^{-1}}\bigl(\lambda(p, v)\bigr)+\alpha_{g^{-1}\tau(h^{-1})}\bigl(\lambda(q, w)\bigr)\\
			&={\rm{ad}}_{(e, g)^{-1}}\tau \bigl(\lambda(p , v)\bigr)-\alpha_{g^{-1}}\bigl(\lambda(p, v)\bigr)+\underbrace{\alpha_{g^{-1}}\bigl({\rm{ad}}_{h^{-1}}(\lambda(q, w))\bigr)}_{{\rm{ad}}_{(h, g)^{-1}}(\lambda(q, w))}\\
			&=\underbrace{{\rm{ad}}_{(h, g)^{-1}}\tau \bigl(\lambda(p , v)\bigr)-\alpha_{g^{-1}}(h^{-1})\,\cdot \bigl({\bar \alpha}_{\alpha_{g^{-1}}(h)}{({\rm ad}_{g^{-1}}(\tau(\lambda(p, v))))}\bigr)}_{{\rm{ad}}_{(e, g)^{-1}}\tau \bigl(\lambda(p , v)\bigr) [\textit {using formulae in \Cref{E:Adjonalgebras}}] }\\
			&\underbrace{-{\rm{ad}}_{(h, g)^{-1}}\bigl((\lambda(p , v) \bigr)+\bigl(\alpha_{g^{-1}}({\rm{ad}}_{h^{-1}}(\lambda(p, v)))-\alpha_{g^{-1}}\bigl(\lambda(p, v)\bigr) \bigr)}_{-\alpha_{g^{-1}}\bigl(\lambda(p, v)\bigr) [\textit{using formulae in \Cref{E:Adjonalgebras}}]}+{\rm{ad}}_{(h, g)^{-1}}(\lambda(q, w))\\
			&={\rm{ad}}_{(h, g)^{-1}}\tau \bigl(\lambda(p , v)\bigr)-{\rm{ad}}_{(h, g)^{-1}}\bigl((\lambda(p , v) \bigr)+{\rm{ad}}_{(h, g)^{-1}}(\lambda(q, w))\\
			&\underbrace{+[-\alpha_{g^{-1}}(h^{-1})\,\cdot \bigl({\bar \alpha}_{\alpha_{g^{-1}}(h)}   {({{\rm{ad}}}_{g^{-1}}(\tau(\lambda(p, v))))\bigr)}
				+\bigl(\alpha_{g^{-1}}({\rm{ad}}_{h^{-1}}(\lambda(p, v)))   - \alpha_{g^{-1}}\bigl(\lambda(p, v)\bigr) \bigr)]}_{\textit{vanishes by (2) of Lemma~\ref{Lemma:Identityfor later}}}\\
			&={\rm{ad}}_{(h, g)^{-1}}\tau \bigl(\lambda(p , v)\bigr)-{\rm{ad}}_{(h, g)^{-1}}\bigl((\lambda(p , v) \bigr)+{\rm{ad}}_{(h, g)^{-1}}(\lambda(q, w)).
		\end{split}
	\end{equation}
	Hence, we showed $\tilde{\omega}$ is a semi-strict conenction on $\pi \colon \mb{E} \ra \mb{X}$.
\end{proof}	

\begin{example}\label{Classical connection as 2-connection}
	For a Lie group $G$, let $\pi \colon P \ra M$ be a traditional principal $G$-bundle over a smooth manifold $M$ (\Cref{Definition: Principal G-bundle}). Any connection 1-form $\omega$ (\Cref{Section: Connection structures on a principal bundle}) on $P\ra M$ defines a strict connection 1-form $(\omega, \omega)$ on the principal $[G\rra G]$-bundle  
	$[P\rra P]\ra [M\rra M]$ over the discrete Lie groupoid $[M\rra M]$. It is obvious that, in this case, a strict connection is the same as a semi-strict connection.
\end{example}
\begin{example}\label{Ex:GGbundleconnection}
	For a Lie group $G$, let $\pi \colon \mb{E} \ra \mb{X}$ be a principal $[G \rra G]$-bundle over a Lie groupoid $\mb{X}$. Recall, we observed in \Cref{Principal 2-bundle with discrete Lie 2 group} that the Lie groupoid $\mb{E}$ is same as the semi-direct product groupoid $[s^*E_0\rra E_0]$ of the underlying action of $\mb{X}$ on $E_0$ (see \Cref{Equivalence of principal G bundles and principal G groupoids} for the precise description of the action).  Then 
	a strict connection $\omega$ on the principal $[G\rra G]$-bundle $\pi \colon [s^*E_0\rra E_0] \ra [X_1 \rra X_0] $ is a connection on the principal $G$-bundle $\pi_0 \colon E_0 \ra X_0$ such that $s^*\omega=t^*\omega$. As the Lie crossed module corresponding to $[G \rra G], $ is  $(G, \lbrace e \rbrace),$ by \Cref{Prop:Semisirictnatural}, a semi-strict connection is the same as a strict connection when the structure group is a discrete  Lie $2$-group

\end{example}
\begin{remark}
	Observe that in \Cref{Ex:GGbundleconnection}, we get back the definition of a connection on a principal $G$-bundle over a Lie groupoid $\mb{X}$, as mentioned in \Cref{Connections on a principal bundle over a Lie groupoid}.
\end{remark}
\begin{example}\label{E:Example of principal 2-bundle ordinary conn}
	For an abelian Lie group $H$, consider the principal  $[H\rra e]$-bundle $\pi \colon [E_1\rra X_0] \ra [X_1\rra X_0]$  over a Lie groupoid $\mb{X}$ (see \Cref{E:Example of principal 2-bundle ordinary}). Then, any strict connection is a classical connection $\omega$ on the principal $H$-bundle $E_1\ra X_1$ such that $\omega$ is a multiplicative $1$-form on the Lie groupoid $[E_1\rra X_0]$. Moreover, it can be readily observed that any $L(H)$-valued $H$-equivariant multiplicative 1-form on $[E_1 \rra X_0]$  defines a semi-strict connection (see \Cref{E:Hemultiplicative}).
\end{example}

\begin{example}\label{E:Exampleprincipalpairlie2conn }
	For a connection $\omega_0$ on a principal $G$-bundle $E_0\ra X_0,$ we define a strict connection on the principal $[G\times G\rra G]$-bundle $[E_1\rra E_0]$ over a Lie groupoid $\mb{X}$ in \Cref{E:Exampleprincipalpairlie2} as
	\begin{equation}\label{no meaning 1}
		\omega_{1 (p, \gamma, q)}(v_1, \mc{X}, v_2)=\bigl(\omega_0(v_1), \omega_0(v_2)\bigr)
	\end{equation}
	for all $(v_1, \mc{X}, v_2)\in T_{(p, \gamma, q)}E_1.$ Given an $L(G)$-valued $1$-form $\lambda$ on $E_0$ that satisfies the condition $\lambda(p.g, v.g)= {\rm{ad}}_{g^{-1}}( \lambda(p.v))$, we have a semi-strict connection 1-form given by  $\tilde{\omega}_0(v)=\omega_0(v) + \lambda(v)$, $\tilde{\omega}_{1 (p, \gamma, q)}(v_1, \mc{X}, v_2)=\bigl(\tilde{\omega}_0(v_1), \tilde{\omega}_0(v_2)\bigr)$.
\end{example}
The following example shows that a connection 1-form on a principal 2-bundle over a Lie groupoid behaves well with the pullback along morphisms of principal 2-bundles over the base, just like its classical counterpart (\Cref{Pullbackconn}).

\begin{example}\label{Lemma:Pullback connection}
	For a Lie 2-group $\mb{G}$, let $\pi \colon \mb{E} \ra \mb{X}$ and $\pi' \colon \mb{E}' \ra \mb{X}$ be a pair of principal $\mb{G}$-bundles over a Lie groupoid $\mb{X}$. Suppose $F =(F_1, F_0): \mb{E} \ra \mb{E}'$ be a morphism of prinicpal $\mb{G}$-bundles over $\mb{X}$.  If $\omega:=(\omega_1,\omega_0):T\mb{E}' \ra L(\mb{G})$ is a strict (resp. semi-strict) connection on $\mb{E}'$ then $F^{*}\omega:= (F_1^{*}\omega_1, F_0^{*}\omega_0):T \mb{E} \ra L(\mb{G})$ is a strict (resp. semi-strict) connection on principal $\mb{G}$-bundle $\pi \colon \mb{E} \ra \mb{X}$. To see this note that the differentials of $F_1$ and $F_0$ induce the morphism of Lie groupoids $F_{*}:=(F_{1_{*}}, F_{0_{*}}): T \mb{E} \ra T \mb{E}' $. Then the required connection on the principal $\mb{G}$ bundle $\mb{E} \ra \mb{X}$ is given by $\omega \circ F_{*}: T \mb{E} \ra L(\mb{G})$. We call $F^{*}\omega$ the \textit{pull-back connection of $\omega$ along $F$}.
\end{example}

\subsection{Categorical correspondence between connections as splittings and connections as Lie 2-algebra valued 1-forms}\label{Categorical correspondence strict and semistrict}
In \Cref{Prop:Corresconndiffform}, we have proved a one-one correspondence between strict  (resp. semi-strict) connections and strict (resp. semi-strict) connection $1$-forms for a principal $2$-bundle over a Lie groupoid. In this subsection, we obtain our first main result of this chapter by extending this correspondence to their respective categories.

\begin{theorem} \label{strict connection=strict forms}
	For a Lie 2-group $\mb{G}$, let $\pi \colon \mb{E} \ra \mb{X}$ be a principal $\mb{G}$-bundle over a Lie groupoid $\mb{X}$.
	\begin{enumerate}[(i)]
		\item The categories $C^{\rm{semi}}_{\mb{E}}$ and $\Omega_{\mb{E}}^{\rm{semi}}$ are isomorphic. 	
		\item The categories $C^{\rm{strict}}_{\mb{E}}$ and $\Omega_{\mb{E}}^{\rm{strict}}$ are isomorphic. 	
	\end{enumerate}	
	\begin{proof} 
		\begin{enumerate}[(i)]
			\item We have already seen the object level correspondence in \Cref{Prop:Corresconndiffform}. Let $(G,H, \tau, \alpha)$ be the associated Lie crossed module of $\mb{G}$, i.e $\mb{G}:=[H \rtimes_{\alpha}G \rra G]$. Suppose $R, R'\colon {\rm At}(\mb{E})\ra {\rm{Ad}}(\mb{E})$ are a pair of strict connections, and let $\eta\colon R \Longrightarrow R'$ a natural transformation such 
			that $\pi\bigl(\eta([p, v])\bigr)=1_{\pi(p)},$ for any $[(p, v)]\in {\rm At}(E_0)=TE_0/G_0.$ Suppose $\omega, \omega'\colon T\mb{E}\ra L(\mb{G})$ are respective semi-strict connection $1$-forms of $R$ and $R'.$ Observe that $R[(p, v)]=[\bigl(p, \omega(p, v)\bigr)], R'[(p, v)]=[\bigl(p, \omega'(p, v)\bigr)].$ Let $\eta_{[p, v]}\colon [\bigl(p, \omega(p, v)\bigr)]\ra [\bigl(p, \omega'(p, v)\bigr)].$ Our claim is that there exists $\omega(p, v)\xrightarrow{\bar \eta(p, v)}\omega'(p, v)\in L(G_1)$ such that $[\bigl(1_p, \bar \eta (p, v)\bigr)]=\eta_{[p, v]}.$	In order to show this, suppose $[(\gamma, K)]\in \eta_{[p, v]}.$ This implies that there exist $g, g'\in G_0$ such that $s(\gamma) g=p, t(\gamma) g'=p, {\rm{ad}}_{g^{-1}}\bigl(s(K)\bigr)=\omega(p, v)$ and  ${\rm{ad}}_{g'^{-1}}\bigl(t(K)\bigr)=\omega'(p, v).$ Then $(\gamma, K)\cdot 1_g=(\gamma 1_g, {\rm{ad}}_{{1_g}^{-1}}(K)):=(\gamma_p, K_{\omega})\in  \eta_{[p, v]}.$ Now, since $\pi\bigl(\eta([p, v])\bigr)=1_{\pi(p)},$ 	we get $\pi(\gamma_p)=\pi(1_p).$ Hence, both $\gamma_p, 1_p$ are elements of the same fiber on the $G_1$-bundle $E_1\ra X_1$ with the same source $p$ and thus there exists a unique $h\in H$ such that $\gamma_p (h, e)=1_p$. 
			
			Comparing the targets $t(\gamma_p)\tau(h)=t(\gamma)g \tau(h)=p=t(\gamma) g'$, we get $g'=g\tau(h)=g\tau(h)g^{-1}g=\tau(\alpha_g(h)) g.$ So, 
			$(\gamma_p, K_{\omega})\cdot (h, e)=\bigl(1_p, {\rm{ad}}_{(h, e)^{-1}}(K_{\omega})\bigr)	\in    \eta_{[p, v]}.$	Note that ${\rm{ad}}_{(h, e)^{-1}} K_{\omega}={\rm{ad}}_{(h, e)^{-1}} {\rm{ad}}_{{1_g}^{-1}}(K)={\rm{ad}}_{(h, e)^{-1}} {\rm{ad}}_{{(e, g)}^{-1}}(K)={\rm{ad}}_{(h^{-1}, g^{-1})} K.$ Using \Cref{E:Adjonalgebras} we immediately see $s({\rm{ad}}_{(h^{-1}, g^{-1})} K)={\rm{ad}}_{g^{-1}}\bigl(s(K)\bigr)=\omega(p, v)$ and $t({\rm{ad}}_{(h^{-1}, g^{-1})} K)={\rm{ad}}_{\tau(h^{-1})g^{-1}}\bigl(t(K)\bigr)={\rm{ad}}_{g'^{-1}}(t(K))=\omega'(p, v).$ We define $$\bar{\eta}(p, v):={\rm{ad}}_{(h, e)^{-1}} K_{\omega}.$$
			It is an easy verification that $\bar \eta$ defines a smooth natural transformation $\omega \Longrightarrow \omega'.$ 
			Now, if $(p', v')=(p, v) \theta$, for some $\theta \in G_0,$ then $(p',  v')\in [(p, v)].$ The above construction then shows $\bigl(1_{p'}, \bar \eta(p', v')\bigr)\in [\bigl(1_p, \bar \eta(p, v)\bigr)]$ and  
			$$\bigl(1_{p'}, \bar \eta(p', v')\bigr)=\bigl(1_p, \bar \eta(p, v)\bigr) 1_{\theta}=\bigl(1_{p \theta}, {\rm{ad}}_{1_{\theta}^{-1}}\bar \eta(p, v)\bigr).$$ 
			Thus, we obtain the $\mb{G}$ equivariancy condition: ${\rm{ad}}_{1_{\theta}^{-1}}\bigl(\bar\eta(p, v)\bigr)=\bar \eta(p \theta , v \theta).$
			
			Conversely, if $\bar \eta\colon \omega\Longrightarrow \omega'$ is a smooth, $\mb{G}$-equivariant natural transformation between the semi-strict connection $1$-forms, we define a natural transformation between the corresponding semi-strict connections 
			$R, R'\colon {\rm At}(\mb{E})\longrightarrow {\rm{Ad}}{\mb{E}},$ as $\eta_{[p, v]}=[\bigl(1_p, \bar \eta(p, v)\bigr)].$ 
			Well-definedness of the map comes from the $\mb{G}$-equivariancy. Now $\pi(\eta_{[(p, v)]})=\pi (1_{p})=1_{\pi(p)}.$ Hence, we get a smooth natural transformation $\eta\colon R \Longrightarrow R'$ which satisfies $\pi(\eta_{[(p, v)]})=1_{\pi(p)}.$
			
			It is straightforward to check that both the maps are functorial and inverses of each other. 
			
			\item The proof is almost similar to the semi-strict connection case.
		\end{enumerate}
	\end{proof}
\end{theorem}
\begin{remark}
	It should be noted that being an isomorphism of categories, \Cref{strict connection=strict forms} induces bijections both on objects and morphisms.
\end{remark}

The following description was not included in our paper \cite{MR4403617}.

\begin{definition}\label{Def:Setstrict connection=strict forms}
	For a Lie 2-group $\mb{G}$, let $\pi \colon \mb{E} \ra \mb{X}$ be a principal $\mb{G}$-bundle over a Lie groupoid $\mb{X}$. Then we define  ${\bar C}_{\mb{E}}^{\rm{strict}},{\bar C}_{\mb{E}}^{\rm{semi}},{\bar \Omega}_{\mb{E}}^{\rm{semi}}$, and ${\bar \Omega}_{\mb{E}}^{\rm{strict}}$	as the respective collections of connected components of the categories  
	${ C}_{\mb{E}}^{\rm{strict}},{C}_{\mb{E}}^{\rm{semi}},{ \Omega}_{\mb{E}}^{\rm{semi}}$, and ${ \Omega}_{\mb{E}}^{\rm{strict}}$.
\end{definition}

Then, as a direct consequence of \Cref{strict connection=strict forms}, we get the following correspondence.
\begin{corollary} \label{Cor:Setstrict connection=strict forms}
	For a Lie 2-group $\mb{G}$, let $\pi \colon \mb{E} \ra \mb{X}$ be a principal $\mb{G}$-bundle over a Lie groupoid $\mb{X}$. Then we have the following correspondences:
	\begin{itemize}
		\item[(a)]  There exists a bijection between ${\bar C}_{\mb{E}}^{\rm{strict}}$ and ${\bar \Omega}_{\mb{E}}^{\rm{strict}}$.
		\item[(b)] There exists a bijection between ${\bar C}_{\mb{E}}^{\rm{semi}}$ and ${\bar \Omega}_{\mb{E}}^{\rm{semi}}$.	
	\end{itemize}
\end{corollary}

\subsection{Connections on decorated principal 2-bundles over Lie groupoids}\label{SS:Conndeco}
In this subsection, we show a way to construct connection structures on decorated principal 2-bundles (\Cref{SS:Decorated}) from a connection data on the underlying principal Lie group bundle over the base Lie groupoid (\Cref{Connections on a principal bundle over a Lie groupoid}). However, we already observed a particular instance of the construction in \Cref{special decorated connections}.

For a Lie crossed module $(G, H, \tau, \alpha)$ and a principal $G$-bundle $\bigl(\pi\colon E_0 \rightarrow X_0, \mu, \mb{X} \bigr)$ over a Lie groupoid $\mb{X}$, consider the associated decorated principal $[H \rtimes_{\alpha}G \rra G]$-bundle $\pi^{\rm{dec}} \colon \mb{E}^{\rm dec} \ra \mb{X}$ (\Cref{Prop:Decoliegpd}).
Before we begin with the construction, we need to compute the differentials of some structure maps and the action maps associated to the Lie groupoid $\mb{E}^{\rm{dec}}$.
The source, target, and the composition map of the tangent Lie groupoid $\bigl[T(s^{*}E_0)^{\rm{dec}} \rightrightarrows TE_0\bigr]$ (\Cref{Tangent Lie groupoid})
can be computed as follows:
\begin{equation}\label{Differential of tangent dec}
	\begin{split}
		&{s}_{* ((\gamma, p), h)}((\mc{X}, v), \mathfrak K)=v,\\
		&{t}_{* ((\gamma, p), h)}((\mc{X}, v), \mathfrak K)=\mu_{* (\gamma, p)} (\mc{X}, v)\cdot \tau(h^{-1})-{\delta_{\mu(\gamma, p)\tau({h}^{-1})}
			( \tau(\mathfrak K)\cdot \tau(h^{-1}))},\\
		&\bigl((\gamma_2, p_2, h_2), (\mc{X}_2, v_2, {\mathfrak K}_2)\bigr)\circ \bigl((\gamma_1, p_1, h_1), (\mc{X}_1, v_1, {\mathfrak K}_1)\bigr)\\
		&=\bigl((\gamma_2\circ \gamma_1, p_1, h_2 h_1), (\mc{X}_2\circ \mc{X}_1, v_1, h_2\cdot {\mathfrak K}_1+{\mathfrak K}_2\cdot h_1\bigr).
	\end{split}
\end{equation}
Here, we adhered to the notational convention adopted in \Cref{Tangent Lie groupoid}, i.e. $\mc{X}_2 \circ \mc{X}_1= m_{*,(\gamma_2,\gamma_1)}(\mc{X}_2,\mc{X}_1)$, where $m$ is the composition map of the Lie groupoid $\mb{E}^{{\rm{dec}}}$.

Observe that the differential of the target map $t$ can be computed by applying the chain-rule in the composition of maps given below:
\[s^*E_0\times H\xra{({\rm Id},^{-1})}s^*E_0\times H\xra{(\mu,\tau)}E_0\times G\xra{} E_0.\]
To compute the vertical vector field generating functor $\delta\colon \mb{E}^{\rm dec}\times L(\mb{G})\to T(\mb{E}^{\rm dec})$ with respect to the action defined in  \Cref{Prop:Decoliegpd}, observe that for a fixed $((\gamma, p), h)\in s^*E_0\times H,$ the map $\delta_{((\gamma, p), h)}\colon H\rtimes G\ra s^*E_0\times H$ is given by $(h',\, g)\mapsto \bigl(\gamma, pg, \alpha_{g^{-1}} (h'^{-1}h)\bigr)$. The first coordinate is the constant map, the second coordinate is the right translation map, whereas the third coordinate can be written as the composition of the following maps:
\[H\times G\xra{^{-1},^{-1}}H\times G\xra{(-,-)}G\times H\xra{({\rm Id}, R_h)}G\times H\xra{\alpha}H.\] 
Then, by applying the chain-rule, we obtain the differential of $\delta$ computed below:
\begin{equation}\label{E:verdeco}
	\begin{split}
		\delta\colon \mb{E}^{\rm dec}\times L(\mb{G})&\ra T(\mb{E}^{\rm dec})\\
		(p, B) &\mapsto \delta_p(B)\\
		\bigl((\gamma, p, h) (A, B)\bigr)&\mapsto  \delta_p(B)-\bar{\alpha}_h(B)-A\cdot h.
	\end{split}
\end{equation}	
The differential of the right action of $\mb{G}$ on $\mb{E}$ are given as, for fixed $g\in G, (h', g)\in G_1,$
\begin{equation}\label{E:Hordeco} 
	\begin{split}
		v&\mapsto v g\\
		(\mc{X}, v, \mathfrak K)&\mapsto \bigl(\mc{X}, v g, \alpha_{g^{-1}}(h'^{-1}\cdot \mathfrak K)\bigr),	
	\end{split}
\end{equation}		
for $v\in T_p E_0, \bigl((\mc{X}, v), \mathfrak K\bigr)\in T_{((\gamma,\, p), h)}\bigl(s^{*}E_0^{\rm dec}\bigr).$
Note that here, we are using the notations of \Cref{Remark:notations}. 

Now, we are ready to begin the construction!

Suppose the underlying principal $G$-bundle $\bigl(\pi\colon E_0 \rightarrow X_0, \mu, [X_1 \rightrightarrows X_0]\bigr)$ over the Lie groupoid $\mb{X}$ admits a connection $\omega$ (\Cref{Connections on a principal bundle over a Lie groupoid}), that is $\omega$ is a connection on the pricnipal $G$-bundle $ \pi \colon E_0 \rightarrow X_0$ over the manifold $X_0$, which satisfies the condition
\begin{equation}\label{Source target connection condition}
	s^*\omega=t^*\omega,
\end{equation}
where $s, t\colon [s^*E_0\rra E_0$] are given by $s\colon (\gamma, p)\mapsto p,$ $t\colon (\gamma, p)\mapsto \mu(\gamma, p)$ respectivey. \Cref{Source target connection condition} is equivalent to the following:
\begin{equation}\label{E:pullconcond}
	\omega_p(v)=\omega_{\mu(\gamma, p)}\bigl(\mu_{*, (\gamma, \, p)}(\mc{X}, v)\bigr),
\end{equation}
for any $(\gamma, p)\in s^{*}E_0, (\mc{X}, v)\in T_{(\gamma, p)}\bigl(s^* E_0\bigr)$.

Now, we define an $L(H\rtimes G)$-valued differential 1-form on $s^*E_0\times H$ as follows: 
\begin{equation}\label{omega dec definition}
	\omega^{\rm dec}_{(\gamma, p,\, h)}(\mc{X}, v, \mathfrak K)={\rm{ad}}_{(h,\, e)}\bigl(\omega_p(v)\bigr)
	-\mathfrak{K} \cdot h^{-1}.
\end{equation}
We will show that $(\omega^{\rm dec}, \omega)$ is a strict connection $1$-form on the principal $\mb{G}$-bundle $\pi^{{\rm{dec}}} \colon \mb{E}^{\rm dec} \ra \mb{X}$ over $\mb{X}.$ To prove our claim, observe that $\omega^{\rm dec}$ can be equivalently writen as
$$\omega^{\rm dec}_{((\gamma,\, p), h)}(\mc{X}, v, \mathfrak K)={\rm{ad}}_{(h, e)}\bigl((s^*\omega)_{((\gamma, p), h)}((\mc{X}, v), \mathfrak K)\bigr)-{\Theta}_h(\mathfrak K),$$  
where $\Theta$ is the Maurer-Cartan  form on $H,$ $\bigl((\gamma, p), h\bigr)\in s^* E_0\times H$ and $\bigl((\mc{X}, v), \mathfrak K\bigr)\in T_{((\gamma,\, p), h)}(s^* E_0\times H)=T_{(\gamma, p)}s^{*}E_0\oplus T_hH.$

In the next lemma, we verify the functoriality of $(\omega^{\rm dec}, \omega)$.
\begin{lemma}\label{lemma:Funcdecconn}
	$(\omega^{\rm dec}, \omega)$ define a functor $T\mb{E}^{\rm dec}\ra L(\mb{G}).$
	\begin{proof}
		From \Cref{E:Adjonalgebras} we obtain
		\begin{equation}\nonumber
			\begin{split}
				&\omega^{\rm dec}_{(\gamma , p, h)}(\mc{X}, v, \mathfrak K)\\
				&={\rm{ad}}_{(h, e)}\bigl(\omega_p(v)\bigr)- \mathfrak{K}\cdot h^{-1}\\
				&=\underbrace{\omega_p(v)}_{\in L(G)}+\underbrace{h\cdot ({\bar \alpha}_{h^{-1}}({\omega_p(v)}))  -\mathfrak{K}\cdot h^{-1}}_{\in L(H)}.
			\end{split}
		\end{equation}
		
		Let $\bigl((\gamma_2, p_2, h_2), (\mc{X}_2, v_2, {\mathfrak K}_2)\bigr)$ and  $\bigl((\gamma_1, p_1, h_1), (\mc{X}_1, v_1, {\mathfrak K}_1)\bigr)$ be a pair of  morphisms in $T{\mb{E}}^{\rm dec}$ such that $$s \Big( \bigl((\gamma_2, p_2, h_2), (\mc{X}_2, v_2, {\mathfrak K}_2)\bigr) \Big)= t \Big(\bigl((\gamma_1, p_1, h_1), (\mc{X}_1, v_1, {\mathfrak K}_1) \Big)$$ which by \Cref{Differential of tangent dec}, imply the following:
		\begin{equation}\label{E:sour-targetmatch}
			\begin{split}
				&p_2=\mu(\gamma_1, p_1)\tau(h_1)^{-1}\\
				&v_2=\mu_{* (\gamma_1, p_1)} (\mc{X}_1, v_1)\cdot \tau(h_1^{-1})-\delta_{\mu(\gamma_1, p_1)\tau({h_1}^{-1})}(\tau({\mathfrak K}_1)\cdot\tau(h_1^{-1})).
			\end{split}
		\end{equation}
		To verify the source-target consistency, observe
		\begin{equation}\label{E:sourcetargetexpli}
			\begin{split}
				&s\bigl(\omega^{\rm dec}_{(\gamma_2 , p_2, h_2)}(\mc{X}_2, v_2, \mathfrak K_2)\bigr)=\omega_{p_2}(v_2),\\
				&t\bigl(\omega^{\rm dec}_{(\gamma_1 , p_1, h_1)}(\mc{X}_1, v_1, \mathfrak K_1)\bigr)=\omega_{p_1}(v_1)+\tau (h_1^{-1}\cdot ({\bar \alpha}_{h_1}({\omega_{p_1}(v_1)}))
				-\mathfrak{K}_1\cdot h_{1}^{-1}).
			\end{split}
		\end{equation}
		Plugging \Cref{E:sour-targetmatch} into the above equations and observing that $\omega$ is a connection  on the principal $G$-bundle $\pi_0 \colon E_0\ra X_0$ satisfying $\omega_{(\mu(\gamma_1, p_1))}(\mu_{*}(\mc{X}_1, v_1))=\omega_{p_1}(v_1)$ (see \Cref{E:pullconcond}), we get the required source-target consistency. 
		
		For verifying the composition, we compute separately,
		$$\bigl(\omega^{\rm dec}_{(\gamma_2 , p_2, h_2)}(\mc{X}_2, v_2, \mathfrak K_2)\bigr)\circ \bigl(\omega^{\rm dec}_{(\gamma_1 , p_1, h_1)}(\mc{X}_1, v_1, \mathfrak K_1)\bigr)$$
		and $$\omega^{\rm dec}_{(\gamma_2 , p_2, h_2)\circ (\gamma_1, p_1, h_1)}\bigl((\mc{X}_2, v_2, \mathfrak K_2)\circ (\mc{X}_1, v_1, \mathfrak K_1)\bigr),$$ using last of the equations in \Cref{Differential of tangent dec}. As previously, we plug in the relations in \Cref{E:sourcetargetexpli} and use the properties of $\omega.$ Following expressions are not difficult to obtain:
		\begin{equation}\label{E:onecompo}
			\begin{split}
				&\bigl(\omega^{\rm dec}_{(\gamma_2 , p_2, h_2)}(\mc{X}_2, v_2, \mathfrak K_2)\bigr)\circ \bigl(\omega^{\rm dec}_{(\gamma_1 , p_1, h_1)}(\mc{X}_1, v_1, \mathfrak K_1)\bigr)\\
				=&\omega_{p_1}(v_1) 
				-\mathfrak{K}_1\cdot h_{1}^{-1}  -\mathfrak{K}_2\cdot h_{2}^{-1}\\
				+& \bigg[h_2\cdot \biggl({\bar \alpha}(h_2^{-1})({\rm ad}_{\tau(h_1)}(\omega_{p_1}(v_1)) )\biggr)+h_1\cdot \big({\bar \alpha}(h_1^{-1}) (\omega_{p_1}(v_1)) \big)\bigg]
			\end{split}
		\end{equation}
		and
		\begin{equation}\label{E:secondcompo}
			\begin{split}
				&\omega^{\rm dec}_{(\gamma_2 , p_2, h_2)\circ (\gamma_1, p_1, h_1)}\bigl((\mc{X}_2, v_2, \mathfrak K_2)\circ (\mc{X}_1, v_1, \mathfrak K_1)\bigr)\\
				=& \omega_{p_1}(v_1)-\mathfrak{K}_1\cdot h_{1}^{-1}  -\mathfrak{K}_2\cdot h_{2}^{-1}
				+ \bigg[h_2h_1 \cdot \big({\bar \alpha}(h_1^{-1}h_2^{-1}) (\omega_{p_1}(v_1))\big)\bigg].		
			\end{split}
		\end{equation}
		Then, \Cref{Lemma:Identityfor later} ensures the terms inside [\, ] on both equations are the same. Compatibility of $(\omega^{\rm dec}, \omega)$ with the unit map is a straightforward verification. Hence, we conclude $(\omega^{\rm dec}, \omega)$ is a functor $T\mb{E}^{\rm dec}\ra L(\mb{G}).$
	\end{proof}
\end{lemma}	
The next result will establish the compatibility of $\omega^{\rm{dec}}$ with $H \rtimes_{\alpha}G$-action and the fundamental vector field.
\begin{lemma}\label{lemma:Connectioncdecconn}
	$\omega^{\rm dec}$ is a connection on $H \rtimes_{\alpha}G=G_1$-bundle  $s^*E_0\times H\ra X_1.$
	\begin{proof}	
		Using \Cref{E:Hordeco,E:verdeco}  one can easily verify
		$$\omega^{\rm dec}_{((\gamma,\, p),\, h)\cdot (h',\, g')}\bigl((\mc{X},\, v,\, \mathfrak K))\cdot (h',\, g')\bigr)={\rm{ad}}_{(h'\, g')^{-1}}\bigl(\omega^{\rm dec}_{((\gamma,\, p),\, h)}\bigl((\mc{X},\, v),\, \mathfrak K\bigr)\bigr)$$
		and $$\omega^{\rm dec}_{((\gamma,\, p),\, \, h)}\bigl(\delta^{(A+B)}_{((\gamma, \, p),\,\, h))})=A+ B$$ for $A\in L(H), B\in L(G).$
\end{proof}	\end{lemma}	
Combining \Cref{lemma:Funcdecconn} and \Cref{lemma:Connectioncdecconn}, we summarise the discussion above in the following proposition:
\begin{proposition}\label{Prop:ConOnDeco}
	Let $\bigl(\pi\colon E_0 \rightarrow X_0, \mu, \mb{X} \bigr)$ 
	be a principal $G$-bundle over a Lie groupoid $\mb{X}$ and $\omega$ a connection (\Cref{Connections on a principal bundle over a Lie groupoid}) on it. Consider a Lie crossed module $(G, H, \tau, \alpha)$. Then  $(\omega^{\rm dec}, \omega)$  is a strict connection $1$-form on the corresponding decorated principal $[H \rtimes_{\alpha}G \rra G]$-bundle $\pi^{{\rm{dec}}} \colon \mb{E}^{\rm dec} \ra \mb{X}$, where $\omega^{\rm dec}$ is as defined in \Cref{omega dec definition}. By the appropriate modification of the connection $(\omega^{\rm{dec}},\omega)$ using \Cref{cor:stricconnto semi}, we get a semi-strict connection.
\end{proposition}
\begin{corollary}\label{Ex:EoXodecobundleconnect}
	Every principal $2$-bundle over a discrete Lie groupoid admits a strict connection.
	\begin{proof}For a Lie 2-group $\mb{G}$, let $\pi \colon \mb{E} \ra [M \rra M]$  be a principal $\mb{G}$-bundle over a discrete Lie groupoid $[M\rra M]$. By \Cref{Corollary:discreteisdecorated}, every such principal $2$-bundle is a decorated principal $\mb{G}$-bundle. Then, \Cref{Prop:ConOnDeco} ensures that any choice of a connection on the underlying principal $G_0$-bundle over $M$ induces a strict connection on $\pi \colon \mb{E} \ra [M \rra M]$. 
	\end{proof}
\end{corollary}

The pair of groups $({\mathbb R}^n\rtimes G, G),$ where $G\subset GL(n, {\mathbb R})$ is of special interest in  Cartan geometry (See \cite{kobayashi1956connections},\cite{slovak2009parabolic}, \cite{cattafi2021cartan}). Here, we casually note an observation that relates our construction with that of Cartan connections (without going into the specific details).  	

\begin{example}
	Let $(G, H, \tau, \alpha)$ be a Lie crossed module and $E \ra M$ a principal $G$-bundle over a manifold $M$. Consider the principal $[H \rtimes_{\alpha}G 
	\rra G]$-bundle $[(E \times H)\rra E]$ over the Lie groupoid $[M \rra M]$ (see \Cref{Ex:EoXodecobundle}). Suppose $\theta$ is a \textit{Cartan connection} for the pair  $\bigl(G\subset {\ker(\tau)\rtimes G}\, , \ker(\tau)\rtimes G\bigr)$ on the principal $G$-bundle $E \ra M$ (\cite[Definition $2.2$]{cattafi2021cartan}), that is a $L(\ker(\tau)\rtimes G)$-valued $1$-form $\theta$ on $E$ which satisfies the following properties
	\begin{equation}\nonumber
		\begin{split}
			&\theta_p\colon T_p E\ra L(\ker(\tau)\rtimes G) \, {\rm is\, a \, linear\,  isomorphism \, for \, every}\, p\in E_0,\\ 
			&\theta_{p g}(v \cdot g)={\rm{ad}}_{(e, g)^{-1}}\bigl(\theta(v)\bigr),\qquad \forall p\in E, g\in G,\\
			&\theta_p \bigl(\delta_p(B)\bigr)=B, \qquad \forall B\in L(G).
		\end{split}
	\end{equation}

	Then we get a connection 
	\begin{equation}\nonumber
		\widetilde \theta_{(p, h)}(v, \mathfrak K)={\rm{ad}}_{(h, e)}\bigl(\theta_p(v)\bigr)-\mathfrak{K}\cdot h^{-1}
	\end{equation}	
	on the principal $H \rtimes_{\alpha}G$-bundle $[(E \times H)\rra E].$ The source maps $E \times H\ra E, (p, h)\mapsto p$ and $H \rtimes_{\alpha}G \to G, (h, g)\mapsto g$ define a map from the principal $H \rtimes_{\alpha}G$-bundle $E\times H\ra M$ to the principal $G$-bundle $E \ra M.$	
	By push forwarding the connection 	$\widetilde \theta$ on 
	$H \rtimes_{\alpha}G$-bundle $E \times H\ra M$ along this map of principal bundles we obtain a connection $\theta_0$ on the $G$-bundle $E \ra M.$ Then $(\widetilde \theta, \theta_0)$ defines a strict connection $1$-form on the principal $[H \rtimes_{\alpha}G 
	\rra G]$-bundle $[E \times H\rra E]$ over $[M \rra M].$ 
\end{example}

\subsection{On the existence of connection structures on a principal $2$-bundle over a Lie groupoid}\label{Existence of strict and semistrict connections}
In this subsection, we deduce an existential criterion for connection structures on a principal $2$-bundle over a Lie groupoid.

For a Lie group $G$ and a Lie groupoid $\mb{X}$, we have already seen that a principal $G$-bundle $\bigl(E_G \ra X_0, \mu, \mb{X}\bigr)$ over a proper \'etale Lie groupoid $\mb{X}$ (\Cref{proper etale Lie groupoid}) admits a connection (\Cref{Existence of connection on a Lie group bundle over a groupoid}). This means, there exists a connection $\omega$ on the principal $G$-bundle $E_G\ra X_0$ which satisfies
$$\omega_p=\omega_{\mu(\gamma,\, p)}.$$ 
In the following proposition, we show that the \Cref{prop:Characterisdecorated} provides us criteria for the existence of 
a strict connection and hence, for a semi-strict connection, on a principal Lie $2$-group bundle over a proper \'etale Lie groupoid.
\begin{proposition}\label{Existence}
	For a Lie crossed module $(G, H, \tau, \alpha)$, let $\pi \colon \mb{E} \ra \mb{X}$ be a principal $\mb{G}=[H \rtimes_{\alpha} G \rra G]$-bundle over a proper \'etale Lie groupoid $\mb{X}$. If $\pi \colon \mb{E} \ra \mb{X}$ admits a categorical connection, it also admits a strict connection.
\end{proposition}
\begin{proof}
	Since the principal $[H \rtimes_{\alpha} G \rra G]$-bundle $\pi \colon \mb{E} \ra \mb{X}$ admits a categorical connection, by \Cref{Lem:isodecgeneral} we have a decorated principal $\mb{G}$-bundle $\mb{E}^{\rm{dec}}:=[s^{*}E_0 \times H \rra E_0] $ over $\mb{X}$,  isomorphic to the principal $\mb{G}$-bundle  $\mb{E} \ra \mb{X}$. As $\mb{X}$ is a proper \'etale Lie groupoid, by \Cref{Existence of connection on a Lie group bundle over a groupoid}, the underlying principal $G$-bundle $E_0$ over $\mb{X}$ admits a connection $\omega$, and hence, by \Cref{Prop:ConOnDeco}, we have a strict connection 1-form $(\omega^{\rm{dec}},\omega)$ on $\pi^{{\rm{dec}}} \colon \mb{E}^{\rm{dec}} \ra \mb{X}$. Then, \Cref{Lemma:Pullback connection} completes the proof.
\end{proof}

\section{Gauge 2-group and its action on the category of connections}\label{Gauge 2-group and its action on the category of connections}
In \Cref{Definoition: Gauge group of a principal bundle}, we have seen the notion of a gauge group of a principal bundle over a manifold. We will see two kinds of generalization here: \textit{gauge 2-group} and \textit{extended gauge 2-group}, in the framework of a principal 2-bundle over a Lie groupoid. The first one offers a direct generalization of the classical one. In contrast, one obtains the second one by intertwining the first one with certain Lie algebra-valued 1-forms that result in a larger symmetry. A slightly weaker notion of the first one can be found in \cite{MR2805195}.
These 2-groups act on the category of strict and semi-strict connections \Cref{Def:Semisemistriccat}, a categorification of the classical one that we already saw in \Cref{Classical action of gauge group on connections}. We see, in particular, that the extended gauge 2-group fits very naturally in the framework of semi-strict connections (\Cref{strict and semi-strict connection definition}).

\subsection{Gauge 2-group of a principal 2-bundle over a Lie groupoid}\label{Gauge 2-group of a principal 2-bundle over a Lie groupoid}

\begin{definition}[Gauge $2$-group]\label{Remark:Gauge2-group}
	For a Lie 2-group $\mb{G}$, the \textit{gauge $2$-group of a principal $\mb{G}$-bundle $\mb{E}\ra \mb{X}$ over a Lie groupoid $\mb{X}$} is defined as the strict automorphism 2-group of the object $\pi \colon \mb{E} \ra \mb{X}$ in the strict $2$-groupoid ${\rm{Bun}}(\mb{X},\mb{G})$ (\Cref{strict 2-groupoid of principal 2-bundles}), which we denote by $\mc{G}(\mb{E})$.
\end{definition}
\footnote{Although we discussed Lie 2-groups extensively, we have not formally defined a strict 2-group. A \textit{strict 2-group} is a group object (\Cref{Group object}) in Cat, the category of small categories; see \cite{MR2068521} for details.}\subsection*{Strict 2-group structure of $\mc{G}(\mb{E})$}
Group products on objects and morphisms are respectively given by 
$(F, F')\mapsto F' \circ F$ and $(\Psi\colon F_1 \Longrightarrow F_2, \Psi '\colon F_1'  \Longrightarrow F_2 ')\mapsto \bigl(\Psi' \circ_{H}\Psi \bigr)\colon F_1'\circ F_1 \Longrightarrow F_2'\circ F_2.$  Here $\Psi\circ_H\Psi'$ is the horizontal composition of natural transformations:
$$\bigl(\Psi '\circ_H\Psi\bigr)_p=\Psi '_{F_2(p)}\circ F_1'(\Psi_p).$$

For readers interested in the proof of `any automorphism 2-group of an object in a strict 2-category is a strict 2-group', we refer to \textbf{Example 8.1.1} of \cite{MR2068521}.

We call the objects and morphisms of $\mc{G}(\mb{E})$ as the \textit{gauge transformations} (\textit{gauge transformations}) and the \textit{2-gauge transformations} of $\pi \colon \mb{E} \ra \mb{X}$, respectively. 

Explicitly, a gauge transformation is a $\mb{G}$-equivariant isomorphism of Lie groupoids $F \colon \mb{E} \rightarrow \mb{E}$ such that it preserves the fiber in the sense, $F\circ \pi=\pi.$ Obviously, a gauge transformation on $\pi\colon \mb{E} \rightarrow \mb{X}$ defines a pair of traditional gauge transformations  on the principal $G_0$-bundle $E_0\ra X_0$ and the principal $G_1$-bundle $E_1\ra X_1$, respectively (\Cref{subsection: gauge group of a principal bundle}). While a 2-gauge transformation between two such 1-gauge transformations $F, F' \colon \mb{E} \ra \mb{E}'$ is given by a smooth natural isomorphism $\eta \colon F \Longrightarrow  F'$ satisfying  
$\eta (p \, g)=\eta(p)1_g,$ and $\pi \bigl(\eta(p)\bigr)=1_{\pi(p)}$
for all $p\in E_0, g\in G_0.$

\subsection*{An alternative but equivalent description of a gauge 2-group}

Recall in \Cref{subsection: gauge group of a principal bundle}, we saw that for a Lie group $G$, the gauge group $\mc{G}(E)$ of a principal $G$-bundle $E \ra M$ over a manifold $M$ has an equivalent description in terms of certain $G$-valued smooth maps defined on the manifold $E$. Here, we see an analog for a principal 2-bundle over a Lie groupoid.

For a Lie 2-group $\mb{G}$, let $\pi \colon \mb{E} \ra \mb{X}$ be a principal $\mb{G}$-bundle over $\mb{X}$. Consider the category whose objects (gauge transformations) are morphisms of Lie groupoids $\sigma\colon \mb{E} \ra \mb{G}$ such that the diagram 
\[	\begin{tikzcd}[sep=small]
	\mb{E}\times \mb{G}\arrow[dd, ""'] \arrow[rrr, "\sigma\times {}^{-1}"] &  &  & \mb{G}\times \mb{G}  \arrow[dd, "{{\rm{Ad}}}"] \\
	&  &  &                                                 \\
	\mb{E} \arrow[rrr, "\sigma"]                                                                            &  &  & \mb{G}                                             
\end{tikzcd}\]
commutes on the nose, and morphisms (2-gauge transformations) are given by natural transformations $\Phi\colon \sigma_1\Longrightarrow \sigma_2$ which satisfies $\Phi(p g)={1_g}^{-1}\Phi(p) {1_g}$ for all $p \in E_0, g \in G_0$. We denote this category by $C^{\infty}(\mb{E},\mb{G})^{\mb{G}}$.
The strict $2$-group structure 
in $C^\infty(\mb{E},\mb{G})^{\mb{G}}$ is induced by pointwise multiplication. The  $2$-group  isomorphism between $\mc{G}(\mb{E})$ and $C^{\infty}(\mb{E},\mb{G})^{\mb{G}}$ is given as
\begin{equation}\label{E:gaugeequi}
	\begin{split}
		&F(x_i)=x_i\sigma(x_i), \forall x_i\in E_i, i\in \{0, 1\}\\
		&	\Psi(p)=1_p\Phi(p), \forall p\in E_0.
	\end{split}
\end{equation}
Consider the Lie 2-group $[H \rtimes_{\alpha}G \rra G]$ associated to a Lie crossed module $(G, H, \tau, \alpha)$. Then any object $\sigma\in C^{\infty}(\mb{E},\mb{G})^{\mb{G}}$ is of the form 
$\sigma(p)\in G$ and $\sigma(\widetilde \gamma)=(\sigma_{\widetilde \gamma}, g_{\widetilde \gamma})\in H \rtimes_{\alpha}G.$ The functoriality of $\sigma$ induces the following conditions on $g_{\widetilde \gamma}\in G, \sigma_{\widetilde \gamma}\in H$ :
\begin{equation}\label{E:Functorgauge}
	\begin{split}
		&g_{\widetilde \gamma}=\sigma_{s(\widetilde \gamma)},\\
		&\sigma_{t(\widetilde \gamma)}=\tau(\sigma_{\widetilde \gamma})\sigma_{s(\widetilde \gamma)},\\
		&\sigma_{{\widetilde \gamma}_2} \sigma_{{\widetilde \gamma}_1}=\sigma_{{\widetilde \gamma}_2\circ {\widetilde \gamma}_1}.
	\end{split}
\end{equation}
Using \Cref{E:Adjongroups}, one arrives at the following $\mb{G}$-equivariancy condition on $(\sigma_{\widetilde \gamma}, \sigma_p)$:
\begin{equation}\label{E:Gaugequi}
	\begin{split}
		&\sigma_{p \, g}={\rm{Ad}}_{g^{-1}}(\sigma_p)\\
		&\sigma_{{\widetilde \gamma}\, (h, g)}=\alpha_{g^{-1}}\bigl({\rm{Ad}}_{h^{-1}}\,(\sigma_{\widetilde \gamma})\bigr).
	\end{split}
\end{equation}
Let $\sigma, \sigma '\colon \mb{E}\ra \mb{G}$ be a pair of objects in $C^{\infty}(\mb{E},\mb{G})^{\mb{G}}$ and $\widetilde \Phi\colon \sigma \Longrightarrow \sigma '$ a morphism in $C^{\infty}(\mb{E},\mb{G})^{\mb{G}}.$ Then ${\widetilde \Phi}(p)\colon \sigma(p)\ra \sigma '(p)$ is of the form $( \Phi_{p},\sigma_p),$ where $\Phi(p)\in H$ satisfies the following conditions:
\begin{equation}\label{E:Natugauge}
	\begin{split}
		&\sigma'(p)=\tau(\Phi(p))\, \sigma(p),\\
		&\sigma'_{\widetilde \gamma}=\Phi_q\, \sigma_{\widetilde \gamma}\, {\Phi_p}^{-1},\\
		&\Phi_{p\, g}=\alpha_{g^{-1}}(\Phi_p),
	\end{split}
\end{equation}
for all $p\xrightarrow{\widetilde \gamma}q\in E_1, g\in G.$

Although our definition of gauge 2-group is a straightforward categorification of the classical one (\Cref{Definoition: Gauge group of a principal bundle}), we will see an enrichment that does not have a counterpart in the classical framework.

\subsection*{Categorical connections and gauge transformations}

For a Lie croseed module $(G, H, \tau, \alpha)$, let $\pi \colon \mb{E} \ra \mb{X}$ be a principal $\mb{G}:=[H \rtimes_{\alpha}G \rra G]$-bundle over a Lie groupoid $\mb{X}$, which admits a  categorical connection (\Cref{Def:categorical connection}). Observe that gauge transformations act naturally on the set of categorical connections by the following action:
\begin{equation}\label{E:Gaugeactcatcon}
	({\mathcal C}\cdot F) (\gamma, p):=F\bigl(\mathcal C(\gamma, F^{-1}(p))\bigr),
\end{equation}
where $F\colon \mb{E}\ra \mb{E}$ is a gauge transformation,  $\mathcal{C}$ is a categorical connection and $(\gamma, p)\in s^* E_0$. Then the corresponding transformation of the underlying action $\mu_{\mc{C}}$ (\Cref{underlying principal G-bundle}) is described by
$\mu_{\mc{C}}(\gamma,p)\mapsto {\mu}_F (\gamma, p):=F\bigl(\mu_{\mc{C}}(\gamma, F^{-1}(p))\bigr).$

Now, we choose  a categorical connection $\mc{C}$ to identify $\pi \colon \mb{E} \ra \mb{X}$ as a decorated  principal $[H \rtimes_{\alpha}G \rra G]$-bundle $\pi^{\rm{dec}} \colon \mb{E}^{\rm{dec}} \ra \mb{X}$ (see \Cref{Lem:isodecgeneral}).  
Suppose $\overline{\mb{E}}^{\rm{dec}}$ is the decorated principal $[H \rtimes_{\alpha}G \rra G]$-bundle corresponding to the action ${\mu}_F.$ Observe that both $\overline{\mb{E}}^{\rm{dec}}$ and  ${\mb{E}}^{\rm{dec}}$ have same set of objects and morphisms. 
\begin{proposition}
	Let $\theta_{{\mathcal C}\cdot F}\colon \overline{\mb{E}}^{\rm{dec}}\ra \mb{E}$ be the map defined in \Cref{Lem:isodecgeneral}.	
	Suppose $\sigma\colon {\mb{E}}\ra \mb{G}$ is a gauge transformation. Then $\sigma\circ \theta_{{\mathcal C}\cdot F} \colon \overline{\mb{E}}^{\rm{dec}}\ra \mb{G}$ defines a gauge transformation on $\overline{\mb{E}}^{\rm{dec}}.$
	\begin{proof}  It is an easy verification that $\sigma\circ \theta_{{\mathcal C}\cdot F} $ is a functor $\overline{\mb{E}}^{\rm{dec}}\ra \mb{G}$ satisfying conditions in \Cref{E:Gaugequi}. 
	\end{proof}	
\end{proposition}

\subsection*{Gauge transformations on decorated principal 2-bundles}
There is a simpler way to describe a gauge transformation on a decorated principal 2-bundle. To see this consider a principal decorated principal $[H \rtimes_{\alpha}G \rra G]$-bundle $ \pi^{{\rm{dec}}} \colon \mb{E} \ra \mb{X}$ associated to a Lie crossed module $(G, H, \tau, \alpha)$ and a principal $G$-bundle $(\pi \colon E \ra X_0, \mu, \mb{X} )$ over a Lie groupoid $\mb{X}$ (\Cref{Prop:Decoliegpd}). Observe that $\bigl(\gamma, p, h\bigr)=\bigl(\gamma, p, e\bigr)\, (h^{-1}, e)$ for $(\gamma,p) \in s^{*}E_0$ and $h \in H$. Hence, for any gauge transformation $F\colon {\mb{E}}^{\rm{dec}}\ra{\mb{E}}^{\rm{dec}}$ we have $$F(\gamma, p, h)=F(\gamma, p, e) (h^{-1}, e)$$ for $(\gamma,p) \in s^{*}E_0$ and $h \in H$. Thus the gauge transformation   $F\colon {\mb{E}}^{\rm{dec}}\ra{\mb{E}}^{\rm{dec}}$ is fully determined by a morphism of Lie groupoids $\bar{F} \colon [s^* E\rra E] \ra  [s^*E_0\times H \rra E]$ satifying $\bar{F}_0(p\, g)=\bar{F}_0(p)g$ for all $p \in E$, $g \in G$ and $\bar{F}_1(\gamma, p\, g)=\bar{F}_1(\gamma, p)\, 1_g$ for $(\gamma,p) \in s^{*}E_0$ and $g \in G$. Now, consider the inclusion $1\colon [G\rra G]\hookrightarrow [H\rtimes G\rra G]$. Then, it is clear that $(\bar{F},1)$ is a morphism of principal $2$-bundles $[s^*E \rra E] \ra [s^{*}E \times H \rra E]$ from the principal $[G\rra G]$-bundle $[s^*E \rra E] \ra \mb{X}$ to 
to the $[H\rtimes G\rra G]$-bundle $\mb{E}^{\rm{dec}}\ra \mb{X}$. The association is given by 
\begin{equation}\nonumber
	\begin{split}
		&F(p)=F_0(p);\\
		&F(\gamma, p, h)=F_0(\gamma, p)(h^{-1}, e).
	\end{split}
\end{equation}
Equivalently, if $\sigma\in C^{\infty}(\mb{E}^{\rm dec},\mb{G})^{\mb{G}}$ corresponds to the gauge transformation $F\colon {\mb{E}}^{\rm{dec}}\ra{\mb{E}}^{\rm{dec}},$ then $\sigma$ is fully determined  by a morphism of Lie groupoids $\bar{\sigma} \colon [s^*E \rra E] \ra [H \rtimes_{\alpha}G \rra G]$ satisfying $\bar{\sigma}_0(p\, g)=g^{-1}\bar{\sigma}_0(p) g$ for all $p \in E, g \in G$ and $\bar{\sigma}_1(\gamma, p \, g)=1_g^{-1}\bar{\sigma}_1(\gamma, p) 1_g$ for $(\gamma,p) \in s^{*}E$ and $g \in G$.
Then, $\sigma$ can be described as
\begin{equation}\nonumber
	\begin{split}
		&\sigma(p)=\bar{\sigma}_0(p),\\
		&\sigma\bigl((\gamma, p), h\bigr)=(h, e)\, \bar{\sigma}_1(\gamma, p)\, (h^{-1}, e).
	\end{split}
\end{equation}
The following proposition summarizes the above discussion:

\begin{proposition}\label{Prop:Decogauge}
	Let $\pi^{\rm{dec}} \colon \mb{E}^{\rm{dec}} \ra \mb{X}$ be the decorated principal $[H \rtimes_{\alpha}G \rra G]$-bundle associated to a Lie crossed module $(G, H, \tau, \alpha)$ and a principal $G$-bundle $(E \ra X_0\, \mu, \mb{X})$ over a Lie groupoid $\mb{X}$.
	Then a gauge transformation $\sigma\colon {\mb{E}}^{\rm{dec}}\ra{\mb{G}}$  is completely determined by a morphism of Lie groupoids $\bar{\sigma} \colon [s^*E \rra E] \ra [H \rtimes_{\alpha}G \rra G]$ satisfying $\bar{\sigma}_0(p\, g)=g^{-1}\bar{\sigma}_0(p) g$ for all $p \in E, g \in G$ and $\bar{\sigma}_1(\gamma, p \, g)=1_g^{-1}\bar{\sigma}_1(\gamma, p) 1_g$ for $(\gamma,p) \in s^{*}E$, $g \in G$ and  is given as
	\begin{equation}\nonumber
		\begin{split}
			&\sigma(p)=\bar{\sigma}_0(p),\\
			&\sigma\bigl((\gamma, p), h\bigr)=(h, e)\, \bar{\sigma}_1(\gamma, p)\, (h^{-1}, e).
		\end{split}
	\end{equation}
	Equivalently, it is completely determined by 
	a morphism of principal 2-bundles $(\bar{F},1) \colon [s^*E \rra E] \ra [s^{*}E \times H \rra E]$ over $\mb{X}$ satisfying $\bar{F}_0(p\, g)=\bar{F}_0(p)g$ for all $p \in E$, $g \in G$ and $\bar{F}_1(\gamma, p\, g)=\bar{F}_1(\gamma, p)\, 1_g$ for $(\gamma,p) \in s^{*}E_0$, $g \in G$ and is given by
	\begin{equation}\nonumber
		\begin{split}
			&F(p)=F_0(p);\\
			&F(\gamma, p, h)=F_0(\gamma, p)(h^{-1}, e),
		\end{split}
	\end{equation} 
	where $F \colon \mb{E}^{\rm{dec}} \ra \mb{E}^{\rm{dec}}$ is the gauge transformation corresponding to $\sigma$.
\end{proposition}

\begin{remark}
	While we touched upon gauge transformations within the context of categorical principal 2-bundles, an exploration of such transformations in the general framework of quasi-principal 2-bundles (\Cref{Definition:Quasicategorical Connection}) is yet to be made.
\end{remark}

\begin{example}
	Let $G$ be a Lie group. Every gauge transformation (as defined in \Cref{subsection: gauge group of a principal bundle}) on a principal $G$-bundle $ \pi \colon E\ra M$ over a manifold $M$ defines a gauge transformation on the principal $[G\rra G]$-bundle $(\pi, \pi) \colon [E\rra E] \ra [M \rra M]$ over the discrete Lie groupoid $[M \rra M]$ and vice versa. 
\end{example}

\begin{example}
	Let $G$ be a Lie group. We know that a principal $[G\rra G]$-bundle $[E_1\rra E_0]\ra [X_1\rra X_0]$  can be identified with the principal $[G \rra G]$-bundle $[s^*E_0\rra E_0]\ra [X_1\rra X_0]$ defined with respect to 
	an action $\mu\colon s^*E_0\ra E_0$ as described in \Cref{Equivalence of principal G bundles and principal G groupoids}. Then, every gauge transformation on a principal $[G\rra G]$-bundle $[E_1\rra E_0]\ra [X_1\rra X_0]$ is a gauge transformation $F\colon E_0\ra E_0$ on the principal $G$-bundle $E_0\ra X_0$ such that it satisfies $F\big(\mu (\gamma, p)\big)=\mu (\gamma,F(p))$.
\end{example}

\begin{example}\label{Ex:EoXodecobundlegaguge}
	Consider the decorated principal $[H \rtimes_{\alpha}G \rra G]$-bundle  ${E}^{\rm dec}:=[E\times H\rra E] \ra [M \rra M]$ over the discrete Lie groupoid $[M \rra M]$ associated to a Lie crossed module $(G, H, \tau, \alpha)$ constructed in \Cref{Ex:EoXodecobundle}. Then, any gauge transformation $F\colon E\ra E$ on the principal $G$-bundle $E\ra M$ over the manifold $M$ defines a gauge transformation $F^{\rm{dec}} \colon {E}^{\rm dec}\ra E^{\rm dec}$ on the $[H \rtimes_{\alpha}G \rra G]$-bundle $E^{\rm{dec}} \ra [M \rra M]$, given by $p\mapsto F(p)$ and $(p, h)\mapsto \bigl(F(p), h\bigr),$ for $p\in E, h\in H.$
\end{example} 

\begin{example}\label{E:Example of principal 2-bundle ordinary Guage}
	For an abelian Lie group $H$, consider the principal $[H\rra e]$-bundle  $[E_1\rra X_0]$ over a Lie groupoid $[X_1\rra X_0]$ (\Cref{E:Example of principal 2-bundle ordinary}).  Then a gauge transformation $F\colon E_1\ra E_1$ on the (classical) principal $H$-bundle $E_1\ra X_1$ determines a gauge transformation on the principal $[H\rra e]$-bundle $[E_1\rra X_0]$ over the Lie groupoid $[X_1\rra X_0]$, given by $x\mapsto x$ and $\widetilde \gamma \mapsto F(\widetilde \gamma),$ for $x\in X_0, \widetilde \gamma \in E_1$.  
\end{example}

\subsection{Action of a gauge $2$-group on the category of connections} \label{Action of Gauge $2$-group on the category of Connections}	
Recall that a gauge group of a classical principal bundle acts on connection 1-forms by pulling it back along gauge transformations (see \Cref{Classical gauge group action}). Here, we will see a generalization of this fact in the framework of a principal 2-bundle over a Lie groupoid equipped with strict or semi-strict connections.

For a Lie 2-group $\mb{G}$, let $\pi\colon \mb{E} \rightarrow \mb{X}$ be a principal $\mb{G}$-bundle over a Lie groupoid $\mb{X}$. Consider a gauge transformation  $F\colon \mb{E}\ra \mb{E}$, i.e an object of $\mc{G}(\mb{E})$ (\Cref{Remark:Gauge2-group}). Then, $F$ induces a $\mb{G}$-equivariant map between the corresponding tangent bundle  by taking the differentials of $F$ on both objects and morphisms, defined as
\begin{equation}\label{Object diff}
	\begin{split}
		F_*\colon& T\mb{E}\ra T\mb{E}\\
		&(p, v)\mapsto F_{*}\bigl(p, v\bigr):=\bigl(p, F_{*, p}(v)\bigr),\\
		& (\widetilde \gamma, \widetilde X)\mapsto  F_{*}(\widetilde \gamma, \widetilde X):=\bigl(\widetilde \gamma, F_{*, \widetilde \gamma}(\widetilde X)\bigr).
	\end{split}
\end{equation}
for any $(p, v)\in TE_0, (\widetilde \gamma, \widetilde X) \in TE_1.$  

On the other hand,  if $\Psi\colon F \Longrightarrow F'$ is a morphism in $\mathcal{G}(\mb{E}),$ then the differential of $\Psi$ yields a $\mb{G}$-equivariant natural isomorphism $\Psi_*\colon F_*\Longrightarrow F_*':$
\begin{equation}\label{Morphism diff}
	(p, v)\mapsto \Psi_*(p, v):=\bigl(1_p, \Psi_{*, p}(v)\bigr)\colon F_*(p, v)\longrightarrow F'_*(p, v).
\end{equation}
In the following proposition, we show that \Cref{Object diff} and \Cref{Morphism diff} together induces an action of the strict 2-group $\mc{G}(\mb{E})$ on the groupoids $\Omega_{\mb{E}}^{\rm{semi}}$ and $\Omega_{\mb{E}}^{\rm{strict}}\,$(\Cref{Def:Semisemistriccat}).
\begin{proposition}\label{Prop:Actiongaugeconnecat}
	For a Lie 2-group $\mb{G}$, let $\pi\colon \mb{E} \rightarrow \mb{X}$ be a principal $\mb{G}$-bundle over a Lie groupoid $\mb{X}$. Then the gauge 2-group $\mc{G}(\mb{E})$ acts on the category $\Omega_{\mb{E}}^{\rm{semi}}$ by  
	\begin{equation}\label{E:gauge2action}
		\begin{split}
			&(F,\omega) \mapsto \omega\circ F_*^{-1} \colon T\mb{E}\ra L(\mb{G}),\\
			&\bigl(\Psi\colon F \Rightarrow F',\zeta\colon \omega \Rightarrow \omega'\bigr) \mapsto \zeta\circ_{H} {\Psi_*}^{-1}\colon \omega\circ F_{*}^{-1}\Rightarrow  \omega\circ F_{*}^{-1},
		\end{split}
	\end{equation}  
	where $\circ_H$ denotes the horizontal composition of natural transformations. Moreover, the action restricts to $\Omega_{\mb{E}}^{\rm{strict}}\subset \Omega_{\mb{E}}^{\rm{semi}}.$

\end{proposition}
\begin{proof}
	Consider an object  $\omega$ and a morphism $\bigl(\zeta\colon \omega\Longrightarrow \omega'\bigr)$ in $\Omega_{\mb{E}}^{\rm{semi}}$. Since the $\mb{G}$-equivariance of $\zeta\colon \omega \Longrightarrow \omega'$ and  $\Psi\colon F \Longrightarrow F'$ induce the $\mb{G}$-equivariance of $\zeta\circ_{H} {\Psi_*}^{-1}$ for a morphism $\Psi$ in $\Omega_{\mb{E}}^{\rm{semi}}$, the only thing that is non-trivial to show is $$\omega\circ F_*^{-1}\in  \Omega_{\mb{E}}^{\rm{semi}}$$ for $F \in \mc{G}(\mb{E})$. To prove it,  we need to show that if there exists a $\mb{G}$-equivariant natural isomorphism  $\kappa\colon \omega \circ \delta \Longrightarrow \rm{pr_2}$ then there exists a $\mb{G}$-equivariant  natural isomorphism  $\kappa^{*}\colon (\omega\circ F_*^{-1}) \circ \delta \Longrightarrow \rm{pr_2},$
	where  $\delta$ is the vertical vector field generating functor for the action of the Lie 2-group $\mb{G}$ on $\mb{E}$ as mentioned in \Cref{Vertical vector field generating functor}. To show this, define $\kappa_{*}\colon \bigl(\omega\circ F_{*}^{-1}\bigr)\circ \delta  \Longrightarrow {\rm pr_2} $ by $$(p, B) \mapsto \kappa\bigl(F^{-1}(p), B\bigr),$$	for $(p, B) \in E_0 \times L(G_0)$. Observe that as a gauge transformation preserves the vertical vector field,  $F^{-1}_{*,\, p}(\delta_p(B))=\delta_{F^{-1}(p)}(B),$ we have 
	$s\bigl(\kappa_{*}(p, B)\bigr)=s\bigl(\kappa\bigl(F^{-1}(p), B\bigr)\bigr)=\omega(\delta_{F^{-1}(p)}(B))=\bigl(\omega\circ F_*^{-1}\bigr)\circ \delta(p, B)$ and $t\bigl(\kappa_{*}(p, B)\bigr)=t\bigl(\kappa\bigl(F^{-1}(p), B\bigr)\bigr)=B.$ Hence, we get $$\kappa_{*}(p, B)\colon \bigl(\omega\circ F_{*}^{-1}\bigr)\circ \delta(p, B) \rightarrow B.$$ Now, since $\kappa\colon \omega \circ \delta \Longrightarrow \rm{pr_2}$ is a natural transformation, for any  $(p, \, B)\xrightarrow{(\widetilde \gamma,\, K)}(q,\, B')\in E_1\times L(G_1),$ we get
	\begin{equation}\nonumber
		K \circ \kappa(F^{-1}(p),\,  B)= \kappa(F^{-1}(q), \,B') \circ \omega\bigl(F^{-1}(\gamma), \delta_{F^{-1}(\gamma)}(K)\bigr). 
	\end{equation}
	Substituting $\delta_{F^{-1}(\widetilde \gamma)}(K)$ above by $F^{-1}_{*,\, \widetilde \gamma}\bigl(\delta_{\widetilde \gamma}(K)\bigr)$, we complete our verification that $\kappa_{*}$ is indeed a natural transformation. Moreover, the following verifies the $\mb{G}$-equivariancy of $\kappa_*,$
	$$\kappa_* \bigl((p, \, B)\, g\bigr)=\kappa \bigl((F^{-1}(p), B) \, g\bigr)=\kappa (F^{-1}(p), B)\, 1_g=\kappa_*(p, \, B)\, 1_g.$$ This proves our claim and hence, $\omega\circ F_*^{-1}\in  \Omega_{\mb{E}}^{\rm{semi}}$. 
	
	Now, observe that from the definition itself, it is obvious that if $\omega\in \Omega_{\mb{E}}^{\rm{strict}}$  then $\omega\circ F_*^{-1}\in  \Omega_{\mb{E}}^{\rm{strict}}$. Thus, it follows immediately that the action of $\mc{G}(\mb{E})$ on $\Omega_{\mb{E}}^{\rm{semi}}$ restricts to $\Omega_{\mb{E}}^{\rm{strict}}\subset \Omega_{\mb{E}}^{\rm{semi}}$.
\end{proof}
\begin{remark}
	If the gauge transformation $F$ in \Cref{Prop:Actiongaugeconnecat} is represented by $\sigma\in C^{\infty}(\mb{E},\mb{G})^{\mb{G}},$ then its action on a strict connection $1$-form $\omega^{\rm st}$ and a semi-strict connection $1$-form $\omega^{\rm se}$ are respectively expressed as
	\begin{equation}\label{E:gaugewithgroup}
		\begin{split}
			&\omega^{\rm st}\mapsto \rm{Ad}_{\sigma}\omega^{\rm st}-(d\sigma) \sigma^{-1},\\
			&\omega^{\rm se}\mapsto \rm{Ad}_{\sigma}\omega^{\rm se}-{\widehat\kappa}_{\omega^{\rm se}}\circ (d\sigma) \sigma^{-1}-(d\sigma) \sigma^{-1},
		\end{split}
	\end{equation}
	where ${\widehat\kappa}_{\omega^{\rm se}}\colon \mb{E}\times L({\mb{G}})\longrightarrow  L({\mb{G}})^{\tau}=[L(H)\rtimes \tau\bigl(L(H)\bigr)\rra \tau\bigl(L(H)\bigr)]$
	is as  in \Cref{Prop:Semisirictnatural}. Note that everything is defined point-wise here.
\end{remark}

	\begin{remark}
	Since by \Cref{strict connection=strict forms}, the categories $\Omega_{\mb{E}}^{\rm{strict}}$ and $\Omega_{\mb{E}}^{\rm{semi}}$ respectively are isomorphic to $C^{\rm{strict}}_\mb{E}$ and $C^{\rm{semi}}_\mb{E},$ the action in  \Cref{Prop:Actiongaugeconnecat} induces actions of $\mathcal{G}(\mb{E})$ on $C^{\rm{strict}}_\mb{E}$ and $C^{\rm{semi}}_\mb{E}.$ This induced action is defined as $(R, F) \mapsto \overline{F} \circ R \circ \overline{T F^{-1}},$
	where $\overline{F}\colon {\rm Ad}(\mb{E}) \rightarrow {\rm Ad}(\mb{E}), [x_i, A_i]$ to $[F(x_i), A_i]$ for any $[x_i, A_i]\in \rm{Ad} (E_i), i\in \{0, 1\}$ and $\overline{TF}\colon {\rm At}(\mb{E}) \rightarrow {\rm At}(\mb{E}), [x_i, v_i]$ to $[F(x_i), F_{{*, x_i}}(v_i)]$ for any $[x_i, v_i]\in {\rm At}(E_i), i\in \{0, 1\}$. Similarly, there is an induced action at the level of morphisms, which is not difficult to write down.
\end{remark}	
As we have defined for connections (see \Cref{Def:Setstrict connection=strict forms}), considering the set of categorical connected components $\bar{\mc{G}}(\mb{E})$ of the gauge $2$-group $\mc{G}(\mb{E})$ for a principal $\mb{G}$-bundle $\pi\colon \mb{E}\ra \mb{X}$, one obtains the following action.
\begin{corollary}
	For a Lie 2-group $\mb{G}$, let $\pi\colon \mb{E} \rightarrow \mb{X}$ be a principal $\mb{G}$-bundle over a Lie groupoid $\mb{X}$. Then \Cref{E:gauge2action}		
	defines an action of ${\bar{\mathcal{G}}(\mb{E})}$   on the category $\bar{\Omega_{\mb{E}}}^{\rm{semi}}$, the set of categorical connected components of $\Omega_{\mb{E}}^{\rm{semi}}$. Moreover, the action restricts to $\bar{\Omega}_{\mb{E}}^{\rm{strict}}\subset \bar{\Omega}_{\mb{E}}^{\rm{semi}}.$ 
\end{corollary}

\subsection{An extended symmetry of semi-strict connections}\label{An extended symmetry of semi-strict connections}
In this section, we observe a particular larger gauge symmetry that the semi-strict connections enjoy. It is interesting that such extended gauge symmetries have already appeared in literature (see \cite{MR3645839} or \cite{martins2011lie}) to describe various physical systems. However, we show here that they fit appropriately in the framework of semi-strict connections. A description of the action of such extended gauge transformations on the semi-strict connections forms our second main result of this chapter.
\subsection*{Description of the extended gauge 2-group of a principal 2-bundle over a Lie groupoid}
We begin by defining specific Lie algebra valued 1-forms, which we intertwine with the elements of the gauge 2-group to produce these extended symmetries.

Let $\mb{G}:=[H \rtimes_{\alpha}G \rra G]$ be a Lie 2-group associated to a Lie crossed module $(G, H, \tau, \alpha)$. Suppose $\pi \colon \mb{E} \ra \mb{X}$ is a principal $\mb{G}$-bundle over a Lie grouipoid $\mb{X}$. Let $\Omega^{G}\bigl({E_0}, L(H)\bigr)$ be the set of $G$-equivariant $L(H)$-valued smooth $1$-forms on $E_0$, or to be more precise consider 
\begin{equation}\label{Extended objects}
	\Omega^{G}\bigl({E_0}, L(H)\bigr):=\bigl\{\lambda\colon TE_0\ra L(H)\big|\, \lambda_{p\cdot g}(v\cdot g)=\alpha_{g^{-1}}\bigl(\lambda_p(v)\bigr)\bigr\}.
\end{equation}
Every such $\lambda\in \Omega^{G}\bigl({E_0}, L(H)\bigr)$ induces a $\mb{G}$-equivariant
$1$-form $\overline\lambda \colon T\mb{E}\ra L(\mb{G})$ on $\mb{E}$ taking values in the sub Lie-groupoid $L(\mb{G})^{\tau}=\big[L(H)\rtimes \tau\bigl(L(H)\bigr)\rra \tau\bigl(L(H)\bigr)\big]\subset L(\mb{G}),$ defined as
\begin{equation}\label{E:lambdabar}
	\begin{split}
		&\overline \lambda_0:=\tau(\lambda)\colon TE_0\ra \tau\bigl(L(H)\bigr),\\
		&\overline \lambda_1:=\bigl(t^*\lambda-s^*\lambda+(\tau(s^*\lambda)\bigr)  \colon TE_1\ra L(H)\oplus \tau(L(H).
	\end{split}
\end{equation}
It is easy to see that the collection of these  $\mb{G}$-equivariant
$1$-forms $\overline\lambda \colon T\mb{E}\ra L(\mb{G})$ (as defined in \Cref{E:lambdabar})   forms a category, whose morphisms are fiber-wise linear natural transformations. We denote this category by $\widehat{\Omega}^{\mb{G}}(E_0,L(H))$.

Now, by intertwining the category  $\widehat{\Omega}^{\mb{G}}(E_0,L(H))$ with the gauge 
$2$-group $\mc{G}(\mb{E})$ (\Cref{Remark:Gauge2-group}), it is not difficult to verify that we get a strict $2$-group
$\mc{G}^{\rm ext}:=\mc{G}{(\mb{E})}\times  \widehat{\Omega}^{\mb{G}}(E_0,L(H))$, whose group products are defined as
\begin{equation}\label{E:semidirectgengauge}\begin{split}
		&(F, \lambda)(F', \lambda ')=\bigl(F F ', \lambda+(\lambda'\circ F_*^{-1})\bigr),\\
		&(\Psi, \theta)(\Psi', \theta ')=\bigl(\Psi \Psi ', \theta+(\theta'\circ_H \Psi_*^{-1})\bigr).
\end{split}\end{equation}
Observe that the  2-vector space stuctures of $L(\mb{G})$ (\Cref{Lie 2-algebra}) is necessary to ensure that $\mc{G}^{\rm ext}$ is indeed a strict $2$-group. Also, note that we have made the following obvious identification in \Cref{E:semidirectgengauge}:
$${\rm{Obj}} \big( \mc{G}^{\rm ext} \big) \simeq {\rm{Obj}} \big( \mc{G}{(\mb{E})} \big) \times  {\Omega}^{G}(E_0,L(H)),  (\textit{see}\,\, \Cref{Extended objects}).$$ As a conclusion to the above description, we get the following definition that captures a larger gauge symmetry than we have for the gauge 2-group of a principal 2-bundle over a Lie groupoid.
\begin{definition}
	
	Consider the Lie 2-group $\mb{G}:=[H \rtimes_{\alpha}G \rra G]$ associated to a Lie crossed module $(G, H, \tau, \alpha)$. Suppose $\pi \colon \mb{E} \ra \mb{X}$  be a principal $\mb{G}$-bundle over a Lie grouipoid $\mb{X}$. Then, the strict 2-group $\mc{G}^{\rm ext}:=\mc{G}{(\mb{E})}\times  \widehat{\Omega}^{\mb{G}}(E_0,L(H))$ (as described in \Cref{E:semidirectgengauge}) is defined as \textit{the extended gauge 2-group of $\pi \colon \mb{E} \ra \mb{X}$.}
\end{definition}
\subsection*{Action of the extended gauge 2-group on the groupoid of connections}
The \Cref{Prop:Semisirictnatural} allows us to define a left action of the extended gauge 2-group $\mc{G}^{\rm ext}$ on  $\Omega_{\mb{E}}^{\rm{semi}}$, given as follows:
\begin{equation}\label{E:GengaugeactdefwithF}
	\begin{split}
		&\bigl((F, \lambda),\omega\bigr)\mapsto \omega\circ F_*^{-1}+\overline \lambda\colon T\mb{E}\ra L(\mb{G}),\\
		&\bigl((\Psi, \theta),\zeta \bigr)\mapsto \zeta\circ_H \Psi_*^{-1}+ \theta.
\end{split}\end{equation}
Using the identification made in \Cref{E:gaugeequi}, the action above can also be expressed as 
\begin{equation}\label{E:Gengaugeactdef}
	\begin{split}
		\bigl((\sigma, \lambda),\omega \bigr)\mapsto \bigl(\rm{Ad}_{\sigma}(\omega)-\kappa_{\omega}\circ (d\sigma) \sigma^{-1}-(d\sigma) \sigma^{-1}+\overline \lambda\bigr)\colon T\mb{E}\ra L(\mb{G}),
	\end{split}
\end{equation}
where $\sigma\in C^{\infty}(\mb{E},\mb{G})^{\mb{G}}$. Similarly, one can write down at the level of morphisms.

We call the action above (\Cref{E:GengaugeactdefwithF}, \Cref{E:Gengaugeactdef})  an \textit{extended gauge transformation}. By \Cref{cor:stricconnto semi}, it is evident that this action does not restrict to $\Omega_{\mb{E}}^{\rm{strict}}$. In particular, the action given in 
\Cref{E:GengaugeactdefwithF} is of importance in higher gauge theory and related areas of physics (for example, in $2$-BF theories). \footnote{ The \textit{BF theory}, a topological field theory, was introduced  to serve as a basis for studying background-free theories \cite{MR1022521}. A 2-BF theory is a BF theory involving Lie 2-groups, see \cite{martins2011lie}. Typically,
the theory involves a Lie algebra $L(G)$ valued differential $1$-form on a principal $G$-bundle $P$ over a $4$-dimensional  manifold, along with an $L(H)$-valued $2$-form, where $\ghta$ is a Lie crossed module. The $1$ and $2$-forms are prescribed to behave certain way under the gauge transformations. Though BF theory does not have any particular relevance with this thesis, it is interesting  to observe that the  global gauge transformation obtained in \Cref{E:GengaugeactdefwithF} coincides with that of differential $1$-form in a $BF$ theory. For more details on BF theory we refer \cite{MR1770708}.} In this context, we refer to the work of Martins and Picken \cite{MR2661492, martins2011fundamental, MR2764890} and works of Martins and Mikovi\' c \cite{martins2011lie}. Examining the following example could provide valuable insights.

\begin{example}\label{Ex:EoXodecobundleconnectgen}
	Consider the decorated principal-$[H \rtimes_{\alpha}G \rra G]$-bundle  ${E}^{\rm dec}:=[E\times H\rra E] \ra [M \rra M]$ over the discrete Lie groupoid $[M \rra M]$ associated to a Lie crossed module $(G, H, \tau, \alpha)$ constructed in \Cref{Ex:EoXodecobundle}. It is easy see that any gauge transformation $\sigma\colon {E}^{\rm dec}\ra \mb{G}$ on  ${E}^{\rm dec}=[E\times H\rra E]$ is expressed as a gauge transformation $\sigma_0\colon E\ra G$ on the $G$-bundle $E\ra M,$ as $\sigma(p)=\sigma_0(p), \sigma(p, h)=\bigl(h, \,\sigma_0(p)\bigr)$, (see \Cref{Ex:EoXodecobundlegaguge}). Hence, an extended gauge transformation $(\sigma, \lambda)$ is given by a classical gauge transformation $\sigma_0\colon E\ra G$ on the $G$-bundle   $E\ra M$ and a $G$-equivariant $L(H)$-valued $1$-form $\lambda\in  \Omega^{G}\bigl({E}, L(H)\bigr)$ on $E.$
	\Cref{Ex:EoXodecobundleconnect} prescribes a way to construct a strict connection $\overline\omega:=(\omega^{\rm dec}, \omega_0)$ from any connection $\omega_0$ on the principal $G$-bundle $E\ra M$. Then the transformation of $\overline\omega$ under the action of $(\sigma, \lambda)$ is given by
	$$\overline\omega \mapsto \rm{Ad}_{\sigma}\overline \omega-(d\sigma)\sigma^{-1}+\overline\lambda.$$
	which reads in detail as
	\begin{equation}\label{E:gaugehighergaue}
		\begin{split}
			\omega_0(p)&\mapsto {\rm{Ad}}_{\sigma_0(p)}\omega_0(p)-(d\sigma_0(p))\sigma_0(p)^{-1}+\tau(\lambda(p)),\\
			\omega^{\rm dec}(p,\, h) &\mapsto {\rm{Ad}}_{(h,\,\sigma_0(p))}\omega^{\rm dec}(p,\, h)-\biggl(d\bigl(h,\, \sigma_0(p)\bigr)\biggr)\bigl(h,\, \sigma_0(p)\bigr)^{-1}\\
			&\hskip 3 cm +\bigl(\tau(\lambda(p))-\lambda(p)+{\rm{Ad}}_h\lambda({p})\bigr).
		\end{split}
	\end{equation}
\end{example}
The gauge transformation that we get here in \Cref{E:gaugehighergaue}
is precisely the gauge transformation of a connection $1$-form in the higher gauge theories (see \cite{MR3645839} or \cite{martins2011lie}).

\chapter{Parallel transport on principal 2-bundles and VB-groupoids}\label{Chapter: Parallel transport on quasi-principal 2-bundles} 



\lhead{Chapter 6. \emph{Parallel transport on quasi-principal 2-bundles and associated VB-groupoids}} 


The results of this chapter stand as concrete examples to demonstrate how ideas of \Cref{chapter 2-bundles} can combine with the concepts of \Cref{chapter 2-bundles copnnection} generically. To be more specific, we will see an instance of how a differential geometric notion of connection-induced horizontal path lifting property combines with a categorical notion of cartesian lifts in a fibered category to produce a differential geometric theory of parallel transport in the framework of principal 2-bundles over Lie groupoids. Further, we test our theory by developing a notion of parallel transport on VB-groupoids (\Cref{subsection VB groupoids}).

Like any other reasonable higher notion of parallel transport, our main result of this chapter categorifies the traditional parallel transport functor (\Cref{Transport functor})  together with the smoothness property it enjoys (\Cref{Smoothness of traditional parallel transport}):
\begin{equation}\label{classic transport rough}
	\begin{split}
		T_{\omega} \colon & \Pi_{{\rm{thin}}}(M) \ra G {\rm{-}} {\rm{Tor}}\\
		& x \mapsto \pi^{-1}(x), x \in M\\
		& [\alpha] \mapsto {\rm{Tr}}_{\omega}^{\alpha}, \alpha \in \frac{PM}{\sim}.
	\end{split}
\end{equation}
However, what sets us apart from the rest is that here we directly replace the left-hand side of \Cref{classic transport rough} i.e. $\Pi_{{\rm{thin}}}(M)$ by developing a notion of \textit{thin fundamental groupoid of a Lie groupoid} $\Pi_{{\rm{thin}}}(\mb{X})$, that enjoys suitable smoothness properties and at the same time generalizes the classical one (\Cref{Definition:Thin homotopy groupoid of a manifold}). Also, being in the framework of Lie 2-group bundle, the right-hand side of \Cref{classic transport rough} is replaced accordingly.
The main non-triviality of our approach is in `making sense' of \Cref{classic transport rough} in the framework of quasi-principal 2-bundles over Lie groupoids equipped with connection structures developed in \Cref{chapter 2-bundles copnnection}, and then showing that `our generalization' is a reasonable choice.

The contents of this chapter is mostly borrowed from our paper \cite{chatterjee2023parallel}.

\section{Lazy Haefliger paths and the thin fundamental groupoid of a Lie groupoid}\label{Lazy Haefliger paths and the thin fundamental groupoid of a Lie groupoid}
This section introduces a notion of a thin fundamental groupoid of a Lie groupoid, a generalization of the classical one (\Cref{Thin homotopy groupoid of a manifold}), and imposes a diffeological structure on it.

We begin with the definition of a `lazy Haefliger path'.

\subsection{Lazy Haefliger paths}\label{Lazy haefliger paths}
\begin{definition}\label{Definition:Haefliger path}
	Suppose $\mb{X}$ is a Lie groupoid and let $x,y \in X_0$. A \textit{lazy $\mb{X}$-path} or a \textit{lazy Haefliger path} $\Gamma$ \textit{from $x$ to $y$} is defined as a sequence $\Gamma :=(\gamma_0, \alpha_1,\gamma_1, \cdots,\alpha_n, \gamma_n)$ for some $n \in \mb{N}$ where 
	\begin{itemize}
		\item[(i)] $\alpha_i:[0,1] \ra X_0$ is a path with sitting instants (see \Cref{Thin homotopy groupoid of a manifold}) for all $1 \leq i \leq n$ and 
		\item[(ii)] $\gamma_i \in X_1$ for all $0 \leq i \leq n$, 
	\end{itemize}
	satisfying the following conditions:
	\begin{itemize}
		\item[(a)] $s(\gamma_0)=x$ and $t(\gamma_n)=y$;
		\item[(b)] $s(\gamma_i)= \alpha_i(1)$ for all $0 < i \leq n$;
		\item[(c)] $t(\gamma_i)= \alpha_{i+1}(0)$ for all $0 \leq i < n$.
	\end{itemize}
\end{definition}

We call $\Gamma $ a \textit{lazy $\mb{X}$-path of order $n$}. We define the \textit{source of $\Gamma$} as  $s(\gamma_0)=x$ and the \textit{target of $\Gamma$} as $t(\gamma_n)=y$. Given a Lie groupoid $\mb{X}$, we shall denote the set of all lazy $\mb{X}$-paths of all orders by $P\mb{X}$. If we remove the sitting instants condition from (i), we recover the existing notion of a Haefliger path as in \Cref{Haefliger definition}.

We are specifically interested in certain equivalence classes of lazy Haefliger paths. Observe that such equivalences are similar to the one discussed in \Cref{Equivalence Haefliger}.
\begin{definition}\label{Definition: Equivalence of X-paths}
	A lazy $\mb{X}$-path $\Gamma:=(\gamma_0, \alpha_1,\gamma_1, \cdots,\alpha_n, \gamma_n)$ is said to be \textit{equivalent} to another lazy $\mb{X}$-path $\bar{\Gamma}$, if one is obtained from the other by a finite sequence of all or some of the following operations:
	\begin{itemize}
		\item[(1)] \textit{Removing/adding a constant path}, that is if $\alpha_{i+1}$ is a constant path in the lazy $\mb{X}$-path $\Gamma$, then by removing it we obtain the lazy $\mb{X}$-path $(\gamma_0, \alpha_1,\gamma_1,..,\gamma_{i+1} \circ \gamma_i, \cdots,\alpha_n, \gamma_n)$, where $i \in \lbrace 1,2, \cdots n-1 \rbrace$. Replacing the word `removing' by `adding' one obtains the condition for `adding a constant path'.
		\[\begin{tikzcd}
			\cdot \arrow[r, "\gamma_i"] & \cdot \arrow["\alpha_{i+1}=\rm{constant}"', dotted, loop, distance=2em, in=305, out=235] \arrow[r, "\gamma_{i+1}"] & \cdot
		\end{tikzcd}\]
		\item[(2)] \textit{Removing/adding an idenitity morphism}, that is if $\gamma_i$ is an identity morphism in the lazy $\mb{X}$-path $\Gamma$, then by removing it we obtain a lazy $\mb{X}$-path $(\gamma_0, \alpha_1,\gamma_1,..,\alpha_{i+1} * \alpha_i, \cdots,\alpha_n, \gamma_n)$, where $*$ is the concatenation of paths and $i \in \lbrace 1,2,..,n-1 \rbrace$. Replacing the word `removing' by `adding' one obtains the condition for `adding an identity morphism'.
		\[\begin{tikzcd}
			\cdot \arrow[r, "\alpha_i", dotted] & \cdot \arrow["\gamma_i=\rm{identity}"', loop, distance=2em, in=305, out=235] \arrow[r, "\alpha_{i+1}", dotted] & \cdot
		\end{tikzcd}\]
		
		\item[(3)] \textit{Replacing $\alpha_i$ by $t \circ \zeta_i$, replacing $\gamma_{i-1}$ by $\zeta_i (0) \circ \gamma_{i-1}$ and $\gamma_{i}$ by $\gamma_i \circ (\zeta_i(1))^{-1}$ for a given path $\zeta_i \colon [0,1] \ra X_1$ with sitting instants, such that $s \circ \zeta_i= \alpha_i$} and $i \in \lbrace 1,2, \cdots, n \rbrace$, that is the portion $(\gamma_{i-1}, \alpha_i, \gamma_i)$ of the lazy $\mb{X}$-path $\Gamma$ is replaced by the portion $\big( \zeta_i(0) \circ \gamma_{i-1}, t \circ \zeta_i, \gamma_i \circ (\zeta(1))^{-1} \big)$ to obtain a lazy $\mb{X}$-path $$(\gamma_0, \alpha_1,\gamma_1,..,\alpha_{i-1},\underbrace{\zeta_i(0) \circ \gamma_{i-1}, t \circ \zeta_i, \gamma_i \circ (\zeta(1))^{-1}},\alpha_{i+1}, \cdots,\alpha_n, \gamma_n).$$ 
		See the diagram below: 
		\[
		\begin{tikzcd}
			\cdot \arrow[r, "\gamma_{i-1}"] & \cdot \arrow[r, "\alpha_i", dotted] \arrow[d, "\zeta_i(0)"'] & \cdot \arrow[r, "\gamma_i"]  & . \\
			& \cdot \arrow[r, "t \circ \zeta_i"', dotted]                & \cdot \arrow[u, "\zeta_i(1)^{-1}"'] &  
		\end{tikzcd}\]
	\end{itemize}
	
\end{definition}
Adhering to the same conventions as in \cite{MR2772614} and \cite{MR2166083}, we say the operations in \Cref{Definition: Equivalence of X-paths} as \textit{equivalences}.

Now, we introduce a notion of \textit{thin deformation of lazy $\mb{X}$-paths}, which can be regarded as a modified thin homotopy analog of the existing notion of deformation of $\mb{X}$-paths discussed in \Cref{deformation}.
\begin{definition}\label{Definition: Thin deformation}
	A \textit{thin deformation} from a lazy $\mb{X}$-path $\Gamma:=(\gamma_0, \alpha_1,\gamma_1, \cdots,\alpha_n, \gamma_n)$ to a lazy $\mb{X}$-path $\Gamma':=(\gamma'_0, \alpha'_1, \gamma'_1 \cdots, \alpha'_n, \gamma'_n)$ of the same order is given by a sequence of smooth paths $ \lbrace \zeta_i: [0,1] \ra X_1 \rbrace_{i =0,1, \cdots,n}$ with $\zeta_i(0)= \gamma_i$ and $\zeta_i(1)= \gamma'_i$  such that 
	\begin{enumerate}[(i)]
		\item$\lbrace \zeta_i: [0,1] \ra X_1 \rbrace_{i =0,1, \cdots,n}$ are  paths with sitting instants;
		\item$\alpha_i$ is thin homotopic to $(s \circ \zeta_i)^{-1} * \alpha_i' * (t \circ \zeta_{i-1})$ for all $i=1,2 \cdots .,n$, where $*$ is the concatenation of paths, as illustrated by the following diagram: 
		\begin{equation}\nonumber
			\begin{tikzcd}
				\cdot \arrow[r] \arrow[r, "\gamma_{i-1}"] & \cdot \arrow[r, "\alpha_i", dotted] \arrow[d, "t \circ \zeta_{i-1}"', dotted] & \cdot \arrow[r, "\gamma_i"]                          & \cdot \\
				\cdot \arrow[r, "\gamma{\rm{'}}_{i-1}"']          & \cdot \arrow[r, "\alpha_i{\rm{'}}"', dotted]                        & \cdot \arrow[u, "(s \circ \zeta_i)^{-1}"', dotted] \arrow[r, "\gamma{\rm{'}}_i"'] & \cdot
			\end{tikzcd};
		\end{equation}
		\item$s \circ \zeta_0$ and $t \circ \zeta_n$ are constant paths in $X_0$.
	\end{enumerate}
\end{definition}
\begin{proposition}\label{thin homotopy of X-paths is an equivalence relation}
	For any Lie groupoid $\mb{X}$, \Cref{Definition: Thin deformation} defines an equivalence relation on $P\mb{X}$.
\end{proposition}
\begin{proof} We prove the equivalence relation on $P\mb{X}$ as follows:

	Reflexive:
 
	Consider a lazy $\mb{X}$-path $\Gamma :=(\gamma_0, \alpha_1,\gamma_1,...,\alpha_n, \gamma_n)$ from $x$ to $y$. A thin deformation from $\Gamma$ to itself is given by the sequence of constant paths $ \lbrace \zeta_{i} \colon [0,1] \ra X_1 \rbrace_{i=0,1...,n}$ defined as $r \mapsto \gamma_i$ for all $r \in [0,1]$.
	
Symmetric:

	Let $ \lbrace \zeta_{i} \colon [0,1] \ra X_1 \rbrace_{i=0,1...,n}$  be a thin deformation from a lazy $\mb{X}$-path $\Gamma :=(\gamma_0, \alpha_1,\gamma_1,...,\alpha_n, \gamma_n)$ to another lazy $\mb{X}$-path $\Gamma' :=(\gamma'_0, \alpha'_1, \gamma'_1..., \alpha'_n, \gamma'_n)$. It follows immediately from the observation $s \circ \zeta_i^{-1}=(s \circ \zeta_i)^{-1}$ and $t \circ \zeta_i^{-1}=(t \circ \zeta_i)^{-1}$ that $ \lbrace \zeta_{i}^{-1} \colon [0,1] \ra X_1 \rbrace_{i=0,1...,n}$  defined as $r \mapsto \zeta(1-r)$, (see \Cref{Notation conventions}), is a thin deformation from $\Gamma'$ to $\Gamma$.
	
	Transitive:
 
	Let $ \lbrace \zeta_{i} \colon [0,1] \ra X_1 \rbrace_{i=0,1...,n}$  be a thin deformation from a lazy $\mb{X}$-path $\Gamma :=(\gamma_0, \alpha_1,\gamma_1,...,\alpha_n, \gamma_n)$ to a lazy $\mb{X}$-path $\Gamma' :=(\gamma'_0, \alpha'_1, \gamma'_1..., \alpha'_n, \gamma'_n)$ and $ \lbrace \delta_{i} \colon [0,1] \ra X_1 \rbrace_{i=0,1...,n}$  be another from  $\Gamma'$ to  $\Gamma'' :=(\gamma''_0, \alpha''_1, \gamma''_1..., \alpha''_n, \gamma''_n)$. We claim that  $ \lbrace \delta_{i} * \zeta_i \colon [0,1] \ra X_1 \rbrace_{i=0,1...,n}$  is a thin deformation from $\Gamma$ to $\Gamma''$ and it follows straightforwardly from the following:
	\begin{itemize}
		\item[(i)] $s(\delta_i)*t(\zeta_i)=s(\delta_i * \zeta_i)$, 
		\item[(ii)] $t(\delta_i)*t(\zeta_i)=t(\delta_i * \zeta_i)$,
		\item[(iii)] $(\delta_i * \zeta_i)^{-1}= \zeta_i^{-1}* \delta_i^{-1}$
	\end{itemize} 
	for all $i= 0,1,..,n$.
	%
\end{proof} 

With a similar spirit as in \Cref{homotopy of Xpaths}, we define the following:
\begin{definition}\label{Definition: Thin homotopy of X-paths}
	A \textit{lazy $\mb{X}$-path thin homotopy} is defined as the equivalence relation on $P\mb{X}$ generated by the equivalence relations in \Cref{Definition: Equivalence of X-paths} and \Cref{thin homotopy of X-paths is an equivalence relation}.
\end{definition}

The corresponding quotient set will be denoted as $\frac{P\mb{X}}{\sim}$. In particular, a pair of lazy $\mb{X}$-paths (with fixed endpoints) is related by lazy $\mb{X}$-path thin homotopy if one is obtained from the other by a finite sequence of equivalences and thin deformations.

 \subsection{Thin fundamental groupoid of a Lie groupoid}\label{Thin fundamental groupoid of a Lie groupoid}
\begin{proposition}\label{Propostioni:thin fundamental groupoid of a Lie groupoid}
	For any Lie groupoid $\mb{X}= [X_1 \rra X_0]$, there is a groupoid $\Pi_{\rm{thin}}(\mb{X})$,  whose object set is $X_0$ and the morphism set is $\frac{P\mb{X}}{\sim}$. The structure maps are given as follows:
	\begin{enumerate}[(i)]
		\item Source: $s: \frac{P\mb{X}}{\sim} \ra X_0$ is defined by $[\Gamma = (\gamma_0, \alpha_1,\gamma_1, \cdots,\alpha_n, \gamma_n)] \mapsto s(\gamma_0);$ 
		\item Target: $t: \frac{P\mb{X}}{\sim} \ra X_0$ is defined by $[\Gamma = (\gamma_0, \alpha_1,\gamma_1, \cdots, \alpha_n, \gamma_n)] \mapsto t(\gamma_n);$ 
		\item Composition: if $s([\Gamma'= (\gamma'_0, \alpha'_1,\gamma'_1, \cdots,\alpha'_n, \gamma'_n)])= t([\Gamma = (\gamma_0, \alpha_1,\gamma_1, \cdots,\alpha_m, \gamma_m)])$, then define
		\begin{equation}\nonumber
  \begin{split}
  		[(\gamma'_0, \alpha'_1,\gamma'_1, \cdots,\alpha'_n, \gamma'_n)] \circ [(\gamma_0, \alpha_1,\gamma_1, \cdots,\alpha_m, \gamma_m)] & \\:=
       [(\gamma_0, \alpha_1,\gamma_1, \cdots,\alpha_m, \gamma'_0 \circ  \gamma_m, \alpha'_1,\gamma'_1, \cdots,\alpha'_n, \gamma'_n)];
   \end{split}
		\end{equation}
		\item Unit: $u : X_0 \ra  \frac{P\mb{X}}{\sim}$ is given by $x \mapsto [(1_x, c_x,1_x)]$ where $c_x \colon [0,1] \ra X_0$ is the constant path at $x \in X_0$;
		\item Inverse: $\mathfrak{i} \colon  \frac{P\mb{X}}{\sim} \ra  \frac{P\mb{X}}{\sim}$ is given by
		\begin{equation}\nonumber
			[(\gamma_0, \alpha_1,\gamma_1, \cdots,\gamma_{n-1},\alpha_n, \gamma_n)] \mapsto [(\gamma_n^{-1}, \alpha^{-1}_{n},\gamma^{-1}_{n-1}, \cdots,\gamma^{-1}_1, \alpha^{-1}_1, \gamma^{-1}_0)].
		\end{equation} (see \Cref{Notation conventions})
	\end{enumerate}
\end{proposition}

\begin{proof}
	From the definition itself, it follows immediately that $s$ and $t$ are well-defined. Now, observe that in order to ensure the well-definedness of the composition map, it is sufficient to consider only the following four cases:
	\begin{itemize}
		\item[(i)] If $\tilde{\Gamma}'$ is obtained from $\Gamma'$ by an equivalence and if $\tilde{\Gamma}$ is obtained from $\Gamma$ by an equivalence, then $\Gamma' \circ \Gamma$ is lazy $\mb{X}$-path thin homotopic to $\tilde{\Gamma}' \circ \tilde{\Gamma}$.
		\item[(ii)] If $\tilde{\Gamma}'$ is obtained from $\Gamma'$ by a thin deformation and if $\tilde{\Gamma}$ is obtained from $\Gamma$ by a thin deformation, then $\Gamma' \circ \Gamma$ is lazy $\mb{X}$-path thin homotopic to $\tilde{\Gamma}' \circ \tilde{\Gamma}$.
		\item[(iii)] If $\tilde{\Gamma}'$ is obtained from $\Gamma'$ by an equivalence and if $\tilde{\Gamma}$ is obtained from $\Gamma$ by a thin deformation, then $\Gamma' \circ \Gamma$ is lazy $\mb{X}$-path thin homotopic to $\tilde{\Gamma}' \circ \tilde{\Gamma}$.
		\item[(iv)]If $\tilde{\Gamma}'$ is obtained from $\Gamma'$ by a thin deformation and if $\tilde{\Gamma}$ is obtained from $\Gamma$ by an equivalence, then $\Gamma' \circ \Gamma$ is lazy $\mb{X}$-path thin homotopic to $\tilde{\Gamma}' \circ \tilde{\Gamma}$.
	\end{itemize}

Case (i):

	From the way the composition is defined, it is clear that successively executing the same operations (See \Cref{Definition: Equivalence of X-paths}) on $\Gamma' \circ \Gamma$,  which were used to obtain $\tilde{\Gamma}$ from $\Gamma$ and $\tilde{\Gamma}'$ from $\Gamma'$,  will produce $\tilde{\Gamma}' \circ \tilde{\Gamma}$.
 
Case (ii):

Note that if $ \lbrace \zeta_{i} \colon [0,1] \ra X_1 \rbrace_{i=0,1...,n}$ and $ \lbrace \zeta_{i}' \colon [0,1] \ra X_1 \rbrace_{i=0,1...,n}$ are thin deformations from $\Gamma$ to $\tilde{\Gamma}$ and $\Gamma'$ to $\tilde{\Gamma}'$ respectively, then $ \lbrace \zeta_0, \zeta_1,...,\zeta_{m-1}, \zeta_0' \circ \zeta_m, \zeta'_1,\zeta'_2,...,\zeta_n \rbrace$ is a thin deformation from $\Gamma' \circ \Gamma$ to $\tilde{\Gamma}' \circ \tilde{\Gamma}$.  	
	
 Case (iii): 
 
 Let $\lbrace \zeta_{i} \colon [0,1] \ra X_1 \rbrace_{i=0,1...,n}$ be a thin deformation from $\Gamma$ to $\tilde{\Gamma}$ and  $\epsilon$ be an equivalence operation on $\Gamma'$ to obtain $\tilde{\Gamma}'$. Then, a thin deformation from $\Gamma' \circ \Gamma$ to $\Gamma' \circ \tilde{\Gamma}$ is given by $d:=\lbrace \zeta_0, \zeta_1,..,\zeta_m, c_{\gamma_0'},c_{\gamma_1'},...,c_{\gamma_n'}   \rbrace $, where $\lbrace c_{\gamma_i'}\rbrace_{i=0,1...n}$ are constant paths in $X_1$ defined by $c_{\gamma_i'}(r)=\gamma_i'$ for all $r \in [0,1]$. Then, we obtain $\tilde{\Gamma}' \circ \tilde{\Gamma}$  by applying the equivalence operation $\epsilon$ on $\Gamma' \circ \tilde{\Gamma}$.
	
Case (iv):

In an exact similar way as in case(iii), we can verify the case(iv).

Associativity of the composition:

	Let $\Gamma, \Gamma', \Gamma''$ be a sequence of three composable lazy $\mb{X}$-paths. To verify the associativity of composition, note that it is sufficient to consider the following three cases:
	\begin{itemize}
		\item[(a)] When $\Gamma$, $\Gamma'$ and $\Gamma''$ are of the form $(\gamma)$,  $(\gamma')$ and $(\gamma'')$ respectively.
		\item[(b)] When $\Gamma$, $\Gamma'$ and $\Gamma''$ are of the form $(\rm{id}, \alpha, \rm{id})$, $(\rm{id}, \alpha', \rm{id})$ and $(\rm{id}, \alpha'', \rm{id})$ respectively, where $\rm{id}$ are the identity morphisms.
		\item[(c)] When $\Gamma$, $\Gamma'$ and $\Gamma''$ are not of the form defined in (a) and (b).
	\end{itemize}
	(a) is a direct consequence of the associativity of the composition in the Lie groupid $\mb{X}$. Since for any three composable paths $\alpha, \beta, \gamma$ with sitting instants in $X_0$, the path  $(\gamma* \beta)* \alpha$ is thin homotopic to $\gamma*(\beta* \alpha)$, hence, (b) follows. (c) follows directly from the definition of composition itself. 
	
	The unit map and the inverse map verifications are immediate. 
	
	Hence we showed $\Pi_{\rm{thin}}(\mb{X})$ is a groupoid.
\end{proof}

\begin{definition} \label{Thin homotopy groupod of a Lie groupoid}
	The Lie groupoid $\Pi_{\rm{thin}}(\mb{X})$ is defined as the \textit{thin fundamental groupoid of the Lie groupoid $\mb{X}$.}		
\end{definition}
The following result ensures that our notion of thin homotopy groupoid of the Lie groupoid is a reasonable generalization of the classical one:

\begin{proposition}\label{Higher analogue of the classical one}
	For any smooth manifold $M$, $\Pi_{\rm{thin}}(M)$  is same as $\Pi_{\rm{thin}}([M \rra M])$. 
\end{proposition}
\begin{proof}
	Since any element $x$ and a path $\alpha$ in a manifold $M$ can respectively be identified with lazy $[M \rra M]$-paths $(1_x, c_x, 1_x)$ and $(1_{\alpha(0)}, \alpha, 1_{\alpha(1)})$ for a constant path $c_x$ at $x$, the $\Pi_{{\rm{thin}}}(M \rra M)$ is same as $\Pi_{{\rm{thin}}}(M)$.
\end{proof}
In the following subsection, we will put a \textit{smooth structure} on the thin fundamental groupoid of a Lie groupoid (\Cref{Thin homotopy groupod of a Lie groupoid}). In particular, we will show that $\Pi_{\rm{thin}}(\mb{X})$ is a diffeological groupoid for any Lie grouopoid $\mb{X}$.

 	\subsection{Smoothness of the thin fundamental groupoid of a Lie groupoid}\label{Subsection Smoothness of thin fundamental groupoid of a Lie groupoid}
Given a Lie groupoid $\mb{X}$, define an infinite sequence of sets as $\lbrace P\mb{X}_n \rbrace_{n \in \mb{N} \cup \lbrace 0 \rbrace}$, where $P\mb{X}_0:=X_1$ and $P\mb{X}_n:= X_1 \times_{t,X_0,ev_0}PX_0 \times_{ev_1,X_0,s}X_1 \times_{t,X_0,ev_0} \cdots \times_{t,X_0,ev_0}PX_0 \times_{ev_1,X_0,s} \times X_1$, for $n \in \mb{N}$. It is obvious from the definition itself that $P\mb{X}_n$ has a natural identification with the set of lazy $\mb{X}$-paths of order $n$ for each $n \in \mb{N}$. Thus, as a set
$P\mb{X}= \cup_{i \in \mb{N} \cup \lbrace 0 \rbrace}P\mb{X}_i=\sqcup_{i \in \mb{N} \cup \lbrace 0 \rbrace}P\mb{X}_i$. 
\begin{proposition}\label{PX is a diffeological space}
	For any Lie groupoid $\mb{X}$, the set of lazy $\mb{X}$-paths $P\mb{X}$ is a diffeological space.
\end{proposition}
\begin{proof} 
	By \Cref{manifold diffeology} and \Cref{Path space diffeology}, respectively, source-target and evaluation maps are maps of diffeological spaces. Thus, the fiber product diffeology (\Cref{Fibre product diffeology}) ensures that for each $n \in \mb{N}$, $P\mb{X}_n$ is a diffeologial space with diffeology given by
	$D_{P\mb{X}_n}:= \biggl\{(p^0_{X_1},p^{1}_{PX_0},p^1_{X_1}, \cdots,p^{n}_{PX_0},p^n_{X_1}) \in D_{X_1} \times D_{PX_0} \times D_{X_1} \times \cdots \times D_{PX_0} \times D_{X_1}: t \circ p^0_{X_1}=ev_0 \circ p^{1}_{PX_0}, ev_1 \circ p^{1}_{PX_0}=s \circ p^1_{X_1}, \cdots, ev_1 \circ p^{n}_{PX_0} =s \circ p^n_{X_1} \biggr\}$. Then, \Cref{sum diffeology} induces the sum diffeology on $P\mb{X}$.
\end{proof}
\begin{corollary}\label{Diffeology of quotient of Xpath space}
	$\frac{P\mb{X}}{\sim}$ is a diffeological space.	
\end{corollary}
\begin{proof}
	An immediate consequence of \Cref{PX is a diffeological space} and \Cref{quotient diffeology}.
\end{proof}
\begin{lemma}\label{composition of PX}
	For any Lie groupoid $\mb{X}$, 
	\begin{itemize}
		\item[(a)] the multiplication map 
		\begin{equation}\nonumber
			\begin{split}
				\tilde{m} \colon  & P\mb{X} \times_{s,X_0,t} P\mb{X} \ra P\mb{X}\\
				& \Big((\gamma'_0, \alpha'_1, \cdots,\alpha'_n, \gamma'_n), (\gamma_0, \alpha_1, \cdots,\alpha_m, \gamma_m) \Big) \mapsto (\gamma_0, \alpha_1, \cdots,\alpha_m, \gamma'_0 \circ  \gamma_m, \alpha'_1, \cdots,\alpha'_n, \gamma'_n),
			\end{split}
		\end{equation}
		\item[(b)] the unit map 
		\begin{equation}\nonumber
			\begin{split}
				\tilde{u} \colon & X_0 \ra P\mb{X}\\
				& x \mapsto (1_x,c_x,1_x),	
			\end{split}
		\end{equation}
		where $c_x$ is the constant path at $x$,
		\item[(c)] the inverse map
		\begin{equation}\nonumber
			\begin{split}
				\tilde{\mathfrak{i}} \colon & P\mb{X} \ra P\mb{X}\\
				& (\gamma_0, \alpha_1,\gamma_1, \cdots,\gamma_{n-1},\alpha_n, \gamma_n) \mapsto (\gamma_n^{-1}, \alpha^{-1}_{n},\gamma^{-1}_{n-1}, \cdots,\gamma^{-1}_1, \alpha^{-1}_1, \gamma^{-1}_0),
			\end{split}
		\end{equation}
	\end{itemize}
	are maps of diffeological spaces, (see \Cref{Notation conventions}).		
\end{lemma}

\begin{proof}:
	Suppose $(p,p') \colon U \ra P\mb{X} \times_{s,X_0,t} P\mb{X}$ is a plot in $P\mb{X} \times_{s,X_0,t} P\mb{X}$ and $x \in U$. By the definition of sum diffeolgy (\Cref{sum diffeology}), there exist open neighbourhoods $U^{n}_x$ and $U^{n'}_x$ of $x$ and indexes $n,n' \in \mb{N} \cup \lbrace 0 \rbrace$ such that $p|_{U^{n}_x} \in D_{P\mb{X}_n}$ and $p|_{U^{n'}_x} \in D_{P\mb{X}_{n'}}$. Thus, by the definition of $\tilde{m}$, it is obvious that $ \big(m \circ (p,p') \big)|_{U_x} \in D_{P\mb{X}_{n+n'}}$, where $U_x=U^{n'}_x \cap U^{n}_x$, and hence, $\tilde{m}$ is smooth, and this proves (a). \\(b) and (c) can be proved using similar techniques as we used in the proof of (a). 
\end{proof}

The following proposition shows that the thin fundamental groupoid of a Lie groupoid (\Cref{Thin homotopy groupod of a Lie groupoid}) is a diffeological groupoid (\Cref{Definition: Diffeological groupoid}).
\begin{proposition}\label{Thin fundamental groupoid of a Lie groupoid is a diffeological groupoid}
	For any Lie groupoid $\mb{X}$, $\Pi_{\rm{thin}}(\mb{X})$ (\Cref{Thin homotopy groupod of a Lie groupoid}) is a diffeological groupoid.
\end{proposition}
\begin{proof}
	As we have already shown, $\frac{P\mb{X}}{\sim}$ is a diffeological space (\Cref{Diffeology of quotient of Xpath space}), the only thing that is left to be shown is the smoothness of the structure maps, i.e., the structure maps descent to maps of diffeological spaces.
	
	\Cref{Technical 1} ensures that the source-target are maps of diffeologial spaces. Now, suppose $(p_1,p_2) \colon U \ra \frac{P {X}}{\sim} \times_{s,X_0,t} \frac{P \mb{X}}{\sim}$ is a plot of $\frac{P \mb{X}}{\sim} \times_{s,X_0,t} \frac{P \mb{X}}{\sim}$. Thus, by the definition of quotient diffeology, there is a cover $ \lbrace U_i \rbrace$ of $U$ such that for each $i$, we have 
	\begin{itemize}
		\item a plot $\bar{p}^{i}_1 \colon U_i \ra P \mb{X}$ and $q \circ \bar{p}^{i}_1 =p_1|_{U_i}$ and
		\item  a plot $\bar{p}^{i}_2 \colon U_i \ra P\mb{X}$ and $q \circ \bar{p}^{i}_2 =p_2|_{U_i}$,
	\end{itemize}
	where $q$ is the quotient map. Hence,  it is clear that $(\bar{p}^{i}_1, \bar{p}^{i}_2 ) \colon U_i \ra P\mb{X} \times_{s,X_0,t} P\mb{X}$ is a plot of $P\mb{X} \times_{s,X_0,t} P\mb{X}$. Then, the  commutativity of the diagram below
	\[
	\begin{tikzcd}
		P\mb{X} \times_{s,X_0,t} P\mb{X} \arrow[d, "{(q,q)}"'] \arrow[r, "\tilde{m}"] & P\mb{X} \arrow[d, "{q}"] \\
		\frac{P \mb{X}}{\sim} \times_{s,X_0,t} \frac{P \mb{X}}{\sim} \arrow[r, "m"']                & \frac{P \mb{X}}{\sim}               
	\end{tikzcd}\]
	guarantees that the composition is a map of diffeological spaces. Here, $\tilde{m}$ is the multiplication map defined in \Cref{composition of PX}.
	
	The smoothness of the unit map and the inverse map can be verified in a similar fashion using, respectively, the commutativity of the following diagrams: 
	
	\[
	\begin{tikzcd}
		X_0 \arrow[d, "{\rm{id}}"'] \arrow[r, "\tilde{u}"] & P\mb{X} \arrow[d, "{q}"] \\
		X_0 \arrow[r, "u"']                & \frac{P \mb{X}}{\sim}               
	\end{tikzcd}\]  
	
	and
	
	\[
	\begin{tikzcd}
		P\mb{X} \arrow[d, "{\rm{q}}"'] \arrow[r, "\tilde{i}"] & P\mb{X} \arrow[d, "{q}"] \\
		\frac{P \mb{X}}{\sim} \arrow[r, "i"']                & \frac{P \mb{X}}{\sim}               
	\end{tikzcd},\]  
	
	where $\tilde{u}$ and $\tilde{i}$  are the maps defined in \Cref{composition of PX}.
	
	Hence, $\Pi_{\rm{thin}}(\mb{X})$ is a diffeological groupoid.

\end{proof}

We end the section with a simple observation:

\begin{lemma}\label{Morphism of thin homotopy groupoid}
	Any morphism of Lie groupoids $F \colon \mb{Y} \ra \mb{X}$ induces a morphism of diffeological groupoids (\Cref{Definition:Morphism of diffeological grouopoids}) between the respective thin fundamental groupoids, $$F_{\rm{thin}} \colon \Pi_{\rm{thin}}(\mb{Y}) \ra \Pi_{\rm{thin}}(\mb{X}),$$ defined as $y \mapsto F(y)$ for each $y \in Y_0$ and a class of lazy $\mb{Y}$-path $[\Gamma]:=[(\gamma_0, \alpha_1, \gamma_1, \cdots \alpha_n,\gamma_n)]$ goes to the class of lazy $\mb{X}$-path $[F(\Gamma)]:=[ \big( F(\gamma_0), F \circ \alpha_1, F(\gamma_1), \cdots ,F \circ \alpha_n, F(\gamma_n) \big)]$. 
\end{lemma}
Recall, we discussed a similar result for fundamental groupoids in \Cref{Induced morphism between haefliger groupoids}.

\section{Parallel transport on a quasi-principal 2-bundle along a lazy Haefliger path}\label{Parallel transport on a principal 2-bundle over a Lie groupooid}
This section introduces the notion of parallel transport along a lazy Haefliger path (\Cref{Definition:Haefliger path}). We proceed in three steps:

Let $\mb{G}$ be a Lie 2-group. Consider a strict connection $\omega$ on a quasi-principal $\mb{G}$-bundle $(\pi \colon \mb{E} \ra \mb{X} ,\mc{C})$ and a lazy $\mb{X}$-path $\Gamma= (\gamma_0, \alpha_1,\gamma_1, \cdots,\alpha_n, \gamma_n)$. 
\begin{itemize}
	\item[Step 1:] For every element $ x_i\xrightarrow{\gamma_i} y_i$ in $X_1$, we will define a $\mb{G}$-equivariant isomorhism of Lie groupoids $T_{\mc{C}, \pi} \colon \pi^{-1}(y_i) \ra \pi^{-1}(x_i)$ induced by the quasi connection $\mc{C}$.
	\item[Step 2:] For every path $\alpha_i \colon [0,1] \ra X_0$ in $X_0$, we will define a $\mb{G}$-equivariant isomorphism of Lie groupoids $T_{\omega}^{\alpha_i} \colon \pi^{-1}(x'_i) \ra \pi^{-1}(y'_i)$ induced from the strict connection $\omega$, where $x'_i= \alpha_i(0)$ and $y_i'= \alpha_i(1)$.
	\item[Step 3:] We will compose the above $\mb{G}$-equivariant isomorphisms of Lie groupoids successively to get $$T_{(\Gamma, \mc{C}, \omega)}: = T_{\mc{C}, \pi}(\gamma_n^{-1}) \circ  T_{\omega}^{\alpha_n} \circ \cdots \circ T_{\omega}^{\alpha_1} \circ  T_{\mc{C}, \pi}(\gamma_0^{-1}).$$
\end{itemize}
The novelty of this approach lies in showing how the functorial nature of our connection structures sync to the underlying fibrational structure of our principal 2-bundles.

\subsection{Step-1 (Transport along morphisms):}\label{Step1}
To suit our purpose, first, we reinterpret some notions discussed in \Cref{subsection Fibered categories}.

\subsection*{Fibered categories and pseudofunctors- Revisited.}\label{Fibered categories revisted}

Following is a direct consequence of the Axiom of Choice:

\begin{lemma}\label{fibered category in terms of function}
	A category $\pi \colon \mb{E} \ra \mb{X}$ over $\mb{X}$ is a fibered category over $\mb{X}$ (\Cref{Definition of fibered categories}) if and only if there exists a function $\bar{\mc{C}} \colon X_1 \times_{t,X_0,\pi_0}E_0 \ra \rm{Cart}(\mb{E})$ such that $\pi_1 \big( \bar{\mc{C}}(\gamma,p) \big)= \gamma$ and $t \big( \bar{\mc{C}}(\gamma,p)\big)=p$, for all $(\gamma,p) \in X_1 \times_{t,X_0,\pi_0}E_0 $, where $\rm{Cart}(\mb{E})$ is the set of cartesian morphisms in $\mb{E}$.
\end{lemma}

In particular, when both $\mb{E}$ and $\mb{X}$ are groupoids, we get the following proposition as a direct consequence of \Cref{lemma: Groupoid Cartesian}.	
\begin{lemma}\label{Cartesian morphisms in groupoids}
	If $\pi \colon \mb{E} \ra \mb{X}$ is  a category over $\mb{X}$ such that both $\mb{E}$ and $\mb{X}$ are groupoids, then $\rm{Cart}(\mb{E})= E_1$.
\end{lemma}
Hence, we get the following characterization:

\begin{lemma}\label{Characterisation of a fibered category in terms of sections}
	If $\mb{E}:=[E_1 \rra E_0]$ and $\mb{X}:=[X_1 \rra X_0]$ are groupoids, then the category $\pi \colon \mb{E} \ra \mb{X}$ over $\mb{X}$ is a fibered category over $\mb{X}$ if and only if the map $P \colon E_1 \ra X_1 \times_{s,X_0,\pi_0}E_0$, $\gamma \mapsto (\pi_1(\gamma), s(\gamma))$ admits a section.
\end{lemma}

\begin{proof}
	If $\pi \colon \mb{E} \ra \mb{X}$ is a fibered category over $\mb{X}$, then by \Cref{fibered category in terms of function} and \Cref{Cartesian morphisms in groupoids}, we have a function $\bar{\mc{C}} \colon X_1 \times_{t,X_0,\pi_0}E_0 \ra E_1$ such that $\pi_1 \big( \bar{\mc{C}}(\gamma,p) \big)= \gamma$ and $t \big( \bar{\mc{C}}(\gamma,p)\big)=p$, for all $(\gamma,p) \in X_1 \times_{t,X_0,\pi_0}E_0 $. Since $\mb{E}$ and  $\mb{X}$ are groupoids, the following map is well defined
	\begin{equation}\nonumber
		\begin{split}
			\mc{C} \colon & X_1 \times_{s,X_0,\pi_0}E_0 \ra E_1\\
			& (\gamma,p) \mapsto \bar{\mc{C}}(\gamma^{-1},p)^{-1}.
		\end{split}
	\end{equation}
	Now, $P \circ \mc{C}(\gamma,p)= \big( \pi_1(\bar{\mc{C}}(\gamma^{-1},p)^{-1}), s(\bar{\mc{C}}(\gamma^{-1},p)^{-1}) \big)= (\gamma,p)$.  Hence, $\mc{C}$ is a section of $P$.
	
	Conversely, let us assume $P$ admits a section $\mc{C} \colon X_1 \times_{s,X_0,\pi_0}E_0 \ra E_1$. Then, consider the function 
	\begin{equation}\nonumber
		\begin{split}
			\bar{\mc{C}} \colon & X_1 \times_{t,X_0,\pi_0}E_0 \ra E_1\\
			& (\gamma,p) \mapsto \mc{C}(\gamma^{-1},p)^{-1}
		\end{split}
	\end{equation}
	Since, $\pi_1 \big( \bar{\mc{C}}(\gamma,p) \big)=\gamma$ and $t\big( \bar{\mc{C}}(\gamma,p) \big)=p$, we get that $\pi \colon \mb{E} \ra \mb{X}$ is a fibered category over $\mb{X}$. 
\end{proof}

The above association defines a pair of one-one correspondences, as we see below:
\begin{proposition}\label{cleavage and section correspondence}
	Let $\mb{E}$ and $\mb{X}$ be groupoids. Then, for any fibered category $\pi \colon \mb{E} \ra \mb{X}$ we have the following one-one correspondences:
	\begin{itemize}
		\item[(i)] the set of cleavages (\Cref{Definition cleavage}) on $\pi \colon \mb{E} \ra \mb{X}$ is in one-one correspondence with the set of sections $\mc{C} \colon X_1 \times_{s,X_0,\pi_0}E_0 \ra E_1$ of the map $P \colon E_1 \ra X_1 \times_{s,X_0,\pi_0}E_0$, $\gamma \mapsto (\pi_1(\gamma), s(\gamma))$,
		\item[(ii)] the set of splitting cleavages (\Cref{Definition splitting cleavage}) on $\pi \colon \mb{E} \ra \mb{X}$ is in one-one correspondence with the set of sections $\mc{C} \colon X_1 \times_{s,X_0,\pi_0}E_0 \ra E_1$ of the map $P \colon E_1 \ra X_1 \times_{s,X_0,\pi_0}E_0$ such that 
		\begin{itemize}
			\item[(a)] $\mathcal{C}(1_x,p)=1_p$  for any $x\in X_0$ and $p\in \pi_{0}^{-1}(x)$;
			\item[(b)] if $(\gamma_2, p_2), (\gamma_1, p_1) \in {s}^{*}E_0$ such that ${s}(\gamma_2)={t}(\gamma_1)$ and $p_2=t\bigl({\mathcal C}(\gamma_1, p_1)\bigr),$ then $\mathcal{C}(\gamma_2 \circ \gamma_1 , p_1)= \mathcal{C}(\gamma_2, p_2) \circ \mathcal{C}(\gamma_1, p_1)$.
		\end{itemize}
	\end{itemize}
\end{proposition}
\begin{proof}
	The proof follows by observing that the association defined in \Cref{Characterisation of a fibered category in terms of sections} are inverse to each other.
\end{proof}

The correspondence \Cref{cleavage and section correspondence} gives a simple yet convenient way of reinterpreting the pseudofunctor $\mc{F} \colon \mc{X}^{{\rm{op}}} \ra {\rm{Cat}}$ associated to a fibered category $\pi \colon \mc{E} \ra \mc{X}$, (as constructed in \Cref{Construction of a pseudofunctor from a fibered category}) when $\mc{E}$ and $\mc{X}$ are groupoids, as we see below:
\begin{proposition}\label{Revisted Pseudo}
	Suppose $\mc{E}:=[E_1 \rra E_0]$ and $\mc{X}:=[X_1 \rra X_0]$  be groupoids. Let $\pi \colon \mc{E} \ra \mc{X}$ be a fibered category over $\mc{X}$ with a cleavage $K$ and let $\mc{C} \colon X_1 \times_{s,X_0,\pi_0}E_0 \ra E_1$ be the associated section of the map $\mc{P} \colon E_1 \ra X_1 \times_{s,X_0,\pi_0}E_0$ corresponding to $K$. Then the pseudofucntor $\mc{F} \colon \mc{X}^{{\rm{op}}} \ra {\rm{Cat}}$  associated to $\pi \colon \mc{E} \ra \mc{X}$ and $K$ is given by the following data:

	\begin{itemize}
		\item[(a)] each $x \in X_0$ is assigned to a groupoid $\mc{F}(x):= \pi^{-1}(x)$,
		\item[(b)]each morphism $x \xrightarrow {\gamma} y$ is assigned  to a functor
		\begin{equation}\nonumber
			\begin{split}
				\mc{F}(\gamma) \colon & \pi^{-1}(y) \ra \pi^{-1}(x)\\
				& p \mapsto t \big(\mc{C}(\gamma^{-1},p) \big);\\
				& (p \xrightarrow[]{\zeta}q) \mapsto \mc{C}(\gamma^{-1},q) \circ \zeta \circ \big(\mc{C}(\gamma^{-1},p) \big)^{-1},
			\end{split}
		\end{equation} 
		\item[(c)]for each $x \in X_0$, we have a natural isomorphism 
		\begin{equation}\nonumber
			\begin{split}
				I_{x} \colon & \mc{F}(1_x) \Longrightarrow 1_{\mc{F}(x)}\\
				& p \mapsto \bigg( t \big(\mc{C}(1_x,p)\big) \xrightarrow[]{\mc{C}(1_x,p)^{-1}}p \bigg),
			\end{split}
		\end{equation}
		\item[(d)]for each pair of composable arrows 
		\begin{tikzcd}
			x \arrow[r, "\gamma_{1}"] & y \arrow[r, "\gamma_2"] & z
		\end{tikzcd},
		we have a natural isomorphism 
		\begin{equation}\nonumber
			\begin{split}
				\alpha_{\gamma_1, \gamma_2} \colon & \mc{F}(\gamma_1) \circ \mc{F}(\gamma_2) \Longrightarrow \mc{F}\gamma_2 \circ \gamma_1)\\
				& p \mapsto \mc{C}(\gamma_1^{-1} \circ \gamma_2^{-1},p) \circ \mc{C}(\gamma_2^{-1},p)^{-1} \circ \mc{C}\big(\gamma_1^{-1}, t(\mc{C}(\gamma_2^{-1},p)) \big)^{-1},
			\end{split}
		\end{equation}
	\end{itemize}
	such that $\alpha_{\gamma_1,\gamma_2}$ and $I_x$ satisfy the necessary coherence laws of \Cref{Coherence diagram 2} and \Cref{Coherence diagram 1} respectively.
\end{proposition}

Equipped with the interpretations discussed above (\Cref{Fibered categories revisted}), we are now ready to complete Step-1.

\subsection*{Transport along morphisms:}		
\begin{definition}\label{Lie 2-group torsor}
	Given a Lie 2-group $\mb{G}$, a \textit{$\mb{G}$-torsor} is defined as a Lie groupoid $\mb{X}$ with an action of $\mb{G}$ such that manifolds $X_0$ and $X_1$ are $G_0$-torsor and $G_1$-torsor respectively. 
\end{definition}
Collection of $\mb{G}$-torsors, $\mb{G}$-equivariant morphisms of Lie groupoids (\Cref{Equivariant morphism of Lie gorupoid}) and $\mb{G}$-equivariant natural transformations (\Cref{Equivariant 2-morphism of Lie groupoid}) form a 2-groupoid which we denote by $\mb{G}$-Tor.
\begin{example}
	For a Lie 2-group $\mb{G}$, let  $\pi \colon \mb{E} \ra \mb{X}$ be a principal $\mb{G}$-bundle over a Lie groupoid $\mb{X}$. Then, for any $x \in X_0$, the \textit{fibre} $\pi^{-1}(x):=[\pi_1^{-1}(1_x) \rra \pi_0^{-1}(x_0)]$ is a $\mb{G}$-torsor.
\end{example}

Now, as a consequence of \Cref{Revisted Pseudo}, we get the following:

\begin{proposition}\label{T_C}
	Let $\mb{G}$ be a Lie 2-group. For a quasi-principal $\mb{G}$-bundle $(\pi \colon \mb{E} \ra \mb{X}, \mc{C})$ over a Lie groupoid $\mb{X}$, there is an associated \textit{$\mb{G}$-${\rm{Tor}}$-valued pseudofunctor } $T_{\mc{C}} \colon \mb{X}^{\rm{op}} \ra \mb{G}$-Tor. Explicitly,
	\begin{itemize}
		\item[(a)]each $x \in X_0$ is assigned to the $\mb{G}$-Torsor $T_{\mc{C}}(x):= \pi^{-1}(x)$,
		\item[(b)]each morphism $x \xrightarrow {\gamma} y$ is assigned  to an isomorphism of $\mb{G}$-torsors
		\begin{equation}\label{Pseudomor}
			\begin{split}
				T_{\mc{C}}(\gamma) \colon & \pi^{-1}(y) \ra \pi^{-1}(x)\\
				& p \mapsto \mu_{\mc{C}}(\gamma^{-1},p));\\
				& (p \xrightarrow[]{\zeta}q) \mapsto \mc{C}(\gamma^{-1},q) \circ \zeta \circ \big(\mc{C}(\gamma^{-1},p) \big)^{-1},
			\end{split}
		\end{equation} 
		\item[(c)]for each $x \in X_0$, we have a smooth $\mb{G}$-equivariant natural isomorphism 
		\begin{equation}\label{Pseudonat1}
			\begin{split}
				I_{x} \colon & T_{\mc{C}}(1_x) \Longrightarrow 1_{\pi^{-1}(x)}\\
				& p \mapsto \bigg( \mu_{\mc{C}}(1_x,p) \xrightarrow[]{\mc{C}(1_x,p)^{-1}}p \bigg),
			\end{split}
		\end{equation}
		\item[(d)]for each pair of composable arrows 
		\begin{tikzcd}
			x \arrow[r, "\gamma_{1}"] & y \arrow[r, "\gamma_2"] & z
		\end{tikzcd},
		we have a smooth $\mb{G}$-equivariant natural isomorphism 
		\begin{equation}\label{pseudonat2}
			\begin{split}
				\alpha_{\gamma_1, \gamma_2} \colon & T_{\mc{C}}(\gamma_1) \circ T_{\mc{C}}(\gamma_2) \Longrightarrow T_{\mc{C}}(\gamma_2 \circ \gamma_1)\\
				& p \mapsto \mc{C}(\gamma_1^{-1} \circ \gamma_2^{-1},p) \circ \mc{C}(\gamma_2^{-1},p)^{-1} \circ \mc{C}\big(\gamma_1^{-1}, t(\mc{C}(\gamma_2^{-1},p)) \big)^{-1},
			\end{split}
		\end{equation}
	\end{itemize}
	such that $\alpha_{\gamma_1,\gamma_2}$ and $I_x$ satisfy the necessary coherence laws of \Cref{Coherence diagram 2} and \Cref{Coherence diagram 1} respectively.
\end{proposition}

\begin{definition}\label{transport along morphism}
	For a Lie 2-group $\mb{G}$, let $(\pi \colon \mb{E} \ra \mb{X}, \mc{C})$ be a quasi-principal $\mb{G}$-bundle over a Lie groupoid $\mb{X}$. If $T_{\mc{C}} \colon \mb{X}^{\rm{op}} \ra \mb{G}$-Tor is the associated $\mb{G}$-Tor valued pseudofunctor (as defined in \Cref{T_C}), then given $\gamma \in X_1$, the $\mb{G}$-equivariant isomorphism of Lie groupoids
	\begin{equation}\nonumber
		\begin{split}
			T_{\mc{C}}(\gamma) \colon & \pi^{-1}(y) \ra \pi^{-1}(x)\\
			& p \mapsto \mu_{\mc{C}}(\gamma^{-1},p))\\
			& (\zeta \colon p \ra q) \mapsto \mc{C}(\gamma^{-1},q) \circ \zeta \circ \big(\mc{C}(\gamma^{-1},p) \big)^{-1}
		\end{split}
	\end{equation}
	is defined as the \textit{transport on $(\pi \colon \mb{E} \ra \mb{X}, \mc{C})$ along the morphism $\gamma$}.
\end{definition}

	\subsection{Step-2 (Transport along paths):}\label{Step2}
First, we recall the notational convention we used in \Cref{Parallel transport on a principal bundle} to describe the traditional connection-induced horizontal lift of paths and the corresponding parallel transport map. 

\textbf{Some conventions and notations:}
For a Lie group $G$, let $A$ be a connection on a traditional principal $G$-bundle $\pi: E \ra M$ over a manifold $M$. Then,  given a  smooth path  $\alpha:[0,1] \ra M$, for each point $p \in \pi^{-1}(\alpha(0))$,  we denote the unique horizontal lift of the path $\alpha$ starting from $p$ by $\tilde{\alpha}_{A}^{p}$, and the associated parallel transport map by ${\rm{Tr}}_A^{\alpha} \colon \pi^{-1}(\alpha(0)) \ra \pi^{-1}(\alpha(1))$, see \Cref{Parallel transport aloing a path}.

%


The following is a consequence of the underlying functorial nature of our connection structures (\Cref{strict ans semi strict connetion 1-forms}).

\begin{lemma}\label{Lemma: source-target strict transport}
	For a Lie 2-group $\mb{G}$, let $\pi: \mb{E} \ra \mb{X}$ be a prinicpal $\mb{G}$-bundle over a Lie groupoid $\mb{X}$. Any strict connection $\omega: T\mb{E} \ra L(\mb{G})$ induces the follwoing:
	
	For any path $\zeta \colon [0,1] \ra X_1$ and $\alpha:[0,1] \ra X_0$, we have the following identities:
	\begin{itemize}
		\item[(1)] ${\rm{Tr}}_{\omega_0}^{s \circ \zeta}(s (\delta)) = s ({\rm{Tr}}_{\omega_1}^{\zeta}( \delta))$ for each $\delta \in \pi_1^{-1}(\zeta(0))$.
		\item[(2)]  ${\rm{Tr}}_{\omega_0}^{t \circ \zeta}(t (\delta)) = t ({\rm{Tr}}_{\omega_1}^{\zeta}( \delta))$ for each $\delta \in \pi_1^{-1}(\zeta(0))$.
		\item[(3)] ${\rm{Tr}}_{\omega_1}^{u \circ \alpha}(u(p))= u ({{\rm{Tr}}_{\omega_0}^{\alpha}(p)})$ for each $p \in \pi_0^{-1}(\alpha(0))$.
	\end{itemize}
\end{lemma}
\begin{proof}
	Observe that to prove the above identities, it is sufficient to show the following:
	\begin{itemize}
		\item[(i)]$\widetilde{s \circ \zeta}_{\omega_0}^{s(\delta)}= s \circ \tilde{\zeta}_{\omega_1}^{\delta}$;
		\item[(ii)]$\widetilde{t \circ \zeta}_{\omega_0}^{t(\delta)}= t \circ \tilde{\zeta}_{\omega_1}^{\delta}$;
		\item[(iii)]$\widetilde{u \circ \alpha}_{\omega_1}^{1_{p}}= u \circ {\tilde{\alpha}_{\omega_0}^{p}}$.
	\end{itemize}
 
Proof of (i):

	Observe that $s \circ \tilde{\zeta}_{\omega_1}^{\delta}(0)= s(\delta)$ and $\pi_0(s \circ \tilde{\zeta}_{\omega_1}^{\delta})= s(\pi_1(\tilde{\zeta}_{\omega_1}^{\delta}))= s \circ \zeta$.  From the functoriality of $\omega$ it is immediate that for any $r \in [0,1]$, $(s \circ \tilde{\zeta}_{\omega_1}^{\delta})_{*,r}(\frac{d}{dt}|_r) $ is horizontal. Hence, by the uniqueness of horizontal path lifting, we have $\widetilde{s \circ \zeta}_{\omega_0}^{s(\delta)}= s \circ \tilde{\zeta}_{\omega_1}^{\delta}$.
	
	One can prove (ii) and (iii) using exactly similar techniques as in (i).
\end{proof}

As a consequence of \Cref{Lemma: source-target strict transport}, we obtain the following:

\begin{proposition}\label{Proposition: Parallel transport on principal 2-bundles}
	Given a Lie 2-group $\mb{G}$, let $\pi: \mb{E} \ra \mb{X}$ be a prinicpal $\mb{G}$-bundle over a Lie groupoid $\mb{X}$ equipped with a strict connection $\omega$. Then for any given  path $\alpha: [0,1] \ra X_0$, there is a $\mb{G}$-equivariant isomorphism of Lie groupoids
	\begin{equation}\nonumber
		\begin{split}
			T^{\alpha}_{\omega} \colon & \pi^{-1}(x) \ra \pi^{1}(y)\\
			& p \mapsto {\rm{Tr}}_{\omega_0}^\alpha(p)\\
			& \gamma \mapsto {\rm{Tr}}_{\omega_1}^ {u \circ \alpha}(\gamma)
		\end{split}
	\end{equation}
	for all $p \in \pi_{0}^{-1}(x)$ and $\gamma \in \pi_{1}^{-1}(1_{x})$, where $\alpha(0)=x$ and $\alpha(1)=y$.
\end{proposition}
\begin{proof}

	To prove $T^{\alpha}_{\omega}$ is a functor, we need to show the following:
	\begin{itemize}
		\item[(a)] $s(T^{\alpha}_{\omega}(\gamma))= T^{\alpha}_{\omega}(s(\gamma)) $ for all $\gamma \in \pi^{-1}(x)$,
		\item[(b)] $t(T^{\alpha}_{\omega}(\gamma))= T^{\alpha}_{\omega}(t(\gamma)) $ for all $\gamma \in \pi^{-1}(x)$,
		\item[(c)] $T_{\omega}^{\alpha}(u(p))= u(T_{\omega}^{\alpha}(p))$ for all $p \in \pi^{-1}(x)$,
		\item[(d)] $T_{\omega}^{\alpha}(\gamma_2 \circ \gamma_1)= T_{\omega}^{\alpha}(\gamma_2) \circ T_{\omega}^{\alpha}(\gamma_1)$ for all suitable $\gamma_1, \gamma_2 \in \pi^{-1}(x)$.
	\end{itemize}
	Note that to prove (a), (b) and (c) it is sufficient to show $$s({\rm{Tr}}_{\omega_1}^{u \circ \alpha}(\gamma))= {\rm{Tr}} _{\omega_0}^{s( u \circ \alpha)}(s(\gamma)),$$  $$t({\rm{Tr}}_{\omega_1}^{u \circ \alpha}(\gamma))= {\rm{Tr}} _{\omega_0}^{t( u \circ \alpha)}(t(\gamma))$$ and 
	$$ {\rm{Tr}}_{\omega_1}^{u \circ \alpha}(u(p)) = u({\rm{Tr}}_{\omega_0}^{\alpha}(p)),$$ all of which follow  directly from \Cref{Lemma: source-target strict transport}. Whereas to prove (d), it is sufficient to show $$\widetilde{u \circ \alpha}_{\omega_1}^{\gamma_2 \circ \gamma_1} = \widetilde{u \circ \alpha}_{\omega_1}^{\gamma_2} \circ \widetilde{u \circ \alpha}_{\omega_1}^{\gamma_1},$$ which can be established in a precisely similar way as in the proof of \Cref{Lemma: source-target strict transport}, using the functoriality of $\pi$, $\omega$ and the uniqueness of horizontal lifting of paths. Smoothness and $\mb{G}$-equivariance of $T^{\alpha}_{\omega}$ directly follow from the traditional theory (\Cref{Parallel transport on a principal bundle}).
\end{proof} 	
\begin{remark}
	When the Lie 2-group $\mb{G}$ is expressed in terms of a Lie crossed module $(G,H, \tau, \alpha)$ i.e when $\mb{G}= [H \rtimes_{\alpha} G \rra G]$, then the definition of $T^{\alpha}_{\omega}$ in \Cref{Proposition: Parallel transport on principal 2-bundles} can be expessed as
	\begin{equation}\nonumber
		\begin{split}
			T^{\alpha}_{\omega} \colon & \pi^{-1}(x) \ra \pi^{1}(y)\\
			& p \mapsto {\rm{Tr}}_{\omega_0}^\alpha(p)\\
			& \gamma \mapsto 1_{{\rm{Tr}}_{\omega_0}^\alpha(r)} (h,e)
		\end{split}
	\end{equation}
	where $s(\gamma)=r$ and $h$ is the unique element in $H$ such that $\gamma= 1_r (h,e)$.
\end{remark}
\begin{definition}\label{transport along path}
	For a Lie 2 group $\mb{G}$, let $\pi: \mb{E} \ra \mb{X}$ be a prinicpal $\mb{G}$-bundle over a Lie groupoid $\mb{X}$  with a strict connection $\omega$. Then given a path $\alpha: [0,1] \ra X_0$, the associated $\mb{G}$-equivariant isomorphism of Lie groupoids $T^{\alpha}_{\omega} \colon  \pi^{-1}(x) \ra \pi^{1}(y)$ defined in \Cref{Proposition: Parallel transport on principal 2-bundles} is called the \textit{transport of $\pi: \mb{E} \ra \mb{X}$ along the path $\alpha$.}
\end{definition}

\subsection{Step-3 (Definition of parallel transport along a lazy Haefliger path):}\label{Step3}
Combining the results of step 1 and step 2, we are now equipped to define a notion of parallel transport on a quasi-principal 2-bundle along a lazy Haefliger path.
\begin{definition}\label{Definition: Parallel transport of X-paths}
	Let $\mb{G}$ be a Lie 2-group. Given a quasi-principal $\mb{G}$-bundle  $(\pi: \mb{E} \ra \mb{X}, \mc{C})$, a strict connection $\omega$ and a lazy $\mb{X}$-path $\Gamma :=(\gamma_0, \alpha_1, \gamma_1, \cdots, \alpha_n, \gamma_n)$ from $x$ to $y$,  the $\mb{G}$-equivariant isomorphism of Lie groupoids $T_{(\Gamma, \mc{C}, \omega)}: = T_{\mc{C}, \pi}(\gamma_n^{-1}) \circ  T_{\omega}^{\alpha_n} \circ \cdots \circ T_{\omega}^{\alpha_1} \circ  T_{\mc{C}, \pi}(\gamma_0^{-1}) $ is defined as the \textit{$(\mc{C}, \omega)$-parallel transport along the lazy $\mb{X}$-path $\Gamma$.}
\end{definition}
\begin{example}[Classical principal bundle]
	Let $\pi \colon [E \rra E] \ra [M \rra M]$ be a principal $[G \rra G]$ bundle  over a discrete Lie groupoid $[M \rra M]$, equipped with the strict connection $\omega:= (\omega,\omega)$ (\Cref{Classical connection as 2-connection}) and the unique categorical connection $\mc{C}$, $(1_x,p) \mapsto 1_p$ for  all $(p,x)$ satisfying $\pi(p)=x$ (\Cref{unique Cat connection on bundle over discrete space}).  Since any lazy $[M \rra M]$-path $\Gamma$ is of the form $(1_{\alpha(0)}, \alpha, 1_{\alpha(1)})$, we have $T_{(\Gamma, \mc{C}, \omega)}= T^{\alpha}_{\omega}$.
\end{example}
\begin{example}[Principal 2-bundle over a manifold]\label{Parallel transport on a Principal 2-bundle over a manifold}
	
	As mentioned earlier in \Cref{Our 2-bundle and Waldrorf}, for a Lie 2-group $ \mb{G}:=[H \rtimes_{\alpha}G \rra G]$, when our base Lie groupoid is of the form $[M \rra M]$, our principal $\mb{G}$-bundle (\Cref{Definition:principal $2$-bundle over Liegroupoid}) coincides with the definition of a principal $\mb{G}$-bundle over a manifold $M$ as defined in \textbf{Definition 3.1.1}, \cite{MR3894086}. Also, it is known that such principal $\mb{G}$-bundle is of the form $\pi \colon [E \times H \rra E] \rra [M \rra M]$ (\Cref{Ex:EoXodecobundle}). Now, \Cref{Ex:EoXodecobundleconnect} implies that any classical connection $\omega \in \Omega(E, L(G))$ on the principal bundle $E \ra M$ induces a strict connection $\bar{\omega}$ on $\pi \colon [E \times H \rra E] \rra [M \rra M]$. On the other hand, $\omega$ also defines a connection $\tilde{\omega}$ in the sense of Waldorf's \textbf{Definition 5.1.1}, \cite{MR3894086} as shown in \textbf{Example 5.1.11}, \cite{MR3894086}. Now, for an arbitrary lazy $[M \rra M]$-path $\Gamma=(1_{\alpha(0)}, \alpha, 1_{\alpha(1)})$ and a quasi connection  $\mc{C}$ (whose existence is ensured by  \Cref{qusi bundle over discrete groupoid}), consider the $\mb{G}$-equivariant anafunctor $\mc{F}_{ \Gamma, \mc{C}, \bar{\omega}}$ induced by $\mc{T}_{(\Gamma, \mc{C}, \bar{\omega})}$ (see \Cref{Lie 2-group equivariant anafunctor} and for the details on the construction of a $\mb{G}$-equivariant anafunctor from a $\mb{G}$-equivariant morphism of Lie groupoids, see \Cref{Functor induced anafunctor}). Then it follows from \textbf{Section 5.2}, \cite{MR3917427}, that $\mc{F}_{ \Gamma, \mc{C}, \bar{\omega}}$ and the parallel transport $\mb{G}$-equivariant anafunctor associated to the path $\alpha$ in $X_0$ and the conection $\tilde{\omega}$ (as defined in \textbf{Propoition 3.26}, \cite{MR3917427}) are related by a $\mb{G}$-equivaraint transformation of anafunctors (\Cref{Lie 2-group equivariant anafunctor}, and see \textbf{Definition 2.4.1(c)}, \cite{MR3894086} for details).

\end{example}
		\section{Lazy $\mb{X}$-path thin homotopy invariance of the parallel transport}\label{Lazy Xpaththinhomotopy}
The purpose of this section is to establish a lazy $\mb{X}$-path thin homotopy invariance (\Cref{Definition: Thin homotopy of X-paths}) for our parallel transport notion introduced in \Cref{Definition: Parallel transport of X-paths}. The intended invariance will be proved in two steps.

\subsection*{Step-1 (Invariance under equivalences defined in  \Cref{Definition: Equivalence of X-paths})}

\begin{proposition}\label{Proposition: Equivalence invariance of parallel transport}
	For a Lie crossed module $(G,H, \tau, \alpha)$, let $(\pi: \mb{E} \ra \mb{X}, \mc{C})$ be a quasi-principal $\mb{G}:=[H \rtimes_{\alpha} G \rra G]$-bundle with a strict connection $\omega: T\mb{E} \ra L(\mb{G})$. If a lazy $\mb{X}$-path $\Gamma  :=(\gamma_0, \alpha_1,\gamma_1, \cdots,\alpha_n, \gamma_n)$ is equivalent (see \Cref{Definition: Equivalence of X-paths}) to a lazy $\mb{X}$-path $\Gamma'$, then there is a smooth $\mb{G}$-equivariant natural isomorphism between $T_{(\Gamma, \mc{C}, \omega)}$ and $T_{(\Gamma', \mc{C}, \omega)}$.
\end{proposition}

\begin{proof}
	As the equivalence relation in \Cref{Definition: Equivalence of X-paths} is generated by the operations (1), (2), and (3) of \Cref{Definition: Equivalence of X-paths}, it is sufficient for us to verify the following three cases:
	\begin{itemize}
		\item[(A)] If $\Gamma'$ is obtained from $\Gamma $ by the $(1)$ of \Cref{Definition: Equivalence of X-paths}, then there is a smooth $\mb{G}$-equivariant natural isomorphism between $T_{(\Gamma, \mc{C}, \omega)}$ and $T_{(\Gamma', \mc{C}, \omega)}$.
		\item[(B)] If $\Gamma'$ is obtained from $\Gamma$ by the $(2)$ of \Cref{Definition: Equivalence of X-paths}, then there is a smooth $\mb{G}$-equivariant natural isomorphism between $T_{(\Gamma, \mc{C}, \omega)}$ and $T_{(\Gamma', \mc{C}, \omega)}$.
		\item[(C)] If $\Gamma'$ is obtained from $\Gamma$ by the $(3)$ of \Cref{Definition: Equivalence of X-paths}, then there is a smooth $\mb{G}$-equivariant natural isomorphism between $T_{(\Gamma, \mc{C}, \omega)}$ and $T_{(\Gamma', \mc{C}, \omega)}$.
	\end{itemize}

 Proof of A:

Using notations as in the $(1)$ of \Cref{Definition: Equivalence of X-paths}, consider the diagram below:

 \[\begin{tikzcd}
			\cdot \arrow[r, "\gamma_i"] & \cdot \arrow["\alpha_{i+1}=\rm{constant}"', dotted, loop, distance=2em, in=305, out=235] \arrow[r, "\gamma_{i+1}"] & \cdot
		\end{tikzcd}\]
From the above diagram and \Cref{pseudonat2}, it is clear that there is a smooth $\mb{G}$-equivariant natural isomorphism between $T_{\mc{C}}(\gamma^{-1}_{i+1}) \circ T^{\alpha_i}_{\omega} \circ T_{\mc{C}}(\gamma^{-1}_{i})$ and $T_{\mc{C}}(\gamma^{-1}_{i+1}) \circ T_{\mc{C}}(\gamma^{-1}_{i})$ as $T^{\alpha_i}_{\omega}$ is identity.

Proof of (B):
 
  Using \Cref{Pseudonat1}, one can prove (B) using a similar technique as in the proof of (A).

Proof of (C):
	
	Suppose $\Gamma'$ is obtained from $\Gamma :=(\gamma_0, \alpha_1,\gamma_1, \cdots,\alpha_n, \gamma_n)$  by $(3) $ of \Cref{Definition: Equivalence of X-paths}. That means, given a path $\zeta_i \colon [0,1] \ra X_1$ with sitting instants, such that $s \circ \zeta_i= \alpha_i$, we replace $\alpha_i$ by $t \circ \zeta_i$,  $\gamma_{i-1}$ by $\zeta_i (0) \circ \gamma_{i-1}$ and $\gamma_{i}$ by $\gamma_i \circ (\zeta_i(1))^{-1}$, $i \in \lbrace 1,2, \cdots, n \rbrace$ to obtain $\Gamma'$. We need to show that
	\begin{equation}\label{Magic}
		T_{\omega}^{\alpha_i} \cong T_{\mc{C}}(\gamma') \circ T_{\omega}^{t \circ \zeta_i} \circ T_{\mc{C}}(\gamma^{-1}),
	\end{equation}
	where $\cong$ is a smooth $\mb{G}$-equivariant natural isomorphism, and $\gamma:= x \xrightarrow {\zeta_{i}(0)} y$, $\gamma':= x' \xrightarrow {\zeta_{i}(1)} y'$ are elements of $X_1$. Then, by the repetitive use of \Cref{T_C}, this is equivalent to showing that  given a square
	\begin{equation}\label{Haefliger square}
		\begin{tikzcd}
			x \arrow[r, "s \circ \zeta_i", dotted] \arrow[d, phantom] \arrow[d, "\gamma"'] & x' \arrow[d, "\gamma'"] \arrow[d] \arrow[d] \arrow[d] \\
			y \arrow[r, "t \circ \zeta_i"', dotted]                                   & y'                                             
		\end{tikzcd}
	\end{equation}
	
	the following square
	\[
	\begin{tikzcd}
		\pi_0^{-1}(y) \arrow[r, "T_{\omega}^{t \circ \zeta_i}"] \arrow[d, phantom] \arrow[d, "T_{\mc{C}}(\gamma)"'] & \pi_0^{-1}(y') \arrow[d, "T_{\mc{C}}(\gamma')"] \arrow[d] \arrow[d] \arrow[d] \\
		\pi_0^{-1}(x) \arrow[r, "T_{\omega}^{s \circ \zeta_i}"']                                   & \pi_0^{-1}(x')                                            
	\end{tikzcd}\]
	
	commutes upto a smooth $\mb{G}$-equivariant natural isomorphism. The dotted lines in \Cref{Haefliger square} represent the paths $s \circ \zeta_i: [0,1] \ra X_0$ and $t \circ \zeta_i: [0,1] \ra X_0$. We claim that the following is our desired smooth $\mb{G}$-equivariant natural isomorphism $\eta \colon  T_{\mc{C}}(\gamma') \circ T_{\omega}^{t \circ \zeta_i} \Longrightarrow T_{\omega}^{s \circ \zeta_i} \circ T_{\mc{C}}(\gamma)$:
	\begin{equation}\label{equivalence invariance}
		p \mapsto \eta_{p} := 1_{\mu_{\mc{C}} \big( \gamma'^{-1}, {\rm{Tr}}_{\omega_0}^{t \circ \zeta_i}(p) \big)} (h_p,e),
	\end{equation}
	where $h_p$ is the unique element in $H$ such that
	\begin{equation}\label{Equivalence h}
		\mc{C}\Big( \gamma'^{-1}, {\rm{Tr}}_{\omega_0}^{t \circ \zeta_i}(p) \Big)(h_p,e)= {\rm{Tr}}^{\mathfrak{i} \circ \zeta_{i}}_{\omega_1}\Big(\mc{C}(\gamma^{-1},p)\Big),
	\end{equation}
	where $\mathfrak{i} \colon X_1 \ra X_1$ is the inverse map. Now, observe from the definition of $\eta_p$ that to show the assignment $p \mapsto \eta_p$ is smooth, it is sufficient for us to prove the smoothness of the map $\mc{Q} \colon E_0 \ra H$ defined as $p \mapsto h_p$. The smoothness of $\mc{Q}$ then follows from the following sequence of composable smooth maps:

\[\begin{tikzcd}
\pi_0^{-1}(y) \arrow[r, "\bar{\mc{Q}}"] & E_1 \times_{\pi_1,X_1,\pi_1} E_1 \arrow[r, "\mc{Z}^{-1}"] & E_1 \times (H \rtimes_{\alpha}G) \arrow[r, "{\rm{pr}}_2"] & H \rtimes_{\alpha}G
\end{tikzcd}\]
where, $\bar{\mc{Q}}$ is defined by $p \mapsto \Bigg( \mc{C}\Big( \gamma'^{-1}, {\rm{Tr}}_{\omega_0}^{t \circ \zeta_i}(p) \Big), {\rm{Tr}}^{\mathfrak{i} \circ \zeta_{i}}_{\omega_1}\Big(\mc{C}(\gamma^{-1},p)\Big) \Bigg)$, $\mc{Z}$ is the diffeomorphism as defined  in the (i) of  
\Cref{Definition: Principal G-bundle}, and ${\rm{pr}}_2$ is the usual 2nd projection map.

 The source consistency of $\eta_p$ is clear from the construction. We check the target consistency as follows:
	\begin{equation}\nonumber
		\begin{split}
			\,\,&t(\eta_p)\\
			&= t((1_{\mu_{\mc{C}}(\gamma'^{-1}, {\rm{Tr}}_{\omega_0}^{t \circ \zeta_i}(p))} (h_p,e)))\\
			&= \mu_{\mc{C}}(\gamma'^{-1}, {\rm{Tr}}_{\omega_0}^{t \circ \zeta_i}(p))) \tau(h_p) \\
			&= t\bigg(\mc{C}\Big(\gamma'^{-1}, {\rm{Tr}}_{\omega_0}^{t \circ \zeta_i}(p)\Big)(h_p,e) \bigg)\\
			&= t\bigg({\rm{Tr}}^{\mathfrak{i} \circ \zeta_{i}}_{\omega_1}\Big(\mc{C}(\gamma^{-1},p)\Big)\bigg) \quad [\textit{by  \Cref{Equivalence h}}].
		\end{split}
	\end{equation}
	Thus, using \Cref{Lemma: source-target strict transport}, we get
	\begin{equation}\label{Target}
		t(\eta_p)= t\bigg({\rm{Tr}}^{\mathfrak{i} \circ \zeta_{i}}_{\omega_1} \Big(\mc{C}(\gamma^{-1},p)\Big)\bigg)= {\rm{Tr}}_{\omega_0}^{s \circ \zeta_i}\Big(\mu_{\mc{C}}(\gamma^{-1},p)\Big).
	\end{equation}
 
	$\mb{G}$-equivariance:
 
	Since
	\begin{equation}\label{E11}
		(h_p,e)=1_g(h_{pg},e)1_g^{-1}\,\, [\textit{by \Cref{Equivalence h}}]
	\end{equation}
	and 
	\begin{equation}\nonumber
		\eta_{pg}= 1_{\mu_{\mc{C}}\big(\gamma'^{-1}, {\rm{Tr}}_{\omega_0}^{t \circ \zeta_i}(p)\big)}1_g (h_{pg},e)\,\, [\textit{by \Cref{equivalence invariance},}],
	\end{equation}
	we have
	\begin{equation}\nonumber
		\eta_{pg}=\eta_p1_g\,\, [\textit{by \Cref{E11}}].
	\end{equation}
	
	Verification of the naturality square:
 
	To ensure that $\eta$ satisfies the naturality square, that is for every $  p \xrightarrow {\delta} q \in \pi^{-1}(y)$, 
		
		$\begin{tikzcd}\label{Natural diagram}
			T_{\mc{C}}(\gamma{\rm{'}}) \circ T_{\omega}^{t \circ \zeta_i}(p) \arrow[r, "\eta_{p}"] \arrow[d, " T_{\mc{C}}(\gamma{\rm{'}}) \circ T_{\omega}^{t \circ \zeta_i}(\delta)"'] & T_{\omega}^{s \circ \zeta_i} \circ T_{\mc{C}}(\gamma)(p) \arrow[d, "T_{\omega}^{s \circ \zeta_i} \circ T_{\mc{C}}(\gamma)(\delta)"] \\
			T_{\mc{C}}(\gamma{\rm{'}}) \circ T_{\omega}^{t \circ \zeta_i}(q) \arrow[r, "\eta_{q}"']                & T_{\omega}^{s \circ \zeta_i} \circ T_{\mc{C}}(\gamma)(q)              
		\end{tikzcd}\,\, {\rm{commutes,}}$
	\begin{equation}\label{121}
		\eta_q \circ \big( T_{\mc{C}}(\gamma') \circ T_{\omega}^{t \circ \zeta_i}(\delta) \big)= \big(T_{\omega}^{s \circ \zeta_i} \circ T_{\mc{C}}(\gamma)(\delta) \big) \circ \eta_p.
	\end{equation}
	As $\delta=1_p(h,e)$ for a unique $h \in H$, we have
	\begin{equation}\label{111}
		q= p \tau(h),
	\end{equation}
	\begin{equation}\label{211}
		T_{\mc{C}}(\gamma') \circ T_{\omega}^{t \circ \zeta_i}(\delta) = 1_{\mu_{C}\big(\gamma'^{-1}, {\rm{Tr}}_{\omega_0}^{t \circ \zeta_i}(p)\big)}(h, e),
	\end{equation} and
	\begin{equation}\label{311}
		{\rm{Tr}}_{\omega_0}^{s \circ \zeta_i} \circ T_{\mc{C}}(\gamma)(\delta)= 1_{{\rm{Tr}}_{\omega_0}^{s \circ \zeta_i}\big(\mu_{\mc{C}}(\gamma^{-1},p)\big)}(h,e) .
	\end{equation} 
	By comparing the left-hand 
	\begin{equation}\nonumber
		\begin{split}
			&\eta_q \circ \big( T_{\mc{C}}(\gamma') \circ T_{\omega}^{t \circ \zeta_i}(\delta) \big)\\
			&= \underbrace{1_{\mu_{\mc{C}}\big(\gamma'^{-1}, {\rm{Tr}}_{\omega_0}^{t \circ \zeta_i}(p)\big)} (h_p,e) \big(e, \tau(h)\big)}_{\quad \eta_q= \eta_p 1_{\tau(h)} [\textit{\Cref{111}}]} \circ \underbrace{1_{\mu_{C}\big(\gamma'^{-1}, {\rm{Tr}}_{\omega_0}^{t \circ \zeta_i}(p)\big)}(h, e)}_{{\textit{by \Cref{211}}.}} \\
			&= 1_{\mu_{\mc{C}}\big(\gamma'^{-1}, {\rm{Tr}}_{\omega_0}^{t \circ \zeta_i}(p)\big)}(h_ph,e) \quad [\textit{by functoriality of the action}]
		\end{split}
	\end{equation}
	and the right-hand sides
	\begin{equation}\nonumber
		\begin{split}
			& {\rm{Tr}}_{\omega_0}^{s \circ \zeta_i} \circ T_{\mc{C}}(\gamma)(\delta) \circ \eta_p\\
			&= \underbrace{1_{{\rm{Tr}}_{\omega_0}^{s \circ \zeta_i}\big(\mu_{\mc{C}}(\gamma^{-1},p)\big)}(h,e)}_{ {\textit{by \Cref{311}}}} \circ \underbrace{1_{\mu_{\mc{C}}\big(\gamma'^{-1}, {\rm{Tr}}_{\omega_0}^{t \circ \zeta_i}(p)\big)} (h_p,e)}_{\textit{by the definition in \Cref{equivalence invariance}}} \\
			&= 1_{t\big(\eta(p)\big)}(h,e) \circ 1_{\mu_{\mc{C}}\big(\gamma'^{-1}, {\rm{Tr}}_{\omega_0}^{t \circ \zeta_i}(p)\big)} (h_p,e)  \quad [\textit{As} \, \, t(\eta_p)= {\rm{Tr}}_{\omega_0}^{s \circ \zeta_i}\big(\mu_{\mc{C}}(\gamma^{-1},p)\big), \textit{by \Cref{Target}} ]\\
			&=  \underbrace{1_{ \mu_{\mc{C}}\big(\gamma'^{-1}, {\rm{Tr}}_{\omega_0}^{t \circ \zeta_i}(p)\big) \tau(h_p)}}_{{\textit{by \Cref{equivalence invariance}}}}(h,e) \circ 1_{\mu_{\mc{C}}\big(\gamma'^{-1}, {\rm{Tr}}_{\omega_0}^{t \circ \zeta_i}(p)\big)} (h_p,e) \\
			&= \underbrace{1_{ \mu_{\mc{C}}\big(\gamma'^{-1}, {\rm{Tr}}_{\omega_0}^{t \circ \zeta_i}(p)\big)}\big(e,\tau(h_p)\big)}_{\textit{by \Cref{E:Identitiesechangeinverse}}}(h,e) \circ 1_{\mu_{\mc{C}}\big(\gamma'^{-1}, {\rm{Tr}}_{\omega_0}^{t \circ \zeta_i}(p)\big)}(h_p,e) \\
			&= 1_{ \mu_{\mc{C}}\big(\gamma'^{-1}, {\rm{Tr}}_{\omega_0}^{t \circ \zeta_i}(p)\big)}\underbrace{\big(h_phh_p^{-1}, \tau(h_p)\big)}_{{\textit{by \Cref{E:Peiffer}.}}} \circ 1_{\mu_{\mc{C}}\big(\gamma'^{-1}, {\rm{Tr}}_{\omega_0}^{t \circ \zeta_i}(p)\big)} (h_p,e)\\
			&= 1_{\mu_{\mc{C}}\big(\gamma'^{-1}, {\rm{Tr}}_{\omega_0}^{t \circ \zeta_i}(p)\big)}(h_ph,e) [\textit{by \Cref{E:Identitiesechangeinverse}}]
		\end{split}
	\end{equation}
	of the \Cref{121}, we conclude the naturality of $\eta$.
\end{proof}

\subsection*{Step-2 (Invariance under the thin deformation defined in \Cref{Definition: Thin deformation})}
The following porposition establishes the invariance of parallel transport on a quasi-principal 2-bundle under the thin deformation of lazy Haefliger paths (\Cref{Definition: Thin deformation}).
\begin{proposition}\label{Proposition:  X-path thin homotopy invariance of parallel transport}
	For a Lie crossed module $(G, H, \tau, \alpha)$, let $(\pi: \mb{E} \ra \mb{X}, \mc{C})$ be a quasi-principal $\mb{G}:=[H \rtimes_{\alpha} G \rra G]$-bundle with a strict connection $\omega: T\mb{E} \ra L(\mb{G})$. If there is a thin deformation from a lazy $\mb{X}$-path $\Gamma  :=(\gamma_0, \alpha_1,\gamma_1, \cdots,\alpha_n, \gamma_n)$ to another lazy $\mb{X}$-path $\Gamma'$, then there is a smooth $\mb{G}$-equivariant natural isomorphism between $T_{(\Gamma, \mc{C}, \omega)}$ and $T_{(\Gamma', \mc{C}, \omega)}$.
\end{proposition}
\begin{proof}
	Suppose $ \lbrace \zeta_i: I \ra X_1 \rbrace_{i =0,1,...,n}$ is a thin deformation from the lazy $\mb{X}$-path $\Gamma :=(\gamma_0, \alpha_1,\gamma_1, \cdots ,\alpha_n, \gamma_n)$ to the lazy $\mb{X}$-path $\Gamma' :=(\gamma_0', \alpha_1',\gamma_1', \cdots,\alpha_n', \gamma_n')$. Let $s(\Gamma)=s(\Gamma')=x$ and $t(\Gamma)=t(\Gamma')=y$. Let us illustrate the thin deformation $ \lbrace \zeta_i: I \ra X_1 \rbrace_{i =0,1,...,n}$ by the following diagram:
	\begin{equation}\label{Main thin deformation diagram}
		\begin{tikzcd}
			& {\cdot} \arrow[r, "\alpha_1", dotted] \arrow[dd, "t \circ \zeta_0"', dotted] & {\cdot} & {\cdot} \arrow[dd, "s \circ \zeta_{i-1}"', dotted] \arrow[r, "\gamma_{i-1}"] & {\cdot} \arrow[r, "\alpha_i", dotted] \arrow[dd, "t \circ \zeta_{i-1}"', dotted] & {\cdot} \arrow[dd, "s \circ \zeta_i"', dotted] \arrow[r, "\gamma_i"] & {\cdot} \arrow[dd, "t \circ \zeta_i", dotted] & {\cdot} \arrow[rd, "\gamma_n"] \arrow[dd, "s \circ \zeta_n", dotted] &   \\
			x \arrow[ru, "\gamma_0"'] \arrow[rd, "\gamma_0'"'] &                                                    &    &                                            &                                                    &                                            &                            &                                            & y \\
			& {\cdot} \arrow[r, "\alpha_1'"', dotted]                         & {\cdot} & {\cdot} \arrow[r, "\gamma_{i-1}'"']                         & {\cdot} \arrow[r, "\alpha_i'"', dotted]                         & {\cdot} \arrow[r, "\gamma_i'"']                         & {\cdot}                         & {\cdot} \arrow[ru, "\gamma_n'"']                        &  
		\end{tikzcd},
	\end{equation}
	
	where the solid arrows are elements of $X_1$, and the dotted arrows are paths in $X_0$.  
	
	Suppose $H_i: I \times I \ra X_0$  are thin homotopies from $\alpha_i$ to $ (s \circ \zeta_i)^{-1}* \alpha'_i *(t \circ \zeta_{i-1})$  for all $i=1,...,n$. Then, $$u \circ H_i : I \times I \ra X_1$$ is a thin homotopy from $u \circ \alpha_i$ to $u \circ \big( (s \circ \zeta_i)^{-1}* \alpha'_i *(t \circ \zeta_{i-1}) \big)$ in $X_1$ for each $i$, as the rank of $u \circ H_i $ is less than rank of $H_i$ at all points. From the thin homotopy invariance of  the parallel transport in classical principal bundles (\Cref{thin homotopy invariance of classical transport}), we get the following family of equations:
	\begin{equation}\label{Magic1}
		T_{\omega}^{\alpha'_i}=T_{\omega}^{( s \circ \zeta_i)} \circ T_{\omega} ^{\alpha_i} \circ T_{\omega}^{t \circ \zeta_{i-1}^{-1}}
	\end{equation}
	obtained from a family of diagrams of the form
	\begin{equation}\label{deformation diagram} 
		\begin{tikzcd}
			\cdot \arrow[r, "\alpha_i", dotted] \arrow[d, "t \circ \zeta_{i-1}"', dotted] & \cdot                         \\
			\cdot \arrow[r, "\alpha_{i}^{{\rm{'}}}"', dotted]                        & \cdot \arrow[u, "(s \circ \zeta_i)^{-1}"', dotted]
		\end{tikzcd}
	\end{equation}
	for $i=1,2....,n$.
	
	Also, observe that the family of diagrams of the form
	\begin{equation}\label{equivalence diagram}
		\begin{tikzcd}
			\cdot \arrow[r, "\gamma_i"] \arrow[d, "s \circ \zeta_i"', dotted] & \cdot                         \\
			\cdot \arrow[r, "\gamma_i^{{\rm{'}}}"']                        & \cdot \arrow[u, "(t \circ \zeta_i)^{-1}"', dotted]
		\end{tikzcd}
	\end{equation}
	induce the following family of smooth $\mb{G}$-equivariant natural isomorphisms 
	\begin{equation}\label{Magic2} 
		T_{\mc{C}}(\gamma_i'^{-1}) \cong T_{\omega}^{t \circ \zeta_{i}} \circ T_{\mc{C}}(\gamma_i^{-1}) \circ T_{\omega}^{(s \circ \zeta_i)^{-1}}  
	\end{equation} 
	for $i=1,..,n-1$, by the same argument  as we have used to prove \Cref{Magic}. 
	
	 As a consequence of \Cref{Magic1} and \Cref{Magic2}, we conclude $T_{(\Gamma, \mc{C}, \omega)} \cong T_{(\Gamma', \mc{C}, \omega)}$.						
\end{proof}
Combining \Cref{Proposition: Equivalence invariance of parallel transport} and \Cref{Proposition:  X-path thin homotopy invariance of parallel transport}, we arrive at our intended lazy $\mb{X}$-path thin homotopy invariance, that we formally state below:
\begin{theorem}\label{lazy Xpath thin homotopy invariance}
	For a Lie crossed module $(G, H, \tau, \alpha)$, let $(\pi: \mb{E} \ra \mb{X}, \mc{C})$ be a quasi-principal $\mb{G}:=[H \rtimes_{\alpha} G \rra G]$-bundle with a strict connection $\omega: T\mb{E} \ra L(\mb{G})$. If a lazy $\mb{X}$-path $\Gamma  :=(\gamma_0, \alpha_1,\gamma_1, \cdots ,\alpha_n, \gamma_n)$ is lazy $\mb{X}$-path thin homotopic to a lazy $\mb{X}$-path $\Gamma'$, then there is a smooth $\mb{G}$-equivariant natural isomorphism between $T_{(\Gamma, \mc{C}, \omega)}$ and $T_{(\Gamma', \mc{C}, \omega)}$.
\end{theorem}

\section{Parallel transport functor of a quasi-principal 2-bundle}\label{Parallel transport functor of a quasi-principal 2-bundle}
With the aid of \Cref{lazy Xpath thin homotopy invariance}, in this section, we construct the parallel transport functor on a quasi-principal 2-bundle over a Lie groupoid (\Cref{Definition:Quasicategorical Connection}). We then obtain our main result of this chapter by establishing its naturality with respect to the connection preserving morphisms and thereby extending the parallel transport functor to a functor between the parallel transport functor category and the groupoid of quasi-principal 2-bundles equipped with connections. We also validate the sanity of our construction by showing that our parallel transport functor enjoys a crucial smoothness property. Furthermore, as a side result, we also establish its naturality with respect to the strong fibered product constructions (\Cref{strong fibered products}).

In order to define the parallel transport functor on a quasi principal 2-bundle, it is necessary to introduce a quotient category $\overline{\mb{G} {\rm{-Tor}}}$ of $\mb{G}$-Tor (\Cref{Lie 2-group torsor}).

\begin{definition}\label{Quotiented G-tor}
	Given a Lie 2-group $\mb{G}$, the category	$\overline{\mb{G} {\rm{-Tor}}}$ is defined as the quotient category of $\mb{G}$-Tor obtained from the congruence relation given as follows: For each pair of $\mb{G}$-torsors $\mb{X}, \mb{Y}$, the equivalence relation on  $\rm{Hom}_{\mb{G} {\rm{-}} {\rm{Tor}}}(\mb{X}, \mb{Y})$ is given by the existence of a smooth $\mb{G}$-equivariant natural isomorphism.
\end{definition}

\begin{theorem}\label{Theorem: Parallel transport on 2-bundles}
	Given a quasi-principal $\mb{G}:=[H \rtimes_{\alpha} G \rra G]$-bundle $(\pi: \mb{E} \ra \mb{X}, \mc{C})$  with a strict connection $\omega: T\mb{E} \ra L(\mb{G})$, there is a functor
	\begin{equation}\nonumber
		\begin{split}
			\mc{T}_{\mc{C}, \omega} \colon & \Pi_{\rm{thin}}(\mb{X}) \ra \overline{\mb{G} \rm{-Tor}}\\
			& x \mapsto \pi^{-1}(x),\\
			& [\Gamma] \mapsto [T_{(\Gamma, \mc{C}, \omega)}].
		\end{split}
	\end{equation}
\end{theorem}
\begin{proof}
	Well-definedness of $\mc{T}_{\mc{C}, \omega}$ is a direct consequence of \Cref{lazy Xpath thin homotopy invariance}. Source-target compatibilities of $\mc{T}_{\mc{C}, \omega} $ are obvious. Consistency with the unit map and the composition follow from  \Cref{Pseudonat1} and  \Cref{pseudonat2}, respectively.
\end{proof}
\begin{definition}\label{Parallel transport functor of quasi-principal 2-bundle}
	For a  Lie crossed module $(G, H, \tau, \alpha)$, let $(\pi: \mb{E} \ra \mb{X}, \mc{C})$ be a quasi-principal $\mb{G}:=[H \rtimes_{\alpha} G \rra G]$-bundle equipped with a strict connection $\omega: T\mb{E} \ra L(\mb{G})$. Then the functor $\mc{T}_{\mc{C, \omega}}$ is defined as the $(\mc{C}, \omega)$-\textit{parallel transport functor of the quasi-principal $\mb{G}$-bundle $(\pi \colon \mb{E} \ra \mb{X}, \mc{C})$}.
\end{definition}

\begin{remark}\label{Classical transport}
	Given a principal $[G \rra G]$-bundle $\pi \colon [E \rra E] \ra [M \rra M]$ over a discrete Lie groupoid $[M \rra M]$, endowed with the strict connection of the form $\omega:= (\omega,\omega)$ (\Cref{Classical connection as 2-connection}) and the unique categorical connection $\mc{C}$ (\Cref{unique Cat connection on bundle over discrete space}), the functor $\mc{T}_{\mc{C}, \omega}$ coincides with the classical one (\Cref{Transport functor}).
\end{remark}

\subsection{Naturality with respect to connection preserving morphisms}\label{Naturality with respect to connection preserving morphisms}
The following proposition will establish the naturality of \Cref{Parallel transport functor of quasi-principal 2-bundle} with respect to the connection preserving morphisms of quasi-principal 2-bundles.

\begin{proposition}\label{Naturality of parallel transport}
	Let $\mb{G}$ be a Lie 2-group. Suppose $\omega$ is a strict connection on $(\pi \colon \mb{E} \ra \mb{X}, \mc{C})$. 	Then, for any morphism of quasi-principal $\mb{G}$-bundles $$F \colon (\pi \colon \mb{E}' \ra \mb{X}, \mc{C}' ) \ra (\pi \colon \mb{E} \ra \mb{X}, \mc{C})$$ over a Lie groupoid $\mb{X}$, the functors $\mc{\tau}_{\mc{C}, \omega}$ and $\mc{\tau}_{\mc{C}', F^{*}\omega}$ (see \Cref{Lemma:Pullback connection}) are naturally isomorphic.
\end{proposition}

\begin{proof}
	The proof follows from the observation that for every $x \xrightarrow {\gamma} y \in X_1$ and for every path $\alpha$ in $X_0$ (with sitting instants) from $p$ to $q$ respectively, the following two diagrams commute in the category of $\mb{G}$-torsors:
	
	\begin{tikzcd}\label{47}
		\pi'^{-1}(y) \arrow[r, "T_{\mc{C}'}(\gamma)"] \arrow[d, "F|_{\pi'^{-1}(y)}"'] & \pi'^{-1}(x) \arrow[d, "F|_{\pi'^{-1}(x)}"] \\
		\pi^{-1}(y) \arrow[r, "T_{\mc{C}}(\gamma)"']                & \pi^{-1}(x)               
	\end{tikzcd}
	\begin{tikzcd}\label{48}
		\pi'^{-1}(p) \arrow[r, "T_{F^{*}\omega}^{\alpha}"] \arrow[d, "F|_{\pi'^{-1}(p)}"'] & \pi'^{-1}(q) \arrow[d, "F|_{\pi'^{-1}(q)}"] \\
		\pi^{-1}(p) \arrow[r, "T_{\omega}^{\alpha}"']                & \pi^{-1}(q).               
	\end{tikzcd}
	
	Commutativity of the right square and the left square follow respectively from \Cref{Classical Parallel transport and connection preserving morphism} and \Cref{Groupoid of quasi principal 2-bundles}.
\end{proof}
The above proposition is a crucial step in obtaining the main result of this chapter, as we see next.

Given a Lie 2-group $\mb{G}$ and a Lie groupoid $\mb{X}$, let $\rm{Bun}_{\rm{quasi}}^{\nabla}(\mb{X}, \mb{G})$ be the category whose objects are quasi-principal $\mb{G}$-bundles equipped with strict connections over the Lie groupoid $\mb{X}$, and arrows are connection preserving morphisms (\Cref{Lemma:Pullback connection}). Suppose $\rm{Trans}(\mb{X},\mb{G})$ is the category whose objects are functors $T \colon \Pi_{\rm{thin}}(\mb{X}) \ra \overline{\mb{G} {\rm{-Tor}}}$ and arrows are natural transformations. Then, the following is an immediate consequence of  \Cref{Naturality of parallel transport}.
\begin{theorem}\label{Equivalence of quasi and functors}
	The map $\Big((\pi \colon \mb{E} \ra \mb{X}, \mc{C} ), \omega \Big) \mapsto \mc{T}_{\mc{C}, \omega}$ defines a functor
	\begin{equation}\nonumber
		\mc{F} \colon  \rm{Bun}_{\rm{quasi}}^{\nabla}(\mb{X}, \mb{G}) \ra \rm{Trans}(\mb{X},\mb{G}),
	\end{equation}
	where $\omega$ is the strict connection on $\pi \colon \mb{E} \ra \mb{X}$.
	
\end{theorem}

\subsection{Naturality with respect to fibered products}\label{Naturality with respect to fibered products}
Let $\mb{G}$ be a Lie 2-group. By \Cref{Proposition: Strict pullback of Lie groupoids},	for any principal $\mb{G}$-bundle  $\pi: \mb{E} \ra \mb{X}$ and a morphism of Lie groupoids $F : \mb{Y} \ra \mb{X}$, the morphism of Lie groupoids ${\rm{pr}}_1: \mb{Y} \times_{F,\mb{X},\pi} \mb{E} \ra \mb{Y}$ is a principal $\mb{G}$-bundle over $\mb{Y}$, where $\mb{Y} \times_{F,\mb{X},\pi} \mb{E}$ is the strong fibered product of $\mb{Y}$ and $\mb{E}$ with respect to $F$ and $\pi$ (\Cref{strong fibered products}). We will denote this $\mb{G}$ bundle by $F^{*}\pi: F^{*}\mb{E} \ra \mb{Y}$.
\begin{equation}\label{pullback principal 2-bundle diagram}
	\begin{tikzcd}
		F^{*}\mb{E} \arrow[d, "F^{*}\pi"'] \arrow[r, "\rm{pr}_2"]                 & \mb{\mb{E}} \arrow[d, "\pi"] \\
		\mb{Y} \arrow[r, "F"']  & \mb{X}               
	\end{tikzcd}.
\end{equation}
\begin{remark}\label{Remark: morphism between string fibered products}
	The above strong fibered product construction naturally extends to a functor
	\begin{equation}\nonumber
		\begin{split}
			\mc{F} \colon & {\rm{Bun}}(\mb{X}, \mb{G}) \ra {\rm{Bun}}(\mb{Y}, \mb{G})\\
			& \big( \pi : \mb{E} \ra \mb{X} \big) \mapsto \big( F^{*}\pi: F^{*}\mb{E} \ra \mb{Y}\big)\\
			& \big( \phi: \mb{E} \ra \mb{E}' \big) \mapsto \big( \phi^{*} : F^{*}\mb{E} \ra \mb{F}^{*}\mb{E}' \big),
		\end{split}
	\end{equation}
	where $\phi_0^{*}(x,p) := (x, \phi_0(p))$  and $\phi_1^{*}(\gamma, \delta) := (\gamma, \phi_1(\delta))$. 
\end{remark}

The following result is easy to prove.
\begin{lemma}\label{naturality lemma}
	If $(\pi \colon \mb{E} \ra \mb{X}, \mc{C})$ is a quasi-principal $\mb{G}$-bundle equipped with a strict connection $\omega$ and $F \colon \mb{Y} \ra \mb{X}$ is any morphism of Lie groupoids, then $(F^{*} \pi \colon \mb{F^{*}\mb{E}} \ra \mb{Y}, F^{*}\mc{C})$ is a quasi-principal $\mb{G}$-bundle with strict connection $\rm{pr}_2^{*}\omega$, where $F^{*}\mc{C} \colon  s^{*}(F_0^{*}E_0) \ra F_1^{*}E_1$ is defined by $\big(\gamma,(x,p)\big) \ra \bigg(\gamma, C \Big( F_1(\gamma),p\Big) \bigg)$ for $\gamma \in X_1$ and $p \in E_0$ satisfying $F_0(s(\gamma))= \pi_0(p)$.
\end{lemma}
\begin{remark}
	Observe that \Cref{naturality lemma} implies that the functor $\mc{F}$ in \Cref{Remark: morphism between string fibered products} restricts to a functor from the subcategory ${\rm{Bun}}_{{\rm{quasi}}}(\mb{X}, \mb{G}) \leq {\rm{Bun}}(\mb{X}, \mb{G})$ to the subcategory  ${\rm{Bun}}_{{\rm{quasi}}}(\mb{Y}, \mb{G}) \leq {\rm{Bun}}(\mb{Y}, \mb{G})$.
\end{remark}
The following result establishes the naturality of  \Cref{Parallel transport functor of quasi-principal 2-bundle} with respect to the strong fibered product of Lie groupoids.
\begin{proposition}\label{Pullback naturality}
	For a Lie 2-group $\mb{G}$, given a quasi-principal $\mb{G}$-bundle $(\pi \colon \mb{E} \ra \mb{X}, \mc{C})$ with a strict connection $\omega$ and a morphism of Lie groupoids $F \colon \mb{Y} \ra \mb{X}$, the functors $\mc{T}_{F^{*}\mc{C}, \rm{pr_2}^{*}\omega}$ and  $\mc{T}_{\mc{C}, \omega} \circ F_{\rm{thin}}$ (see \Cref{Morphism of thin homotopy groupoid}) are naturally isomorphic.
\end{proposition}
\begin{proof}
	We claim that $\eta \colon Y_0 \ra (\overline{\mb{G} {\rm{-Tor}}})_1$ defined as $y \mapsto \eta_y := [{\rm{pr}}_2|_{(F^{*}\pi)^{-1}(y)}]$, is the required natural isomorphism, where ${\rm{pr}}_2 \colon F^{*}\mb{E} \ra \mb{E}$ is the 2nd projection functor from the strong fibererd product. Our claim follows from the following pair of straightforward observations:
	\begin{itemize}
		\item[(i)] For every $x \xrightarrow {\gamma} y \in Y_1$, we have
		\begin{equation}\nonumber
			[T_{\mc{C}}\big(F(\gamma)\big)] \circ \eta_{y}= \eta_x \circ [T_{F^{*}\mc{C}}(\gamma)],
		\end{equation}
		\item[(ii)]for every path (with sitting instants)  $\alpha \colon [0,1] \ra Y_0$ such that $\alpha(0)=a$ and $\alpha(1)=b$, we have 
		\begin{equation}\nonumber
			[T_{\omega}^{F(\alpha)^{-1}}] \circ \eta_{b}=\eta_a \circ [T_{{\rm{pr}}_2^{*}\omega}^{\alpha^{-1}}].
		\end{equation}
	\end{itemize}
\end{proof}
\subsection{Smoothness of the parallel transport functor of a quasi-principal 2-bundle}\label{Smoothness of parallel transport}
Let $\overline{\rm{Aut}(\mb{E})}$ denote the automorphism group of the $\mb{G}$-torsor $\mb{E}$ in the groupoid $\overline{\mb{G} {\rm{-Tor}}}$ (\Cref{Quotiented G-tor}). Note that the quotient functor $\mb{G}-{\rm{Tor}} \ra \overline{\mb{G}-{\rm{Tor}}}$ descends to a quotient map $q \colon {\rm{Aut}}(\mb{E}) \ra \overline{\rm{Aut}(\mb{E})}$. Before stating a smoothness condition for the parallel transport functor, we prove that  $\overline{\rm{Aut}(\mb{E})}$  is a diffeological group.
\subsection*{Smoothness of $\overline{\rm{Aut}(\mb{E})}$}
We start with the following observation:

\begin{lemma}\label{diffeology on AutE}
	For any $\mb{G}:=[H \rtimes_{\alpha}G \rra G]$-torsor $\mb{E}$, the group of automorphisms ${\rm{Aut}}(\mb{E}):= {\rm{Hom}}_{\mb{G}-{\rm{Tor}}} ( \mb{E}, \mb{E})$ is  canonicially isomorphic to the Lie group $G$.
\end{lemma}
\begin{proof}
	For any  Lie group $G$, a $G$-torsor $E$ and a point $z \in E$, we have a group isomorphism defined as 
	\begin{equation}\label{canonical Lie group structure on Aut}
		\begin{split}
			\psi_z \colon & {\rm{Aut}}(E):= {\rm{Hom}}_{G-{\rm{Tor}}}( E, E) \ra G\\
			& f \mapsto \delta \big(z,f(z)\big),
		\end{split}
	\end{equation}
	where $\delta \colon E \times E \ra G$ is a smooth map defined implicitly as $x \cdot \delta(x,y)=y$ (see \Cref{Automorphism group of the fibre}). This isomorphism is independent of the choice of $z$, and so, ${\rm{Aut}}(E)$ can be canonically identified as a Lie group (see \Cref{Automorphism group of the fibre}). Hence, it is sufficient to show that the following map 
	\begin{equation}\label{AutMbE}
		\begin{split}
			\theta \colon& {\rm{Aut}}(\mb{E}) \ra {\rm{Aut}}(E_0)\\
			& F:=(F_1,F_0) \mapsto F_0
		\end{split}
	\end{equation}
	is an isomorphism of groups. 
	
	It is obvious that $\theta$ is a group homomorphism.  To show $\theta$ is injective, let $\theta(F)=\theta(F')$ for $F, F' \in {\rm{Aut}}(\mb{E})$. Suppose $\delta \in E_1$. Then there exists unique $h_{\delta} \in H$, such that $\delta=1_{s(\delta)}(h_{\delta},e)$. Thus, $F_1(\delta)=F_1\big(1_{s(\delta)}(h_{\delta},e)\big)=1_{F'_0\big(s(\delta)\big)}(h_{\delta},e)=F'_1(\delta)$. 
	
	Now, suppose $f \in {\rm{Aut}}(E_0)$. For $\delta \in E_1$, define $F_1(\delta) :=1_{f\big(s(\delta)\big)}(h_{\delta},e)$. Note that as for any $(h,g) \in H \rtimes_{\alpha}G$ the following identity holds
	\begin{equation}\nonumber
		(h_{\delta (h,g)},e)=\big(\alpha_{g^{-1}}(h_{\delta}h),e \big),
	\end{equation}
	it follows $F_1$ is a  morphism of $H \rtimes_{\alpha}G$-torsor.
	Hence, to show $\theta$ is onto, it is enough to prove $(F_1,f)$ is a functor.  Note that the compatibility with the source, target, and unit maps are obvious while the consistency with the composition map follows from the observation that for any composable $\delta_2,\delta_1 \in E_1$, we have 
	\begin{equation}\nonumber
		h_{\delta_2 \circ \delta_1} = h_{\delta_1} h_{\delta_2}.
	\end{equation}

	\begin{proposition}\label{Autbar E}
		For any $\mb{G}:=[H \rtimes_{\alpha}G \rra G]$-torsor $\mb{E}$, the group $\overline{\rm{Aut}}(\mb{E})$ is isomorphic to the quotient group $G/\tau(H)$.  
	\end{proposition}
	\begin{proof}

		Consider the quotient map $q \colon {\rm{Aut}}(\mb{E}) \ra \overline{\rm{Aut}(\mb{E})}$. Observe that to show $\overline{{\rm{Aut}}(\mb{E})} \cong G/ \tau(H)$,  by the first isomorphism theorem it is enough to show
		\begin{equation}\nonumber
			\psi_z \circ \theta \big(\ker(q) \big)= \tau(H),
		\end{equation}
		for some $z \in E_0$, where $\psi_z$ and $\theta$ are maps as defined in \Cref{diffeology on AutE}.
		The inclusion $\psi_z \circ \theta (\ker(q)) \subseteq \tau(H)$ follows, since for any $F \in \ker(q)$, there is a smooth $\mb{G}$-equivariant natural isomorphism $\eta \colon \rm{Id}_{\mb{E}} \Longrightarrow F$ and thus we get the unique element $h_z \in H$ satisfying $\eta(z)=1_z(h_z,e)$, for which $\psi_z \circ \theta(F)= \tau(h_z)$. On the other hand, for any $h \in H$, one can define $f \colon E_0 \ra E_0$ as $z.g \mapsto z\tau(h)g$ for each $g \in G$, and hence we get an element $(F_1, f)  \in \rm{Aut}(\mb{E})$ (as in \Cref{AutMbE}). Then it is easy to see that $(F_1, f) \in \ker(q)$ as the prescription $z.g \mapsto 1_z(h,e)(e,g)$ for each $g \in G$ defines a smooth $\mb{G}$-equivariant natural isomorphim $\eta \colon {\rm{id}}_{\mb{E}} \Longrightarrow (F_1, f)$.
	\end{proof}
	
\end{proof}
\begin{corollary}\label{Aut Diff}
	For any $\mb{G}:=[H \rtimes_{\alpha}G \rra G]$-torsor $\mb{E}$, $\overline{\rm{Aut}}(\mb{E})$ is a diffeologial group.
\end{corollary}
\begin{proof}
	By \Cref{Autbar E}, $\overline{\rm{Aut}}(\mb{E})$ is isomorphic to $G/ \tau(H)$. Now, since $G$ is a Lie group, $\overline{\rm{Aut}(\mb{E})}$ is a diffeological group equipped with the quotient diffeology, see \Cref{quotient group diffeology}.
\end{proof}

Now, we are ready to show that the parallel transport functor of a quasi-principal 2-bundle (\Cref{Parallel transport functor of quasi-principal 2-bundle}) is smooth in an `appropriate sense', which will be made precise in the following theorem.
\subsection*{Smoothness of the parallel transport functor}

\begin{theorem}\label{Smoothness of parallel transport functor}
	For a Lie crossed module $(G, H, \tau, \alpha)$, let $(\pi: \mb{E} \ra \mb{X}, \mc{C})$ be a quasi-principal $\mb{G}:=[H \rtimes_{\alpha} G \rra G]$-bundle with a strict connection $\omega: T\mb{E} \ra L(\mb{G})$. Then for each $x \in X_0$, the restriction map $\mc{T}_{{\mc{C}, \omega}}|_{{\Pi_{\rm{thin}}(\mb{X},x)}} \colon \Pi_{\rm{thin}}(\mb{X},x) \ra \overline{\rm{Aut}(\pi^{-1}(x))}$ is a map of diffeological spaces, where $\Pi_{\rm{thin}}(\mb{X},x)$ is the automorphiosm group of $x$ in the diffeological groupoid $\Pi_{\rm{thin}}(\mb{X})$.
\end{theorem}
\begin{proof}
	Suppose $P\mb{X}_x$ denotes the set of lazy $\mb{X}$-paths which start and end at $x \in X_0$. $P\mb{X}_x$ being a subset of $P\mb{X}$, is a diffeological space by (\Cref{subsapce diffeology}). Similarly, $\Pi_{\rm{thin}}(\mb{X},x)$ is also equipped with the subspace diffeology induced from the diffeology on $\frac{P\mb{X}}{\sim}$ (see  \Cref{quotient diffeology}). Let $q^{P\mb{X}_x} \colon P\mb{X}_x  \ra \Pi_{\rm{thin}}(\mb{X},x)$ be the quotient map. Observe that from \Cref{Technical 1}, it is sufficient to show that for any plot $ \big( p \colon U \ra P\mb{X}_x \big) \in D_{P\mb{X}_x}$, $\mc{T}_{{\mc{C}, \omega}}|_{\Pi_{\rm{thin}}(\mb{X},x)} \circ  q^{P\mb{X}_x} \circ p \in  D_{\overline{\rm{Aut}(\pi^{-1}(x))}}$.  Let $x \in U$, then by \Cref{sum diffeology}, there is an open neighbourhood $U_x$ around $x$ such that $p|_{U_x}$ is of the form 
	\begin{equation}\nonumber
		p|_{U_x}=(p^0_{X_1}, p^{1}_{PX_0},p^1_{X_1},\cdots, p^{n}_{PX_0}, p^n_{X_1}) \colon U \ra P\mb{X}_n 
	\end{equation}
	for some $n \in \mb{N} \cup \lbrace 0 \rbrace$. Note that the smoothness of the map 
	\begin{equation}\nonumber
		\begin{split}
			\theta \colon & U_x \ra {\rm{Aut}}\big(\pi^{-1}(x)\big)\\
			& u \mapsto T_{ \big( p|_{U_x}(u), \mc{C}, \omega \big)} \quad [\rm{see} \,\, \Cref{Definition: Parallel transport of X-paths}.]
		\end{split}
	\end{equation}
	will imply $\mc{T}_{{\mc{C}, \omega}}|_{\Pi_{\rm{thin}}(\mb{X},x)} \circ  q^{P\mb{X}_x} \circ p \in  D_{\overline{\rm{Aut}}(\pi^{-1}(x))}$.		
	
	Due to the smooth structure on ${\rm{Aut}}(\pi^{-1}(x))$ (\Cref{diffeology on AutE}), $\theta$ is smooth if and only if the following map 
	\begin{equation}\nonumber
		\begin{split}
			\bar{\theta} \colon & U_x \ra \pi_{0}^{-1}(x)\\
			& u \mapsto  \Big( T_{ \big( p|_{U_x}(u), \mc{C}, \omega \big)} \Big)_{0}(z)
		\end{split}
	\end{equation}
	is smooth for some choice of $z \in \pi^{-1}(x)$. But, the smoothness of $\bar{\theta}$ follows easily from the following sequence of facts:
	\begin{equation}\nonumber
		\begin{split}
			& 	U_x \ra \pi_{0}^{-1} \Big( t\big(p^0_{X_1}(u)\big) \Big), \, \, u \mapsto t \Big( \mc{C}\big(p^0_{X_1}(u),z\big) \Big) \,\, {\rm{is}} \, \, {\rm{smooth}},\,\, {\rm{and}}\\
			& U_x \ra \pi_0^{-1} \Big(ev_0\big(p^{1}_{PX_0}\big) \Big), \ \, u \mapsto {\rm{Tr}}_{\omega}^{p^{1}_{PX_0}(u)}\Big(t \big( \mc{C}(p^0_{X_1}(u),z \big) \Big)
		\end{split}
	\end{equation}
	is smooth due to \textbf{Lemma 3.13,} \cite{MR3521476}. Progressing in this fashion for the sequence of maps in $p|_{U_x}=(p^0_{X_1}, p^{1}_{PX_0},p^1_{X_1},\cdots, p^{n}_{PX_0}, p^n_{X_1}) \colon U \ra P\mb{X}_n $, we finish the proof.
\end{proof}
\begin{remark}\label{smoothness remark}
	The smoothness of $\mc{T}_{\mc{C}, \omega}$ in \Cref{Classical transport} obtained from \Cref{Smoothness of parallel transport functor} coincides with that of \Cref{Smoothness of traditional parallel transport} for the parallel transport functor of the classical principal $G$-bundle $\pi \colon E \ra M$ over the manifold $M$. 	Recall in \Cref{Equivalence of quasi and functors}, we defined a functor $\mc{F} \colon {\rm{Bun}}_{\rm{quasi}}^{\nabla}(\mb{X}, \mb{G}) \ra {\rm{Trans}}(\mb{X},\mb{G})$. Currently, it remains inconclusive whether $\mc{F}$ provides a categorified analog of \textbf{Theorem 4.1 }of \cite{MR3521476} or not, when we impose the above smoothness condition on the objects $T \colon \Pi_{\rm{thin}}(\mb{X}) \ra \overline{\mb{G} \rm{-Tor}}$ of ${\rm{Trans}}(\mb{X},\mb{G})$ i.e

 `\textit{for each $x \in X_0$, the restriction map $T |_{{\Pi_{\rm{thin}}(\mb{X},x)}} \colon \Pi_{\rm{thin}}(\mb{X},x) \ra \overline{\rm{Aut}(T(x))}$ is a map of diffeological spaces, where $\Pi_{\rm{thin}}(\mb{X},x)$ is the automorphiosm group of $x$ in the diffeological groupoid $\Pi_{\rm{thin}}(\mb{X})$ and $\overline{\rm{Aut}(T(x))}$ is as defined in the beginning of \Cref{Smoothness of parallel transport}}'. 

\end{remark}

\section{Induced parallel transport on VB-groupoids along lazy Haefliger paths}\label{Associated PAPER VERSION}
As an application of the theory developed in the preceding sections, here we study parallel transports on VB-groupoids along lazy Haefliger paths. 

\subsection*{Construction of a VB-groupoid associated to a principal 2-bundle over a Lie groupoid}
For a Lie 2-group $\mb{G}:=[G_1 \rra G_0]$, let $\pi \colon \mb{E} \ra \mb{X}$ be a principal $\mb{G}$-bundle over a Lie groupoid $\mb{X}$. Suppose there is a left action of $\mb{G}$ on a 2-vector space $\mb{V}:=[V_1 \rra V_0]$ as in \Cref{Action of a Lie 2-group on a vector 2-space}, such that it induces linear representations of Lie groups $G_1$ and $G_0$ on $V_1$ and $V_0$ respectively.  Then, by the usual associated vector bundle construction (as discussed in \Cref{Associate bundle}), we get a pair of vector bundles $\lbrace \pi^{\mb{V}}_i \colon \frac{E_i \times V_i}{G_i} \ra X_i \rbrace_{i=0,1}$, defined by $[p_i,v_i] \mapsto \pi_i(p)$ respectively. From the definition of the right action of $G_i$ on $E_i \times V_i$, $ (p_i,v_i)g \mapsto (p_ig_i, g_i^{-1}v_i)$ for $g_i \in G_i$,  it is obvious that the quotient map $E_i \times V_i \ra \frac{E_i \times V_i}{G_i}$ is a surjective submersion for each $i$. Hence, it follows that the maps $\frac{E_1 \times F_1}{G_1} \mapsto \frac{E_0 \times F_0}{G_0}$, $[\delta, \zeta] \mapsto [s(\delta), s(\zeta)]$ and $[\delta, \zeta] \mapsto [t(\delta), t(\zeta)]$ are surejective submersions. This ensures that the pair of  manifolds $\lbrace \frac{E_i \times V_i}{G_i} \rbrace_{i=0,1}$  defines a Lie groupoid $\frac{\mb{E} \times \mb{V}}{\mb{G}}:= [\frac{E_1 \times V_1}{G_1} \rra \frac{E_0 \times V_0}{G_0}]$ whose structure maps are given as
\begin{itemize}
	\item[(i)] Source: $s \colon \frac{E_1 \times V_1}{G_1} \mapsto \frac{E_0 \times V_0}{G_0}$ given by $[\delta, \zeta] \mapsto [s(\delta), s(\zeta)]$
	\item[(ii)] Target: $t \colon \frac{E_1 \times V_1}{G_1} \mapsto \frac{E_0 \times V_0}{G_0}$ given by $[\delta, \zeta] \mapsto [t(\delta), t(\zeta)]$
	\item[(iii)] Composition: If $s([\delta_2, \zeta_2])= t([\delta_1,\zeta_1])$, then define $[\delta_2,\zeta_2] \circ [\delta_1, \zeta_1]=[\delta_2 \circ \delta_1, \zeta_2 \circ \zeta_1]$
	\item[(iv)] Unit: $u \colon \frac{E_0 \times V_0}{G_0} \mapsto  \frac{E_1 \times V_1}{G_1}$ given by $[p,f] \mapsto [1_p, 1_f]$
	\item[(v)] Inverse: $\mathfrak{i} \colon \frac{E_1 \times V_1}{G_1} \ra \frac{E_1 \times V_1}{G_1}$ given by $[\delta, \zeta] \mapsto [\delta^{-1}, \zeta^{-1}]$.
\end{itemize}
To see that composition makes sense, observe that there exists $(\delta_2', \zeta_2') \in E_1 \times F_1$ and $(\delta_1', \zeta_1') \in E_1 \times F_1$ such that $[\delta_2',\zeta_2']=[\delta_2,\zeta_2]$, $[\delta_1',\zeta_1']=[\delta_1,\zeta_1]$, $s(\delta_2')=t(\delta_1')$ and $s(\zeta_2')=t(\zeta_1')$. Then, one has to use the functoriality of  Lie 2-group action to ensure that the multiplication map is well defined. Now, it is a straightforward but lengthy verification that the pair of vector bundles $ \lbrace \pi^{\mb{V}}_i \colon \frac{E_i \times V_i}{G_i} \ra X_i \rbrace_{i=0,1}$ defines a VB-groupoid $\pi^{\mb{V}} \colon \frac{\mb{E} \times \mb{V}}{\mb{G}} \ra \mb{X}$ over the Lie groupoid $\mb{X}$. We call it an \textit{associated VB-groupoid of $\pi \colon \mb{E} \ra \mb{X}$}.

\begin{remark}\label{general associated groupoid bundle construction}
	The above construction can be considered as a particular case of the associated groupoid bundle construction mentioned in the \textbf{Remark 3.13} of \cite{MR4621032}, where instead of a 2-vector space, the authors considered an ordinary Lie groupoid.
\end{remark}

\begin{example}[Adjoint VB-groupoid]
	 The adjoint VB-groupoid $\rm{Ad}(\mb{E})$ of a principal $\mb{G}$-bundle $\pi \colon \mb{E} \ra \mb{X}$,  as defined in the \Cref{Adjoint VB construction}, can be realized as an associated VB groupoid $\pi^{L(\mb{G})} \colon \frac{\mb{E} \times L(\mb{G})}{\mb{G}} \ra \mb{X}$ of $\pi \colon \mb{E} \ra \mb{X}$, with respect to the usual adjoint action of $\mb{G}$ on $L(\mb{G})$ (\Cref{Adj action}). 
\end{example}

The following observation is straightforward.
\begin{proposition}\label{Proposition: Associated VB-groupoidPaper}
	For a Lie 2-group $\mb{G}$, let $(\pi \colon \mb{E} \ra \mb{X}, \mc{C})$ be a  quasi-principal $\mb{G}$-bundle over a Lie groupoid $\mb{X}$. Suppose there is a left action of $\mb{G}$ on a 2-vector space $\mb{V}$. Then the associated VB-groupoid $\pi^{\mb{V}} \colon \frac{\mb{E} \times \mb{V}}{\mb{G}} \ra \mb{X}$  over $\mb{X}$ admits a linear cleavage, 
	\begin{equation}\nonumber
		\begin{split}
			\mc{C}^{\mb{V}} \colon & \Big( X_1 \times_{s, X_0, \pi^{\mb{V}}_0} \frac{E_0 \times V_0}{G_0}\Big) \ra \frac{E_1 \times V_1}{G_1}\\
			& \quad \hskip 2cm \big(\gamma, [p,v] \big) \mapsto [\mc{C}(\gamma,p), 1_v].
		\end{split}
	\end{equation}
	Furthermore,  if $\mc{C}$ is a unital, then so is $\mc{C}^{\mb{V}}$ and likewise if $\mc{C}$ is a categorical connection then $\mc{C}^{\mb{V}}$ is flat.
\end{proposition}

Combining \Cref{Proposition: Associated VB-groupoidPaper} and \Cref{T_C}, we obtain a  \textit{2Vect-valued pseudofunctor} corresponding to an associated VB-groupoid of a quasi principal 2-bundle, as we see next. Using the notations as in \Cref{T_C}, we have the following:
\begin{proposition}\label{KLiegrpdvalued pseudoPaper}
	For a Lie 2-group $\mb{G}$, let $(\pi \colon \mb{E} \ra \mb{X}, \mc{C})$ be a  quasi-principal $\mb{G}$-bundle over a Lie groupoid $\mb{X}$ with a left action of $\mb{G}$ on a vector 2-space $\mb{V}$. Then there is a  \textit{2Vect-valued pseudofunctor}  $$T_{\mc{C}^{\mb{V}}} \colon \mb{X}^{\rm{op}} \ra {\rm{2}}{\rm{Vect}}$$ defined as
	\begin{itemize}
		\item[(i)] Each $x \in X_0$ is assigned to the vector 2-space $(\pi^{\mb{V}})^{-1}(x)$.
		\item[(ii)] Each morphism $x \xrightarrow{\gamma} y$ is assigned to an isomorphism of 2-vector spaces (i.e., a bijective functor internal to Vect) (see \Cref{Section: 2-vector space}), defined as
		\begin{equation}\label{Assosiate pseudo morphism}
			\begin{split}
				\gamma* \colon & (\pi^{\mb{V}})^{-1}(y)  \rightarrow(\pi^{\mb{V}})^{-1}(x)\\
				& [p,v] \mapsto[T_{\mc{C}}(\gamma)(p),v]\\
				& ([\delta, \zeta] \colon [p,v] \ra [q,v']) \mapsto [T_{\mc{C}}(\gamma)(\delta), \zeta],
			\end{split}
		\end{equation}
		where $T_{\mc{C}}(\gamma)$ is as defined in \Cref{Pseudomor}.
		\item[(iii)]  For each $x \in X_0$, we have a natural isomorphism internal to Vect
		\begin{equation}\nonumber
			\begin{split}
				I^{\mb{V}}_{x} & \colon 1_{x}^{*} \Longrightarrow 1_{\pi^{-1}(x)} \\
				& [p,v] \mapsto [I_x(p), 1_v],
			\end{split}
		\end{equation}
		where $I_{x}(p)$ is as defined in \Cref{Pseudonat1}.
		
		
		\item[(iv)] For each pair of composable arrows 
		\begin{tikzcd}
			x \arrow[r, "\gamma_{1}"] & y \arrow[r, "\gamma_2"] & z
		\end{tikzcd},
		we have a natural isomorphism internal to Vect
		\begin{equation}\nonumber
			\begin{split}
				\alpha^{\mb{F}}_{\gamma_1, \gamma_2} \colon & \gamma_1^{*} \circ \gamma_2^{*} \Longrightarrow (\gamma_2 \circ \gamma_1)^{*}\\
				& [p,f] \mapsto [\alpha_{\gamma_1, \gamma_2}(p), 1_f].
			\end{split}
		\end{equation}
		where $\alpha_{\gamma_1, \gamma_2}$ is as defined in \Cref{pseudonat2}. 
		
	\end{itemize}
	Moreover, $\alpha^{\mb{F}}_{\gamma_1,\gamma_2}$ and $I^{\mb{F}}_x$ satisfy the necessary coherence laws for all composable $\gamma_1,\gamma_2 \in X_1$ and $x \in X_0$  respectively, as menitoned in \Cref{T_C}.
\end{proposition}

\begin{remark}
	In the context of categorical principal bundles over path groupoids and categorical vector spaces, a cursory mention of an analog of the object level map in \Cref{Assosiate pseudo morphism} has been made in the \textbf{Section 11} of \cite{MR3126940}. However, our setup is different.
\end{remark}

The following result follows from the traditional notion of induced parallel transport on associated vector bundles (\Cref{Induced parallel transport on associated fibre bundles}) and \Cref{Proposition: Parallel transport on principal 2-bundles}. 
\begin{proposition}\label{Associated transport along path}
	For a Lie 2-group $\mb{G}$, let $(\pi \colon \mb{E} \ra \mb{X}, \mc{C})$ be a  quasi-principal $\mb{G}$-bundle over a Lie groupoid $\mb{X}$ equipped with a strict connection $\omega \colon T \mb{E} \ra L\mb{G}$. Suppose there is a left action of $\mb{G}$ on a 2-vector space $\mb{V}$. Then, given a path $\alpha \colon [0,1] \ra X_0$ from $x$ to $y$ in $X_0$, there is an isomorphism of 2-vector spaces $T_{\omega, \mb{V}}^{\alpha} \colon  (\pi^{\mb{V}})^{-1}(x) \ra (\pi^{\mb{V}})^{-1}(y)$ defined as 
	\begin{equation}\nonumber
		\begin{split}
			T_{\omega, \mb{V}}^{\alpha} \colon & (\pi^{\mb{V}})^{-1}(x) \ra (\pi^{\mb{V}})^{-1}(y)\\
			& [p,v] \mapsto [{\rm{Tr}}_{\omega_0}^{\alpha}(p), v],\\
			& [\delta, \zeta] \mapsto [{\rm{Tr}}_{\omega_1}^{u \circ \alpha}(\delta), \zeta].
		\end{split}
	\end{equation}
\end{proposition}

Combining \Cref{KLiegrpdvalued pseudoPaper} and \Cref{Associated transport along path}, we arrive at a notion of parallel transport on an associated VB-groupoid of a quasi principal 2-bundle with a strict connection along a lazy Haefliger path, as we see below:
\begin{definition}\label{Comega associated transportpaper}
	Suppose a Lie 2-group $\mb{G}$ acts on a 2-vector space $\mb{V}$, and $(\pi \colon \mb{E} \ra \mb{X}, \mc{C})$ be a  quasi-principal $\mb{G}$-bundle over a Lie groupoid $\mb{X}$, with a strict connection  $\omega \colon T \mb{E} \ra L(\mb{G})$. Then the isomorphism of 2-vector spaces $T^{\mb{V}}_{(\Gamma, \mc{C}, \omega)}:= T_{\mc{C}^{\mb{V}}}(\gamma_n^{-1}) \circ  T_{\omega, \mb{V}}^{\alpha_n} \circ\cdots \circ T_{\omega,\mb{V}}^{\alpha_1} \circ  T_{\mc{C}^{\mb{V}}}(\gamma_0^{-1})$  will be called 
	the \textit{$(\mc{C}, \omega)$-parallel transport on the associated VB-groupoid $\pi^{\mb{V}}$  along the lazy $\mb{X}$-path $\Gamma =(\gamma_0, \alpha_1,\gamma_1,\cdots, \alpha_n, \gamma_n)$}.
\end{definition}
\begin{remark}
	Using  \Cref{Definition: Parallel transport of X-paths}, $T^{\mb{V}}_{(\Gamma, \mc{C}, \omega)}$ in the above definition can be expressed in terms of $T_{\Gamma, \mc{C}, \omega}$ as follows:
	\begin{equation}\nonumber
		\begin{split}
			T^{\mb{V}}_{(\Gamma, \mc{C}, \omega)} \colon & (\pi^{\mb{V}})^{-1}(x) \ra (\pi^{\mb{V}})^{-1}(y)\\
			& [p,v] \mapsto [T_{(\Gamma, \mc{C}, \omega)}(p), v],\\
			& [\delta,\zeta] \mapsto [T_{(\Gamma, \mc{C}, \omega)}(\delta), \zeta].
		\end{split}
	\end{equation}
\end{remark}

Furthermore, the following theorem is an immediate consequence of \Cref{Theorem: Parallel transport on 2-bundles} and \Cref{Comega associated transportpaper}.
\begin{theorem}\label{Parallel transport on associated VB}
	For a Lie 2-group $\mb{G}$, let $(\pi \colon \mb{E} \ra \mb{X}, \mc{C})$ be a  quasi-principal $\mb{G}$-bundle over a Lie groupoid $\mb{X}$ with strict connection $\omega$. Suppose there is a  left action of $\mb{G}$ on a 2-vector space $\mb{V}$. Then there is a functor given by
	\begin{equation}\nonumber
		\begin{split}
			\mc{T}_{\mc{C}, \omega} \colon & \Pi_{\rm{thin}}(\mb{X}) \ra \overline{{\rm{2Vect}}}\\
			& x \mapsto (\pi^{\mb{V}})^{-1}(x),\\
			&[\Gamma] \mapsto [T^{\mb{V}}_{(\Gamma, \mc{C}, \omega)}].
		\end{split}
	\end{equation}
\end{theorem}
where $\overline{{\rm{2Vect}}}$ is a category whose objects are 2-vector spaces and morphisms are functors internal to Vect identified up to a natural isomorphism internal to Vect (see \Cref{Section: 2-vector space}).
\begin{remark}
	Although our discussion was restricted to the notion of parallel transport on an associated VB-groupoid of a quasi-principal 2-bundle equipped with a strict connection, one can generalize the results discussed in this section straightforwardly to obtain a notion of parallel transport on an associated groupoid bundle mentioned in \Cref{general associated groupoid bundle construction}.
\end{remark}


\chapter{Future directions of research}\label{Future} 


\lhead{Chapter 7. \emph{Future directions of research}} 


This chapter explores some research possibilities that may stem from the findings in this thesis.

\section{Characterizations of a pseudo-principal Lie crossed module bundle over a Lie groupoid}\label{Characterization of a pseudo-principal Lie crossed module bundle over a Lie groupoid}
In \Cref{Different characterizations of principal Lie bundles over Lie groupoids}, we discussed several equivalent ways to characterize a principal Lie group bundle over a Lie groupoid. It would be interesting to obtain similar characterizations for a pseudo-principal Lie crossed module-bundle over a Lie groupoid (\Cref{Definition:PseudoprincipalLiecrossedmodulebundle}) and extend the results of \cite{MR2270285} and \cite{MR3521476} suitably.  Particularly, expressing it as an anafunctor could be significant in obtaining a notion of \textit{quasi-principal 2-bundle  over a differentiable stack} and relating our notion of principal 2-bundles to the one discussed in \cite{MR3480061}.

\section{Higher parallel transport theory for a principal Lie 2-group bundle over a Lie groupoid}\label{Higher parallel transport}
In this thesis, we restricted ourselves to parallel transport along a lazy Haefliger path. To check the strength of the ideas developed in (\Cref{Chapter: Parallel transport on quasi-principal 2-bundles}) for a Higher gauge theory framework, one should obtain a notion of a \textit{bigon in a Lie groupoid} that generalizes the concept of a bigon in a smooth manifold (\cite{MR3917427}) and develops a suitable higher parallel transport theory for our principal 2-bundles which relates to the one in \cite{MR3917427} when the base Lie groupoid is discrete.  

\section{Construction of a quasi-principal 2-bundle with a strict connection from a transport functor/ 2-functor}
Recall in \Cref{Equivalence of quasi and functors}, we extended our parallel transport construction (\Cref{Theorem: Parallel transport on 2-bundles}) to  a functor 
\begin{equation}\nonumber
	\mc{F} \colon  \rm{Bun}_{\rm{quasi}}^{\nabla}(\mb{X}, \mb{G}) \ra \rm{Trans}(\mb{X},\mb{G}).
\end{equation}
Now, consider a functor $T \colon \Pi_{\rm{thin}}(\mb{X}) \ra \overline{\mb{G} {\rm{-Tor}}}$, that is an object  in $\rm{Trans}(\mb{X},\mb{G})$ such that it satisfies the smoothness property mentioned in \Cref{smoothness remark}. It would be interesting to see whether one can construct a quasi-principal $\mb{G}$-bundle $(\pi \colon \mb{E} 
\ra \mb{X}, \mc{C})$ with a strict connection $\omega$ from the functor $T$ satisfying the said  smoothness condition. If not, what extra information do we need to construct one? Finally, one may consider obtaining a higher analog of (\textbf{Theorem 4.1 } \cite{MR3521476}) and  (\textbf{Theorem 6.11}, \textbf{Theorem 6.13}, \cite{MR3917427}).
\section{Semi-strict connection induced parallel transport along a Haefliger path}
Although the notion of semi-strict connection (\Cref{strict ans semi strict connetion 1-forms}) seems to be an interesting consequence of our categorified framework, especially through its relation to gauge transformations (\Cref{An extended symmetry of semi-strict connections}), its role in the parallel transport theory is not yet explored. It would be interesting to extend the constructions of (\Cref{Chapter: Parallel transport on quasi-principal 2-bundles}) in the framework of semi-strict connections.
\section{Local gauge theory of a principal 2-bundle over a Lie groupoid } 
We have yet to explore the local aspects of the ideas developed in this thesis. It may be interesting to relate the local theory to non-abelian cocyle gerbes.


\addtocontents{toc}{\vspace{2em}} 

\appendix 




\backmatter
\lhead{\emph{Bibliography}} 

\bibliographystyle{amsplain} 

\bibliography{Bibliography} 

\providecommand{\bysame}{\leavevmode\hbox to3em{\hrulefill}\thinspace}
\providecommand{\MR}{\relax\ifhmode\unskip\space\fi MR }
\providecommand{\MRhref}[2]{%
  \href{http://www.ams.org/mathscinet-getitem?mr=#1}{#2}
}
\providecommand{\href}[2]{#2}
\begin{thebibliography}{100}

\bibitem{MR2240597}
Ji\v{r}\'{\i} Ad\'{a}mek, Horst Herrlich, and George~E. Strecker, \emph{Abstract and concrete categories: the joy of cats}, Repr. Theory Appl. Categ. (2006), no.~17, 1--507, Reprint of the 1990 original [Wiley, New York; MR1051419]. \MR{2240597}

\bibitem{MR3107517}
Camilo Arias~Abad and Marius Crainic, \emph{Representations up to homotopy and {B}ott's spectral sequence for {L}ie groupoids}, Adv. Math. \textbf{248} (2013), 416--452. \MR{3107517}

\bibitem{MR2117631}
Paolo Aschieri, Luigi Cantini, and Branislav Jur\v{c}o, \emph{Nonabelian bundle gerbes, their differential geometry and gauge theory}, Comm. Math. Phys. \textbf{254} (2005), no.~2, 367--400. \MR{2117631}

\bibitem{baez2004higher}
John Baez and Urs Schreiber, \emph{Higher gauge theory: 2-connections on 2-bundles}, arXiv:hep-th/0412325 (2004).

\bibitem{MR1770708}
John~C. Baez, \emph{An introduction to spin foam models of {$BF$} theory and quantum gravity}, Geometry and quantum physics ({S}chladming, 1999), Lecture Notes in Phys., vol. 543, Springer, Berlin, 2000, pp.~25--93. \MR{1770708}

\bibitem{baez2002higher}
John~C Baez, \emph{Higher yang-mills theory}, arXiv preprint hep-th/0206130 (2002).

\bibitem{MR2978538}
John~C. Baez, Aristide Baratin, Laurent Freidel, and Derek~K. Wise, \emph{Infinite-dimensional representations of 2-groups}, Mem. Amer. Math. Soc. \textbf{219} (2012), no.~1032, vi+120. \MR{2978538}

\bibitem{MR2068522}
John~C. Baez and Alissa~S. Crans, \emph{Higher-dimensional algebra. {VI}. {L}ie 2-algebras}, Theory Appl. Categ. \textbf{12} (2004), 492--538. \MR{2068522}

\bibitem{MR1664990}
John~C. Baez and James Dolan, \emph{Categorification}, Higher category theory ({E}vanston, {IL}, 1997), Contemp. Math., vol. 230, Amer. Math. Soc., Providence, RI, 1998, pp.~1--36. \MR{1664990}

\bibitem{MR2817410}
John~C. Baez and Alexander~E. Hoffnung, \emph{Convenient categories of smooth spaces}, Trans. Amer. Math. Soc. \textbf{363} (2011), no.~11, 5789--5825. \MR{2817410}

\bibitem{MR2825807}
John~C. Baez and John Huerta, \emph{An invitation to higher gauge theory}, Gen. Relativity Gravitation \textbf{43} (2011), no.~9, 2335--2392. \MR{2825807}

\bibitem{MR2068521}
John~C. Baez and Aaron~D. Lauda, \emph{Higher-dimensional algebra. {V}. 2-groups}, Theory Appl. Categ. \textbf{12} (2004), 423--491. \MR{2068521}

\bibitem{MR2342821}
John~C. Baez and Urs Schreiber, \emph{Higher gauge theory}, Categories in algebra, geometry and mathematical physics, Contemp. Math., vol. 431, Amer. Math. Soc., Providence, RI, 2007, pp.~7--30. \MR{2342821}

\bibitem{MR4612595}
Michael Bailey and Marco Gualtieri, \emph{Integration of generalized complex structures}, J. Math. Phys. \textbf{64} (2023), no.~7, Paper No. 073503, 24. \MR{4612595}

\bibitem{bakovic2008bigroupoid}
Igor Bakovi{\'c}, \emph{Bigroupoid 2-torsors}, Ph.D. thesis, Ludwig Maximilian University of Munich, 2008.

\bibitem{bakovic2009simplicial}
Igor Bakovic, \emph{The simplicial interpretation of bigroupoid 2-torsors}, arXiv:0902.3436 (2009).

\bibitem{MR2709030}
Tobias~Keith Bartels, \emph{Higher gauge theory: 2-bundles}, ProQuest LLC, Ann Arbor, MI, 2006, Thesis (Ph.D.)--University of California, Riverside. \MR{2709030}

\bibitem{MR2817778}
Kai Behrend and Ping Xu, \emph{Differentiable stacks and gerbes}, J. Symplectic Geom. \textbf{9} (2011), no.~3, 285--341. \MR{2817778}

\bibitem{MR2748598}
Indranil Biswas, \emph{The {A}tiyah bundle and connections on a principal bundle}, Proc. Indian Acad. Sci. Math. Sci. \textbf{120} (2010), no.~3, 299--316. \MR{2748598}

\bibitem{MR4592876}
Indranil Biswas, Saikat Chatterjee, Praphulla Koushik, and Frank Neumann, \emph{Atiyah sequences and connections on principal bundles over {L}ie groupoids and differentiable stacks}, J. Noncommut. Geom. \textbf{17} (2023), no.~2, 407--437. \MR{4592876}

\bibitem{biswas2023connections}
Indranil Biswas, Saikat Chatterjee, Praphulla Koushik, and Frank Neumann, \emph{Connections on lie groupoids and chern-weil theory}, 2023.

\bibitem{MR3150770}
Indranil Biswas and Frank Neumann, \emph{Atiyah sequences, connections and characteristic forms for principal bundles over groupoids and stacks}, C. R. Math. Acad. Sci. Paris \textbf{352} (2014), no.~1, 59--64. \MR{3150770}

\bibitem{MR1291599}
Francis Borceux, \emph{Handbook of categorical algebra. 1}, Encyclopedia of Mathematics and its Applications, vol.~50, Cambridge University Press, Cambridge, 1994, Basic category theory. \MR{1291599}

\bibitem{MR1313497}
\bysame, \emph{Handbook of categorical algebra. 2}, Encyclopedia of Mathematics and its Applications, vol.~51, Cambridge University Press, Cambridge, 1994, Categories and structures. \MR{1313497}

\bibitem{MR2183393}
Lawrence Breen and William Messing, \emph{Differential geometry of gerbes}, Adv. Math. \textbf{198} (2005), no.~2, 732--846. \MR{2183393}

\bibitem{MR0419643}
Ronald Brown and Christopher~B. Spencer, \emph{{$G$}-groupoids, crossed modules and the fundamental groupoid of a topological group}, Indag. Math. \textbf{38} (1976), no.~4, 296--302, Nederl. Akad. Wetensch. Proc. Ser. A {\bf 79}. \MR{419643}

\bibitem{MR3451921}
Henrique Bursztyn, Alejandro Cabrera, and Matias del Hoyo, \emph{Vector bundles over {L}ie groupoids and algebroids}, Adv. Math. \textbf{290} (2016), 163--207. \MR{3451921}

\bibitem{MR2565034}
Henrique Bursztyn, Alejandro Cabrera, and Cristi\'{a}n Ortiz, \emph{Linear and multiplicative 2-forms}, Lett. Math. Phys. \textbf{90} (2009), no.~1-3, 59--83. \MR{2565034}

\bibitem{MR2068969}
Henrique Bursztyn, Marius Crainic, Alan Weinstein, and Chenchang Zhu, \emph{Integration of twisted {D}irac brackets}, Duke Math. J. \textbf{123} (2004), no.~3, 549--607. \MR{2068969}

\bibitem{cattafi2021cartan}
Francesco Cattafi, \emph{Cartan geometries and multiplicative forms}, Differential Geometry and its Applications \textbf{75} (2021), 101722.

\bibitem{chatterjee2023parallel}
Saikat Chatterjee and Adittya Chaudhuri, \emph{Parallel transport on a lie 2-group bundle over a lie groupoid along haefliger paths}, arXiv preprint arXiv:2309.05355 (2023).

\bibitem{MR4403617}
Saikat Chatterjee, Adittya Chaudhuri, and Praphulla Koushik, \emph{Atiyah sequence and gauge transformations of a principal 2-bundle over a {L}ie groupoid}, J. Geom. Phys. \textbf{176} (2022), Paper No. 104509, 29. \MR{4403617}

\bibitem{MR3126940}
Saikat Chatterjee, Amitabha Lahiri, and Ambar~N. Sengupta, \emph{Path space connections and categorical geometry}, J. Geom. Phys. \textbf{75} (2014), 129--161. \MR{3126940}

\bibitem{MR3213404}
\bysame, \emph{Twisted actions of categorical groups}, Theory Appl. Categ. \textbf{29} (2014), No. 8, 215--255. \MR{3213404}

\bibitem{MR3504595}
\bysame, \emph{Construction of categorical bundles from local data}, Theory Appl. Categ. \textbf{31} (2016), Paper No. 14, 388--417. \MR{3504595}

\bibitem{MR3521476}
Brian Collier, Eugene Lerman, and Seth Wolbert, \emph{Parallel transport on principal bundles over stacks}, J. Geom. Phys. \textbf{107} (2016), 187--213. \MR{3521476}

\bibitem{MR2772614}
Hellen Colman, \emph{On the 1-homotopy type of {L}ie groupoids}, Appl. Categ. Structures \textbf{19} (2011), no.~1, 393--423. \MR{2772614}

\bibitem{MR0996653}
A.~Coste, P.~Dazord, and A.~Weinstein, \emph{Groupo\"{\i}des symplectiques}, Publications du {D}\'{e}partement de {M}ath\'{e}matiques. {N}ouvelle {S}\'{e}rie. {A}, {V}ol. 2, Publ. D\'{e}p. Math. Nouvelle S\'{e}r. A, vol. 87-2, Univ. Claude-Bernard, Lyon, 1987, pp.~i--ii, 1--62. \MR{996653}

\bibitem{MR1355904}
Louis Crane, \emph{Clock and category: is quantum gravity algebraic?}, J. Math. Phys. \textbf{36} (1995), no.~11, 6180--6193. \MR{1355904}

\bibitem{MR1295461}
Louis Crane and Igor~B. Frenkel, \emph{Four-dimensional topological quantum field theory, {H}opf categories, and the canonical bases}, vol.~35, 1994, Topology and physics, pp.~5136--5154. \MR{1295461}

\bibitem{MR3757473}
Matias del Hoyo and Rui~Loja Fernandes, \emph{Riemannian metrics on {L}ie groupoids}, J. Reine Angew. Math. \textbf{735} (2018), 143--173. \MR{3757473}

\bibitem{MR3968895}
Matias del Hoyo and Rui Loja~Fernandes, \emph{Riemannian metrics on differentiable stacks}, Math. Z. \textbf{292} (2019), no.~1-2, 103--132. \MR{3968895}

\bibitem{MR4126305}
Matias del Hoyo and Cristian Ortiz, \emph{Morita equivalences of vector bundles}, Int. Math. Res. Not. IMRN (2020), no.~14, 4395--4432. \MR{4126305}

\bibitem{del2008homotopy}
Matias~L del Hoyo, \emph{On the homotopy type of a cofibred category}, arXiv preprint arXiv:0810.3063 (2008).

\bibitem{MR0116360}
Charles Ehresmann, \emph{Cat\'{e}gories topologiques et cat\'{e}gories diff\'{e}rentiables}, Colloque {G}\'{e}om. {D}iff. {G}lobale ({B}ruxelles, 1958), Librairie Universitaire, Louvain, 1959, pp.~137--150. \MR{116360}

\bibitem{MR0197529}
\bysame, \emph{Cat\'{e}gories structur\'{e}es}, Ann. Sci. \'{E}cole Norm. Sup. (3) \textbf{80} (1963), 349--426. \MR{197529}

\bibitem{MR2331238}
Josep Elgueta, \emph{Representation theory of 2-groups on {K}apranov and {V}oevodsky's 2-vector spaces}, Adv. Math. \textbf{213} (2007), no.~1, 53--92. \MR{2331238}

\bibitem{fiorenza2011cech}
Domenico Fiorenza, Urs Schreiber, and Jim Stasheff, \emph{Cech cocycles for differential characteristic classes: an infinity-{L}ie theoretic construction}, Adv. Theor. Math. Phys \textbf{16} (2012), no.~1, 149--250.

\bibitem{MR4598921}
Alfonso Garmendia and Sylvie Paycha, \emph{Principal bundle groupoids, their gauge group and their nerve}, J. Geom. Phys. \textbf{191} (2023), Paper No. 104865, 21. \MR{4598921}

\bibitem{MR3480061}
Gr\'{e}gory Ginot and Mathieu Sti\'{e}non, \emph{{$G$}-gerbes, principal 2-group bundles and characteristic classes}, J. Symplectic Geom. \textbf{13} (2015), no.~4, 1001--1047. \MR{3480061}

\bibitem{MR3696590}
Alfonso Gracia-Saz and Rajan~Amit Mehta, \emph{Vb-groupoids and representation theory of {L}ie groupoids}, J. Symplectic Geom. \textbf{15} (2017), no.~3, 741--783. \MR{3696590}

\bibitem{grothendieck2004revetements}
Alexander Grothendieck and Michele Raynaud, \emph{Rev\^etements \'etales et groupe fondamental (sga 1)}, 2004.

\bibitem{MR2218759}
K.~Guruprasad and A.~Haefliger, \emph{Closed geodesics on orbifolds}, Topology \textbf{45} (2006), no.~3, 611--641. \MR{2218759}

\bibitem{MR2166083}
Simone Gutt, John Rawnsley, and Daniel Sternheimer (eds.), \emph{Poisson geometry, deformation quantisation and group representations}, London Mathematical Society Lecture Note Series, vol. 323, Cambridge University Press, Cambridge, 2005, Lectures from the EuroSchool (PQR2003) held at the Universit\'{e} Libre de Bruxelles, Brussels, June 13--17, 2003. \MR{2166083}

\bibitem{MR0285027}
Andr\'{e} Haefliger, \emph{Homotopy and integrability}, Manifolds--{A}msterdam 1970 ({P}roc. {N}uffic {S}ummer {S}chool), Lecture Notes in Math., vol. Vol. 197, Springer, Berlin-New York, 1971, pp.~133--163. \MR{285027}

\bibitem{MR1086659}
\bysame, \emph{Orbi-espaces}, Sur les groupes hyperboliques d'apr\`es {M}ikhael {G}romov ({B}ern, 1988), Progr. Math., vol.~83, Birkh\"{a}user Boston, Boston, MA, 1990, pp.~203--213. \MR{1086659}

\bibitem{MR3837560}
Mark J.~D. Hamilton, \emph{Mathematical gauge theory}, Universitext, Springer, Cham, 2017, With applications to the standard model of particle physics. \MR{3837560}

\bibitem{MR2417440}
Eli Hawkins, \emph{A groupoid approach to quantization}, J. Symplectic Geom. \textbf{6} (2008), no.~1, 61--125. \MR{2417440}

\bibitem{MR3556124}
Benjam\'{\i}n~A. Heredia and Josep Elgueta, \emph{On the representations of 2-groups in {B}aez-{C}rans 2-vector spaces}, Theory Appl. Categ. \textbf{31} (2016), Paper No. 32, 907--927. \MR{3556124}

\bibitem{herreracarmona2023chernweillecomte}
Juan~Sebastian Herrera-Carmona and Cristian Ortiz, \emph{The chern-weil-lecomte characteristic map for $l_{\infty}$-algebras}, 2023.

\bibitem{MR4621032}
Juan~Sebasti\'{a}n Herrera-Carmona and Fabricio Valencia, \emph{Isometric {L}ie 2-group actions on {R}iemannian groupoids}, J. Geom. Anal. \textbf{33} (2023), no.~10, Paper No. 323, 36. \MR{4621032}

\bibitem{MR1022521}
Gary~T. Horowitz, \emph{Exactly soluble diffeomorphism invariant theories}, Comm. Math. Phys. \textbf{125} (1989), no.~3, 417--437. \MR{1022521}

\bibitem{huan20222representations}
Zhen Huan, \emph{2-representations of lie 2-groups and 2-vector bundles}, 2022.

\bibitem{MR3025051}
Patrick Iglesias-Zemmour, \emph{Diffeology}, Mathematical Surveys and Monographs, vol. 185, American Mathematical Society, Providence, RI, 2013. \MR{3025051}

\bibitem{MR4261588}
Niles Johnson and Donald Yau, \emph{2-dimensional categories}, Oxford University Press, Oxford, 2021. \MR{4261588}

\bibitem{MR3351282}
Branislav Jur\v{c}o, Christian S\"{a}mann, and Martin Wolf, \emph{Semistrict higher gauge theory}, J. High Energy Phys. (2015), no.~4, 087, front matter+66. \MR{3351282}

\bibitem{MR3548195}
\bysame, \emph{Higher groupoid bundles, higher spaces, and self-dual tensor field equations}, Fortschr. Phys. \textbf{64} (2016), no.~8-9, 674--717. \MR{3548195}

\bibitem{MR1278735}
M.~M. Kapranov and V.~A. Voevodsky, \emph{{$2$}-categories and {Z}amolodchikov tetrahedra equations}, Algebraic groups and their generalizations: quantum and infinite-dimensional methods ({U}niversity {P}ark, {PA}, 1991), Proc. Sympos. Pure Math., vol. 56, Part 2, Amer. Math. Soc., Providence, RI, 1994, pp.~177--259. \MR{1278735}

\bibitem{MR0854594}
M.~V. Karas\"{e}v, \emph{Analogues of objects of the theory of {L}ie groups for nonlinear {P}oisson brackets}, Izv. Akad. Nauk SSSR Ser. Mat. \textbf{50} (1986), no.~3, 508--538, 638. \MR{854594}

\bibitem{MR4177087}
Hyungrok Kim and Christian Saemann, \emph{Adjusted parallel transport for higher gauge theories}, J. Phys. A \textbf{53} (2020), no.~44, 445206, 52. \MR{4177087}

\bibitem{kobayashi1956connections}
Sh{\^o}shichi Kobayashi, \emph{On connections of cartan}, Canadian Journal of Mathematics \textbf{8} (1956), 145--156.

\bibitem{MR1393940}
Shoshichi Kobayashi and Katsumi Nomizu, \emph{Foundations of differential geometry. {V}ol. {I}}, Wiley Classics Library, John Wiley \& Sons, Inc., New York, 1996, Reprint of the 1963 original, A Wiley-Interscience Publication. \MR{1393940}

\bibitem{koushik2021geometric}
Praphulla Koushik, \emph{Geometric structures on lie groupoids and differentiable stacks}, arXiv preprint arXiv:2112.13472 (2021).

\bibitem{MR2493616}
Camille Laurent-Gengoux, Mathieu Sti\'{e}non, and Ping Xu, \emph{Non-abelian differentiable gerbes}, Adv. Math. \textbf{220} (2009), no.~5, 1357--1427. \MR{2493616}

\bibitem{MR2270285}
Camille Laurent-Gengoux, Jean-Louis Tu, and Ping Xu, \emph{Chern-{W}eil map for principal bundles over groupoids}, Math. Z. \textbf{255} (2007), no.~3, 451--491. \MR{2270285}

\bibitem{leinster1998basic}
Tom Leinster, \emph{Basic bicategories}, arXiv preprint math/9810017 (1998).

\bibitem{MR2778793}
Eugene Lerman, \emph{Orbifolds as stacks?}, Enseign. Math. (2) \textbf{56} (2010), no.~3-4, 315--363. \MR{2778793}

\bibitem{MR2806566}
David Li-Bland and Pavol \v{S}evera, \emph{Quasi-{H}amiltonian groupoids and multiplicative {M}anin pairs}, Int. Math. Res. Not. IMRN (2011), no.~10, 2295--2350. \MR{2806566}

\bibitem{MR1712872}
Saunders Mac~Lane, \emph{Categories for the working mathematician}, second ed., Graduate Texts in Mathematics, vol.~5, Springer-Verlag, New York, 1998. \MR{1712872}

\bibitem{MR1932333}
Marco Mackaay and Roger Picken, \emph{Holonomy and parallel transport for abelian gerbes}, Adv. Math. \textbf{170} (2002), no.~2, 287--339. \MR{1932333}

\bibitem{MR1001474}
Kirill Mackenzie, \emph{Classification of principal bundles and {L}ie groupoids with prescribed gauge group bundle}, J. Pure Appl. Algebra \textbf{58} (1989), no.~2, 181--208. \MR{1001474}

\bibitem{MR896907}
Kirill C.~H. Mackenzie, \emph{Lie groupoids and {L}ie algebroids in differential geometry}, London Mathematical Society Lecture Note Series, vol. 124, Cambridge University Press, Cambridge, 1987. \MR{896907}

\bibitem{MR2103015}
\bysame, \emph{Duality and triple structures}, The breadth of symplectic and {P}oisson geometry, Progr. Math., vol. 232, Birkh\"{a}user Boston, Boston, MA, 2005, pp.~455--481. \MR{2103015}

\bibitem{MR2157566}
\bysame, \emph{General theory of {L}ie groupoids and {L}ie algebroids}, London Mathematical Society Lecture Note Series, vol. 213, Cambridge University Press, Cambridge, 2005. \MR{2157566}

\bibitem{MR2831518}
\bysame, \emph{Ehresmann doubles and {D}rinfel'd doubles for {L}ie algebroids and {L}ie bialgebroids}, J. Reine Angew. Math. \textbf{658} (2011), 193--245. \MR{2831518}

\bibitem{MR3744376}
Kirill C.~H. Mackenzie, Anatol Odzijewicz, and Aneta Sli\.{z}ewska, \emph{Poisson geometry related to {A}tiyah sequences}, SIGMA Symmetry Integrability Geom. Methods Appl. \textbf{14} (2018), Paper No. 005, 29. \MR{3744376}

\bibitem{MR2661492}
Jo\~{a}o~Faria Martins and Roger Picken, \emph{On two-dimensional holonomy}, Trans. Amer. Math. Soc. \textbf{362} (2010), no.~11, 5657--5695. \MR{2661492}

\bibitem{MR2764890}
\bysame, \emph{Surface holonomy for non-abelian 2-bundles via double groupoids}, Adv. Math. \textbf{226} (2011), no.~4, 3309--3366. \MR{2764890}

\bibitem{martins2011lie}
Jo{\~a}o Martins, Aleksandar Mikovi{\'c}, et~al., \emph{Lie crossed modules and gauge-invariant actions for 2-bf theories}, Advances in Theoretical and Mathematical Physics \textbf{15} (2011), no.~4, 1059--1084.

\bibitem{martins2011fundamental}
Joao~Faria Martins and Roger Picken, \emph{The fundamental gray 3-groupoid of a smooth manifold and local 3-dimensional holonomy based on a 2-crossed module}, Differential Geometry and its Applications \textbf{29} (2011), no.~2, 179--206.

\bibitem{MR2861783}
Rajan~Amit Mehta and Xiang Tang, \emph{From double {L}ie groupoids to local {L}ie 2-groupoids}, Bull. Braz. Math. Soc. (N.S.) \textbf{42} (2011), no.~4, 651--681. \MR{2861783}

\bibitem{MR2166453}
I.~Moerdijk and J.~Mr\v{c}un, \emph{Lie groupoids, sheaves and cohomology}, Poisson geometry, deformation quantisation and group representations, London Math. Soc. Lecture Note Ser., vol. 323, Cambridge Univ. Press, Cambridge, 2005, pp.~145--272.

\bibitem{Moerdijk2}
Ieke Moerdijk, \emph{Introduction to the language of stacks and gerbes}, arXiv:math/0212266 (2002).

\bibitem{MR1950948}
\bysame, \emph{Orbifolds as groupoids: an introduction}, Orbifolds in mathematics and physics ({M}adison, {WI}, 2001), Contemp. Math., vol. 310, Amer. Math. Soc., Providence, RI, 2002, pp.~205--222. \MR{1950948}

\bibitem{MR2012261}
Ieke Moerdijk and J.~Mr\v{c}un, \emph{Introduction to foliations and {L}ie groupoids}, Cambridge Studies in Advanced Mathematics, vol.~91, Cambridge University Press, Cambridge, 2003. \MR{2012261}

\bibitem{MR1466622}
Ieke. Moerdijk and D.~A. Pronk, \emph{Orbifolds, sheaves and groupoids}, $K$-Theory \textbf{12} (1997), no.~1, 3--21. \MR{1466622}

\bibitem{MR4037666}
Jeffrey~C. Morton and Roger Picken, \emph{2-group actions and moduli spaces of higher gauge theory}, J. Geom. Phys. \textbf{148} (2020), 103548, 21. \MR{4037666}

\bibitem{MR4529816}
Ji\v{r}\'{\i} N\'{a}ro\v{z}n\'{y}, \emph{Generalised {A}tiyah's theory of principal connections}, Arch. Math. (Brno) \textbf{58} (2022), no.~4, 241--256. \MR{4529816}

\bibitem{MR3423073}
Thomas Nikolaus, Urs Schreiber, and Danny Stevenson, \emph{Principal {$\infty$}-bundles: general theory}, J. Homotopy Relat. Struct. \textbf{10} (2015), no.~4, 749--801. \MR{3423073}

\bibitem{MR3385700}
\bysame, \emph{Principal {$\infty$}-bundles: presentations}, J. Homotopy Relat. Struct. \textbf{10} (2015), no.~3, 565--622. \MR{3385700}

\bibitem{MR3089401}
Thomas Nikolaus and Konrad Waldorf, \emph{Four equivalent versions of nonabelian gerbes}, Pacific J. Math. \textbf{264} (2013), no.~2, 355--419. \MR{3089401}

\bibitem{nlab:2-vector_space}
{nLab authors}, \emph{2-vector space}, \url{https://ncatlab.org/nlab/show/2-vector+space}, October 2023, \href{https://ncatlab.org/nlab/revision/2-vector+space/60}{Revision 60}.

\bibitem{MR0941624}
Jean Pradines, \emph{Remarque sur le groupo\"{\i}de cotangent de {W}einstein-{D}azord}, C. R. Acad. Sci. Paris S\'{e}r. I Math. \textbf{306} (1988), no.~13, 557--560. \MR{941624}

\bibitem{MR2800361}
Christopher~J. Schommer-Pries, \emph{Central extensions of smooth 2-groups and a finite-dimensional string 2-group}, Geom. Topol. \textbf{15} (2011), no.~2, 609--676. \MR{2800361}

\bibitem{schreiber2005loop}
Urs Schreiber, \emph{From loop space mechanics to nonabelian strings}, arXiv:hep-th/0509163 (2005).

\bibitem{schreiber2013differential}
Urs Schreiber, \emph{Differential cohomology in a cohesive infinity-topos}, arXiv:1310.7930 (2013).

\bibitem{MR2520993}
Urs Schreiber and Konrad Waldorf, \emph{Parallel transport and functors}, J. Homotopy Relat. Struct. \textbf{4} (2009), no.~1, 187--244. \MR{2520993}

\bibitem{MR2803871}
\bysame, \emph{Smooth functors vs. differential forms}, Homology Homotopy Appl. \textbf{13} (2011), no.~1, 143--203. \MR{2803871}

\bibitem{MR3084724}
\bysame, \emph{Connections on non-abelian gerbes and their holonomy}, Theory Appl. Categ. \textbf{28} (2013), 476--540. \MR{3084724}

\bibitem{slovak2009parabolic}
Jan Slov{\'a}k and Andreas Cap, \emph{Parabolic geometries i, background and general theory}, vol. 1000, American Mathematical Society, 2009.

\bibitem{MR0316529}
Alexandru Solian, \emph{Groupe dans une cat\'{e}gorie}, C. R. Acad. Sci. Paris S\'{e}r. A-B \textbf{275} (1972), A155--A158. \MR{316529}

\bibitem{MR3357822}
Emanuele Soncini and Roberto Zucchini, \emph{A new formulation of higher parallel transport in higher gauge theory}, J. Geom. Phys. \textbf{95} (2015), 28--73. \MR{3357822}

\bibitem{MR2119241}
Jean-Louis Tu, Ping Xu, and Camille Laurent-Gengoux, \emph{Twisted {$K$}-theory of differentiable stacks}, Ann. Sci. \'{E}cole Norm. Sup. (4) \textbf{37} (2004), no.~6, 841--910. \MR{2119241}

\bibitem{MR3585539}
Loring~W. Tu, \emph{Differential geometry}, Graduate Texts in Mathematics, vol. 275, Springer, Cham, 2017, Connections, curvature, and characteristic classes. \MR{3585539}

\bibitem{MR4535273}
Nesta van~der Schaaf, \emph{Diffeological {M}orita equivalence}, Cah. Topol. G\'{e}om. Diff\'{e}r. Cat\'{e}g. \textbf{62} (2021), no.~2, 177--238. \MR{4535273}

\bibitem{MR3566125}
David Viennot, \emph{Non-abelian higher gauge theory and categorical bundle}, J. Geom. Phys. \textbf{110} (2016), 407--435. \MR{3566125}

\bibitem{MR2223406}
Angelo Vistoli, \emph{Grothendieck topologies, fibered categories and descent theory}, Fundamental algebraic geometry, Math. Surveys Monogr., vol. 123, Amer. Math. Soc., Providence, RI, 2005, pp.~1--104. \MR{2223406}

\bibitem{MR3894086}
Konrad Waldorf, \emph{A global perspective to connections on principal 2-bundles}, Forum Math. \textbf{30} (2018), no.~4, 809--843. \MR{3894086}

\bibitem{MR3917427}
\bysame, \emph{Parallel transport in principal 2-bundles}, High. Struct. \textbf{2} (2018), no.~1, 57--115. \MR{3917427}

\bibitem{MR3645839}
Wei Wang, \emph{On the global 2-holonomy for a 2-connection on a 2-bundle}, J. Geom. Phys. \textbf{117} (2017), 151--178. \MR{3645839}

\bibitem{MR0866024}
Alan Weinstein, \emph{Symplectic groupoids and {P}oisson manifolds}, Bull. Amer. Math. Soc. (N.S.) \textbf{16} (1987), no.~1, 101--104. \MR{866024}

\bibitem{MR1103911}
Alan Weinstein and Ping Xu, \emph{Extensions of symplectic groupoids and quantization}, J. Reine Angew. Math. \textbf{417} (1991), 159--189. \MR{1103911}

\bibitem{MR2805195}
Christoph Wockel, \emph{Principal 2-bundles and their gauge 2-groups}, Forum Math. \textbf{23} (2011), no.~3, 565--610. \MR{2805195}

\bibitem{MR3529236}
Roberto Zucchini, \emph{On higher holonomy invariants in higher gauge theory {I}}, Int. J. Geom. Methods Mod. Phys. \textbf{13} (2016), no.~7, 1650090, 59. \MR{3529236}

\end{thebibliography}

\newpage
\thispagestyle{plain}\chapter*{Publications arising out of the PhD
	thesis\hfill} \addcontentsline{toc}{chapter}{Publications arising
	out of the PhD thesis}
\textbf{Published papers}
		\begin{itemize}
			\item[(1)] Saikat Chatterjee, Adittya Chaudhuri, and Praphulla Koushik, \textit{Atiyah sequence and gauge transformations of a principal 2-bundle over a Lie groupoid.} J. Geom. Phys., 176:Paper No. 104509, 29,
			2022 (\cite{MR4403617}).
			\end{itemize}
			\textbf{Preprints}
			\begin{itemize}
				\item[(1)] Saikat Chatterjee and Adittya Chaudhuri, \textit{Parallel transport on a Lie 2-group bundle over a Lie
				groupoid along Haefliger paths}, 2023, 	arXiv:2309.05355  (\cite{chatterjee2023parallel}).
				\end{itemize}

\end{document}